\documentclass[10pt]{article}
\usepackage{amssymb}

\usepackage[a4paper,left=30mm,top=40mm,body={155mm,230mm}]{geometry}
\usepackage{stmaryrd}
\usepackage{bbding,bbm}
\usepackage{enumerate}

\usepackage{fancybox}
\usepackage{algorithm}
\usepackage{algpseudocode}

\usepackage{fancyhdr,color,graphicx,amsmath}

\usepackage{multirow}

\usepackage{enumerate}

\usepackage{graphicx}
\usepackage{color}
\usepackage{ifpdf}

\usepackage{amsmath}

\usepackage{amsthm}
\allowdisplaybreaks[4]
\usepackage{mathrsfs}

\usepackage[figuresright]{rotating}

\usepackage{xcolor}
\usepackage{authblk}

\newtheorem{theorem}{\qquad \bf Theorem}[section]

\newtheorem{lemma}[theorem]{\qquad\bf Lemma}
\newtheorem{corollary}[theorem]{\qquad \bf Corollary}

\newtheorem{assumption}[theorem]{\qquad \bf Assumption}
\newtheorem{definition}[theorem]{\qquad \bf Definition}

\newcommand{\dist}{\mbox{dist}}

\newcommand{\be}{\begin{equation}}
\newcommand{\ee}{\end{equation}}
\newcommand{\ba}{\begin{array}}
\newcommand{\ea}{\end{array}}
\newcommand{\bea}{\begin{eqnarray}}
\newcommand{\eea}{\end{eqnarray}}

\begin{document}

\title{A Sequential Quadratic Programming Method for Optimization with {Expectation} Objective Functions, Deterministic Inequality Constraints and Robust Subproblems}

\author[a]{Songqiang Qiu}

\affil[a]{School of mathematics, China university of mining and technology, Xuzhou, P. R. China\\ email: {sqqiu@cumt.edu.cn}}

\author[b]{Vyacheslav Kungurtsev}
\affil[b]{Faculty of electrical engineering,Czech technical university in Prague,Prague,Czech Republic}

\date{}


%
\maketitle


\abstract{In this paper, a robust sequential quadratic programming method for constrained optimization is generalized to problem with an {expectation}
objective function {and} deterministic equality and inequality constraints. A stochastic line search scheme is employed to globalize the steps. {We show theoretically that sequences generated by the algorithm converge almost surely to a Karush-Kuhn-Tucker point under the assumption of the extended Mangasarian-Fromovitz constraint qualification}.
Encouraging numerical results are reported.}



\section{Introduction}
We are interested in the following constrained  {optimization problem with an expectation of a stochastic function as an objective}:
\begin{equation}\label{eq:prob}
\begin{array}{rl}
\min\limits_{x\in\mathbb{R}^n} & f(x)=\textbf{E}[F(x,\xi)] \\
\text{s.t. } & h(x) =0,\\
& c(x) \le 0,
\end{array}
\end{equation}
where $h: \mathbb{R}^n\rightarrow \mathbb{R}^{m_1}$  and $c: \mathbb{R}^n\rightarrow \mathbb{R}^{m_2}$ are continuously differentiable, $f:\mathbb{R}^n\rightarrow \mathbb{R}$,  $\xi$ is a random variable with associated probability space $(\mathcal{D},\mathcal{F}, P)$, $F:\mathbb{R}^n\times\mathcal{D}\rightarrow \mathbb{R}$ and $\textbf{E}[\cdot]$ represents expectation taken with respect to $P$. We assume that the values and the gradients of the constraint functions are accessible, but the objective function $f(x)$ and its first-order derivative cannot be evaluated, or too expensive to compute.  Instead, by taking individual instances or sample averages, $f(x)$ and $\nabla f(x)$ can be approximated by sampling $\xi$. In this paper, we do not assume that the estimates of $f(x)$ and $\nabla f(x)$ are unbiased.
Problems of this type arise prominently in important applications such as machine learning \cite{Chen2018a,Roy2018,Nandwani2019,Ravi2019}, statistics {\cite{Kirkegaard1972,Kaufman1978,Nagaraj1991,Aitchison1958,Silvey1959,Sen1979,Dupacova1988,Shapiro2000} and other domains\cite{ConnGT2000}.}

As for this writing, there are just a few algorithms that have been proposed for {solving problems} of this form or similar. {By and large, algorithms of the Stochastic Approximation (SA) type to minimize expectation objective functions have considered unconstrained problems, or constraints on which a projection operator is straightforward to compute. The case of nonlinear functional constraints has received relatively scant attention, especially until recently. There is a classic work in the spirit of Stochastic Approximation analysis presented in~\cite{wang2008stochastic}. While notable for the early contribution to a challenging problem, the SA approach uses a priori diminishing stepsizes that are known to be difficult to tune and typically exhibit slow convergence. Furthermore there is a highly specific set of Assumptions required for the convergence guarantees.} In the optimization community, To the best of our knowledge, the first algorithm for nonlinear constrained stochastic optimization was proposed in \cite{Berahas2021}. It is {also an} SQP method, {however only considers} problems with equality constraints. {The convergence theory of this method requires} that the singular values of the constraints' gradients are bounded away from zero, {which one can see implies the classical} linear independence constraint qualification (LICQ) for all feasible points. Global convergence in the case where the
penalty factor remains constant in the limit is proven. The authors also discussed, {informally, situations in which algorithm sequences generate an unbounded penalty parameter}. Later Berahas et al presented an improved algorithm in \cite{berahas2021stochastic}, {also for strictly equality constrained problems}. In this algorithm, the SQP method is {performed} with a step decomposition strategy. {In doing so}, the presented algorithm {maintains global convergence guarantees even} with constraints {whose Jacobian may be rank-deficient at limit points of Algorithm sequences}. A worst-case complexity analysis for this algorithm is presented in Curtis et al \cite{Curtis2021a}.
Na et al, {continuing the trend on strengthening the guarantees originally presented in} \cite{Berahas2021}, proposed another stochastic SQP method for equality constrained optimization \cite{na2021adaptive}. This algorithm uses a differentiable exact penalty function {to enforce global converence} and {ultimately is shown to satisfy stronger convergence guarantees, although also under the LICQ condition}. This work generalized the {techniques and associated theory} in \cite{paquette2020stochastic} {studying a stochastic line search to the problem with nonlinear constraints, a non-trivial exercise due to multiple technical challenges.}
A recent work of Curtis et al extends the methods in \cite{Berahas2021,berahas2021stochastic} to allow for
inexact subproblem solutions. The inexact solution strategy make the approach applicable for \eqref{eq:prob} ({albeit with only equality constraints}) in a large scale setting.
A very recent work of Berahas et al \cite{Berahas2023a} considered the use of variance reduced gradients in a stochastic SQP method {for} accelerating the method.

All the mentioned stochastic SQP methods above deal with equality constraints only. To our knowledge, only recently,
Na et al \cite{Na2021} considered SQP for constrained stochastic optimization with inequality constraints. This work is actually an extension of the previously mentioned work of the authors \cite{na2021adaptive}. Their algorithm uses an active set strategy to handle the inequality constraints, {a technique that enables relatively straightforward extensions of the algorithm and its analysis for the problem with only equalities}. {The presented method shows satisfactory convergence properties, albeit with some tailored assumptions. In addition, after the preprint release of the first version of our paper was uploaded, another work studying SQP in the context of deterministic inequality constraints appeared~\cite{curtis2023sequential}, which similarly extends techniques from the equality constrained case. Our contribution goes beyond these two works with a robust subproblem strategy that enables convergence guarantees without stringent or a posteriori (iterate-dependent) assumptions.}

Our method is similar to the methods mentioned above in incorporating sampling techniques with classical SQP tools. As one of its primary goals, our method makes no assumption on the solvability of certain subproblems, and instead uses classical techniques in robust SQP to enforce feasible subproblems~\cite{burke1989robust}. For globalizing the steps, we generalize the stochastic line search method in Paquette and Scheinberg \cite{paquette2020stochastic}. This generalization is non-trivial since:
\begin{enumerate}[(a)]
    \item Due to the constraints, a merit function with varying penalty factors instead of a fixed objective function is used in the line search.
    \item Due to the randomness of the estimates, boundedness of the penalty factors is not guaranteed even in the presence of the Mangasarian-Fromowitz constraint qualification (MFCQ) or even the stronger Linear Independence Constraint Qualification (LICQ).
    \item Due to the presence of inequality constraints, an analytical expression formulated to compute the search direction is generally not available.
\end{enumerate}
As a result, a naive and simple integration of the analysis in \cite{paquette2020stochastic} and the stochastic SQP schemes mentioned above cannot not sufficient.

{This paper presents a few key novel strategies that enable an otherwise natural and classic algorithm to be associated with strong convergence guarantees.} We use the scheme in \cite{burke1989robust} to remedy the standard issue of inconsistent subproblems for SQP methods, which will significantly simplify the  theoretical analysis of the algorithm's behavior compared to alternatives. We introduce a Lipschitz continuous measure for approximate stationarity in order to extend the general framework proposed in \cite{blanchet2019} to consider constraints. We also design a specific batch sampling technique for estimating the gradient of the objective function to ensure that the probability of the penalty factor growing unbounded is zero.

The remainder of the paper is structured as follows. In Section~\ref{sec:dset}, an SQP method for solving deterministic constrained optimization, which is a special case of the method in \cite{Burke1986}, is introduced, as a warmup for assisting the reader in obtaining familiarity with the tools used in this paper. The stochastic SQP method is formulated and analyzed in Section \ref{sect:algs}, where we give a detailed description of the algorithm, show the global convergence guarantees and discuss limiting behavior of the penalty parameter. Section \ref{sect:num} reports some numerical experiments and Section \ref{sect:conc} concludes the paper.

\emph{Notation}

Throughout the paper, the subscript $k$ will denote an iteration counter and for any function $g(x)$, $g_k$ will be the
shorthand for $g(x_k)$.
The feasible region of the problem is denoted as
$$
\mathcal{X}=\{x\in \mathbb{R}^n:\ h(x)=0,\ c(x)\leq 0 \}.
$$
 If $C\subset \mathbb{R}^{m_1+m_2}$, then the $\ell_{\infty}$ distance function is defined as
\[
\text{dist}(w \mid C ):=\inf\{\|w-y\|_{\infty} : y\in C\}.
\]
We also use $K=\{0\}_{\mathbb{R}^{m_1}}\times \mathbb{R}_-^{m_2}$, where $\{0\}_{\mathbb{R}^{m_1}}$ is the vector of all 0 in $\mathbb{R}^{m_1}$ and $\mathbb{R}_-^{m_2}=\{w\in \mathbb{R}^{m_2}: w\leq 0\}$.

\section{A Robust Deterministic SQP Algorithm}\label{sec:dset}
As a first step for considering \eqref{eq:prob}, we begin by reviewing Burke and Han's line search SQP method \cite{burke1989robust} for the deterministic
counterpart of \eqref{eq:prob}, that is,
\begin{equation}\label{eq:prob_determ}
\begin{array}{rl}
\min\limits_x & f(x) \\
\text{s.t. } & h(x) =0,\\
& c(x) \le 0,
\end{array}
\end{equation}
where $f:\mathbb{R}^n\rightarrow \mathbb{R}$ is continuously differentiable.  In \cite{burke1989robust}, a measure quantifying infeasibility
\begin{equation}\label{eq:defphi}
\phi(x)=\text{dist}\left[\begin{pmatrix}h(x)\\ c(x)\end{pmatrix}\mid K\right]
\end{equation}
and an exact penalty function
$$
\Psi(x;\rho)=f(x)+\rho\ \text{dist}\left[\begin{pmatrix}h(x)\\ c(x)\end{pmatrix}\mid K\right]=f(x)+\rho\phi(x),
$$
where $\rho>0$ is a penalty parameter, are defined to monitor the behavior of the algorithm.

The following stationary conditions are standard in the study of constrained optimization.
\begin{definition}
    \begin{enumerate}[(1)]
        \item A point $x$ satisfies the Fritz-John point conditions for \eqref{eq:prob_determ}, if there exists a scalar $\gamma\geq 0$, a vector $\lambda\in \mathbb{R}^{m_1}$ and a vector $\mu\in \mathbb{R}^{m_2}$ such that
        \begin{eqnarray*}
            && \gamma\nabla f(x)+\nabla h(x)\lambda +\nabla c(x)\mu=0,\\
            && h(x)=0,\\
            && \mu\geq 0, c(x)\leq 0, \mu^Tc(x)=0.
        \end{eqnarray*}
        A point satisfying the Fritz-John conditions is called a Fritz-John point.
    \item A point $x$ satisfies the Karush–Kuhn–Tucker (KKT) conditions for \eqref{eq:prob_determ}, if there exists a vector $\lambda\in \mathbb{R}^{m_1}$ and a vector $\mu\in \mathbb{R}^{m_2}$ such that
        \begin{eqnarray*}
            && \nabla f(x)+\nabla h(x)\lambda +\nabla c(x)\mu=0,\\
            && h(x)=0,\\
            && \mu\geq 0, c(x)\leq 0, \mu^Tc(x)=0
        \end{eqnarray*}
        A point satisfying the KKT conditions is called a KKT point.
    \end{enumerate}
\end{definition}

Conventional SQP methods are an iterative process based on repeatedly solving a quadratic subproblem
\begin{eqnarray}
        \min\limits_d &\ \nabla f(x)^Td+\frac12d^THd\nonumber\\
        \text{s.t.}&\ h(x)+\nabla h(x)^Td=0,\label{eq:qpcon1}\\
        &\ c(x)+\nabla c(x)^Td\leq 0\label{eq:qpcon2}
\end{eqnarray}
for a step $d$ that is subsequently ascertained as to its quality and adjusted with a line search or a trust region scheme. A well-known limitation of many SQP methods is that the subproblem's constraints \eqref{eq:qpcon1} and \eqref{eq:qpcon2} must be consistent, which is not guaranteed even if $\mathcal{X}$ is nonempty. Several approaches have been developed to overcome this shortcoming. In \cite{Powel1978}, an extra variable is introduced into the subproblem to ensure its consistency. Burke and Han \cite{burke1989robust} proposed a generic SQP method with an additional perturbation to the form of the subproblem constraints. Burke and Han's method has some similarity to the works of Omojokun \cite{Omojo1989}, where a composite step technique was introduced. In FilterSQP \cite{FletcL2002,FletcLT2002}, an extra phase called ``the feasibility restoration phase'' is used to generate iterates closer to the feasible region when the QP subproblem becomes inconsistent.

We herein use the framework of Burke and Han, adopting the $\ell_{\infty}$ norm to measure the magnitudes and the distances. Although, this method has been well studied in \cite{burke1989robust}, we would like to sketch the description of the algorithm and state its convergence guarantees to make this paper self-contained for the reader.

For each iterate $k$, given a current primal estimate $x_k\in \mathbb{R}^n$, we first solve
\begin{equation}\label{eq:dist0}
\begin{array}{rl}
\min\limits_{p\in \mathbb{R}^n}&\text{dist}\left[\begin{pmatrix}h_k+\nabla h_k^Tp\\ c_k+\nabla c_k^Tp\end{pmatrix}\mid K\right]\\
\text{s.t.}& \|p\|_{\infty}\leq\sigma_k:=\min\{\sigma_u,\kappa_u\phi_k\},
\end{array}
\end{equation}
where $\sigma_u$ and $\kappa_u$ are preset positive parameters. Note that we choose a different definition for $\sigma_k$ from its counterpart in \cite{burke1989robust}. 
Requiring that  $\|p\|_\infty\leq \kappa_u\phi_k$ is an idea that
can also be seen in, e.g., Curtis et al \cite{CurtiNW2009} for their composite-step SQP method for equality constrained optimization. It was also used as one of the conditions/assumptions on the quasi/inexact normal steps, see, for instance, \cite{Dennis1997} in the context of equality constrained optimization and \cite{FletcGLTW2002} for general constrained optimization.

Since $\ell_\infty$ distance is used, the subproblem defined above \eqref{eq:dist0} is equivalent to the following linear program
\begin{equation}\label{eq:dist}
\begin{array}{rl}
\min\limits_{p\in\mathbb{R}^n,y\in\mathbb{R}} & y\\
\text{s.t.}& -ye\leq h_k+\nabla h_k^Tp\leq ye,\\
& c_k+\nabla c_k^Tp\leq ye,\\
&\|p\|_{\infty}\leq\sigma_k, y\geq 0,
\end{array}
\end{equation}
where $e$ is a vector of all ones with proper dimension.
Let $(p_k,y_k)$ be a minimizer of \eqref{eq:dist}.

The search direction $d_k$ is obtained by solving
 \begin{equation}\label{eq:subp0}
    \begin{array}{rl}
         \min\limits_d & \nabla f_k^T d+\frac{1}{2} d^T H_k d \\
         \text{s.t.} & -y_ke\le h_k+\nabla h_k^T d\le y_ke, \\
         & c_k+\nabla c_k^Td \le y_ke \\
         & \|d\|_{\infty} \le \beta_k,
    \end{array}
    \end{equation}
where $\beta_k\in [\beta_l, \beta_u]$ with \begin{equation}\label{eq:relatebetasig}
0<2\sigma_u<\beta_l< \beta_u
\end{equation}

The improvement in linearized feasibility is estimated by
\begin{equation}\label{eq:defdeltaphi}
\Delta_{\phi}(x_k,\sigma_k):=\phi_k-y_k
\end{equation} and the predicted reduction in $\Psi(x;\rho)$ associated with taking the step $d_k$ is defined as
\[
\Delta_{\Psi}(x_k, d_k;\rho)=-\nabla f_k^Td_k+\rho\Delta_{\phi}(x_k,\sigma_k).
\]
The penalty parameter is updated in a way such that $d_k$ exhibits sufficient reduction in $\Psi(x;\rho)$. If
\begin{equation}\label{eq:condpen}
\Delta_{\Psi}(x_k,d_k;\rho_k)\geq \frac12d_k^TH_kd_k,
\end{equation}
 then set $\rho_{k+1}=\rho_k$; otherwise, set
\begin{equation}\label{eq:uppen}
\rho_{k+1}=\max\left\{\frac{\nabla f_k^Td_k+ \frac12d_k^TH_kd_k}{\Delta_{\phi}(x_k,\sigma_k)},2\rho_k\right\}.
\end{equation}
By employing these rules for updating $\rho_k$ we can observe that the following condition holds for all $k$:
\begin{equation}\label{eq:condpen2}\Delta_{\Psi}(x_k, d_k;\rho_{k+1})\geq \frac12d_k^TH_kd_k.
\end{equation}
To globalize the algorithm, a line search is performed to determine a stepsize $\alpha_k$ satisfying
\begin{equation}\label{eq:linesearch}
\Psi(x_k;\rho_{k+1})-\Psi(x_k+\alpha_k d_k;\rho_{k+1})\geq\theta\alpha_k\Delta_{\Psi}(x_k,d_k;\rho_{k+1}),
\end{equation}
where $\theta\in(0,1)$. The algorithm is presented formally as Algorithm \ref{sqp}.

\begin{algorithm}[ht]
\caption{Deterministic SQP}
\label{sqp}
\begin{algorithmic}
\State{Parameters: $\beta_l$, $\beta_u$ and $\sigma_u$ satisfying \eqref{eq:relatebetasig}, $\kappa_u>0$, $\gamma>1$, $\theta\in(0,1)$ and $\rho_0>0$. Starting Point $x_0\in \mathbb{R}^n$.}
\For{$k=0,1,2,\cdots$}
\State Compute $(p_k,y_k)$ from \eqref{eq:dist}.
\If{$p_k=0$ and $\phi_k>0$}
\State Return an infeasible stationary point $x_k$.
\EndIf
\State Choose $\beta_k\in[\beta_l,\beta_u]$.
\State Solve \eqref{eq:subp0} for $d_k$.
\If{$d_k=0$ and $\phi_k=0$}
\State Return a first order stationary point $x_k$.
\EndIf
\If{\eqref{eq:condpen} holds}
\State {Set $\rho_{k+1}=\rho_k$}
\Else
\State {Set $\rho_{k+1}$  as \eqref{eq:uppen}}
\EndIf
\State Let $\alpha_k=1$. Set $dols=\texttt{true}$.
\While{dols}
\If{\eqref{eq:linesearch} holds}
\State {Set $s_k=\alpha_k d_k$, $x_{k+1}=x_k+s_k$ and let $dols = \texttt{false}$.
}
\Else
\State {Set $\alpha_{k}=\gamma^{-1}\alpha_k$.}
\EndIf
\EndWhile
\EndFor
\end{algorithmic}
\end{algorithm}

Algorithm \ref{sqp} is a special case of the model algorithm given in \cite[Section 4]{burke1989robust}. We refer the readers to \cite[Section 6]{burke1989robust} for details on the how the iterates genereated by the algorithm can be proven to be well defined as well as the convergence guarantees.

Now we present the needed assumptions, some estimates that will be useful later, and the main convergence theorem concerning Algorithm~\ref{sqp}.
\begin{assumption}\label{ass1} There is a bounded closed set $\Omega$ that contains all the iterates $\{x_k\}$ and trial points $\{x_k+\alpha d_k\}$, and satisfies
\begin{enumerate}[~~~~({A}1)]
    \item the objective function $f$ is continuously differentiable over $\Omega$, and $\nabla f$ is
    Lipschitz continuous with constant $L_f$;
    \item the constraints $h$ and their gradients $\nabla h$ are continuously differentiable over $\Omega$, each gradient $\nabla h_i$ is Lipschitz continuous
    with constants $L_{h,i}$ for all $i\in\{1,\cdots,m_1\}$;
    \item the constraints $c$ and their gradients $\nabla c$ are continuously differentiable over $\Omega$, each gradient $\nabla c_i$ is Lipschitz continuous
    with constants $L_{c,i}$ for all $i\in\{1,\cdots,m_2\}$.
    \item there are positive constants $\eta$ and $M_H$ (suppose $M_H>1$, without loss of generality) such that $2\eta\|d\|^2\leq d^TH_kd\leq M_H\|d\|^2$ for all $d\in \mathbb{R}^n$.
\end{enumerate}
\end{assumption}
Under these assumptions, there are two positive constants $M_f$ and $M_{d}$ such that
\begin{eqnarray}
&\max\{\lvert f(x)\rvert,\|h(x)\|_{\infty},\|\max\{c(x),0\}\|_{\infty}\}\leq M_f,\label{eq:bnd_fns}\\
&\max\{\|\nabla f(x)\|,\{\|\nabla h_i(x)\|\}_{i=1,\cdots,m_1},\{\|\nabla c_j(x)\|\}_{j=1\cdots,m_2}\}\leq M_d\label{eq:bnd_grds}\end{eqnarray}
hold for all $x\in\Omega$. The following Lemma, which grounds the subsequent convergence Theorem, can be readily proven {using the assumptions of the boundedness of the objective and constraint functions and gradients~\eqref{eq:bnd_fns}-\eqref{eq:bnd_grds}}.
\begin{lemma}\cite{burke1989robust}\label{lem:phi}
Suppose that Assumption \ref{ass1} holds. Let
$$L_{hc}=\max\{\{L_{h,i}\}_{i=1,\cdots,m_1}\cup\{L_{c,i}\}_{i=1,\cdots,m_2}\}.$$
Then {for all $k=0,1,2,...$ such that the Algorithm has not yet terminated},
\begin{equation*}
    \phi(x_k+\alpha d_k)\leq\phi_k-\alpha\Delta_{\phi}(x_k,\sigma_k)+\frac12L_{hc}\alpha^2\|d_k\|^2.
\end{equation*}
\end{lemma}

\begin{theorem}
Let $\{x_k\}$ be a sequence generated by Algorithm \ref{sqp}.  {Suppose that Assumption \ref{ass1} holds. }Then either\begin{enumerate}[(1)]
\item $\{x_k\}$ terminates at $\tilde x$, where $\tilde x$ is either a KKT point for \eqref{eq:prob_determ}, a Fritz John point for \eqref{eq:prob_determ} or a stationary point for $\phi$ that is infeasible; or
\item there is a limit point of $\{x_k\}$ that is a KKT point for \eqref{eq:prob_determ}, a Fritz John point for \eqref{eq:prob_determ} or a stationary point for $\phi$ that is infeasible.\end{enumerate}
\end{theorem}

\section{Stochastic setting}\label{sect:algs}
\subsection{Algorithm}
Now consider \eqref{eq:prob}. We present a stochastic SQP algorithm modeled after Algorithm \ref{sqp}, however using a noisy sample for the subproblem objective gradient
and a stochastic line search procedure as given in \cite{paquette2020stochastic}.

At each iteration, {from the state of $\sigma$-algebra $\mathcal{F}_{k-\frac{1}{2}}$,we compute a sample minibatch stochastic gradient $g_k\sim \frac{1}{N_k}\sum\limits_{i=1}^{N_k} \nabla F(x_k,\xi_k)$. Subsequently} we compute a direction $d_k${, defined as a random variable in, now, $\mathcal{F}_k$, }as arising from the solution of the following QP subproblem
 \begin{equation}\label{eq:s_subp0}
    \begin{array}{rl}
         \min_d & g_k^T d+\frac{1}{2} d^T H_k d \\
         \text{s.t.} & -y_ke\le h_k+\nabla h_k^T d\le y_ke, \\
         & c_k+\nabla c_k^Td \le y_ke \\
         & \|d\|_{\infty} \le \beta_k,
    \end{array}
    \end{equation}
where $y_k$ and $\sigma_k$ are determined by the same mechanism as described in Section \ref{sec:dset}.

Then, we compute the stochastic function estimates at the current iterate and at the new trial iterate, respectively denoted $f_k^0$ and $f_k^s$. Correspondingly, we define the stochastic merit function estimates
\begin{equation}\label{eq:stochmerit}
\Psi_k^0(\rho)=f_k^0+\rho\phi(x_k,\sigma_k),\ \text{and}\  \Psi_k^s(\rho)=f_k^s+\rho\phi(x_k+\alpha_kd_k,\sigma_k).
\end{equation}
The predicted reduction estimates in $\Psi(x;\rho)$ is defined as \begin{equation}\label{eq:defdeltapsi}
{\Delta_\Psi^g}(x_k,d_k;\rho)
:=-g_k^Td_k+\rho{\Delta_{\phi}}(x_k,\sigma_k).
\end{equation}
where $\Delta_{\phi}$ is given in~\eqref{eq:defdeltaphi}.

Like in the deterministic setting, it is checked if
\begin{equation}\label{eq:s_condpen}
  {\Delta_\Psi^g}(x_k,d_k,\rho_k)\geq \frac12d_k^T H_kd_k
\end{equation}
holds. If so, set $\rho_{k+1}=\rho_k.$ Otherwise, set
\begin{equation}\label{eq:s_uppen}
\rho_{k+1}=\max\left\{\frac{g_k^T d_k+\frac12d_k^TH_kd_k}{\Delta_{\phi}(x_k,\sigma_k)},2\rho_k\right\}.
\end{equation}
Therefore, we always have
\begin{equation}\label{eq:s_condpen2}
\Delta_{\Psi}^g(x_k, d_k;\rho_{k+1})\geq \frac12d_k^TH_kd_k.
\end{equation}
The stochastic counterpart of \eqref{eq:linesearch}, the line search acceptance criterion, is
\begin{equation}\label{eq:s_linesearch}
    \Psi_k^0(\rho_{k+1})-\Psi_k^s(\rho_{k+1})\geq\theta\alpha_k{\Delta_\Psi^g}(x_k, d_k;\rho_{k+1}).
\end{equation}

For constrained optimization, there are three major challenges with formulating a well defined line search procedure:
\begin{itemize}
    \item the possibility of a consecutive sequence of erroneous unsuccessful steps cause $\alpha_k$ to become arbitrarily small;
    \item steps may satisfy \eqref{eq:s_linesearch} but, in fact $$
     \Psi(x_k;\rho_{k+1})<\Psi(x_k+\alpha_k d_k;\rho_{k+1});
    $$
    (i.e., $d_k$ is a direction of increase for the complete merit function);
    \item the condition \eqref{eq:s_condpen} is repeatedly and indefinitely violated, which corresponds to $\rho_k$ approaching $+\infty$.
\end{itemize}
To handle the first two difficulties, we employ a strategy introduced by Paquette and Scheinberg \cite{paquette2020stochastic}, which we will briefly describe in the next section. Next, we introduce a special sampling method on estimating $\nabla f_k$ to deal with the third difficulty.

It is well-known that {the $l1$ penalty function is exact, meaning that the the set of limit points of the algorithm satisfy the MFCQ, for $\rho$ sufficiently large, the set of local minimizers of $\Psi(x;\rho)$ coincide to those of~\eqref{eq:prob}. This implies that in practice, optimization algorithms that check for the progress in the value of a merit function incorporating a penalty term can distinguish a favorable and unfavorable scenario regarding the asymptotic behavior of $\rho_k$, with constraint regularity yielding the favorable bounded penalty factor (see, for instance, \cite{Bertsekas2016}). In the stochastic case, the presence of noise implies that the corresponding guarantee of appropriate bounded terms cannot be ensured with certainty, and instead we present a, novel to this work, scheme to ensure almost sure penalty parameter boundedness. Specifically, if $\|g_k\|$ s larger than some term, increase the
accuracy requirement and scale the appropriate parameters so that
if $\|g_k\|$ grows without bound, so does the certainty that it is a
correct estimate of the gradient.}.

 With these modifications, we can present the algorithmic framework of the stochastic SQP method, which is formulated in Algorithm \ref{ssqp}.

\begin{algorithm}[ht]
\caption{Stochastic SQP}
\label{ssqp}
\begin{algorithmic}
\State Parameters: $\beta_l$, $\beta_u$ and $\sigma_u$ satisfying \eqref{eq:relatebetasig}, $\kappa_u>0$, $\gamma>1$, $\theta\in(0,1)$, $\alpha_{\max}>0$, $\rho_0>0$, $0<p_f,p_0^g<1$, $\zeta_0\gg 0$ and $\zeta_c>0$. An increasing sequence $\{a_j\}$ satisfying $a_0\geq p_0^g$, $a_j\in(0,1), \forall j$  and $\sum_{j=0}^{+\infty}(1-a_j)<+\infty$.
\State Initialization: Choose a starting point  $x_0\in \mathbb{R}^n$ and set $\alpha_0=\gamma^{j_0}\alpha_{\max}$ for some $j_0<0$. Set $j=1$.
\For{$k=0,1,2,\cdots$}
\If{$k=0$ or $s_k\neq 0$}
\State choose $\sigma_k>0$ such that \eqref{eq:dist} is feasible.
\State compute $(p_k,y_k)$ as a solution of \eqref{eq:dist}.
\State compute $H_k$ {satisfying Assumption~\ref{ass1}, Part 4}.
\EndIf
\State Construct a probabilistic model to compute $g_k$. {This defines the transition from $\sigma$-algebra $\mathcal{F}_{k-\frac{1}{2}}$ to $\mathcal{F}_k$.}
\State {Compute $d_k$ from ~\eqref{eq:s_subp0}}

\If{$\|g_k\|>\zeta_k$}
\State set $\zeta_{k+1}:=\max\{\zeta_k+\zeta_c,\|g_k\|\}$,set $p_{k+1}^g=a_j$.
\State let $j=j+1$.
\Else
\State set $\zeta_{k+1}=\zeta_k$ and $p_{k+1}^g=p_k^g$.
\EndIf
\State
\If{\eqref{eq:s_condpen} does not hold}
\State {set $\rho_{k+1}$  by \eqref{eq:s_uppen}, }
\Else
\State set $\rho_{k+1}=\rho_k$.
\EndIf
\State {Compute stochastic function estimates $f^0_k$ and $f^s_k$. This defines the transition from $\sigma$-algebra $\mathcal{F}_{k}$ to $\mathcal{F}_{k+\frac{1}{2}}$.}
\If{ \eqref{eq:s_linesearch} holds}
\State set $s_k=\alpha_k d_k$, $x_{k+1}=x_k+s_k$ and $\alpha_{k+1}=\min\{\gamma\alpha_k$,$\alpha_{\max}$\}.
\Else
\State set $s_k=0$, $x_{k+1}=x_k$  and $\alpha_{k+1}=\gamma^{-1}\alpha_k$.
\EndIf
\EndFor
\end{algorithmic}
\end{algorithm}

\subsection{Random gradient and function estimate}\label{sect:gfest}

Algorithm \ref{ssqp} generates a random process given by the sequence $\{\hat x_k,\hat f_k^0,\hat f_k^s,\hat g_k,\hat\alpha_k,\hat s_k,\hat \rho_k\}$, where $\hat s_k=\hat\alpha_k\hat d_k$. In what follows we will use a ``$\ \hat{\ }\ $'' over a quantity to
denote the related random quantity.

Let $\mathcal{F}_{k-1}$ be the $\sigma-$algebra that is conditional on the random variables $\hat g_0$, $\hat g_1$, $\cdots$, $\hat g_{k-1}$
and $\hat f_0^0$, $\hat f_0^s$, {$\hat f_1^0$}, $\hat f_1^s$, $\cdots$, $\hat f_{k-1}^0$, $\hat f_{k-1}^s$ and let $\mathcal{F}_{k-1/2}$ be the $\sigma-$algebra conditional on the random variables $\hat g_0$, $\hat g_1$, $\cdots$, $\hat g_{k}$
and $\hat f_0^0$, $\hat f_0^s$, {$\hat f^0_1$}, $\hat f_1^s$, $\cdots$, $\hat f_{k-1}^0$, $\hat f_{k-1}^s$. {That is, the evaluation of the stochastic gradient estimate $g_k$ distinguishes $\mathcal{F}_{k-\frac{1}{2}}$ and $\mathcal{F}_k$ and the sampling of $f^0_k$ and $f^s_k$ distinguishes $\mathcal{F}_k$ and $\mathcal{F}_{k+\frac{1}{2}}$. We consider $\omega_k$ to be the random elements, i.e. for a given realization $\hat{x}_k(\omega_k)=x_k$ and $\Xi=\cup_k\omega_k$.}

The stochastic gradient and stochastic function values should be ``close'' to their true counterparts with certain
reasonable probabilities, which means that, at each iterate, the distances between the random estimates and the true values are bounded using the current step length.

We use the definitions given in \cite{paquette2020stochastic}, with slight modifications, to measure the accuracy of the gradient estimates $\hat g_k$ and the function estimates $\hat f_k^0$ and $\hat f_k^s$.

\begin{definition}\label{def:gaccu}
We say a sequence of random gradient estimates $\{\hat g_k\}$ is $p_k^g$-probabilistically $\varepsilon_g$-sufficiently accurate for Algorithm \ref{ssqp} for the corresponding sequence $\{\hat x_k,\hat\alpha_k,\hat d_k\}$ if  the event
\[I_k:=\{\|\hat g_k-\nabla f(\hat x_k)\|\leq\varepsilon_g\hat\alpha_k\|\hat d_k\|\}
\]
satisfies the condition{
\[
\textbf{Pr}(I_k\mid \mathcal{F}_{k-\frac 12})=\textbf{E}[\textbf{1}_{I_k}\mid \mathcal{F}_{k-\frac 12}]\geq p_k^g.
\]}
\end{definition}


\begin{definition}
A sequence of random function estimates $\{\hat f_k^0,\hat f_k^s\}$ is said to be $p_f$-probabilistically $\varepsilon_f$-sufficiently accurate for Algorithm \ref{ssqp} given the corresponding random instantiations $\{x_k,\alpha_k,d_k\}$ if the event
\[J_k:=\{\|\hat f_k^0-f(x_k)\|\leq\varepsilon_f\alpha_k\|d_k\|^2
\ \text{and}\ \|\hat f_k^s-f(x_k+ s_k)\|\leq\varepsilon_f\alpha_k\|d_k\|^2\}\]
satisfies the condition{
\[
\textbf{Pr}(J_k\mid\mathcal{F}_{k})=\textbf{E}[\textbf{1}_{J_k}\mid\mathcal{F}_{k}]\geq p_f.
\]}
\end{definition}
{We denote by $I^c_k$ and $J^c_k$ the complement of these events (that is, the event that they do not hold), respectively}.

{Let $p_k^g\in(0,1]$ and $p_f\in(0,1]$ be given. In the course of establishing our convergence theory we shall present specific quantities indicating reasonable problem-dependent lower bounds for these probabilities. We assume that they define the probabilistic accuracy of our sequence.
\begin{assumption}\label{ass2}
The following hold for the quantities in the algorithm
\begin{enumerate}[(i)]
    \item The random gradient $\hat g_k$ are $p_k^g$-probabilistically $\varepsilon_g$-sufficiently $p_k^g$ probabilistically accurate .
    \item The estimates $\{\hat f_k^0,\hat f_k^s\}$ generated by Algorithm \ref{ssqp} are $p_f$-probablilistically $\varepsilon_f$-accurate estimates for some $\varepsilon_f>0$ and probability $p_f$.
    \item The sequence of estimates $\{\hat f_k^0,\hat f_k^s\}$ generated by Algorithm \ref{ssqp} satisfies a $\kappa_f$-variance condition for all $k\geq 0$.
    \begin{eqnarray}
        &\textbf{E}[\lvert\hat f_k^s-f(\hat x_k+\hat s_k)\rvert^2\mid\mathcal{F}_k]\leq\kappa_f^2\alpha_k^2(\Delta_{\Psi}^g(x_k,d_k; \rho_{k+1}))^2, \text{ and } \nonumber\\
         &\textbf{E}[\lvert\hat f_k^0-f(\hat x_k)\rvert^2\mid\mathcal{F}_k]\leq\kappa_f^2\alpha_k^2(\Delta_{\Psi}^g( x_k, d_k; \rho_{k+1}))^2.\nonumber
    \end{eqnarray}
\end{enumerate}
\end{assumption}
}

We follow the approach given in \cite{paquette2020stochastic} for computing stochastic estimates that satisfy Assumption \ref{ass2}.
Assume, only for this section, that {
\[
\textbf{E}\left(\|\nabla F(x,\xi)-\nabla f(x)\|^2\mid\mathcal{F}_{k-\frac 12}\right)\leq V_g\ \text{and}\ \textbf{E}\left(\lvert F(x,\xi)- f(x)\rvert^2\mid\mathcal{F}_{k-\frac 12}\right)\leq V_f.
\]}
The stochastic gradient $g_k$ is computed as the average
\[
g_k=\frac1{\lvert\mathcal{S}_k^g\rvert}\sum_{i\in\mathcal{S}_k^g}\nabla F(x_k,\xi_i),
\]
where $\mathcal{S}_k^g$ is the sample set and $\lvert \cdot\rvert$ denotes the cardinality of a set. The stochastic estimate
$f_k^0$ is computed by
\[
f_k^0=\frac1{\lvert\mathcal{S}_k^f\rvert}\sum_{i\in\mathcal{S}_k^f} F(x_k,\xi_i)
\]
and $f_k^s$ is computed analogously, where $\mathcal{S}_k^f$ is the sample set.

{We now discuss how one can practically define the samples so as to satisfy the conditions of Assumption~\ref{ass2} on the sampling accuracy in the Algorithm. }Similar to Subsection 2.3 of \cite{paquette2020stochastic}, we can derive the lower bounds of the sample sizes $|\mathcal{S}_k^g|$ and $|\mathcal{S}_k^f|$ using Chebyshev inequality.  To satisfy Assumption \ref{ass2}, it suffices to require the sample sizes $\lvert\mathcal{S}_k^g\rvert$ and $\lvert\mathcal{S}_k^f\rvert$ to satisfy
\[
\lvert\mathcal{S}_k^g\rvert\geq\tilde O\left(\frac{V_g}{\varepsilon_g^2\alpha_k^2\|d_k\|^2}{\frac1{1-p_k^g}}\right)\]
and
\[\lvert\mathcal{S}_k^f\rvert\geq\max\left\{\tilde O\left(\frac{V_f}{\varepsilon_f^2\alpha_k^2\|d_k\|^4}{\frac1{1-p_f}}\right),\frac{V_f}{\kappa_f^2\hat\alpha_k^2(\Delta_{\Psi}^g(\hat x_k,\hat d_k;\hat \rho_{k+1}))^2}\right\}.
\]
In practice $d_k$ is not known when the initial sample size is chosen. Paquette and Scheinberg \cite{paquette2020stochastic} suggest a simple loop
to choose the sample sizes: guessing the value of $d_k$  and increasing the
number of samples until these conditions are satisfied. Discussions on this procedure can be found in \cite{Cartis2018}.

{We also compute a positive definite matrix $H_k$ that must satisfy, as by Assumption~\ref{ass1}-4, that is $2\eta\|d\|^2\le d^2 H_k d\le M_h \|d\|^2$ for all $d\in\mathbb{R}^n$. In the implementation, we employ a simple diagonal expression $H_k=\tau I$ with $2\eta\le\tau\le M_h$, as that is the most simple choice that satisfies the criteria and we are looking to perform a proof of concept. Choices of $H_k$ that incorporate second order information, appropriately restricted to its positive definite components while maintaining Newton local convergence rate guarantees, are the preferred option in deterministic SQP~\cite{kungurtsev2013second}. Such considerations, and the analysis of how to best incorporate them in the stochastic setting, is beyond the scope of the current paper and left to future work.}

Similar to \cite[Lem. 2.5]{paquette2020stochastic}, we have the following lemma.{
\begin{lemma}\label{lem:jc}
Let Assumption \ref{ass2} hold, that is, suppose that $\{\hat f_k^0,\hat f_k^s\}$ are $p_f$-probabilistically accurate estimates. Then we have
\begin{eqnarray}
    \textbf{E}[\textbf{1}_{J_k^c}\lvert f_k^s-f(  x_k+ s_k)\rvert\mid\mathcal{F}_{k}]\leq(1-p_f)^{1/2}\kappa_f\hat\alpha_k\Delta_{\Psi}^g( x_k, d_k; \rho_{k+1}),\nonumber\\
    \textbf{E}[\textbf{1}_{J_k^c}\lvert\hat f_k^0-f( x_k)\rvert\mid\mathcal{F}_{k}]\leq(1-p_f)^{1/2}\kappa_f\hat\alpha_k\Delta_{\Psi}^g( x_k, d_k; \rho_{k+1}).\nonumber
\end{eqnarray}
\end{lemma}}
The following inequalities will be useful for the subsequent analysis:
\[
\textbf{E}[\textbf{1}_{I_k\cap J_k}\mid\mathcal{F}_{k-\frac 12}]\geq p_k^gp_f,\ \textbf{E}[\textbf{1}_{I_k^c\cap J_k}\mid\mathcal{F}_{k-\frac 12}]\leq 1-p_k^g\ \text{and}\
\textbf{E}[\textbf{1}_{ J_k^c}\mid\mathcal{F}_{k}]\leq 1-p_f
\]

{Finally, we remark on the presence of an external accuracy forcing term $\{a_j\}$ for $p^g$. The intention is ultimately corroborated in the last set of theoretical results below regarding asymptotic penalty parameter behavior. In particular, asymptotic full gradient computations are necessary to prevent the low probability event of a sequence of increasingly larger magnitude gradients which are also inaccurate. Meanwhile, the sequence is chosen so as to be
 make it the slowest possible increasing sequence towards one that can be satisfied}.
\subsection{Preliminary results for the convergence analysis}

We shall use the following set, which contains all the iterations at which the step is successful
\begin{equation*}
\mathbf{S}:=\{k: \text{ at iteration }k\text{ the step is successful, {i.e., $d_k$ satisfies~\eqref{eq:s_linesearch}}}\}.
\end{equation*}

\subsubsection{A Renewal-reward process.}

We will use a general framework introduced in \cite{blanchet2019} to analyze the behavior of our stochastic SQP method. This framework, where a general random process and its stopping time are defined, was used to explore the behaviors of a stochastic trust region method in \cite{blanchet2019} and a stochastic line search method in \cite{paquette2020stochastic}.

\begin{definition}
Given a discrete time stochastic process $\{\hat x_k\}$, a random variable $T$ is a stopping  time if the event $\{T=k\}\in\sigma(\hat x_0,\hat x_1,\cdots,\hat x_k).$, {that is the $\sigma$-field generated by $(\hat x_1,\cdots,\hat x_k)$.}
\end{definition}
Denote by $\{T_\varepsilon\}_{\varepsilon>0}$ a family of stopping times with respect to $\{\mathcal{F}_k\}_{k\geq 0}$, parameterized by
$\varepsilon$.

Let $\{\Phi_k,\hat\alpha_k,W_k\}$ be a random process such that $\Phi_k\in[0,\infty)$ and $\hat\alpha_k\in[0,\infty)$ for all $k\geq 0$, $W_k$ is a biased random walk process, defined on the same probability space as $\{\Phi_k,\hat\alpha_k\}$ and $\mathcal{F}_k$ is the $\sigma$-algebra generated by
\[
\{\Phi_0,\hat\alpha_0,W_0,\cdots,\Phi_k,\hat\alpha_k,W_k\},
\]
We let $W_0=1$ and $W_k$ obey the following transition rules,
\begin{equation}\label{eq:w}
Pr(W_{k+1}=1\mid \mathcal{F}_k)=p\ \text{and}\ Pr(W_{k+1}=-1\mid \mathcal{F}_k)=1-p
\end{equation}
with $\frac12<p<1$.

The following assumptions are used in \cite{blanchet2019} to derive a bound on $\textbf{E}[T_\varepsilon]$. {We shall ultimately show that Algorithm~\ref{ssqp} satisfies these conditions. A key part of the stochastic line search is to ensure the step sequence is appropriately bounded in magnitude. This is delicate -- it is bounded so as to mitigate the unfortunate impact of a low probability event of a large magnitude highly inaccurate gradient estimate, while maintaining enough adaptivity in the line search to achieve the most possible descent along a subproblem solution in more favorable circumstances. }
\begin{assumption}\label{ass4}
The following hold for the process $\{\Phi_k,\hat\alpha_k,W_k\}$.
\begin{enumerate}[(i)]
    \item The random variable $\hat\alpha_0$ is taken to be a constant w.p.1. For some $\alpha_{\max}>0$, there exists a constant $\lambda\in(0,\infty)$ and $j_{\max}\in\mathbb{N}$ such that $\alpha_{\max}=\hat\alpha_0e^{\lambda j_{\max}}$ and for all realizations of the process, $\hat{\alpha}_k$ satisfy $\hat\alpha_k\leq\alpha_{\max}$ for all $k\geq0$.
    \item There exists a constant $\bar \alpha=\hat\alpha_0e^{\lambda \bar j}$ for some $\bar j\in \mathbb{N}$ such that the following holds for all $k\geq 0$:
    \[
    \textbf{1}_{\{T_\varepsilon >k\}}\hat\alpha_{k+1}\geq \textbf{1}_{\{T_\varepsilon >k\}}\min\left\{\hat\alpha_ke^{\lambda W_{k+1}},\bar\alpha\right\},
    \]
    where $W_{k+1}$ satisfies \eqref{eq:w} with $p>\frac12$.
    \item There exists a nondecreasing function $q:[0,\infty)\rightarrow(0,\infty)$ and a constant $\Theta>0$, {independent of the stochastic process noise as well as $k$}, such that
    \[
    \textbf{1}_{\{T_\varepsilon>k\}}\textbf{E}[\Phi_{k+1}\mid\mathcal{F}_k]\leq \textbf{1}_{\{T_\varepsilon>k\}}(\Phi_k-\Theta q(\hat\alpha_k)).
    \]
\end{enumerate}
\end{assumption}

A bound on $\textbf{E}[T_\varepsilon]$ is given in the following theorem.
\begin{theorem}\cite[Theorem 2]{blanchet2019}\label{thm:bounde}
Under Assumption \ref{ass4},
\[
\textbf{E}[T_\varepsilon]\leq\frac{p}{2p-1}\frac{\Phi_0}{\Theta q(\bar \alpha)}+1.
\]
\end{theorem}

\subsubsection{A measure for optimality}
In order to define a merit function $\Phi$ satisfying Assumption \ref{ass4}(iii) with $\hat\alpha$, we introduce a measure for approximate optimality.  Recalling the constant $\beta_l$ from~\eqref{eq:relatebetasig}, let $t(x)$ solve
\begin{equation}\label{eq:lpchi}
    \begin{split}\min_{t\in\mathbb{R}^n}&\ \nabla f(x)^Tt\\
    s.t.&\ \nabla h(x)^Tt=0,\\
        &\ \nabla c(x)^Tt\leq 0,\\
        &\ \|t\|_\infty\leq \frac12\beta_l.
    \end{split}
\end{equation}
Then we define
\begin{equation*}
\chi(x)=-\nabla f(x)^Tt(x).
\end{equation*}
We also define an estimate of $\chi(x)$ as
$\chi^g(x,g)=-g^Tv(x)$ that uses a gradient estimate $g$ and the approximate subproblem solution $v(x)$ of
\begin{equation*}\label{eq:lpchig}
    \begin{split}\min_{v\in\mathbb{R}^n}&\ g(x)^Tv\\
    s.t.&\ \nabla h(x)^Tv=0,\\
        &\ \nabla c(x)^Tv\leq 0\\
        &\ \|v\|_\infty\leq \frac12\beta_l.
    \end{split}
\end{equation*}
Neither linear program needs to be solved in Algorithm \ref{ssqp}. These are simply used for theoretical purposes.

The following theorem is trivial.
\begin{theorem}
A point $x$ is a KKT point of \eqref{eq:prob} if and only if $\phi(x)=0$ and $\chi(x)=0$.
\end{theorem}

We are going to show the Lipschitz continuity of $\chi(x)$ with respect to changes of coefficients. This is one of the key points of the convergence analysis to be given.
Consider a linear program of the following general form
\begin{equation}\label{eq:glp}
    \begin{split}
        \min\limits_{z\in\mathbb{R}^n}&\ \pi^Tz\\
        \mathrm{s.t.}&\ Az=0,\\
        &\ Bz\leq 0,\\
        &\ \|z\|_{\infty}\leq\frac12\beta_l,
    \end{split}
\end{equation}
where $\pi\in \mathbb{R}^n$, $A\in \mathbb{R}^{m_1\times n}$, $B\in \mathbb{R}^{m_2\times n}$. Let $$\Upsilon=\{\upsilon=(\pi,A,B)\mid \pi\in \mathbb{R}^n,A\in \mathbb{R}^{m_1\times n},B\in \mathbb{R}^{m_2\times n}\}.$$
For an instance $\upsilon\in\Upsilon$, we denote the optimal value of \eqref{eq:glp} by $\chi(\upsilon)$, by a slight abuse of notation.

{
\begin{lemma}\label{lem:llc}
Let $\upsilon_0\in\Upsilon$. Then there exist $\delta_0(\upsilon_0)>0$ and $L_0(\upsilon_0)>0$ such that for any two instances $\upsilon_1,\upsilon_2$ in $\Upsilon$ satisfying $\|\upsilon_j-\upsilon_0\|\leq\delta_0(\upsilon_0)$, for $j=1,2$, it  holds that
\[\lvert\chi(\upsilon_1)-\chi(\upsilon_2)\rvert\leq L_0(\upsilon_0)\|\upsilon_1-\upsilon_2\|.\]
\end{lemma}
\begin{proof}
    Since problem \eqref{eq:glp} has a bounded feasible set, it follows from  strong duality (see \cite{WrighN1999} for instance) that both the primal solution set and the dual solution set of \eqref{eq:lpchi} are nonempty and bounded. Then, by
\cite[Thm. 1]{Robinson1977}, the solution to problem \eqref{eq:lpchi} is solveable and locally Lipschitz stable with respect to all small but arbitrary perturbations in the coefficients. Then the local Lipschitz continuity of $\chi(\cdot)$, which corresponds to the optimal value function of~\eqref{eq:lpchi}, follows from \cite[Thm. 4.3]{Canovas2006a}.
\end{proof}}

If the optimal value function $\chi(\upsilon)$ is restricted on a compact subset of $\Upsilon$, then it is globally Lipschiz continuous.
\begin{theorem}
Let $\mathcal{K}$ be a compact subset of $\Upsilon$. Then there is a constant $L_\upsilon>0$ such that
\[\lvert\chi(\upsilon_1)-\chi(\upsilon_2)\rvert\leq L_\upsilon\|\upsilon_1-\upsilon_2\|\]
holds for any two instances $\upsilon_1,\upsilon_2\in\mathcal{K}$.
\end{theorem}
\begin{proof}
This follows directly from Lemma \ref{lem:llc} and standard compactness arguments for Lipschitz functions (see, for instance, \cite[Prop. 3.3.2]{Schaeffer2016}).
\end{proof}

{Assumption~\ref{ass1} permits us to define a bounded compact set for the gradient values that could be evaluated:}
\[
\{(\nabla f(x),\nabla h(x),\nabla c(x))\mid x\in\Omega\}\subset\mathcal{K}.
\]
Therefore, by Assumptions (A1)-(A3) we have
\begin{corollary}\label{thm:glc}
Suppose that Assumptions (A1)-(A3) hold and $L_\chi=M_dL_\upsilon$. Then for any $x_1,x_2\in\Omega$, we have
\[\lvert\chi(x_1)-\chi(x_2)\rvert\leq L_\chi\|x_1-x_2\|.\]
\end{corollary}

\subsubsection{Extended MFCQ}

In order to accommodate infeasible accumulation points, we need to extend the {definition of the MFCQ to include infeasible points}.
\begin{definition}\label{def:eMFCQ}
A point $x$ is said to satisfy the extended MFCQ (eMFCQ for short) if
\begin{enumerate}[(1)]
    \item  the gradients $\{\nabla h_i(x):i=1,\cdots,m_1\}$ are linearly independent;
    \item there is a vector $z\in \mathbb{R}^n$ that satisfies
    \begin{eqnarray}
&&\nabla h( x)^T z=0,\label{eq:eMFCQ1}\\
&&\nabla c_i( x)^T z<0,i\in\{i : c_i(x)\geq 0\}\label{eq:eMFCQ2}
\end{eqnarray}
\end{enumerate}
\end{definition}

The definition of eMFCQ goes back at least to Di Pillo and Grippo \cite{DiPillo1988}
and Di Pillo and Grippo \cite{DiPillo1989}. Note that when $x$ is feasible, eMFCQ is equivalent to MFCQ.

\begin{assumption}\label{ass3}
The eMFCQ holds at any accumulation point of
$\{x_k\}$ {for any realization of Algorithm 2}.
\end{assumption}

Assuming eMFCQ enables us to derive a lower bound on $\Delta_{\phi}(x_k,\sigma_k)${, that is a constant $\tilde\tau_1$ depending on the noise realization of Algorithm 2, that permits a bound on the descent sequence independent of $k$}.

\begin{lemma}\label{lem:localu}
{Consider any realization $\{x_k\}$ of Algorithm 2}. Suppose that Assumptions \ref{ass1} and \ref{ass3}
hold and that $\tilde x$ is an accumulation point of $\{x_k: k\in\mathbf{S}\}$. Suppose also that $\phi(\tilde x)>0$. Then there exists a neighborhood $\tilde{\mathcal{N}}$ of $\tilde x$  and a constant $\tilde \tau_1$, which is a function of $\tilde x$, such that for any $x_k\in \tilde{\mathcal{N}}$,  there is
\begin{equation}\label{eq:predphi}
    \Delta_{\phi}(x_k,\sigma_k)\geq\tilde\tau_1\phi_k.
\end{equation}
\end{lemma}
\begin{proof}
In this proof, we use the notation $A^+=(A^TA)^{-1}A^T$. By linear independence and continuity, there exists a constant $\kappa_{gh}>0$ and a neighborhood $\tilde{\mathcal{N}}$ of $\tilde x$ such that $\nabla h_k^+$ exists and
$\|\nabla h_k^+\|_\infty\leq\kappa_{gh}$ for all $x_k\in \tilde{\mathcal{N}}$ . We denote $n_k=-(\nabla h_k^+)^Th_k$, which is  the smallest norm element of the linearized manifold. By its definition, we have $\|n_k\|_\infty\leq \kappa_{gh}\|h_k\|_\infty\leq\kappa_{gh}\phi_k$.

On the other hand, by Definition \ref{def:eMFCQ}, there is a nonzero vector $\tilde z\in \mathbb{R}^n$ such that \eqref{eq:eMFCQ1} and \eqref{eq:eMFCQ2} hold with $x=\tilde x$, i.e.,
\begin{eqnarray}
&&\nabla h(\tilde x)^T\tilde z=0,\nonumber\\
&&\nabla c_i(\tilde x)^T\tilde z<0,\ i\in\{i: c_i(\tilde x)\geq 0\}.\nonumber 
\end{eqnarray}

Let $\bar c =-\max\limits_{i}\{c_i(\tilde x): c_i(\tilde x)<0\}$. Define the sequence $\{z_k\}$ by
\[
z_k:=\frac{\frac12\min\left\{\frac{\bar c}{M_d},\sigma_u,\kappa_u\phi_k\right\}(I-\nabla h_k\nabla h_k^+)\tilde z}{\|(I-\nabla h_k\nabla h_k^+)\tilde z\|_\infty}.
\]
Observe that this implies $\nabla h_k^Tz_k=0$.

Consider the line segment defined by
\[
p_k(\tau)=\tau n_k+(1-\tau)z_k,\ \tau\in [0,1].
\]
Then
 \[h_k+\nabla h_k^Tp_k(\tau)=(1-\tau)h_k+\tau(h_k+\nabla h_k^Tn_k)+(1-\tau)\nabla h_k^Tz_k=(1-\tau)h_k,\] which implies that
 \begin{equation}\label{eq:hpred}
 \| h_k+\nabla h_k^Tp_k(\tau)\|_\infty\leq (1-\tau)\phi_k.
 \end{equation}

 Now consider any  $i\in \{i: c_i(\tilde x)\geq 0\}$. There is
    \begin{equation*}
    \begin{split}
    &\nabla c_i(x_k)^Tz_k
    =\frac{\min\left\{\frac{\bar c}{M_d},\sigma_u,\kappa_u\phi_k\right\}\nabla c_i(x_k)^T(I-\nabla h_k\nabla h_k^+)\tilde z}{2\|(I-\nabla h_k\nabla h_k^+)\tilde z\|_\infty}\\
    &\stackrel{x_k\rightarrow\tilde x}{\longrightarrow}\frac{\min\left\{\frac{\bar c}{M_d},\sigma_u,\kappa_u\phi(\tilde x)\right\}\nabla c_i(\tilde x)^T\tilde z}{2\|\tilde z\|_\infty}<0,
    \end{split}\end{equation*}
    which implies by continuity the existence of a (possibly smaller) neighborhood $\tilde{\mathcal{N}}$ of $\tilde x$ and a positive constant $\tilde\varepsilon$ such that for $x_k\in\tilde{\mathcal{N}}$,
    \[
    c_i(x_k)+\nabla c_i(x_k)^Tz_k\leq c_i(x_k)-\tilde\varepsilon.
    \]
  Then
    \begin{equation*}
    \begin{split}
        &c_i(x_k)+\nabla c_i(x_k)^Tp_k(\tau)\\
        =&c_i(x_k)+\tau\nabla c_i(x_k)^Tn_k+(1-\tau)\nabla c_i(x_k)^Tz_k\\
        =&\tau(c_i(x_k)+\nabla c_i(x_k)^Tn_k)+(1-\tau)(c_i(x_k)+\nabla c_i(x_k)^Tz_k)\\
        \leq &\tau(c_i(x_k)+\nabla c_i(x_k)^Tn_k)+(1-\tau)(c_i(x_k)-\tilde \varepsilon)\\
        =&\tau(c_i(x_k)+\nabla c_i(x_k)^Tn_k)-(1-\tau)\tilde\varepsilon +(1-\tau)c_i(x_k).
    \end{split}
    \end{equation*}
By Assumption \ref{ass1}, if $\tau\in\left(0,\frac{\tilde\varepsilon}{\tilde\varepsilon+M_f+M_d\kappa_{gh}M_f}\right]$, then
\[
\tau(c_i(x_k)+\nabla c_i(x_k)^Tn_k)-(1-\tau)\tilde\varepsilon
=\tau(c_i(x_k)+\nabla c_i(x_k)^Tn_k+\tilde\varepsilon)-\tilde\varepsilon\leq 0.
\]
Hence we have
\begin{equation}\label{eq:cpred}
    c_i(x_k)+\nabla c_i(x_k)^Tp_k(\tau)\leq (1-\tau)c_i(x_k)\leq(1-\tau)\phi(x_k).
\end{equation}

For any $i\not\in \{i : c_i(\tilde x)\geq 0\}$, we choose $\tau\in\left(0,\frac{\bar c}{2M_d\kappa_{gh}M_f}\right]$, which implies
\begin{equation}\label{eq:cpredn}
\begin{split}
& c_i(x_k)+\nabla c_i(x_k)^Tp_k(\tau)\\
=& c_i(x_k)+\tau\nabla c_i(x_k)^Tn_k+(1-\tau)\nabla c_i(x_k)^Tz_k\\
\leq& -\bar c+\frac{\bar c}2+\frac{\bar c}2\\
=& 0.\end{split}\end{equation}

Next, let us consider $\|p_k(\tau)\|_{\infty}$.
We have
\begin{equation}\label{eq:dpred}
\begin{split}
&\|p_k(\tau)\|_\infty
\leq\tau\|n_k\|_\infty+(1-\tau)\| z_k\|_\infty\\
\leq&\tau\kappa_{gh}\phi_k+\frac12(1-\tau)\min\left\{\frac{\bar c}{M_d},\sigma_u,\kappa_u\phi_k\right\}\\
\leq&\tau\kappa_{gh}\phi_k+\frac12(1-\tau)\min\left\{\sigma_u,\kappa_u\phi_k\right\}
\end{split}\end{equation}
If $\sigma_u\leq \kappa_u\phi_k$, then from \eqref{eq:dpred} together with the fact that $\phi(x)\leq M_f$ (cf. \eqref{eq:bnd_fns}), it follows that
\begin{equation*}
        \|p_k(\tau)\|_{\infty}\leq\tau\kappa_{gh}M_f+\frac12(1-\tau)\sigma_u\leq \sigma_u=\min\{\sigma_u,\kappa_u\phi_k\}
\end{equation*}
for any $\tau\in(0,\sigma_u/(2\kappa_{gh}M_f))$.

Otherwise, i.e., $\sigma_u> \kappa_u\phi_k$, we have $\min\{\sigma_u,\kappa_u\phi_k\}=\kappa_u\phi_k$. Then \eqref{eq:dpred} implies
\begin{equation*}
    \|p_k(\tau)\|_{\infty}\leq \tau\kappa_{gh}\phi_k+\frac12(1-\tau)\kappa_u\phi_k.
\end{equation*}
For $\tau\in(0,{\kappa_u}/(2\kappa_{gh})]$, it holds that,
\[
\|p_k(\tau)\|_{\infty}\leq\frac12\kappa_u\phi_k+\frac12\kappa_u\phi_k=\kappa_u\phi_k=\min\{\sigma_u,\kappa_u\phi_k\}.
\]
From this we obtain the expression,
\begin{equation}\label{eq:dpredx}
    \|p_k(\tau)\|_{\infty}\leq\min\{\sigma_u,\kappa_u\phi_k\}, \forall \tau\in\left(0,\min\left\{\frac{\sigma_u}{2\kappa_{gh}M_f},\frac{\kappa_u}{2\kappa_{gh}}\right\}\right].
\end{equation}

Finally, it follows from \eqref{eq:hpred}, \eqref{eq:cpred}, \eqref{eq:cpredn} and \eqref{eq:dpredx} that $(p_k(\tau),(1-\tau)\phi_k)$ is a feasible solution of \eqref{eq:dist} for all
\[
\tau\in\left(0,\min\left\{\frac{\tilde\varepsilon}{\tilde\varepsilon+M_f+M_d\kappa_{gh}M_f},\frac{\bar c}{2M_d\kappa_{gh}M_f},\frac{\sigma_u}{2\kappa_{gh}M_f},\frac{\kappa_u}{2\kappa_{gh}}\right\}\right].
\]
Therefore we have $y_k\leq(1-\tau)\phi_k$ for all $\tau$ in the interval above, which implies \eqref{eq:predphi}.
\end{proof}

\subsection{Descent properties}

In this subsection, we are going to provide lower bounds on the predicted reduction in $\Psi$.
\begin{lemma}\label{lem:bnd_dpsi}
{Consider any realization $\{x_k\}$ of Algorithm 2 such} that Assumptions \ref{ass1} and \ref{ass3}
hold and that $\tilde x$ is an accumulation point of $\{x_k\mid\ k\in\mathbf{S}\}$ but not a stationary point of \eqref{eq:prob}. There is a neighborhood $\tilde{\mathcal{N}}$ of $\tilde{x}$  and a  constant $M_1>0$ such that
\begin{equation}\label{eq:bnd_dpsi}
\begin{split}    \Delta_{\Psi}^g(x_k,d_k;\rho_{k+1})\geq&
   \left(\tilde\tau_1\rho_{k+1}- \left(\|g_k\|+\frac12M_H(\beta_l+\kappa_uM_f)\right)\kappa_u\right)\phi_k\\
   &+\min\left\{\frac18M_H\beta_l^2,\frac{2(\chi_k^g)^2}{M_H\beta_l^2}\right\}.
   \end{split}
\end{equation}
for all $x_k\in\tilde{\mathcal{N}}$
\end{lemma}
\begin{proof}
From Lemma \ref{lem:localu}, there is a neighborhood $\tilde{\mathcal{N}}$ of $\tilde x$ such that for any $x_k\in\tilde{\mathcal{N}}$, \eqref{eq:dist} has a solution $(y_k,p_k)$ with $p_k\neq 0$ and that $\eqref{eq:predphi}$ holds.
Let $v_k$ be a solution of \eqref{eq:lpchig}. Define $d_k(\tau)=p_k+\tau v_k$, $\tau\in(0,1]$ at $x_k$. Then for any $\tau\in\left(0,1\right]$, the vector $d_k(\tau)$ is feasible for \eqref{eq:s_subp0}.

Recall from~\eqref{eq:dist0} that $\|p_k\|_{\infty}\le \kappa_u \phi_k$. Since $d_k$ is an optimal solution of \eqref{eq:s_subp0}, we have
\begin{equation*}
    \begin{split}
        &g_k^Td_k+\frac12d_k^TH_kd_k\\
        \leq& g_k^Td_k(\tau)+\frac12d_k(\tau)^TH_kd_k(\tau)\\
        =&(g_k+\tau H_kv_k)^Tp_k+\frac12p_k^TH_kp_k+\tau g_k^Tv_k+\frac12\tau^2v_k^TH_kv_k\\
        =&(g_k+\tau H_kv_k)^Tp_k+\frac12p_k^TH_kp_k+\tau\left(-\chi_k^g+\frac12\tau v_k^TH_kv_k\right)\\
        \leq&\left(\|g_k\|+\frac12M_H\beta_l\right)\kappa_u\phi_k+\frac12M_H\kappa_u^2\phi_k^2+\tau\left(-\chi_k^g+\frac18\tau M_H\beta_l^2\right)\\
        \leq& \left(\|g_k\|+\frac12M_H(\beta_l+\kappa_uM_f)\right)\kappa_u\phi_k
    +\tau\left(-\chi_k^g+\frac18\tau M_H\beta_l^2\right),
    \end{split}
\end{equation*}
where the last inequality is due to \eqref{eq:bnd_fns}.
If $\frac{4\chi_k^g}{M_H\beta_l^2}\geq 1$, then let $\tau=1$ and we have that
\begin{equation}\label{eq:bnd_gk2}\begin{split}
   &g_k^Td_k+\frac12d_k^TH_kd_k\\ \leq&\left(\|g_k\|+\frac12M_H(\beta_l+\kappa_uM_f)\right)\kappa_u\phi_k
    -\chi_k^g+\frac18\frac{4\chi_k^g}{M_H\beta_l^2}M_H\beta_l^2\\
    =&\left(\|g_k\|+\frac12M_H(\beta_l+\kappa_uM_f)\right)\kappa_u\phi_k-\frac12\chi_k^g\\
    \leq&\left(\|g_k\|+\frac12M_H(\beta_l+\kappa_uM_f)\right)\kappa_u\phi_k-\frac18M_H\beta_l^2.
\end{split}\end{equation}
Otherwise, by letting $\tau=\frac{4\chi_k^g}{M_H\beta_l^2}$ we have
\begin{equation}\label{eq:bnd_gk3}\begin{split}
    &g_k^Td_k+\frac12d_k^TH_kd_k\\
    \leq&\left(\|g_k\|+\frac12M_H(\beta_l+\kappa_uM_f)\right)\kappa_u\phi_k
    +\frac{4\chi_k^g}{M_H\beta_l^2}\left(-\chi_k^g+\frac18\frac{4\chi_k^g}{M_H\beta_l^2}M_H\beta_l^2\right)\\
    =&\left(\|g_k\|+\frac12M_H(\beta_l+\kappa_uM_f)\right)\kappa_u\phi_k-\frac{2(\chi_k^g)^2}{M_H\beta_l^2}.
\end{split}\end{equation}
Combing \eqref{eq:bnd_gk2} and \eqref{eq:bnd_gk3}, we have that
\begin{equation}\label{eq:predg2}
\begin{split}
&g_k^Td_k+\frac12d_k^TH_kd_k \\
\leq&\left(\|g_k\|+\frac12M_H(\beta_l+\kappa_uM_f)\right)\kappa_u\phi_k-\min\left\{\frac18M_H\beta_l^2,\frac{2(\chi_k^g)^2}{M_H\beta_l^2}\right\}.
\end{split}\end{equation}

{But from \eqref{eq:predg2} and Assumption (A4), we have $g_k^Td_k\leq g_k^Td_k+\frac12d_k^TH_kd_k$. Thus, using the definition of $\Delta_{\Psi}^g$ (cf. \eqref{eq:defdeltapsi}) together with Lemma \ref{lem:localu}, we obtain}
\begin{eqnarray*}    \lefteqn{\Delta_{\Psi}^g(x_k,d_k;\rho_{k+1})\geq-g_k^Td_k+\rho_{k+1}\tilde\tau_1\phi_k}\nonumber\\
    &&{\geq}\left(\rho_{k+1}\tilde\tau_1-\left(\|g_k\|+\frac12M_H(\beta_l+\kappa_uM_f)\right)\kappa_u\right)\phi_k+\min\left\{\frac18M_H\beta_l^2,\frac{2(\chi_k^g)^2}{M_H\beta_l^2}\right\}.\nonumber
\end{eqnarray*}
\end{proof}

{From \eqref{eq:predg2} in the proof of Lemma \ref{lem:bnd_dpsi}, we can derive the next Lemma concerning the update of the penalty parameter.}
\begin{lemma}\label{lem:pen_updt}
{Consider any realization $\{x_k\}$ of Algorithm 2 such} that Assumptions \ref{ass1}, \ref{ass2} and \ref{ass3}
hold and that $\tilde x$ is an accumulation point of $\{x_k\mid\ k\in\mathbf{S}\}$. Let $\tilde{\mathcal{N}}$ and $M_1>0$
are the same as given in Lemma \ref{lem:bnd_dpsi}. Then
\begin{equation}\label{eq:pen_updt}
    \Delta_{\Psi}^g(x_k,d_k;\rho_k)-\frac12d_k^TH_kd_k\geq\left(\tilde\tau_1\rho_{k}- \left(\|g_k\|+\frac12M_H(\beta_l+\kappa_uM_f)\right)\kappa_u\right)\phi_k
\end{equation}
for all $x_k\in\tilde{\mathcal{N}}$.
\end{lemma}
\begin{proof}
By the definition of $\Delta_{\Psi}^g$ (cf. \eqref{eq:defdeltapsi}), we have
\begin{equation*}
\begin{split}
    &\Delta_{\Psi}^g(x_k,d_k;\rho_k)-\frac12d_k^TH_kd_k\\
    =&-g_k^Td_k-\frac12d_k^TH_kd_k+\rho_k\Delta_\phi(x_k,\sigma_k)\\
    \stackrel{(a)}{\geq}& \left(\tilde\tau_1\rho_{k}- \left(\|g_k\|+\frac12M_H(\beta_l+\kappa_uM_f)\right)\kappa_u\right)\phi_k+\min\left\{\frac18M_H\beta_l^2,\frac{2(\chi_k^g)^2}{M_H\beta_l^2}\right\}\\
    \geq& \left(\tilde\tau_1\rho_{k}- \left(\|g_k\|+\frac12M_H(\beta_l+\kappa_uM_f)\right)\kappa_u\right)\phi_k,
\end{split}\end{equation*}
where (a) follows from Lemma \ref{lem:localu} and \eqref{eq:predg2}.
\end{proof}

\subsection{Guarantees with approximately accurate estimates}
The main purpose of this subsection is to show the well-definedness of Algorithm \ref{ssqp}, estimate of the difference between $\Psi(x_{k+1};\rho_{k+1})$ and $\Psi(x_{k};\rho_{k+1})$ and some other useful results.
\begin{lemma}\label{lem:fs}
{Consider any realization $\{x_k\}$ of Algorithm 2 such that Assumption \ref{ass1} holds and that for some $k$,} $g_k$ is $\varepsilon_g$-sufficiently accurate and $\{f_k^0,f_k^s\}$ are $\varepsilon_f$-sufficiently accurate. Then
\begin{equation*}
    f_k^s\leq f_k^0+\alpha_kg_k^Td_k+\alpha_k^2\left(\varepsilon_g+\frac12L_f+2\varepsilon_f\right)\|d_k\|^2.
\end{equation*}
\end{lemma}
\begin{proof}
Using Lipschitz continuity of $\nabla f$ and Definition \ref{def:gaccu}, we have
\begin{eqnarray*}
\lefteqn{f(x_k+\alpha_kd_k)\leq f(x_k)+\alpha_k\nabla f_k^Td_k+\frac12L_f\alpha_k^2\|d_k\|^2}\\
&& = f(x_k)+\alpha_k(\nabla f_k-g_k)^Td_k+\alpha_kg_k^Td_k+\frac12L_f\alpha_k^2\|d_k\|^2\\
&&\leq f(x_k)+\alpha_k\|\nabla f_k-g_k\|\|d_k\|+\alpha_kg_k^Td_k+\frac12L_f\alpha_k^2\|d_k\|^2\\
&&\leq f(x_k)+\alpha_k^2\varepsilon_g\|d_k\|^2+\alpha_kg_k^Td_k+\frac12\alpha_k^2\|d_k\|^2\\
&&=f(x_k)+\alpha_kg_k^Td_k+\alpha_k^2\left(\varepsilon_g+\frac12L_f\right)\|d_k\|^2.
\end{eqnarray*}
Since $\{f_k^0,f_k^s\}$ are $\varepsilon_f$-accurate, we have
\begin{eqnarray*}
\lefteqn{f_k^s\leq f(x_k+\alpha_kd_k)+\varepsilon_f\alpha_k^2\|d_k\|^2}\\
&&\leq f(x_k)+\alpha_kg_k^Td_k+\alpha_k^2\left(\varepsilon_g+\frac12L_f+\varepsilon_f\right)\|d_k\|^2\\
&&=f(x_k)-f_k^0+f_k^0+\alpha_kg_k^Td_k+\alpha_k^2\left(\varepsilon_g+\frac12L_f+\varepsilon_f\right)\|d_k\|^2\\
&&\leq f_k^0+\|f(x_k)-f_k^0\|+\alpha_kg_k^Td_k+\alpha_k^2\left(\varepsilon_g+\frac12L_f+\varepsilon_f\right)\|d_k\|^2\\
&&\leq f_k^0+\alpha_kg_k^Td_k+\alpha_k^2\left(\varepsilon_g+\frac12L_f+2\varepsilon_f\right)\|d_k\|^2.
\end{eqnarray*}
\end{proof}

\begin{lemma}\label{lem:bndalp}
{Consider any realization $\{x_k\}$ of Algorithm 2 such} that Assumption \ref{ass1} holds and that {for some $k$}, $g_k$ is $\varepsilon_g$-sufficiently accurate and $\{f_k^0,f_k^s\}$ are $\varepsilon_f$-sufficiently accurate. If it holds that,
\[
\alpha_k\leq\frac{(1-\theta)\eta}{\varepsilon_g+\frac12L_f+2\varepsilon_f+\frac12\rho_{k+1}L_{hc}},
\]
then the $k$th step is successful. {That is, for any realization and all $k$, these conditions are incompatible with an unsuccessful step.}
\end{lemma}
\begin{proof}
By Lemmas \ref{lem:phi} and \ref{lem:fs}, we have that the stochastic merit function~\eqref{eq:stochmerit} satisfies,
\begin{eqnarray*}
\lefteqn{\Psi_k^s(\rho_{k+1})=f_k^s+\rho_{k+1}\phi(x_k+\alpha_kd_k)}\\
&&\leq f_k^0+\alpha_kg_k^Td_k+\alpha_k^2\left(\varepsilon_g+\frac12L_f+2\varepsilon_f\right)\|d_k\|^2\\
&&\ \ \ \  +\rho_{k+1}\left(\phi_k-\alpha_k{\Delta_\phi(x_k,\sigma_k)}+\alpha_k^2\frac12L_{hc}\|d_k\|^2\right)\\
&&=f_k^0+\rho_{k+1}\phi_k-\alpha_k(-g_k^Td_k+\rho_{k+1}\Delta_{\phi}(x_k,\sigma_k))\\
&&\ \ \ \  +\alpha_k^2\left(\varepsilon_g+\frac12L_f+2\varepsilon_f+\frac12\rho_{k+1}L_{hc}\right)\|d_k\|^2\\
&&=\Psi_k^0(\rho_{k+1})-\alpha_k{\Delta_\Psi^g}(x_k,d_k;\rho_{k+1})+\alpha_k^2\left(\varepsilon_g+\frac12L_f+2\varepsilon_f+\frac12\rho_{k+1}L_{hc}\right)\|d_k\|^2\\
&&=\Psi_k^0(\rho_{k+1})-\theta\alpha_k{\Delta_\Psi^g}(x_k,d_k;\rho_{k+1})-(1-\theta)\alpha_k{\Delta_\Psi^g}(x_k,d_k;\rho_{k+1})\\
&&\ \ +\alpha_k^2\left(\varepsilon_g+\frac12L_f+2\varepsilon_f+\frac12\rho_{k+1}L_{hc}\right)\|d_k\|^2\\
&&\leq \Psi_k^0(\rho_{k+1})-\theta\alpha_k{\Delta_\Psi^g}(x_k,d_k;\rho_{k+1})-(1-\theta)\eta\alpha_k\|d_k\|^2\\
&&\ \ \ \ +\alpha_k^2\left(\varepsilon_g+\frac12L_f+2\varepsilon_f+\frac12\rho_{k+1}L_{hc}\right)\|d_k\|^2,
\end{eqnarray*}
where the last inequality is due to \eqref{eq:s_condpen2} and Assumption (A4).
Thus, the $k$th step is successful for any step size satifying
\[
\alpha_k\leq\frac{(1-\theta)\eta}{\varepsilon_g+\frac12L_f+2\varepsilon+\frac12\rho_{k+1}L_{hc}}.
\]
\end{proof}

\begin{lemma}
{Consider any realization $\{x_k\}$ of Algorithm 2. It holds that for all $k$},
\begin{equation}\label{eq:bnd_chi}
    \chi_k\leq\frac{\sqrt{n}M_d\beta_l}2.
\end{equation}
Furthermore, if $g_k$ is $\varepsilon_g$-sufficiently accurate, then it holds that
\begin{eqnarray}
\|g_k\|\leq M_d+\varepsilon_g\alpha_{\max}\beta_u,\label{eq:bnd_g}\\
\chi_k^g\leq \sqrt{n}(M_d+\varepsilon_g\alpha_{\max}\beta_u)\beta_u.\label{eq:bnd_chig1}
\end{eqnarray}
and there exists a constant $\kappa_1>0$ such that \begin{equation}\label{eq:bnd_chig2}
 (\chi_k^g)^2\geq\frac12\chi_k^2-\kappa_1\Delta_{\Psi}^g(x_k,d_k;\rho_{k+1}).  \end{equation}
\end{lemma}
\begin{proof}
The bound of $\chi_k$ is directly followed from \eqref{eq:bnd_grds}, \eqref{eq:lpchi} and the well-known equivalence between $\ell_2$ norm and $\ell_\infty$ norm (see, for example, \cite{Sun2006}).

For a $\varepsilon_g$-sufficiently accurate estimate $g_k$, it follows from Definition \ref{def:gaccu}, the bound constraint in \eqref{eq:s_subp0} and the rules for updating $\alpha_k$ in Algorithm \ref{ssqp} that
\begin{equation*}
\|g_k\|\leq\|\nabla f_k\|+\varepsilon_g\alpha_k\|d_k\|\leq M_d+n\varepsilon_g\alpha_{\max}\beta_u.
\end{equation*}
Observe that \eqref{eq:bnd_chig1} immediately follows because
$\chi_k^g\leq\|g_k\|\|v_k\|$.

By \eqref{eq:condpen2} and Assumption (A4), we have
\begin{equation}\label{eq:dlepred}
    \|d_k\|^2\leq\eta^{-1}{\Delta_\Psi^g}(x_k,d_k;\rho_{k+1}).
\end{equation}
Corollary \ref{thm:glc} together with the property of $\varepsilon_g$-sufficient accuracy implies that
\begin{equation*}
    \chi_k\leq\chi_k^g+L_{\chi}\varepsilon_g\alpha_{\max}\|d_k\|.
\end{equation*}
Squaring both sides yields the following inequality:
\begin{equation*}
        (\chi_k)^2=(\chi_k^g+L_\chi\varepsilon_g\alpha_{\max}\|d_k\|)^2\leq 2(\chi_k^g)^2+2L_{\chi}^2\varepsilon_g^2\alpha_{\max}^2\|d_k\|^2,
\end{equation*}
which implies that
\begin{equation*}
(\chi_k^g)^2\geq\frac12\chi_k^2-L_{\chi}^2\varepsilon_g^2\alpha_{\max}^2\|d_k\|^2\geq \frac12\chi_k^2
-L_{\chi}^2\varepsilon_g^2\alpha_{\max}^2\eta^{-1}{\Delta_\Psi^g}(x_k,d_k;\rho_{k+1}).
\end{equation*}
Hence the Lemma is true with
$
\kappa_1:=L_{\chi}^2\varepsilon_g^2\alpha_{\max}^2\eta^{-1}.
$
\end{proof}

\begin{lemma}
{Consider any realization $\{x_k\}$ of Algorithm 2.} Suppose that {for some $k$} $\{f_k^0,f_k^s\}$ are $\varepsilon_f$-accurate estimates and that $\rho_{k+1}$ satisfies \eqref{eq:s_condpen2}.  Suppose also that $\varepsilon_f<\frac{\eta\theta}{4}$. If the step is successful, then the improvement in the value of the merit function across iterations satisfies
\begin{equation}\label{eq:realf_red}
    \Psi(x_{k+1};\rho_{k+1})\leq\Psi(x_k;\rho_{k+1})-\frac\theta2\alpha_k{\Delta_\Psi^g}(x_k,d_k;\rho_{k+1}).
\end{equation}
\end{lemma}
\begin{proof}

Under the assumptions of this Lemma, we can derive the following sequence of inequalities:
\begin{equation*}
    \begin{split}        &\Psi(x_{k+1};\rho_{k+1})=f(x_{k+1})+\rho_{k+1}\phi(x_{k+1})\\
        =&f(x_{k+1})-f_k^s+f_k^s+\rho_{k+1}\phi(x_{k+1})\\
        =&f(x_{k+1})-f_k^s+\Psi_{k}^s(\rho_{k+1})\\
        \leq&f(x_{k+1})-f_k^s+\Psi_k^0(\rho_{k+1})-\theta\alpha_k{\Delta_\Psi^g}(x_k,d_k;\rho_{k+1})\\
        \leq&f(x_{k+1})-f_k^s+f_k^0-f_k+f_k+\rho_{k+1}\phi(x_k)-\theta\alpha_k{\Delta_\Psi^g}(x_k,d_k;\rho_{k+1})\\
        \leq&2\varepsilon_f\alpha_k\|d_k\|^2+\Psi(x_k;\rho_{k+1})-\theta\alpha_k{\Delta_\Psi^g}(x_k,d_k;\rho_{k+1})\\
        \leq&\Psi(x_k;\rho_{k+1})+2\varepsilon_f\alpha_k\eta^{-1}{\Delta_\Psi^g}(x_k,d_k;\rho_{k+1})-\theta\alpha_k{\Delta_\Psi^g}(x_k,d_k;\rho_{k+1})\\
        \leq&\Psi(x_k;\rho_{k+1})-\alpha_k(\theta-2\varepsilon_f\eta^{-1}){\Delta_\Psi^g}(x_k,d_k;\rho_{k+1})
    \end{split}
\end{equation*}
where~\eqref{eq:dlepred} is applied at the second to last line.
Then \eqref{eq:realf_red} follows if $\varepsilon_f<\frac{\eta\theta}{4}$.
\end{proof}

\begin{lemma}\label{lem:bnd_ldpg}
{Consider any realization $\{x_k\}$ of Algorithm 2 such that Assumptions \ref{ass1} and \ref{ass2} hold. Suppose for some $k$ } that $g_k$ is $\varepsilon_g$-sufficiently accurate. Then there exist positive constants $\kappa_3$ and $\kappa_4$ such that
\begin{equation*}
    \Delta_{\Psi}^g(x_k,d_k;\rho_{k+1})\geq\kappa_3(\tilde\tau_1\rho_{k+1}-\kappa_2)\phi_k^2+\kappa_4\chi_k^2.
\end{equation*}
\end{lemma}
\begin{proof}
Denote
\[
\kappa_2=\left(M_d+\varepsilon_g\alpha_{\max}\beta_u+\frac12M_H(\beta_l+\kappa_uM_f)\right)\kappa_u.
\]
It follows from \eqref{eq:bnd_dpsi}  and \eqref{eq:bnd_g} that
\begin{equation*}
        \Delta_{\Psi}^g(x_k,d_k;\rho_{k+1})\geq \left(\tilde\tau_1\rho_{k+1}-\kappa_2 \right)\phi_k+\min\left\{\frac18M_H\beta_l^2,\frac{2(\chi_k^g)^2}{M_H\beta_l^2}\right\}.
\end{equation*}
{Now we consider two cases. \emph{Case 1:}} If $\frac18M_H\beta_l^2\leq\frac{2(\chi_k^g)^2}{M_H\beta_l^2}$, it follows from \eqref{eq:bnd_chi} that
\begin{equation*}
    \begin{split}
        \Delta_{\Psi}^g(x_k,d_k;\rho_{k+1})
        \geq&(\tilde\tau_1\rho_{k+1}-\kappa_2)\phi_k+\frac18M_H\beta_l^2\\
        =&(\tilde\tau_1\rho_{k+1}-\kappa_2)\phi_k+\frac18M_H\beta_l^2\frac{\chi_k^2}{\chi_k^2}\\
        \geq&(\tilde\tau_1\rho_{k+1}-\kappa_2)\phi_k+\frac18M_H\beta_l^2\frac{4\chi_k^2}{nM_d^2\beta_l^2}\\
        =&(\tilde\tau_1\rho_{k+1}-\kappa_2)\phi_k+\frac{M_H}{2n M_d^2}\chi_k^2\\
        \geq&\frac{(\tilde\tau_1\rho_{k+1}-\kappa_2)}{M_f}\phi_k^2+\frac{M_H}{2nM_d^2}\chi_k^2.
    \end{split}
\end{equation*}
{\emph{Case 2:} Otherwise, i.e. $\frac18M_H\beta_l^2>\frac{2(\chi_k^g)^2}{M_H\beta_l^2}$}, then we have from \eqref{eq:bnd_chig2} that
\begin{equation*}
    \Delta_{\Psi}^g(x_k,d_k;\rho_{k+1})\geq(\tilde\tau_1\rho_{k+1}-\kappa_2)\phi_k+\frac{\chi_k^2}{M_H\beta_l^2}-\frac{2\kappa_1\Delta_{\Psi}^g(x_k,d_k;\rho_{k+1})}{M_H\beta_l^2}.
\end{equation*}
By moving the third term to the left hand side, we obtain
\[
\left(
1+\frac{2\kappa_1}{M_H\beta_l^2}
\right)\Delta_{\Psi}^g(x_k,d_k;\rho_{k+1})\geq (\tilde\tau_1\rho_{k+1}-\kappa_2)\phi_k+\frac{\chi_k^2}{M_H\beta_l^2},
\]
{Dividing both sides by the positive constant on the left hand-side yields}
\begin{equation*}
\begin{split}
    \Delta_{\Psi}^g(x_k,d_k;\rho_{k+1})
    \geq&\frac{M_H\beta_l^2(\tilde\tau_1\rho_{k+1}-\kappa_2)}{M_H\beta_l^2+2\kappa_1}\phi_k+\frac{\chi_k^2}{M_H\beta_l^2+2\kappa_1}\\
    \geq& \frac{M_H\beta_l^2(\tilde\tau_1\rho_{k+1}-\kappa_2)}{M_f(M_H\beta_l^2+2\kappa_1)}\phi_k^2+\frac{\chi_k^2}{M_H\beta_l^2+2\kappa_1},
\end{split}\end{equation*}
where the second inequality is due to the fact $\phi(x)\leq M_f$, which is a direct consequence of the definition of $\phi(x)$ (cf. Eq. \eqref{eq:defphi}) and the bound \eqref{eq:bnd_fns}.
Let
\[
\kappa_3=\min\left\{\frac1{M_f},\frac{M_H\beta_l^2}{M_f(M_H\beta_l^2+2\kappa_1)}\right\},
\kappa_4=\min\left\{\frac{M_H}{2nM_d^2},\frac{1}{M_H\beta_l^2+2\kappa_1}\right\}.
\]
Then the conclusion holds with $\kappa_3$ and $\kappa_4$ given above.
\end{proof}

\begin{lemma}
{Consider some realization $\{x_k\}$. Suppose that for some $k$} $\{f_k^0,f_k^s\}$ are $\varepsilon_f$-accurate estimates, with $\varepsilon_f<\frac{\eta\theta}{4}$, that $g_k$ is $\varepsilon_g$ sufficiently accurate and that $\rho_{k+1}$ satisfies \eqref{eq:s_condpen2}. If the step is successful, then the improvement in merit function value is
\begin{equation}\label{eq:realf_red2}
\begin{split}
    &\Psi(x_{k+1};\rho_{k+1})\\
    \leq&\Psi(x_k;\rho_{k+1})-\frac{\theta\alpha_k}{4}\Delta_{\Psi}^g(x_k,d_k;\rho_{k+1})
    -\frac{\theta\alpha_k}4(\kappa_3(\tilde\tau_1\rho_{k+1}-\kappa_2)\phi_k^2+\kappa_4\chi_k^2).\end{split}
\end{equation}
\end{lemma}
\begin{proof}
It follows from \eqref{eq:realf_red} and Lemma \ref{lem:bnd_ldpg} that
\begin{equation*}\begin{split}
    &\Psi(x_{k+1};\rho_{k+1})\leq \Psi(x_k;\rho_{k+1})-\frac{\theta\alpha_k}{2}\Delta_{\Psi}^g(x_k,d_k;\rho_{k+1})\\
    \leq& \Psi(x_k;\rho_{k+1})-\frac{\theta\alpha_k}{4}\Delta_{\Psi}^g(x_k,d_k;\rho_{k+1})-\frac{\theta\alpha_k}4(\kappa_3(\tilde\tau_1\rho_{k+1}-\kappa_2)\phi_k^2+\kappa_4\chi_k^2).\end{split}
\end{equation*}
\end{proof}

\begin{lemma}\label{lem:lipchi}
{Consider some realization $\{x_k\}$. Suppose that for some $k$} the step is successful. Then
\begin{equation*}
    \begin{split}
        &\chi(x_{k+1})^2\leq 2\left(\chi(x_k)^2+L_{\chi}^2\alpha_k^2\eta^{-1}\Delta_{\Psi}^g(x_k,d_k;\rho_{k+1})\right),\\
        &\phi(x_{k+1})^2\leq2(\phi(x_k)^2+M_d^2\alpha_k^2\eta^{-1}\Delta_{\Psi}^g(x_k,d_k;\rho_{k+1})).
    \end{split}
\end{equation*}
\end{lemma}
\begin{proof}
It follows from the Lipschitz continuity of $\chi(x;\rho)$ that
\[\chi(x_{k+1})\leq \chi(x_k)+L_{\chi}\alpha_k\|d_k\|.
\]
Then by squaring both sides and applying the bound $(a+b)^2\leq 2(a^2+b^2)$, we obtain that
\begin{equation*}
    \begin{split}
        \chi(x_{k+1})^2\leq&(\chi(x_k)+L_{\chi}\alpha_k\|d_k\|)^2\\
        \leq&2(\chi(x_k)^2+L_{\chi}^2\alpha_k^2\|d_k\|^2)\\
        \leq&2\left(\chi(x_k)^2+L_{\chi}^2\alpha_k^2\eta^{-1}\Delta_{\Psi}^g(x_k,d_k;\rho_{k+1})\right),
    \end{split}
\end{equation*}
where \eqref{eq:dlepred} is used in the last inequality.

Consider now
\[
\lvert\phi(x_{k+1})-\phi(x_k)\rvert\leq M_d\alpha_k\|d_k\|
\]
and apply the same derivation as for $\chi(x_{k+1})^2$ to obtain
\[\phi(x_{k+1})^2\leq2(\phi(x_k)^2+M_d^2\alpha_k^2\eta^{-1}\Delta_{\Psi}^g(x_k,d_k;\rho_{k+1})).\]
\end{proof}

\subsection{Convergence with bounded penalty parameters}
In this section, we shall consider convergence in the case where the sequence of gradient estimates $\{g_k\}$ is bounded. In this case, one can know from the mechanism of Algorithm \ref{ssqp} that both $\{\zeta_k\}$ and $\{p_k^g\}$ will eventually remain constant.
Without loss of generality, we assume in this section that $\zeta_k=\bar\zeta$ and $p_k^g=p_g$ for all $k\geq 0$.
\begin{lemma}
{Consider a realization $\{\hat{x}_k(\omega_k)=x_k\}$ of Algorithm 2 such that Assumptions \ref{ass1} and \ref{ass3} hold and that, furthermore, there is a positive constant $M_g(\Xi)$ such that $\|g_k\|\leq M_g(\Xi)$ for all $k$. Then eventually, there is $\rho_k=\rho^*(\Xi)$ for some $\rho^*(\Xi)>0$}.
\end{lemma}
\begin{proof}
{Fix $\Xi$}. By Lemma \ref{lem:pen_updt} and compactness of $\{x_k : k\in\mathbf{S}\}\cup \{\bar x : \exists \mathcal{K}\subset \mathbb{N},\,x_{k_l\in\mathcal{K}}\to \bar x\},$, one can deduce that \eqref{eq:pen_updt} holds with a  $\tilde x(\Xi)$ for any accumulation point of $\{x_k: k\in\mathbf{S}\}$. Therefore, the condition \eqref{eq:s_condpen} holds if
\begin{equation*}
    \rho_k\geq\frac{\left(M_g+\frac12 M_H(\beta_l+\kappa_uM_f)\right)\kappa_u}{\tilde\tau_1},
\end{equation*}
which implies that $\rho_k$ will eventually remain constant. {Note that this depends on $\Xi$ as the constants $M_g$, etc. which ultimately arise from Assumption~\ref{ass1} depend on $\Xi$.}
\end{proof}

We are going to apply the techniques in \cite{paquette2020stochastic} to our stochastic SQP framework to prove its convergence properties.


The key point of the arguments to be presented is the expected decrease of the following random merit function
\begin{equation}\label{eq:phi}
    \Phi(\hat x_k)=\nu(\Psi(\hat x_k;\bar\rho)-\Psi_{\min})+(1-\nu)\hat\alpha_k(\phi(\hat x_k)^2+\chi(\hat x_k)^2)
\end{equation}
with some $\nu\in(0,1)$, $\bar\rho>0$ and $\Psi_{\min}\leq\Psi(x;\bar\rho)$ for all $x\in\Omega$.
Note that \eqref{eq:phi} is only for theoretical use and is never evaluated in the algorithms introduced before.
The parameter $\bar\rho$ is chosen so that
\[
\bar\rho(\Xi)\geq\max\{\rho^*(\Xi),\kappa_2/\tilde\tau_1\}.
\]
It can be seen that all the results established in the previous Sections still hold if $\rho_{k+1}$ is replaced by $\bar\rho$.

We are going to show that the expected decrease in $\Phi$, with proper
$\nu, p_f,p_g>0$, can be bounded below by a value proportional to $\phi(\hat x_k)^2+\chi(\hat x_k)^2$. {Note that we have to make an additional assumption that the bounded penalty parameter is uniformly bounded w.p.1. This will ultimately depend on if the iterates are bounded with a uniform bound on dense measure, or otherwise that event, which ensures the Assumptions in the Theorem}.

\begin{theorem}\label{thm:redinex}
Suppose that Assumptions \ref{ass1}, \ref{ass2} and \ref{ass3} hold {for a set of dense realizations $\{\hat{x}(\omega_k)\}$, and moreover the iterates are uniformly bounded, implying a uniform $M_f$ and $M_g$}. {Note that this implies that there exists $\bar{\rho}<\infty$ such that $\bar{\rho}\ge \rho^*(\Xi)$ for a.e. $\Xi$}. Then there exist probabilities $p_g, p_f$ and a constant $\nu\in(0,1)$ such that the expected decrease in $\Phi_k$ almost surely satisfies
\begin{equation}\label{eq:bnd_dphi_e}
    \textbf{E}[\Phi(\hat x_k)-\Phi(\hat x_{k+1})\mid\mathcal{F}_{k-1}]\geq
    p_fp_g(1-\nu)(1-\gamma^{-1})\hat\alpha_k(\hat\phi(\hat x_k)^2+\chi(\hat x_k)^2).
\end{equation}
\end{theorem}
\begin{proof}
The proof to be given is structured along  similar arguments as the proof of Theorem 4.6 in \cite{paquette2020stochastic}. We consider three separate cases with respect to step success and accuracy and establish bounds for the change in $\Phi_k$ for each case, and then add them together weighed by their probabilities of occurrence.

Case 1 (accurate gradients and estimates, $\textbf{1}_{I_k\cap J_k}=1$).

\begin{enumerate}[(i)]
    \item \emph{Successful step ($\textbf{1}_{\mathbf{S}}=1$).}
    In this case, a decrease in $\Psi(x;\bar\rho)$ occurs as \eqref{eq:realf_red2}.

    As the iterate changes, the term $\hat\alpha_k\chi(\hat x_k)^2$ may increase. It follows from Lemma \ref{lem:lipchi} that
    \begin{equation}\label{eq:chi_case1i}
        \begin{split}
            &\hat\alpha_k\chi(\hat x_k)^2-\hat\alpha_{k+1}\chi(\hat x_{k+1})^2\\
            =&\hat\alpha_k\chi(\hat x_k)^2-\gamma\hat\alpha_k\chi(\hat x_{k+1})^2\\
            \geq&\hat\alpha_k\chi(\hat x_k)^2-2\gamma\hat\alpha_k(\chi(x_k)^2+L_{\chi}^2\hat\alpha_k^2\eta^{-1}\Delta_{\Psi}^g(\hat x_k,\hat d_k;\bar\rho))\\
            \geq&-2\gamma L_{\chi}^2\alpha_{\max}^2\hat\alpha_k\eta^{-1}\Delta_{\Psi}^g(\hat x_k,\hat d_k;\bar\rho)-(2\gamma-1)\hat\alpha_k\chi(\hat x_k)^2
        \end{split}
    \end{equation}
    and, similarly,
    \begin{equation}\label{eq:phi_case1i}
    \begin{split}
        &\hat\alpha_k\phi(\hat x_k)^2-\hat\alpha_{k+1}\phi(\hat x_{k+1})^2\\
        \geq& -2\gamma M_d^2\alpha_{\max}^2\hat\alpha_k\eta^{-1}\Delta_{\Psi}^g(\hat x_k,\hat d_k;\bar\rho)
        -(2\gamma-1)\hat\alpha_k\phi(\hat x_k)^2.
        \end{split}
    \end{equation}

We choose $\nu$ satisfying
 \begin{equation}\label{eq:bnd_nu2}
     \begin{split}
         \nu\geq\max\left\{\frac{2\gamma(L_{\chi}^2+M_d^2)\alpha_{\max}^2}{\frac{\theta\eta}{8}+2\gamma(L_{\chi}^2+M_d^2)\alpha_{\max}^2},\frac{1}{\frac{\theta\kappa_3(\tilde{\tau_1}\bar\rho-\kappa_2)}{8(2\gamma-1)}+1},\frac{1}{\frac{\theta\kappa_4}{8(2\gamma-1)}+1}\right\}.
     \end{split}
 \end{equation}
Then, by combining \eqref{eq:realf_red2}, \eqref{eq:chi_case1i}and \eqref{eq:phi_case1i}, we conclude that
 \begin{equation*}\begin{split}
     &\Phi(\hat x_k)-\Phi(\hat x_{k+1})\\
     \geq&\frac{\nu\theta}{8}\hat\alpha_k\Delta_{\Psi}^g(\hat x_k,\hat d_k;\bar\rho)+\frac{\nu\theta}{8}\hat\alpha_k\kappa_3(\tilde\tau_1\bar\rho-\kappa_2)\phi(\hat x_k)^2+\frac{\nu\theta}{8}\kappa_4\hat\alpha_k\chi(\hat x_k)^2\\
     \geq&\frac{\nu\theta}{8}\hat\alpha_k\Delta_{\Psi}^g(\hat x_k,\hat d_k;\bar\rho)+\frac{\nu\theta}{8}\hat\alpha_k\tilde\kappa(\phi(\hat x_k)^2+\chi(\hat x_k)^2),
 \end{split}\end{equation*}
where $\tilde\kappa=\min\{\kappa_3(\tilde\tau_1\bar\rho-\kappa_2),\kappa_4\}$.


Taking conditional expectations with respect to $\mathcal{F}_{k-1}$ and using Assumption \ref{ass2}, we have
\begin{equation*}
\begin{split}
   &\textbf{E}[\textbf{1}_{I_k\cap J_k}\textbf{1}_{\mathbf{S}}(\Phi(\hat x_k)-\Phi(\hat x_{k+1}))]\stackrel{a.s.}{=}\textbf{1}_{\mathbf{S}}\textbf{E}[\textbf{1}_{I_k\cap J_k}(\Phi(\hat x_k)-\Phi(\hat x_{k+1}))]\\
    \geq&\textbf{1}_{\mathbf{S}}\left[ p_fp_g\left(\frac{\nu\theta}{8}\hat\alpha_k\Delta_{\Psi}^g(\hat x_k,\hat d_k;\bar\rho)+\frac{\nu\theta}{8}\hat\alpha_k\tilde\kappa(\phi(\hat x_k)^2+\chi(\hat x_k)^2)\right)\right]\\
    =& p_fp_g\frac{\nu\theta}{8}\textbf{1}_{\mathbf{S}}\hat\alpha_k\Delta_{\Psi}^g(\hat x_k,\hat d_k;\bar\rho)+ p_fp_g\frac{\nu\theta}{8}\textbf{1}_{\mathbf{S}}\hat\alpha_k\tilde\kappa(\phi(\hat x_k)^2+\chi(\hat x_k)^2)
\end{split}
\end{equation*}
where the first equality follows from \cite[Theorem 8.14(iii)]{Klenke2013} and  the ``a.s.'' stands for ``almost surely''.

 \item \emph{Unsuccessful step. $(\textbf{1}_{{\mathbf{S}}^c}=1)$.} In this case, $\hat x_{k+1}=\hat x_k$ and $\hat\alpha_{k+1}=\gamma^{-1}\hat\alpha_k$. Consequently, we conclude that
 \begin{eqnarray*}
 \Psi(\hat x_k;\bar\rho)-\Psi(\hat x_{k+1};\bar\rho)=0,\\
 \hat\alpha_k \chi(\hat x_k)^2-\hat\alpha_{k+1}\chi(\hat x_{k+1})^2=(1-\gamma^{-1})\hat\alpha_k\chi(\hat x_k)^2,\\
 \hat\alpha_k \phi(\hat x_k)^2-\hat\alpha_{k+1}\phi(\hat x_{k+1})^2=(1-\gamma^{-1})\hat\alpha_k\phi(\hat x_k)^2,
 \end{eqnarray*}
 which yields
 \begin{equation}\label{eq:decr_phi_us}
     \Phi(\hat x_k)-\Phi(\hat x_{k+1})=(1-\nu)(1-\gamma^{-1})\hat\alpha_k(\phi(\hat x_k)^2+\chi(\hat x_k)^2).
 \end{equation}
\end{enumerate}




Case 2 (inaccurate gradients and accurate function evaluations, $\textbf{1}_{I_k^c\cap J_k}=1$)
In this case, the value of $\Phi(\hat x_k)$ may increase even if the step is successful, since decrease along the step obtained from the model \eqref{eq:s_subp0} with an inaccurate gradient estimate $g_k$ may not be bounded as in Lemma \ref{lem:bnd_ldpg}. As a result, the decrease from $\Psi$ may not dominate the increase from $\chi(x)$.

\begin{enumerate}[(i)]
    \item \emph{Successful steps ($\textbf{1}_{\mathbf{S}}=1.$)}

    It follows from \eqref{eq:realf_red}, \eqref{eq:chi_case1i} and \eqref{eq:phi_case1i} that, for $\nu$
    satisfying \eqref{eq:bnd_nu2},
    \begin{equation*}
    \begin{split}
        &\Phi(\hat x_k)-\Phi(\hat x_{k+1})\\
        \geq&\frac{\nu\theta}8\hat\alpha_k\Delta_{\Psi}^g(\hat x_k,d_k;\bar\rho)-(1-\nu)(2\gamma-1)\hat\alpha_k(\phi(\hat x_k)^2+\chi(\hat x_k)^2)\\
        \geq& -(1-\nu)(2\gamma-1)\hat\alpha_k(\phi(\hat x_k)^2+\chi(\hat x_k)^2).
\end{split}
    \end{equation*}
Taking conditional expectation in $\mathcal{F}_{k-1}$, using \cite[Theorem 8.14(iii)]{Klenke2013} and noting that $E[\textbf{1}_{I_k^c}\mid\mathcal{F}_{k-1}]\leq 1-p_k^g$ as in Assumption \ref{ass2}, we have
 \begin{equation*}
 \begin{split}
     &\textbf{E}[\textbf{1}_{I_k^c\cap J_k}\textbf{1}_{\mathbf{S}}(\Phi(\hat x_k)-\Phi(\hat x_{k+1}))]\\
     \stackrel{a.s.}{=}&\textbf{1}_{\mathbf{S}}\textbf{E}[\textbf{1}_{I_k^c\cap J_k}(\Phi(\hat x_k)-\Phi(\hat x_{k+1}))]\\
     \geq&-(1-p_g)(1-\nu)(2\gamma-1)\textbf{1}_{\mathbf{S}}\hat\alpha_k(\phi(\hat x_k)^2+\chi(\hat x_k)^2).
     \end{split}
     \end{equation*}

    \item \emph{Unsuccessful step. ($\textbf{1}_{{\mathbf{S}}^c}=1$.)}

    Like (ii) of Case 1, we also have \eqref{eq:decr_phi_us}.


\end{enumerate}


Case 3 (inaccurate function evaluations, $\textbf{1}_{J_k^c}=1$)

In this case, Assumption \ref{ass2} (iii) is used to bound the increase in $\Phi_k$. We have
\begin{eqnarray}
    \lvert\Psi(\hat x_k;\bar\rho)-\Psi_k^0(\bar\rho)\rvert=\lvert f_k-f_k^0\rvert,\nonumber\\
    \lvert\Psi(\hat x_{k+1};\bar\rho)-\Psi_{k}^s(\bar\rho)\rvert=\lvert f_{k+1}-f_{k}^s\rvert.\nonumber
\end{eqnarray}

As before, we consider the following two cases.
\begin{enumerate}[(i)]
    \item \emph{Successful steps ($\textbf{1}_{\mathbf{S}}=1$)}.

    With a successful step we have
    \begin{equation*}
    \begin{split}
        &\Psi(\hat x_k;\bar\rho)-\Psi(\hat x_{k+1};\bar\rho)\\
        \geq&\Psi_k^0(\bar\rho)-\Psi_k^s(\bar\rho)-\lvert\Psi(\hat x_k;\bar\rho)-\Psi_k^0(\bar\rho)\rvert-\lvert\Psi(\hat x_{k+1};\bar\rho)-\Psi_{k}^s(\bar\rho)\rvert\\
        \geq&\frac{\theta}2\hat\alpha_k\Delta_{\Psi}^g(\hat x_k,d_k;\bar\rho)-\lvert\Psi(\hat x_k;\bar\rho)-\Psi_k^0(\bar\rho)\rvert-\lvert\Psi(\hat x_{k+1};\bar\rho)-\Psi_{k}^s(\bar\rho)\rvert.
    \end{split}
    \end{equation*}
    By choosing $\nu$ satisfying
    \begin{equation}\label{eq:bnd_nu4}
    \nu\geq\frac{2\gamma(M_d^2+L_{\chi}^2)\alpha_{\max}^2\eta^{-1}}{\frac{\theta}4+2\gamma(M_d^2+L_{\chi}^2)\alpha_{\max}^2\eta^{-1}}
    \end{equation}
   and using \eqref{eq:realf_red}, \eqref{eq:chi_case1i} and \eqref{eq:phi_case1i}, it follows that
   \begin{equation*}\begin{split}
   &\Phi(\hat x_k)-\Phi(\hat x_{k+1})\\ \geq&\frac{\nu\theta}4\hat\alpha_k\Delta_{\Psi}^g(\hat x_k,\hat d_k;\bar\rho)-(1-\nu)(2\gamma-1)\hat\alpha_k(\phi(\hat x_k)^2+\chi(\hat x_k)^2)\\
   &-\nu\lvert\Psi(\hat x_k;\bar\rho)-\Psi_k^0(\bar\rho)\rvert-\nu\lvert\Psi(\hat x_{k+1};\bar\rho)-\Psi_{k+1}^0(\bar\rho)\rvert\\
   \geq&-(1-\nu)(2\gamma-1)\hat\alpha_k(\phi(\hat x_k)^2+\chi(\hat x_k)^2)\\
   &-\nu\lvert\Psi(\hat x_k;\bar\rho)-\Psi_k^0(\bar\rho)\rvert-\nu\lvert\Psi(\hat x_{k+1};\bar\rho)-\Psi_{k+1}^0(\bar\rho)\rvert.
   \end{split}\end{equation*}
Taking conditional expectation in $\mathcal{F}_{k-1}$, using Lemma \ref{lem:jc}, \cite[Theorem 8.14(iii)]{Klenke2013} and noting that $E[\textbf{1}_{J_k^c}\mid\mathcal{F}_{k-1}]\leq 1-p_f$ as in Assumption \ref{ass2}, we have
 \begin{equation*}
 \begin{split}
    &\textbf{E}\left[\textbf{1}_{J_k^c}\textbf{1}_{\mathbf{S}}(\Phi(\hat x_k)-\Phi(\hat x_{k+1}))\right]\\
    \stackrel{a.s.}{=}&\textbf{1}_{\mathbf{S}}\textbf{E}\left[\textbf{1}_{J_k^c}(\Phi(\hat x_k)-\Phi(\hat x_{k+1}))\right]\\
    \geq&-(1-p_f)(1-\nu)(2\gamma-1)\textbf{1}_{\mathbf{S}}\hat\alpha_k(\phi(\hat x_k)^2+\chi(\hat x_k)^2)\\
   & -2\nu (1-p_f)^{3/2}\kappa_f\textbf{1}_{\mathbf{S}}\hat\alpha_k\Delta_{\Psi}^g(\hat x_k,\hat d_k;\bar\rho).
\end{split}
\end{equation*}

   \item \emph{Unsuccessful step ($\textbf{1}_{\mathbf{S}^c}=1$)}.

   As in Case 1(ii), we have \eqref{eq:decr_phi_us}
\end{enumerate}

Now we combine the expectations for successful steps and unsuccessful steps. For successful steps, we have
\begin{align}
        &\textbf{E}[\textbf{1}_{\mathbf{S}}(\Phi(\hat x_k)-\Phi(\hat x_{k+1}))\lvert \mathcal{F}_{k-1}]\nonumber\\
        \stackrel{a.s.}{=}& \textbf{1}_{\mathbf{S}}\textbf{E}[(\Phi(\hat x_k)-\Phi(\hat x_{k+1}))\lvert \mathcal{F}_{k-1}]\nonumber\\
        =&\textbf{1}_{\mathbf{S}}\textbf{E}[(\textbf{1}_{I_k\cap J_k}+\textbf{1}_{I_k^c\cap J_k}+\textbf{1}_{J_k^c})(\Phi(\hat x_k)-\Phi(\hat x_{k+1}))\lvert \mathcal{F}_{k-1}]\nonumber\\
        \geq&p_fp_g\frac{\nu\theta}{8}\textbf{1}_{\mathbf{S}}\hat\alpha_k\Delta_{\Psi}^g(\hat x_k,\hat d_k;\bar\rho)+ p_fp_g\tilde\kappa\frac{\nu\theta}{8}\textbf{1}_{\mathbf{S}}\hat\alpha_k(\phi(\hat x_k)^2+\chi(\hat x_k)^2)\nonumber\\
        &-(1-p_g)(1-\nu)(2\gamma-1)\textbf{1}_{\mathbf{S}}\hat\alpha_k(\phi(\hat x_k)^2+\chi(\hat x_k)^2)\nonumber\\
        &-(1-p_f)(1-\nu)(2\gamma-1)\textbf{1}_{\mathbf{S}}\hat\alpha_k(\phi(\hat x_k)^2+\chi(\hat x_k)^2)\nonumber\\
   & -2\nu (1-p_f)^{3/2}\kappa_f\textbf{1}_{\mathbf{S}}\hat\alpha_k\Delta_{\Psi}^g(\hat x_k,\hat d_k;\bar\rho)\nonumber\\
   \stackrel{(a)}{\geq}&p_fp_g\frac{\nu\theta}{8}\textbf{1}_{\mathbf{S}}\hat\alpha_k\Delta_{\Psi}^g(\hat x_k,\hat d_k;\bar\rho)+ p_fp_g\tilde\kappa\frac{\nu\theta}{8}\textbf{1}_{\mathbf{S}}\hat\alpha_k(\phi(\hat x_k)^2+\chi(\hat x_k)^2)\nonumber\\
        &-(1-p_g)(1-\nu)(2\gamma-1)\textbf{1}_{\mathbf{S}}\hat\alpha_k(\phi(\hat x_k)^2+\chi(\hat x_k)^2)\nonumber\\
        &-(1-p_f)(1-\nu)(2\gamma-1)\textbf{1}_{\mathbf{S}}\hat\alpha_k(\phi(\hat x_k)^2+\chi(\hat x_k)^2)\nonumber\\
   & -2\nu (1-p_f)\kappa_f\textbf{1}_{\mathbf{S}}\hat\alpha_k\Delta_{\Psi}^g(\hat x_k,\hat d_k;\bar\rho)\nonumber\\
   =&\left(p_fp_g\frac{\nu\theta}{8}-2\nu (1-p_f)\kappa_f\right)\textbf{1}_{\mathbf{S}}\hat\alpha_k\Delta_{\Psi}^g(\hat x_k,\hat d_k;\bar\rho)\nonumber\\
   &+\left(p_fp_g\tilde\kappa\frac{\nu\theta}{8}-(2-p_f-p_g)(1-\nu)(2\gamma-1)\right)\hat\alpha_k(\phi(\hat x_k)^2+\chi(\hat x_k)^2) \label{eq:exp_s1}
\end{align}
where (a) follows from $(1-p_f)^{3/2}\leq 1-p_f$.
If we choose
\begin{equation}\label{eq:ch_pfg}
p_f\geq \frac{2\kappa_f}{2\kappa_f+\frac{\theta}{16}p_g^0}
\end{equation}
and
\begin{equation}\label{eq:bnd_nux1}
    \nu\geq\frac{(2-p_f-p_g)(2\gamma-1)}{p_fp_g\tilde\kappa\frac{\theta}{16} +(2-p_f-p_g)(2\gamma-1)},
\end{equation}
then it follows from \eqref{eq:exp_s1} and  the inequality $p_g^0\leq p_g$ that
\begin{equation*}
    \begin{split}
        &\textbf{E}[\textbf{1}_{\mathbf{S}}(\Phi(\hat x_k)-\Phi(\hat x_{k+1}))\lvert \mathcal{F}_{k-1}]\\
        \stackrel{a.s.}\geq& p_fp_g\frac{\nu\theta}{16}\textbf{1}_{\mathbf{S}}\hat\alpha_k\Delta_{\Psi}^g(\hat x_k,\hat d_k;\bar\rho)+p_fp_g\tilde\kappa\frac{\nu\theta}{16}\textbf{1}_{\mathbf{S}}\hat\alpha_k(\phi(\hat x_k)^2+\chi(\hat x_k)^2)\\
        \geq&p_fp_g\tilde\kappa\frac{\nu\theta}{16}\textbf{1}_{\mathbf{S}}\hat\alpha_k(\phi(\hat x_k)^2+\chi(\hat x_k)^2).
    \end{split}
\end{equation*}
For unsuccessful steps, note that \eqref{eq:decr_phi_us} always holds, we have
\begin{equation*}
    \textbf{E}[\textbf{1}_{\mathbf{S}^c}\Phi(\hat x_k)-\Phi(\hat x_{k+1})\lvert \mathcal{F}_{k-1}]\stackrel{a.s.}{=}(1-\nu)(1-\gamma^{-1})\textbf{1}_{\mathbf{S}^c}\hat\alpha_k(\phi(\hat x_k)^2+\chi(\hat x_k)^2).
\end{equation*}
If $\nu$ is chosen such that
\begin{equation}\label{eq:bnd_nux}
    \nu\geq\frac{1+\gamma^{-1}}{\frac{\tilde\kappa\theta}{16}+1-\gamma^{-1}},
\end{equation}
then we have
\begin{eqnarray*}
    \textbf{E}[\textbf{1}_{\mathbf{S}}(\Phi(\hat x_k)-\Phi(\hat x_{k+1}))\lvert \mathcal{F}_{k-1}]\stackrel{a.s.}{\geq} p_fp_g(1-\nu)(1-\gamma^{-1})\hat\alpha_k(\phi(\hat x_k)^2+\chi(\hat x_k)^2), \text{ and }\\
    \textbf{E}[\textbf{1}_{\mathbf{S}^c}(\Phi(\hat x_k)-\Phi(\hat x_{k+1}))\lvert \mathcal{F}_{k-1}]\stackrel{a.s.}{\geq} p_fp_g(1-\nu)(1-\gamma^{-1})\hat\alpha_k(\phi(\hat x_k)^2+\chi(\hat x_k)^2)
\end{eqnarray*}
where $p_fp_g<1$ is used in the second inequality.
Therefore, we have
\begin{equation*}
    \begin{split}
    &\textbf{E}[(\Phi(\hat x_k)-\Phi(\hat x_{k+1}))\lvert \mathcal{F}_{k-1}]\\
    =&\textbf{E}[(\textbf{1}_{\mathbf{S}}+\textbf{1}_{\mathbf{S}^c})(\Phi(\hat x_k)-\Phi(\hat x_{k+1}))\lvert \mathcal{F}_{k-1}]\\
    =&\textbf{E}[(\textbf{1}_{\mathbf{S}})(\Phi(\hat x_k)-\Phi(\hat x_{k+1}))\lvert \mathcal{F}_{k-1}]
    +\textbf{E}[(\textbf{1}_{\mathbf{S}^c})(\Phi(\hat x_k)-\Phi(\hat x_{k+1}))\lvert \mathcal{F}_{k-1}]\\
    \stackrel{a.s.}{\geq}&p_fp_g(1-\nu)(1-\gamma^{-1})\hat\alpha_k(\phi(\hat x_k)^2+\chi(\hat x_k)^2).
    \end{split}
\end{equation*}
\end{proof}

\subsection{Convergence rate with bounded penalty factors}

We are going to bound the expected number of steps that the algorithm takes until
\[
\phi(\hat x_k)^2+\chi(\hat x_k)^2 <\varepsilon.
\]
Define the stopping time
\[
T_\varepsilon=\inf\{k\geq0:\phi(\hat x_k)^2+\chi(\hat x_k)^2<\varepsilon\},
\]
 a function
$q(A_k)=A_k\varepsilon$ and $W_k=2(\textbf{1}_{I_k\cap J_k}-\frac12)$.

The proof is similar to Section 4.3 in \cite{paquette2020stochastic}. The main argument requires us is to check whether the random process
$\{\Phi_k,\hat\alpha_k,W_k\}$ satisfies Assumption \ref{ass4}, where $\Phi_k$ is defined by \eqref{eq:phi}.

Assumption \ref{ass4}(i) is ensured by the rule of updating $\alpha_k$ and
Assumption \ref{ass4}(ii) is due to Lemma \ref{lem:bndalp} and the application of similar arguments as in the proof of \cite[Lemma 4.8]{paquette2020stochastic}.

Given
\eqref{eq:bnd_dphi_e}, by multiplying both sides by the indicator $\textbf{1}_{\{T_\varepsilon>k\}}$, we have
\[
\textbf{1}_{\{T_\varepsilon>k\}}\textbf{E}[\Phi(\hat x_k)-\Phi(\hat x_{k+1})]\geq p_fp_g(1-\nu)(1-\gamma^{-1})\varepsilon\hat\alpha_k=\Theta q(\hat\alpha_k)
\]
with $\Theta=p_fp_g(1-\nu)(1-\gamma^{-1})$. Therefore, Assumption \ref{ass4}(iii) holds.

Using Theorem \ref{thm:bounde} and noting that \eqref{eq:bnd_nu2} dominates \eqref{eq:bnd_nu4}, we have

\begin{lemma}
Let $p_f$ and $p_g$ satisfy \eqref{eq:ch_pfg} and $p_fp_g\geq1/2$. Suppose that $\nu$ is chosen such that \eqref{eq:bnd_nu2}, \eqref{eq:bnd_nux1} and \eqref{eq:bnd_nux} hold. Then Assumption \ref{ass4} is satisfied for $\lambda=\log(\gamma)$, $p=p_fp_g$ and
\[\bar\alpha =\frac{(1-\theta)\eta}{\varepsilon_g+\frac12L_f+2\varepsilon_f+\frac12\bar\rho L_{hc}}\]
\end{lemma}
\begin{proof}
Using Lemma \ref{lem:bndalp} and a proof exactly the same as that of \cite[Lemma 4.8]{paquette2020stochastic}, we can prove this Lemma.
\end{proof}

\begin{theorem}\label{thm:rate}
    Under the same assumptions in Theorem \ref{thm:redinex}, suppose that $p_f$ and $p_g$ satisfy  \eqref{eq:ch_pfg} and $p_fp_g\geq1/2$. Suppose also that $\nu$ is chosen such that \eqref{eq:bnd_nu2}, \eqref{eq:bnd_nux1} and \eqref{eq:bnd_nux} hold. Then
    \[
    \textbf{E}[T_{\varepsilon}]\leq\frac{1}{2p_fp_g-1}\frac{(\varepsilon_g+\frac12L_f+2\varepsilon_f+\frac12\bar\rho L_{hc})}{(1-\nu)(1-\gamma^{-1})(1-\theta)\eta\varepsilon}\Phi_0+1.
    \]
\end{theorem}

As a simple corollary to Theorem \ref{thm:rate}, just like Theorem 4.10 in \cite{paquette2020stochastic}, we have, {this time without any requirement on $\bar{\rho}(\Xi)$ since the result is asymptotic.}
\begin{corollary}
Let Assumptions \ref{ass1}, \ref{ass2} and \ref{ass3} hold. Then the sequence of random iterates generated by Algorithm \ref{ssqp}, $\{\hat x_k\}$, almost surely satisfies
\[
\liminf_{k\rightarrow\infty}\phi(\hat x_k)^2+\chi(\hat x_k)^2=0.
\]
\end{corollary}

\subsection{Unbounded penalty parameter}
Now we consider the case where $\rho_k\rightarrow+
\infty$. Even in the presence of eMFCQ, the boundedness of the sequence of penalty factors $\{\rho_k\}$ is not guaranteed. As the constraints are deterministic and by Lemma \ref{lem:pen_updt}, the possible unboundedness of $\{\rho_k\}$ is caused by the randomness of $\hat g_k$.  We will show that, under the sampling method in Algorithm \ref{ssqp}, the event that $\rho_k$ tends to $+\infty$ will happen with probability 0.

We will use Borel-Cantelli Lemma (see, for example, \cite{Klenke2013}), which is stated as
\begin{theorem}[Borel-Cantelli Lemma]\label{thm:bc}
Let $A_k$, $k=1,2,\cdots$ be a sequence of random events. If $\sum_{k=1}^{+\infty}\textbf{Pr}(A_k)<+\infty$, then
\[
\textbf{Pr}(A_k\ i.o.)=0,
\]
 where $i.o.$ stands for infinitely often, in other words,
$\{A_k\ i.o.\}=\limsup A_k$.
\end{theorem}

We will derive the probability of $\{\lim\hat\rho_k=+\infty\}$ by estimating probability of a sequence of random events.
Define
\[\begin{array}{c}
T_0=0,\\
T_1=\inf\{\ k: \|\hat g_k\|>\hat\zeta_k,k> T_0\},\\
T_2=\inf_k\{\ k:\|\hat g_k\|>\hat\zeta_k,k> T_1\},\\
\cdots\\
T_j=\inf_k\{\ k:\|\hat g_k\|>\hat\zeta_k,k> T_{j-1}\},\\
\cdots
\end{array}\]
and $B_j=\{\|\hat g_{T_j}\|>\hat\zeta_{T_j}\}$, $j=1,2,\cdots$.
By construction of Algorithm \ref{ssqp} we have
\[
\{\lim\hat\rho_k=+\infty\}=\{B_j\ i.o.\}.
\]

Note that $\lim_{j\rightarrow+\infty}\zeta_{T_j}=+\infty$. We assume without loss of generality that $\zeta_0>M_d+\varepsilon_g\alpha_{\max}\beta_u^2$. Therefore, for any time $T_j$, there is
\begin{equation*}
    \{\|\hat g_{T_j}-\nabla f(\hat x_{T_j})\|\leq \hat\zeta_{T_j}-\|\nabla f(\hat x_{T_j})\|\}\supset \{\|\hat g_{T_j}-\nabla f(\hat x_{T_j})\|\leq \varepsilon_g\hat\alpha_{T_j}\|\hat d_{T_j}\|^2\},
\end{equation*}
which implies
\begin{equation}\label{eq:pbj}
\begin{split}
    &\textbf{Pr}(B_j)=\textbf{Pr}\{\|\hat g_{T_j}\|>\hat\zeta_{T_j}\}\\
    \leq& \textbf{Pr}\{\|\hat g_{T_j}-\nabla f(\hat x_{T_j})\|>\hat \zeta_{T_j}-\|\nabla f(\hat x_{T_j})\|\}\\
    =&1-\textbf{Pr}\{\|\hat g_{T_j}-\nabla f(\hat x_{T_j})\|\leq\hat \zeta_{T_j}-\|\nabla f(\hat x_{T_j})\|\}\\
    \leq&1-\textbf{Pr}\{\|\hat g_{T_j}-\nabla f(\hat x_{T_j})\|\leq \varepsilon_g\hat\alpha_{T_j}\|\hat d_{T_j}\|^2\}\\
    \leq&1-a_j.
\end{split}\end{equation}
Therefore, we get $\sum_{j=0}^{+\infty}\textbf{Pr}(B_j)\leq\sum_{j=0}^{+\infty}(1-a_j)<+\infty$, which, by Theorem \ref{thm:bc}, yields
\begin{theorem}
Let Assumptions \ref{ass1}, \ref{ass2} and \ref{ass3} hold. Then \[\textbf{Pr}(\lim_{k\rightarrow+\infty}\hat\rho_k=+\infty)=0.\]
\end{theorem}
\begin{proof}
It follows from \eqref{eq:pbj} and $\sum_{j=0}^{+\infty}(1-a_j)<+\infty$ that $\sum_{j=0}^{+\infty}\textbf{Pr}(B_j)<+\infty$.
Then we have by Theorem \ref{thm:bc} that $\textbf{Pr}(B_j\ i.o.)=0$, which is equivalent to $\textbf{Pr}(\lim_{k\rightarrow+\infty}\hat\rho_k=+\infty)=0$.
\end{proof}
{We note that this does not preclude a set of $\epsilon_k\to 0$ measure on which $\rho_k$ increases, so is not equivalent to there existing some $\bar{\rho}$ such that $\hat{\rho}_k\le \bar{\rho}$ w.p. 1. However, we can say that the iteration complexity results hold with an incremental balance of probability to total iterations.}

\section{Numerical experiments} \label{sect:num}

{\subsection{Tests on a set of simple problems}
In this subsection, we demonstrate the empirical performance of Algorithm \ref{ssqp} on a set of simple problems  adapted from a subset of problems in two standard collections \cite{Hock1981,Schit1987}. }

Specifically, we are interested in those problems
whose objective functions are of the form
\begin{equation}\label{eq:obj1}
F(x)=\sum_{i=1}^m a_iF_i^2(x)
\end{equation}
and which have at least one non-degenerate solution. Here a point is said to be non-degenerate if it satisfies LICQ. The non-degeneracy was verified in the process of choosing test problems. A total of 58 problems are included in our experiments. Each problem comes with an initial point and a set of solutions (or, at least, well approximated solutions).

The selected problems are deterministic and only a part of them have both equality and inequality constraints. Therefore, we modified them to fit the framework of interest.
For the objective function, we perturb each component in \eqref{eq:obj1} by $\xi_i\sim N(0,\sigma^2)$ for some $\sigma>0$ and then take expectation with respect to $\xi=(\xi_1,\xi_2,\cdots,\xi_m)$, i.e., the adapted objective function is
\[
f(x)=\textbf{E}[F(x,\xi)]=\textbf{E}\left[\sum_{i=1}^m a_i(F_i(x)+\xi_i)^2\right].
\]
For problems having only equality constraints,
we add a new inequality constraints of the form
\begin{equation}\label{eq:adinq}
c_{lt}(x-e)\leq b,
\end{equation}
where $c_{lt}$ is the last equality constraint, $e$ is the $n-$dimensional vector of all one and $b$ is chosen such that \eqref{eq:adinq} is active at the first solution given in \cite{Hock1981,Schit1987} of the original problem.
For problems having only inequality constraints, we introduce
an equality constraint which is generated by
shifting the last nonlinear active constraint (if exists) horizontally by $1$ unit and vertically by a proper distance such that ``$=$'' holds at the first solution of the original problem. If none of the nonlinear inequality constraints is active, then we shift the last one vertically so that it becomes active and make it an equality. {By the construction, the solution set of an adapted problem is contained by that of the corresponding original problem.}

In what follows, the test problems are numbered in the same way as \cite{Hock1981,Schit1987}, for instance, ``HS06'' means that this problem is modified from Problem No. 6 in Hock and Schittkowski's collection \cite{Hock1981} and ``S216'', Problem No. 216 in Schittkowski's set \cite{Schit1987}.

\subsubsection{Experimental Details}
We are particularly interested in variation in the algorithm's performance with respect to different levels of noise and sample size.
We consider three tiers of $\sigma$
\[
\sigma \in \{10,1,0.1\}
\]
and five levels of sample size
\[
\tilde{S}=\{50,500,5 000,50 000,100 000\}
\]
Note that here we set $\lvert{\mathcal{S}}_k^f\rvert=\lvert{\mathcal{S}}_k^g\rvert$ in experiments.

The solution set of the test problems is not difficult to determine since the solution set of the original problems can be founded in \cite{Hock1981} and \cite{Schit1987}. Here, just for experimental concerns, we measure the optimality by the following distance
\[
\dist(x_k,\mathcal{S}_{opt})=\min_{\tilde x\in\mathcal{S}_{opt}}\|x_k-\tilde x\|,
\]
where $\mathcal{S}_{opt}$ is the solution set.
For easier informativeness, we display a $\log_{10}$ scale of the distances.

For the experiments, the parameters were set as $\sigma_u=10^6$, $\beta_l=100$, $\beta_u=500$, $\kappa_\mu=2$, $\rho_0=10$, $\alpha_0=1$, $\alpha_{\max}=2$ and $\gamma=2$.
We use the solver ``linprog'' in MATLAB's Optimization Toolbox to solve the linear program \eqref{eq:dist}, and ``quadprog'' for \eqref{eq:s_subp0}.

\subsubsection{Empirical performance}
We now examine the performance of our method for these test problems.
{The empirical performance of the algorithm is reported by means of box plot.}  For each instance, we ran Algorithm \ref{ssqp} 20 times and record the $\log_{10}\dist(x_{k},\mathcal{S}_{opt})$ at all $k\in\tilde{K}$, then draw a box plot based on the recorded data.

{Broadly speaking, for the majority of the test problems, our method was able to find an approximate solution with reasonable precision. There are also a few problems ,for instance, S394 and S395, that our algorithm failed to find an approximate solution.

 For brevity, we only list 6 randomly chosen box plots, but the complete figures for performance on the additional test problems can be found in Appendix. Box plots for these 6 problems are given in Figures \ref{fig:b1}-\ref{fig:b12}.

\begin{figure*}[ht]
  \begin{minipage}[t]{0.5\linewidth}
    \centering
    \includegraphics[scale=0.35]{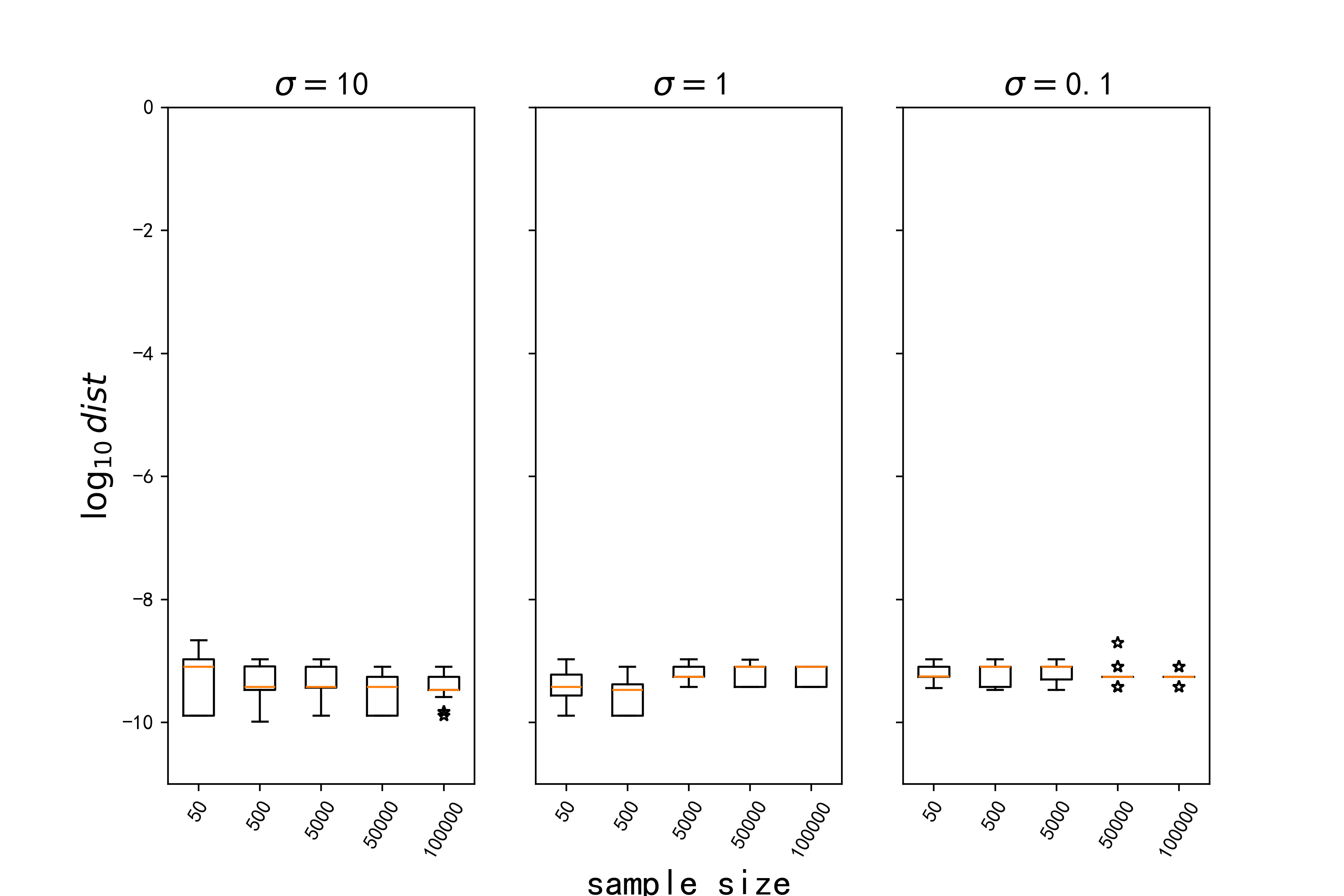}
    \caption{\scriptsize{Performance on HS11 w.r.t after 50 iterations.}}
    \label{fig:b1}
  \end{minipage}%
  \begin{minipage}[t]{0.5\linewidth}
    \centering
    \includegraphics[scale=0.35]{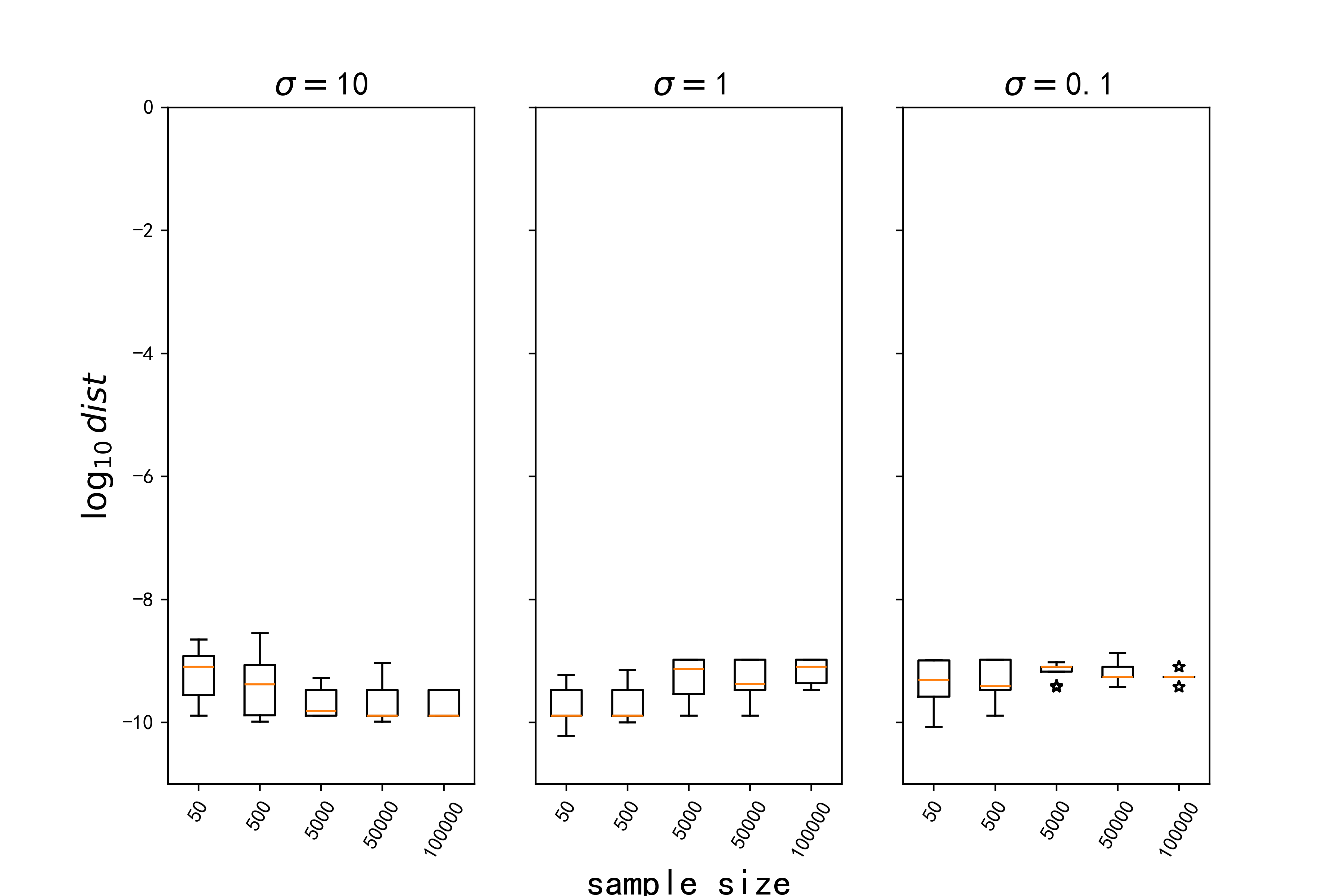}
    \caption{\scriptsize{Performance on HS11 w.r.t after 1 500 iterations.}}
  \end{minipage}
\end{figure*}

\begin{figure*}[ht]
  \begin{minipage}[t]{0.5\linewidth}
    \centering
    \includegraphics[scale=0.35]{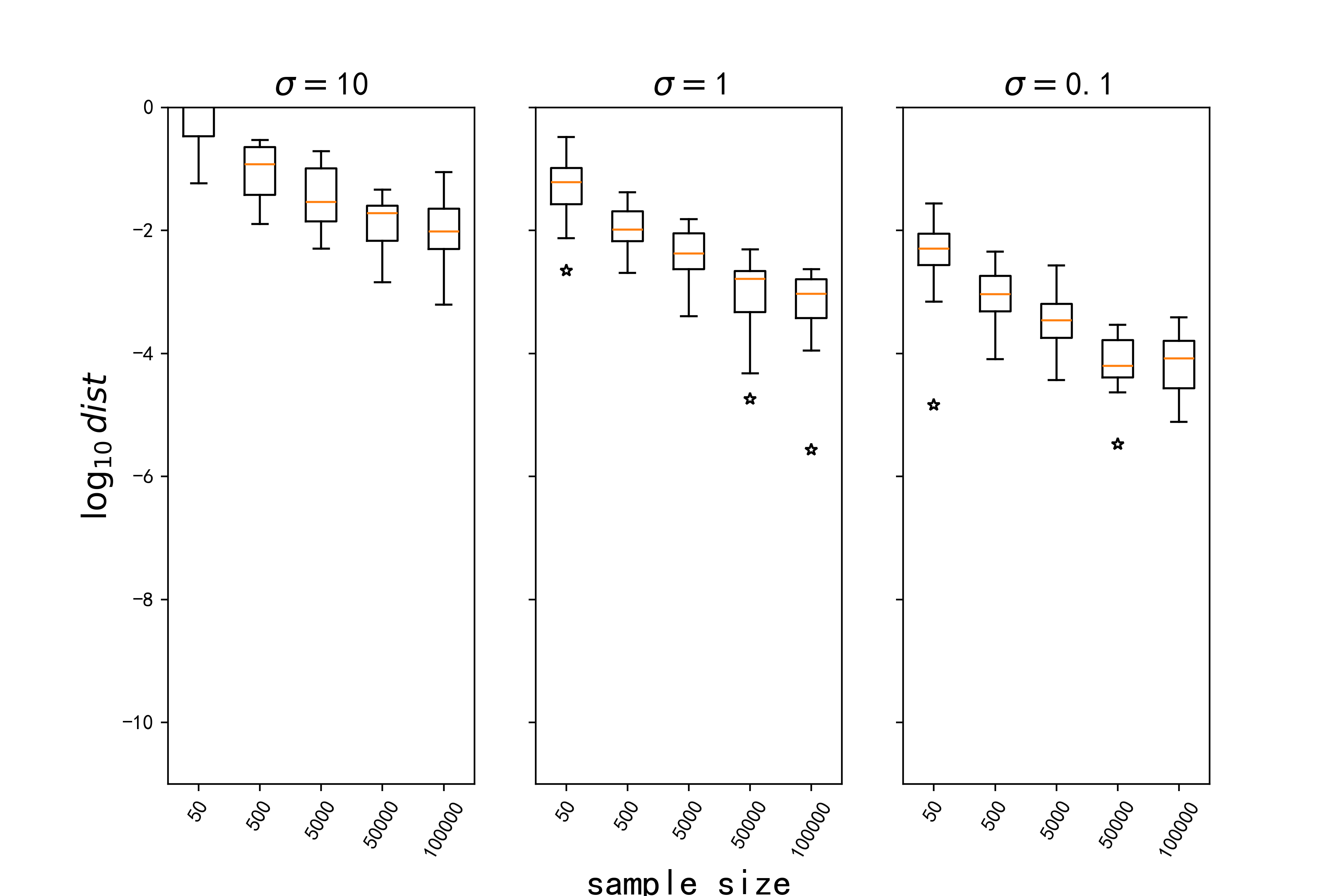}
    \caption{\scriptsize{Performance on HS30 w.r.t after 50 iterations.}}
  \end{minipage}%
  \begin{minipage}[t]{0.5\linewidth}
    \centering
    \includegraphics[scale=0.35]{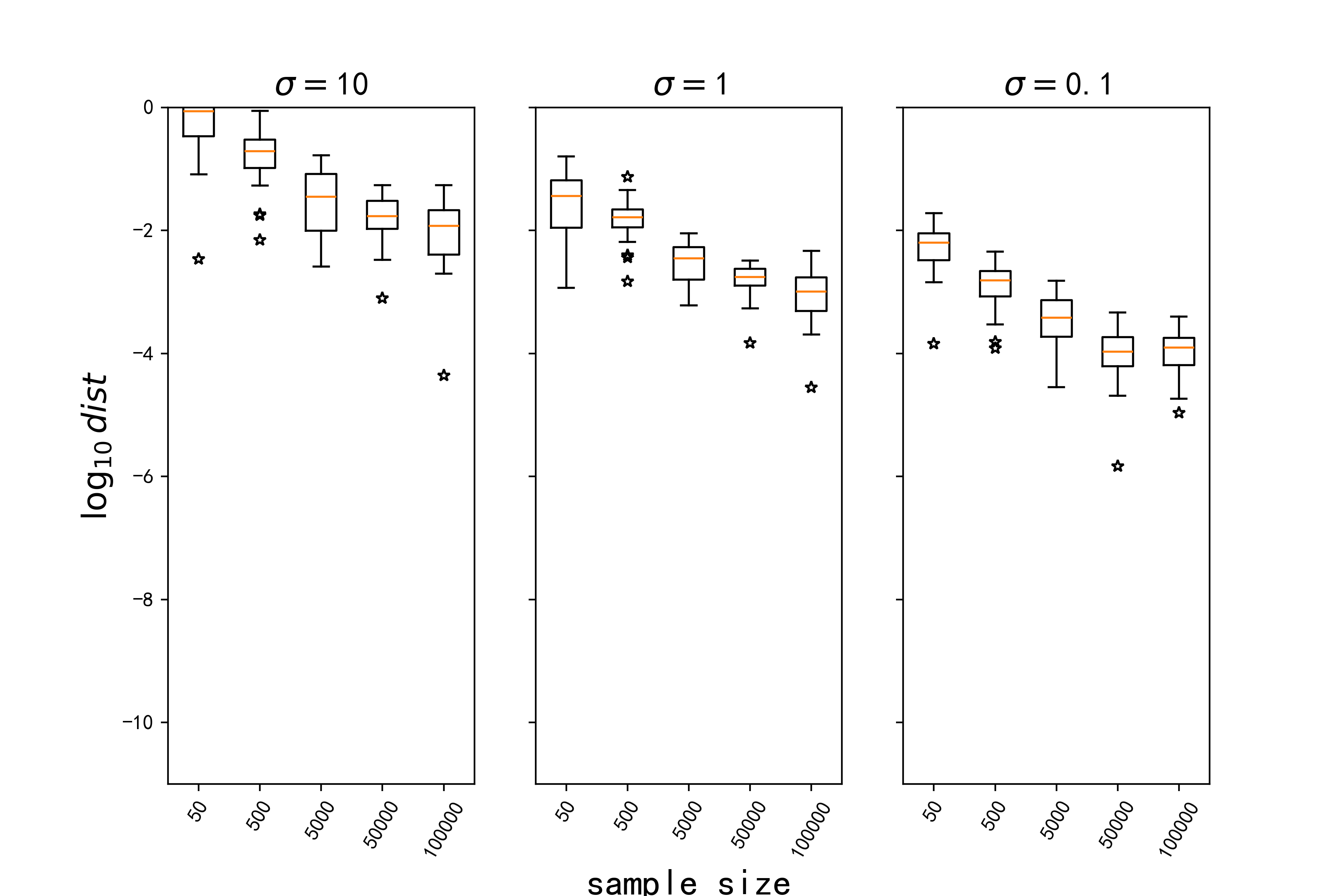}
    \caption{\scriptsize{Performance on HS30 w.r.t after 1 500 iterations.}}
  \end{minipage}
\end{figure*}

\begin{figure*}[ht]
  \begin{minipage}[t]{0.5\linewidth}
    \centering
    \includegraphics[scale=0.35]{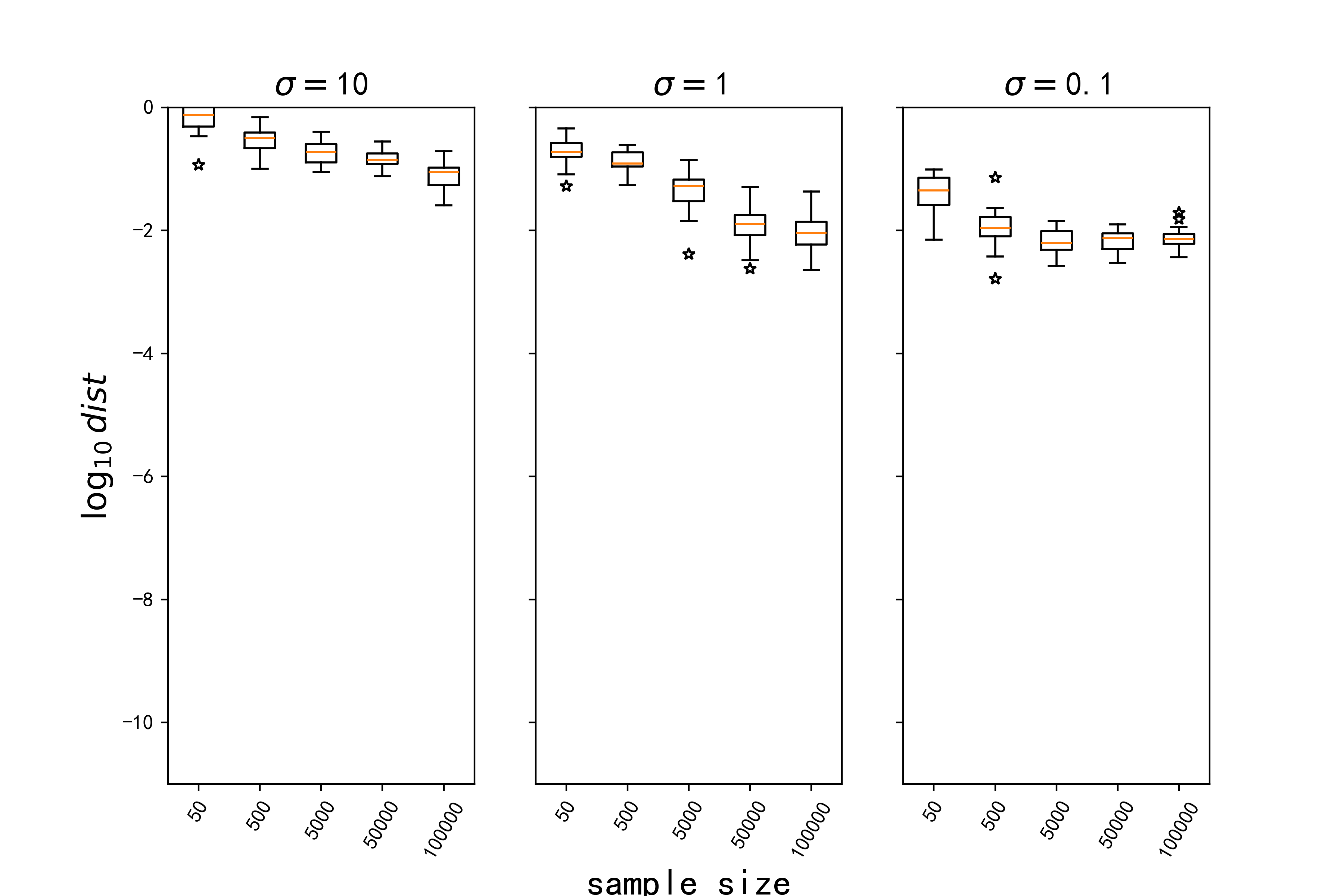}
    \caption{\scriptsize{Performance on HS46 w.r.t after 50 iterations.}}
  \end{minipage}%
  \begin{minipage}[t]{0.5\linewidth}
    \centering
    \includegraphics[scale=0.35]{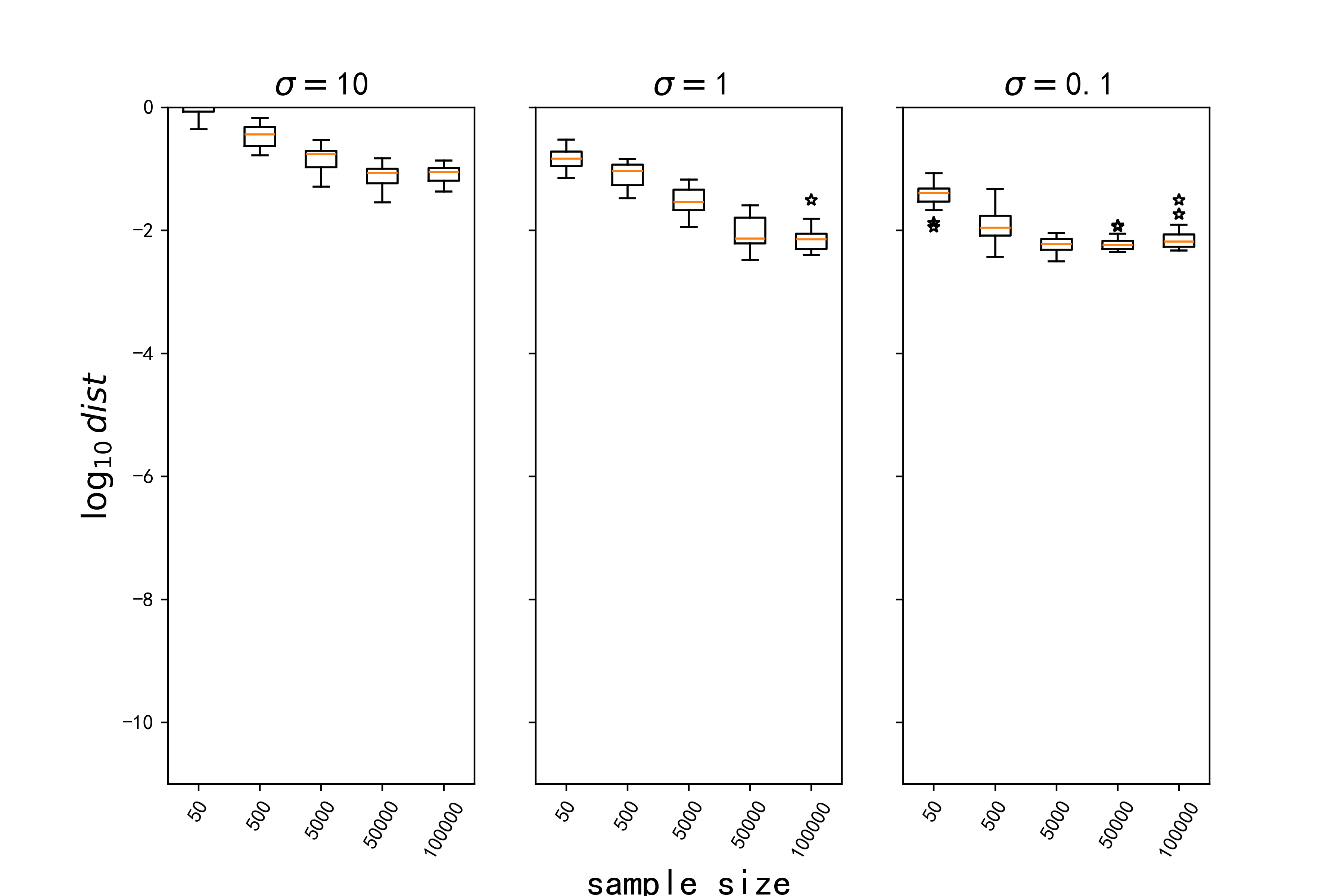}
    \caption{\scriptsize{Performance on HS46 w.r.t after 1 500 iterations.}}
  \end{minipage}
\end{figure*}

\begin{figure*}[ht]
  \begin{minipage}[t]{0.5\linewidth}
    \centering
    \includegraphics[scale=0.35]{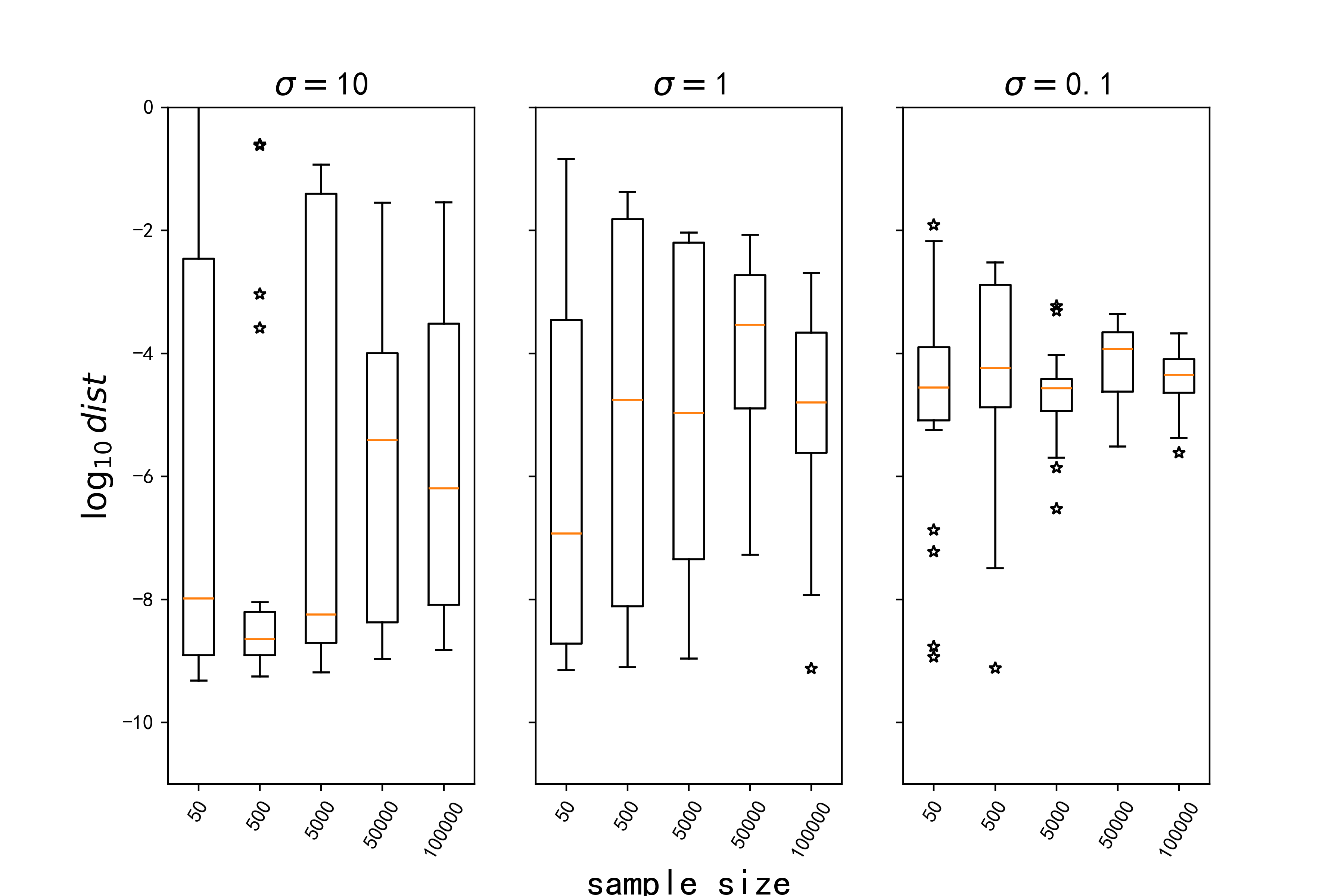}
    \caption{\scriptsize{Performance on HS61 w.r.t after 50 iterations.}}
  \end{minipage}%
  \begin{minipage}[t]{0.5\linewidth}
    \centering
    \includegraphics[scale=0.35]{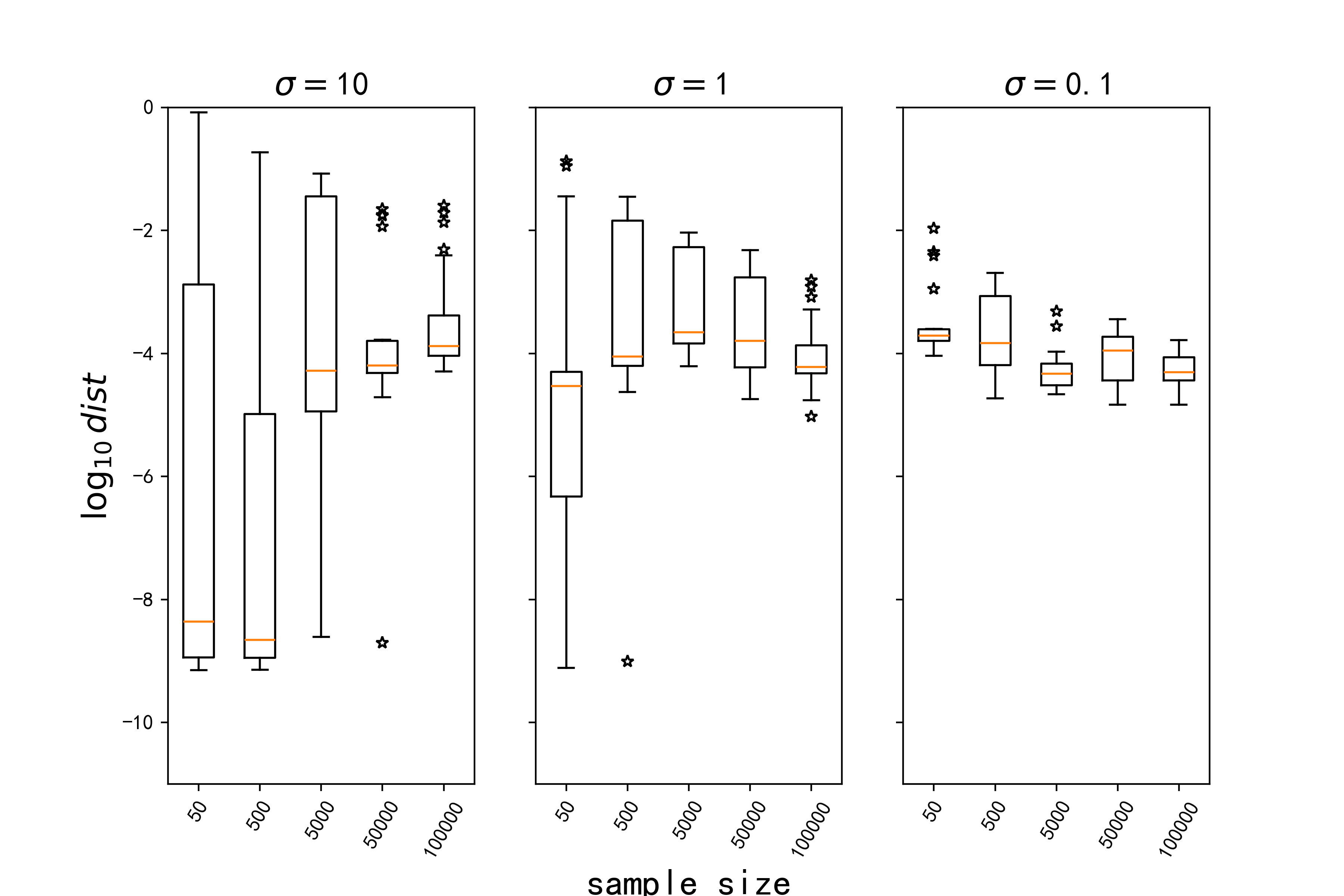}
    \caption{\scriptsize{Performance on HS61 w.r.t after 1 500 iterations.}}
  \end{minipage}
\end{figure*}

\begin{figure*}[ht]
  \begin{minipage}[t]{0.5\linewidth}
    \centering
    \includegraphics[scale=0.35]{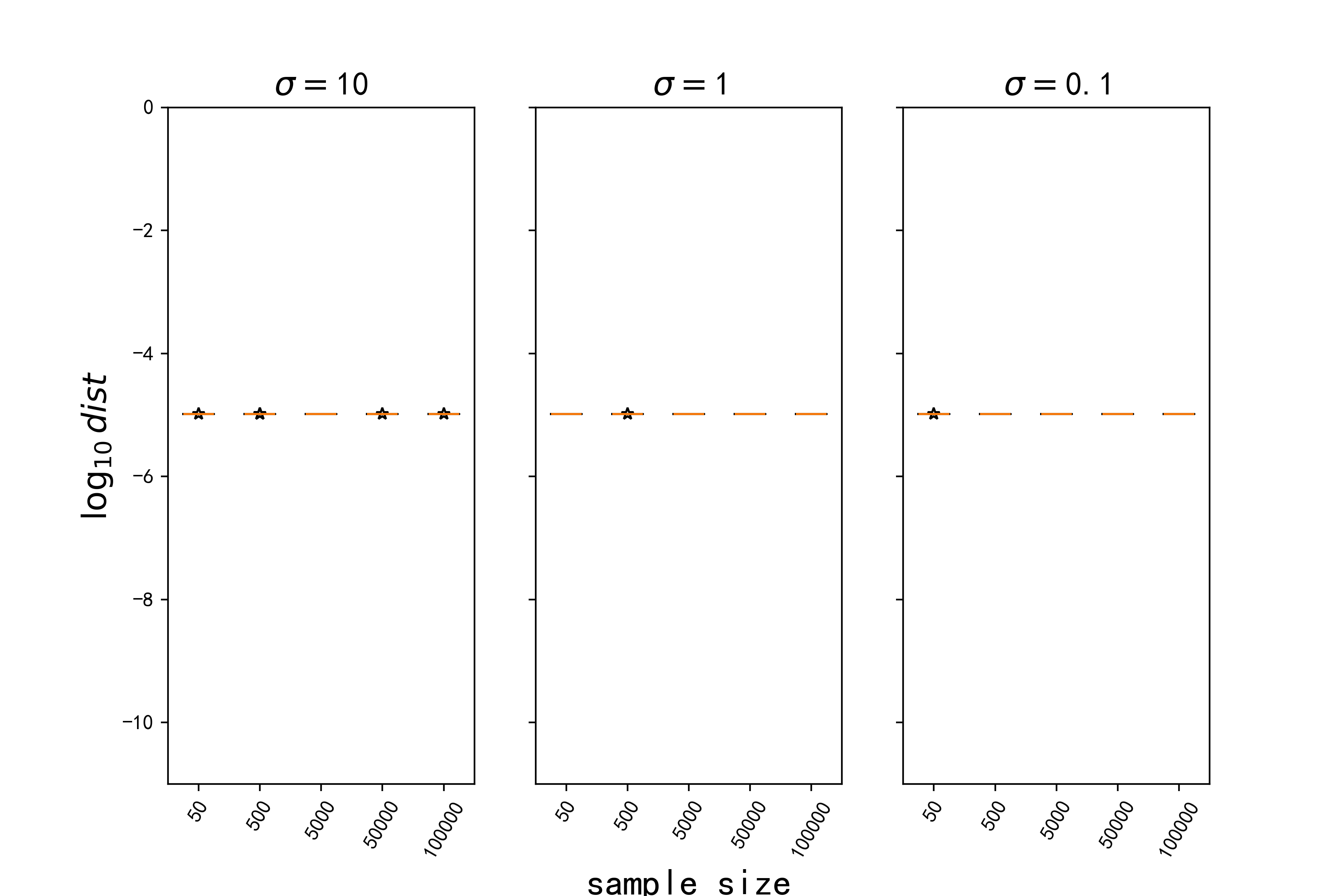}
    \caption{\scriptsize{Performance on S324 w.r.t after 50 iterations.}}
  \end{minipage}%
  \begin{minipage}[t]{0.5\linewidth}
    \centering
    \includegraphics[scale=0.35]{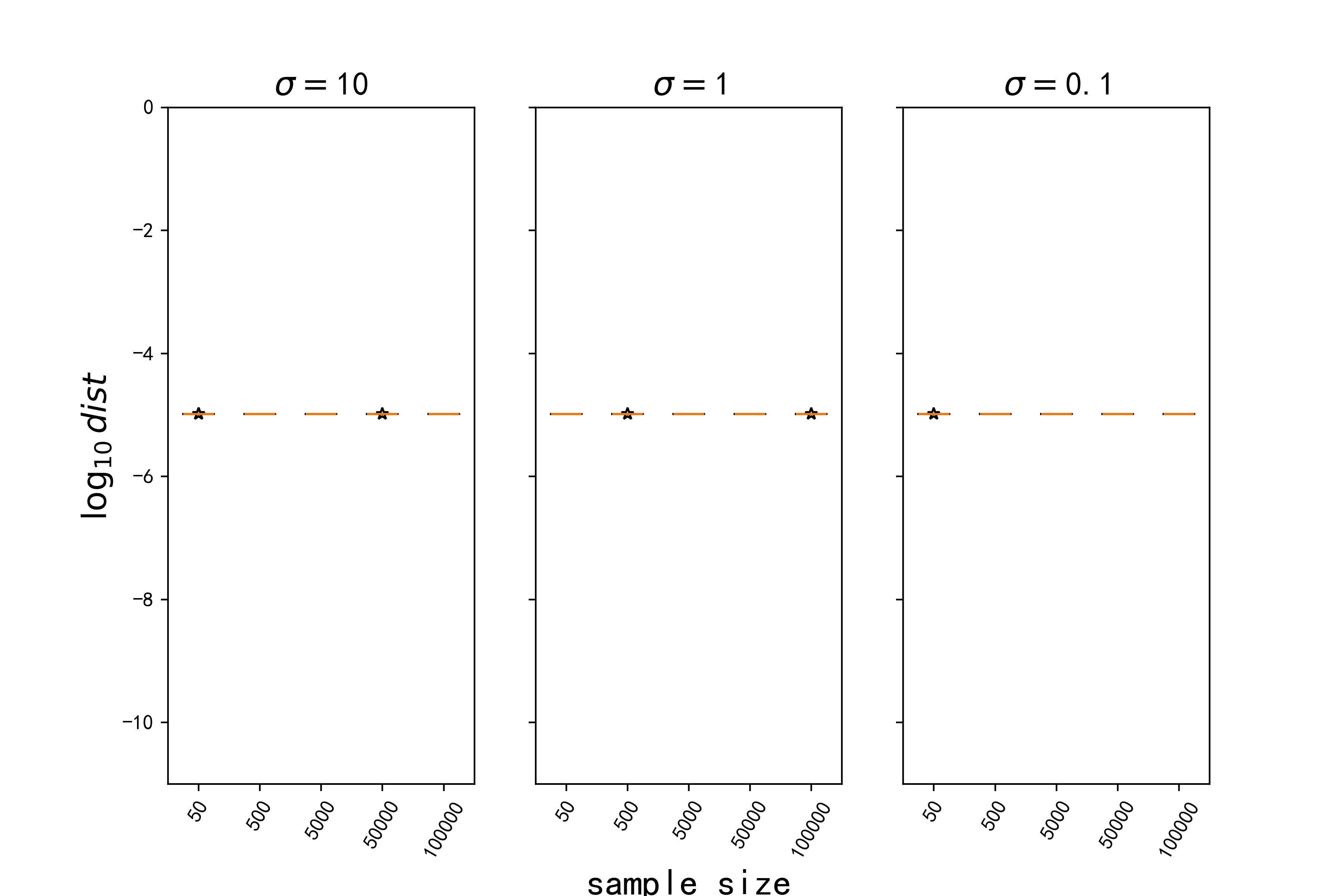}
    \caption{\scriptsize{Performance on S324 w.r.t after 1 500 iterations.}}
  \end{minipage}
\end{figure*}

\begin{figure*}[ht]
  \begin{minipage}[t]{0.5\linewidth}
    \centering
    \includegraphics[scale=0.35]{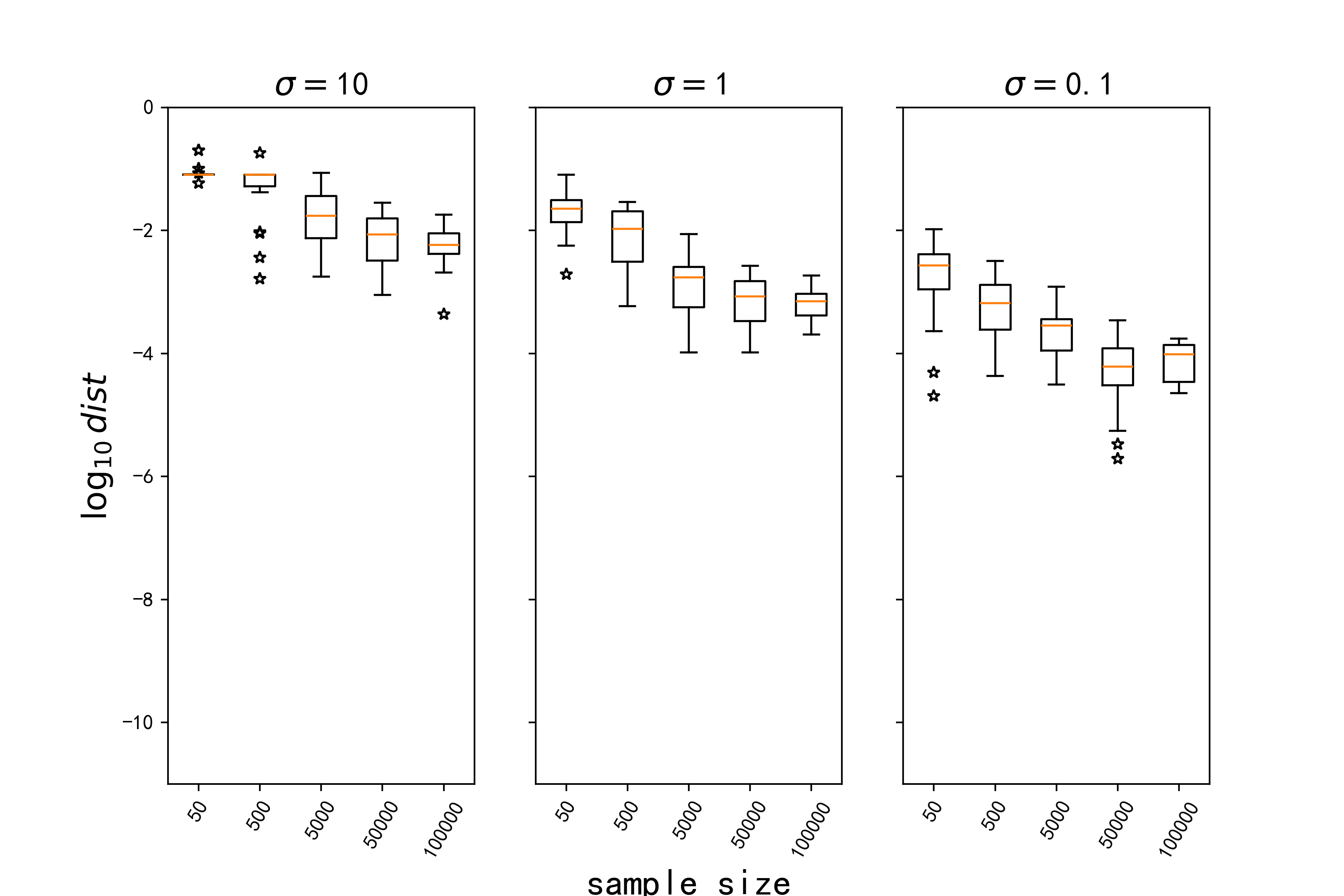}
    \caption{\scriptsize{Performance on S327 w.r.t after 50 iterations.}}
  \end{minipage}%
  \begin{minipage}[t]{0.5\linewidth}
    \centering
    \includegraphics[scale=0.35]{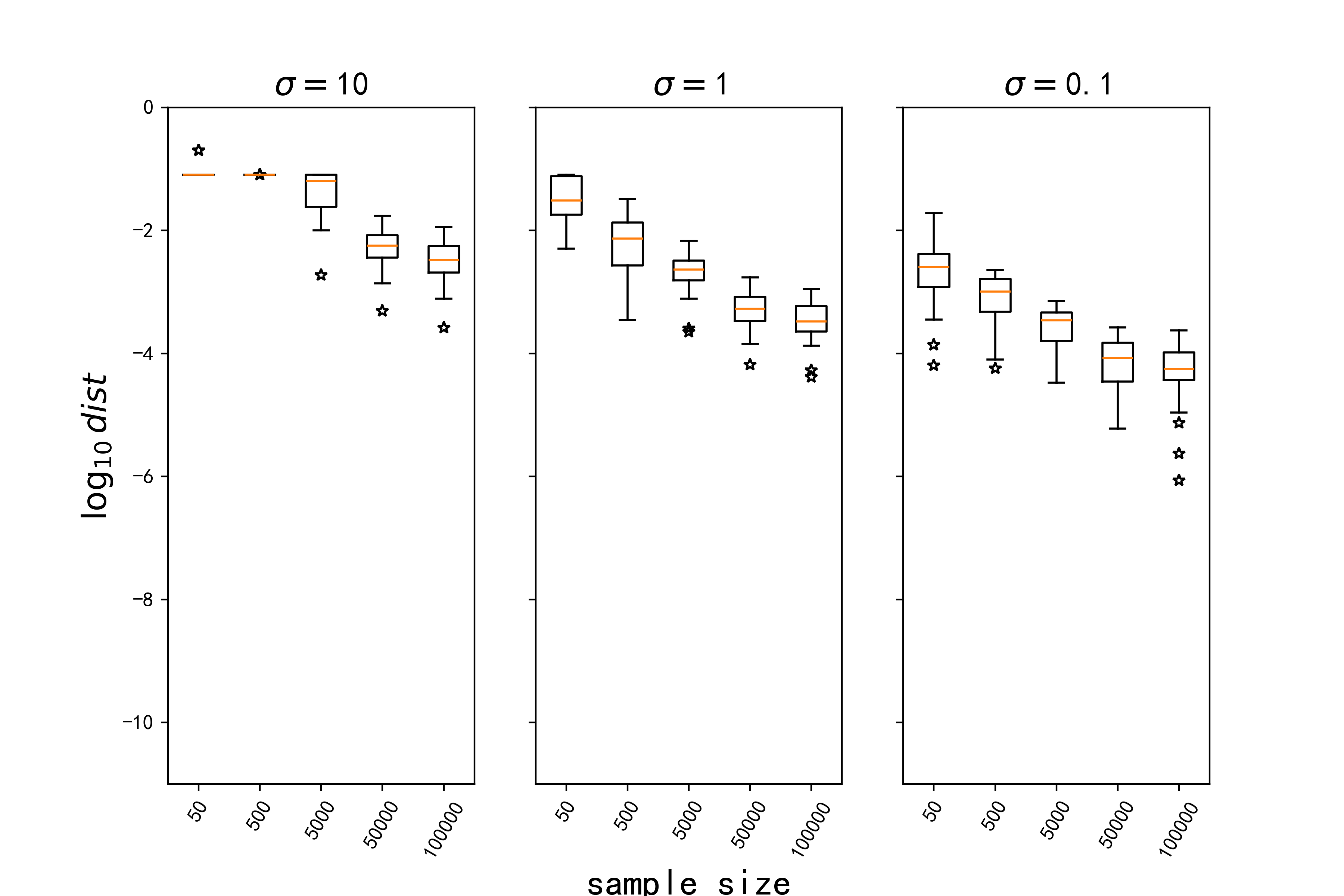}
    \caption{\scriptsize{Performance on S327 w.r.t after 1 500 iterations.}}
    \label{fig:b12}
  \end{minipage}
\end{figure*}}

We can note several general patterns that can be seen from the figures informative in regards to the specific properties of the Algorithm as well as broad insights for stochastic constrained optimization in general.

From these figures, one can see that Algorithm \ref{ssqp} performs better on instances with small levels of noise, which is to be expected.

In addition, we see that using a larger sample size does seem to reduce the overall variation in the objective performance, however not its expected value. Furthermore, there is often a threshold at which the additional variance reduction is marginal with increasing sample size. This threshold varies depending on the problem.

Finally, we see that as standard with stochastic optimization, there is an exponential increase in the number of iterations required to achieve an order of magnitude. Moreover, there is considerable variation across problems as to how steep this relationship is. This confirms the overall understanding that there is considerable cost in total samples necessary to achieve precision.

{
\subsection{Learning directed acyclic graphs}

We have used Algorithm 2 to learn the structure (in the form of directed acyclic graphs (DAGs) and weights of Bayesian Networks (BN) from data in the framework of NOTEARS \cite{NEURIPS2018_e347c514}. In this work, the problem of estimating the connections of the DAG defining the BN is formulated as purely continuous constrained  optimization problem.

BNs  find wide applications in a variety of fields, ranging from bio-informatics to health care, image processing, economic risk analysis, causal inference, and engineering fault diagnosis \cite{Pourret2008,Kitson2023}. Learning DAGs from data is known as an NP-hard problem \cite{Chickering1996},due to the combinatorial acyclicity constraint.

One of the most used approaches to learning structure of DAGs is score-based learning. Loosely speaking, the main task is to find a highest scoring graph according to some function measuring fit and parsimony. Traditionally, the goal is to solve the following combinatorial optimization problem:
\begin{equation}\label{eq:dags_com}
    \begin{split}
        \min\limits_{W\in\mathbb{R}^{d\times d}}&\ F(W)\\
        \text{s.t.}&\ G(W)\in \textrm{DAGs,}
    \end{split}
\end{equation}
where $G(W)$ is the $d$-node graph introduced by the weighted adjacency matrix $W$, $F$ is a score function.

The task is to learn a DAG for the joint distribution of a d-dimensional random vector $X=(X_1,X_2,\cdots,X_d)$, given a data matrix of $n$ i.i.d observations of $\textbf{X}$, which is modeled via a structural equation model (SEM) defined by a weighted adjacency matrix $W\in\mathbb{R}^{d\times d}$.
}{
\subsubsection{A sketch of NOTEARS}
Zheng et al. \cite{NEURIPS2018_e347c514} proposed NOTEARS, an approach for score-based learning of BNs, where problem \eqref{eq:dags_com} is converted into a continuous program:
\begin{equation}\label{eq:notears}
    \begin{split}
        \min\limits_{W\in\mathbb{R}^{d\times d}}&\ F(W)\\
        \text{s.t.}&\ h(W)=0,
    \end{split}
\end{equation}
where $h:\mathbb{R}^{d\times d}\rightarrow\mathbb{R}$ is a smooth function over real matrices. The equation $h(W)=0$ exactly characterizes the acyclic graphs, i.e.
\[
 G(W)\in \textrm{DAGs,}\Longleftrightarrow h(W)
=0.\]

In \cite{NEURIPS2018_e347c514}, the authors focus on the setting of linear SEM:
\begin{equation}\label{eq:sample}
    X_j={X}W_{\cdot j}+z_j, \forall j=1,2,,d,
\end{equation}
where $W_{\cdot j}$ is the $j-$th column of $W$ and $z_j$ is random noise. The score function $F(W)$ is chosen to be a regularized least square loss function:
\begin{equation}\label{eq:fw}
   F(W)=\frac1{2n}\|\textbf{X}-\textbf{X}W\|_F^2+\lambda\|W\|_1,
\end{equation}
where $\|W\|_1=\|vec(W)\|_1$ and $\lambda\geq0$ is a regularization parameter. The function $h(W)$ is defined as
\begin{equation}\label{eq:hw}
h(W)=tr(W\circ W)-d,
\end{equation}
where $\circ$ is the Hadamard product and $e^A$ is the matrix exponential of A. NOTEARS uses a classical augmented Lagrangian method to solve \eqref{eq:notears}.

In order to reduce the number of false discoveries, a weight thresholding is used: after getting a approximate stationary point $\tilde{W}$ of \eqref{eq:notears}, given a fixed threshold $w>0$, set any weights smaller than $w$ in absolute value to zero. It is shown that $w=0.3$ works well.

\subsubsection{Numerical results}
We applied Algorithm \ref{ssqp} to \eqref{eq:notears} with the settings \eqref{eq:fw} and \eqref{eq:hw}. We follow the ideas introduced in \cite{NEURIPS2018_e347c514} to generate the test graph: A random graph $G$ was generated from one of two common random graph models, Erd\"{o}s-R\'{e}nyi (ER) or scale-free (SF), and was assigned a weight matrix $W$ from Unif$([-2,-0.5]\cup[0.5,2])$. We sampled \eqref{eq:sample} from Gaussian (Gauss) and Exponential (Exp) noise models with number of nodes $d\in\{10,15\}$, number of edges $ed\in \{d,2d,3d\}$ and sample size $n=1000$. For weight thresholding, we use the setting in NOTEARS, i.e., $w=0.3$. The regularization parameter is $\lambda = 0.1$.

 Learning DAGs is understood to be difficult and solving \eqref{eq:notears} at large scale often becomes more of a scientific computing rather than algorithmic optimization exercise. Here we aim to assess the ability of Algorithm \ref{ssqp} to solve \eqref{eq:notears}, rather than to develop an efficient method for it. The latter needs a large number of additional issues, such as linear algebra tricks, sparse matrix techniques, to enhance efficiency. Therefore, the tests were performed on small graphs only, and run Algori. Practical implementation issues for large scale models will be considered in a further work.

The results of the test, compared with NOTEARS are listed in Table \ref{tab:tab1}. Like \cite{NEURIPS2018_e347c514}, we used four common graph metrics to evaluate the learned graph: 1) False discovery rate (FDR), 2) True positive rate (TPR), 3)False positive rate (FPR), and 4) Structure Hamming distance (SHD).
We also recorded the distance in Frobenius norm from the estimated graph $W^*$ to the ground truth graph $W_{true}$. From the reported results, the behaviors of Algorithm \ref{ssqp} on learning DAGs are encouraging, though more tests, especially tests on large graphs are needed.
}

\begin{sidewaystable}[htbp]
  \centering
  \caption{Numerical Results, Compared with NOTEARS}
    \begin{tabular}{|c|c|c|c|c|c|c|c|c|c|c|c|c|c|}
    \hline
    \multirow{2}{*}{Graph} & \multirow{2}{*}{Noise} & \multirow{2}{*}{Nodes} & \multicolumn{1}{c|}{\multirow{2}{*}{Edges}} & \multicolumn{2}{c|}{FDR} & \multicolumn{2}{c|}{TPR} & \multicolumn{2}{c|}{FPR } & \multicolumn{2}{c|}{SHD} & \multicolumn{2}{c|}{$\|W_{true}-W^*\|$} \\ \cline{5-14}
          &       &       &       & \multicolumn{1}{l|}{Alg.} & \multicolumn{1}{l|}{\scriptsize{NOTEARS}} & \multicolumn{1}{l|}{Alg. \ref{ssqp}} & \multicolumn{1}{l|}{\scriptsize{NOTEARS}} & \multicolumn{1}{l|}{Alg. \ref{ssqp}} & \multicolumn{1}{l|}{\scriptsize{NOTEARS}} & \multicolumn{1}{l|}{Alg. \ref{ssqp}} & \multicolumn{1}{l|}{\scriptsize{NOTEARS}} & \multicolumn{1}{l|}{Alg. \ref{ssqp}} & \multicolumn{1}{l|}{\scriptsize{NOTEARS}} \\
    \hline
    \multirow{12}{*}{ER} & \multirow{6}{*}{Guass} & \multirow{3}{*}{10} & 10    & 0.00  & 0.00  & 1.00  & 1.00  & 0.00  & 0.00  & 0     & 0     & 0.786  & 0.626  \\
          &       &       & 20    & 0.00  & 0.15  & 0.90  & 0.85  & 0.00  & 0.12  & 2     & 4     & 1.265  & 2.607  \\
          &       &       & 30    & 0.14  & 0.04  & 0.80  & 0.83  & 0.27  & 0.07  & 9     & 5     & 3.024  & 2.820  \\ \cline{3-14}
          &       & \multirow{3}{*}{15} & 15    & 0.00  & 0.00  & 1.00  & 1.00  & 0.00  & 0.00  & 0     & 0     & 0.784  & 0.468  \\
          &       &       & 30    & 0.00  & 0.00  & 0.97  & 0.97  & 0.00  & 0.00  & 1     & 1     & 1.205  & 1.232  \\
          &       &       & 45    & 0.03  & 0.12  & 0.82  & 0.82  & 0.02  & 0.08  & 9     & 9     & 2.943  & 3.218  \\ \cline{2-14}
          & \multirow{6}{*}{Exp} & \multirow{3}{*}{10} & 10    & 0.20  & 0.00  & 0.80  & 1.00  & 0.06  & 0.00  & 3     & 0     & 1.831  & 0.556  \\
          &       &       & 20    & 0.00  & 0.00  & 0.75  & 0.85  & 0.00  & 0.00  & 5     & 3     & 2.069  & 1.654  \\
          &       &       & 30    & 0.23  & 0.00  & 0.90  & 0.90  & 0.53  & 0.00  & 11    & 3     & 3.270  & 2.419  \\ \cline{3-14}
          &       & \multirow{3}{*}{15} & 15    & 0.08  & 0.00  & 0.73  & 0.93  & 0.01  & 0.00  & 5     & 1     & 2.084  & 0.887  \\
          &       &       & 30    & 0.21  & 0.03  & 0.87  & 0.93  & 0.09  & 0.01  & 11    & 2     & 3.255  & 1.341  \\
          &       &       & 45    &  0.29     & 0      &0.78       & 0.87      &0.23       &0.00       &23      &6       &4.581       &2.530  \\
    \hline
    \multirow{12}{*}{SF} & \multirow{6}{*}{Guass} & \multirow{3}{*}{10} & 10    & 0.00  & 0.00  & 1.00  & 1.00  & 0.00  & 0.00  & 0     & 0     & 0.293  & 0.274  \\
          &       &       & 20    & 0.26  & 0.14  & 0.82  & 0.71  & 0.18  & 0.07  & 7     & 6     & 2.704  & 2.534  \\
          &       &       & 30    & 0.00  & 0.00  & 0.92  & 0.92  & 0.00  & 0.00  & 0     & 0     & 4.646  & 3.492  \\ \cline{3-14}
          &       & \multirow{3}{*}{15} & 15    & 0.00  & 0.00  & 1.00  & 1.00  & 0.00  & 0.00  & 0     & 0     & 0.349  & 0.341  \\
          &       &       & 30    & 0.00  & 0.00  & 1.00  & 1.00  & 0.00  & 0.00  & 0     & 0     & 0.609  & 0.573  \\
          &       &       & 45    & 0.03  & 0.03  & 0.95  & 0.95  & 0.02  & 0.02  & 3     & 3     & 2.315  & 2.385  \\ \cline{2-14}
          & \multirow{6}{*}{Exp} & \multirow{3}{*}{10} & 10    & 0.10      &0.00       &1.00       &1.00       &0.03       &0.00       &1       & 0      &1.781       &0.243  \\
          &       &       & 20    & 0.00      &0.00       &0.941       &0.882       &0.00       &0.00       & 1      &2       &1.146       &1.467  \\
          &       &       & 30    & 0.00      &0.05       &0.83       &0.79       &0.00       &0.05       &4       & 6      &2.535       &3.245  \\ \cline{3-14}
          &       & \multirow{3}{*}{15} & 15    & 0.07      & 0.00      &1.00       &1.00       &0.01       &0.00       &1       &0       &0.727       &0.324  \\
          &       &       & 30    & 0.00      &0.00       &0.96       &0.96       &0.00       &0.00       & 1      &1       & 1.418      &1.303  \\
          &       &       & 45    & 0.17      &0.00       &0.87       &1.00       &0.11       &  0.00     & 12      & 0      & 3.077      &2.089  \\
    \hline
    \end{tabular}%
  \label{tab:tab1}%
\end{sidewaystable}%

\section{Conclusions} \label{sect:conc}

In this paper, we proposed a robust SQP method for optimization with stochastic objective functions and deterministic constraints. The presented method generalizes Paquette and Scheinberg's line search method for unconstrained stochastic optimizations to the SQP method for  constrained stochastic optimization of the form \eqref{eq:prob}. Ideas of Burke and Han's robust SQP \cite{burke1989robust} method are adopted to ensure the consistency of QP subproblems. Global convergence in the case where the penalty parameter keeps bounded is proved, meanwhile the probability of the penalty parameter approaching infinity is shown to be 0, with a specific sampling method.

Some advanced SQP schemes in deterministic nonlinear programming, for example, the inexact SQP scheme \cite{ByrdCN2008}, the filterSQP technique \cite{FletcL2002,FletcLT2002}, the stabilized SQP scheme \cite{Wright1998} and etc, can be used to extend this work. It will be also interesting and deserves studying to exploit other classical schemes for nonlinear constrained optimizations, such as the augmented Lagrangian method and interior point method, to solve stochastic problems. Finally, establishing fast local convergence as well as real time iteration and other online schemes of using SQP, two of the classical strengths of SQP methods, are of interest.

\section*{Acknowledgement}
Computing resources for the paper were supported by the OP VVV project:
CZ.02.1.01/0.0/0.0/16\_019/0000765 ``Research Center for Informatics''.


\section*{Funding}
Songqiang Qiu was supported by a scholarship granted by the China Scholarship Council (No. 202006425023).
Vyacheslav Kungurtsev was supported by the European Union's Horizon Europe research and innovation programme under grant agreement No. 101084642.

\bibliographystyle{plain}
\bibliography{stosqprv}
\newpage
\appendix
\section{Remaining figures of the numerical experiments}
\begin{figure*}[!h]
  \begin{minipage}[t]{0.5\linewidth}
    \centering
    \includegraphics[scale=0.35]{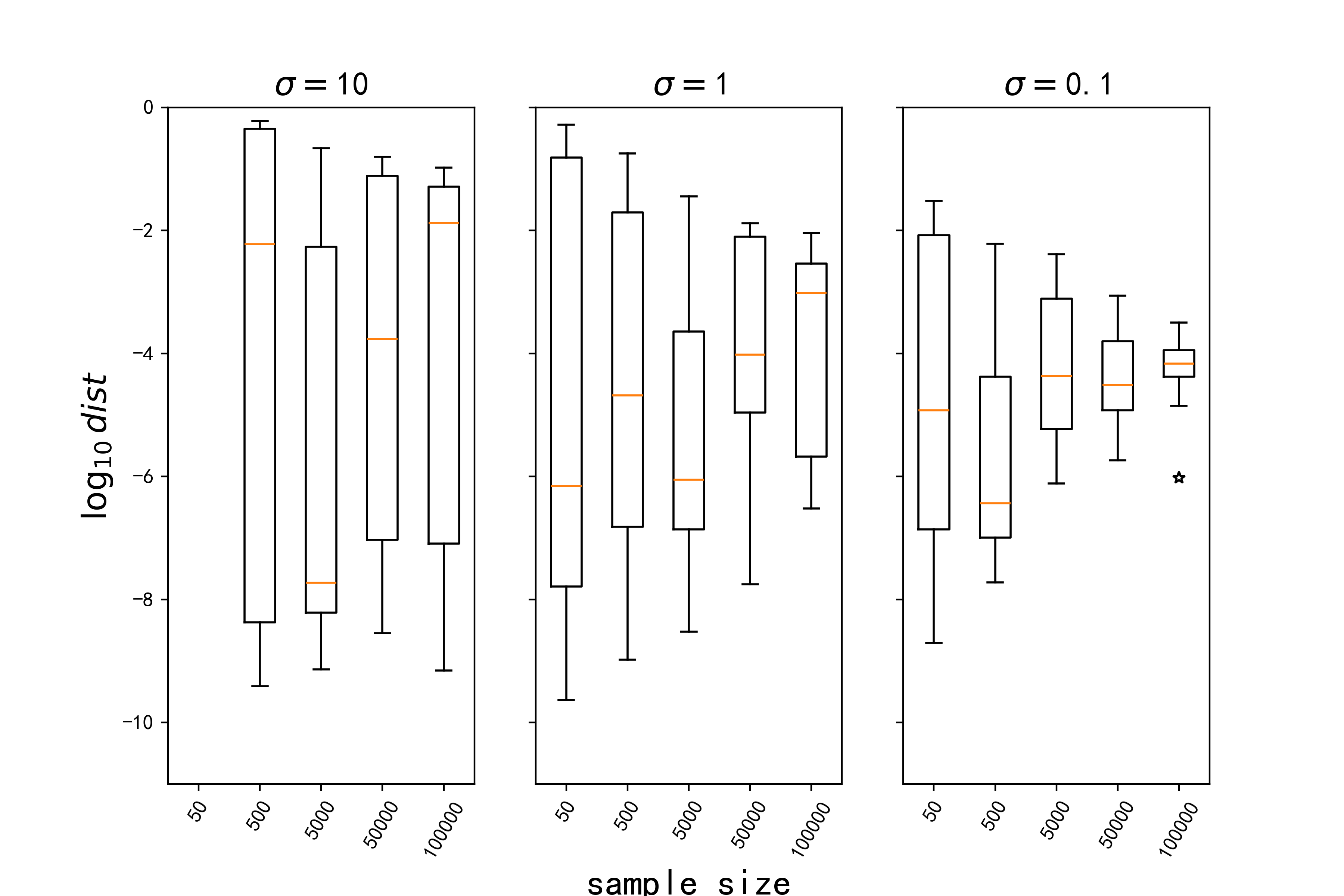}
    \caption{\scriptsize{Performance on HS06 w.r.t after 50 iterations.}}
  \end{minipage}%
  \begin{minipage}[t]{0.5\linewidth}
    \centering
    \includegraphics[scale=0.35]{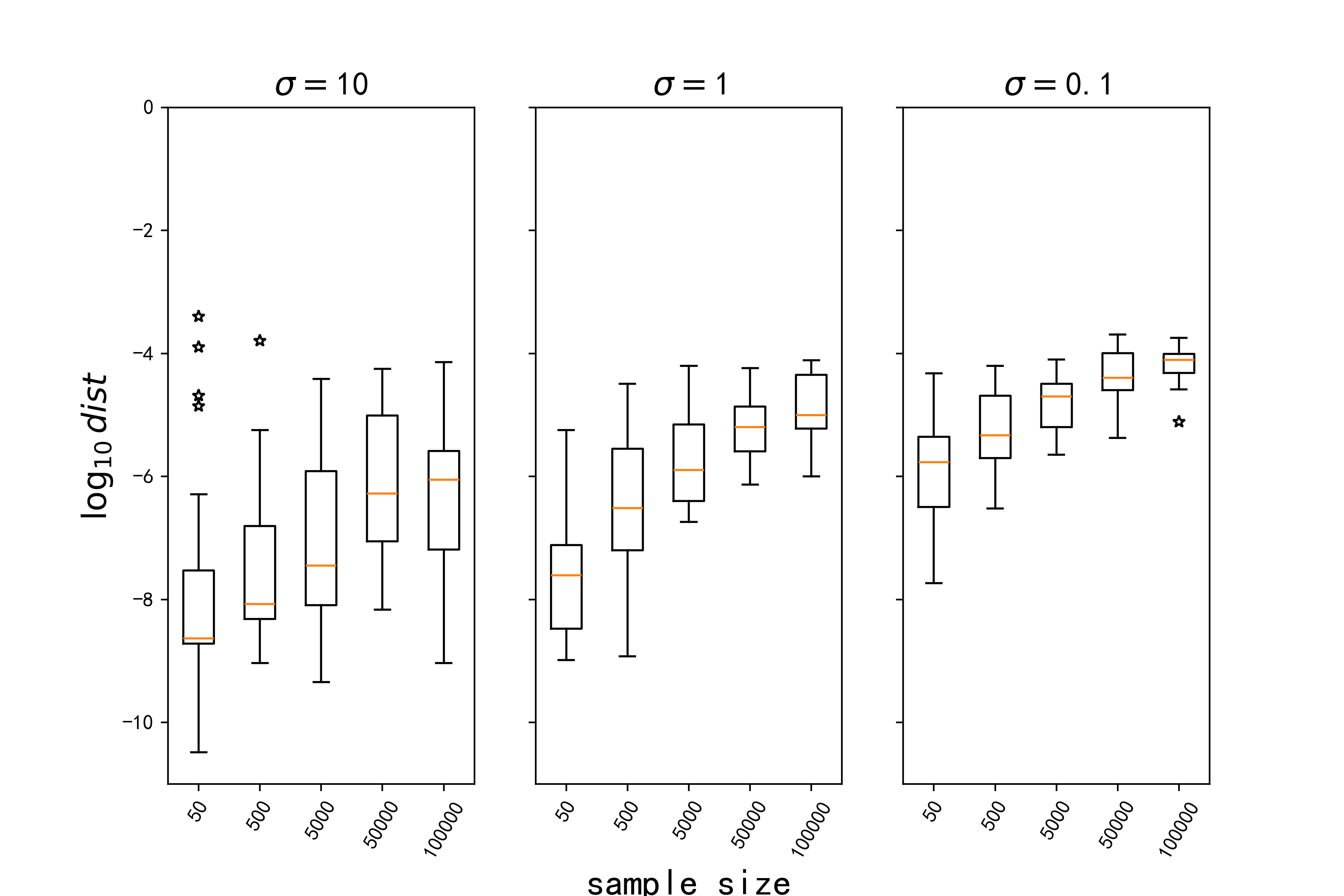}
    \caption{\scriptsize{Performance on HS06 w.r.t after 1 500 iterations.}}
  \end{minipage}
\end{figure*}

\begin{figure*}[!h]
  \begin{minipage}[t]{0.5\linewidth}
    \centering
    \includegraphics[scale=0.35]{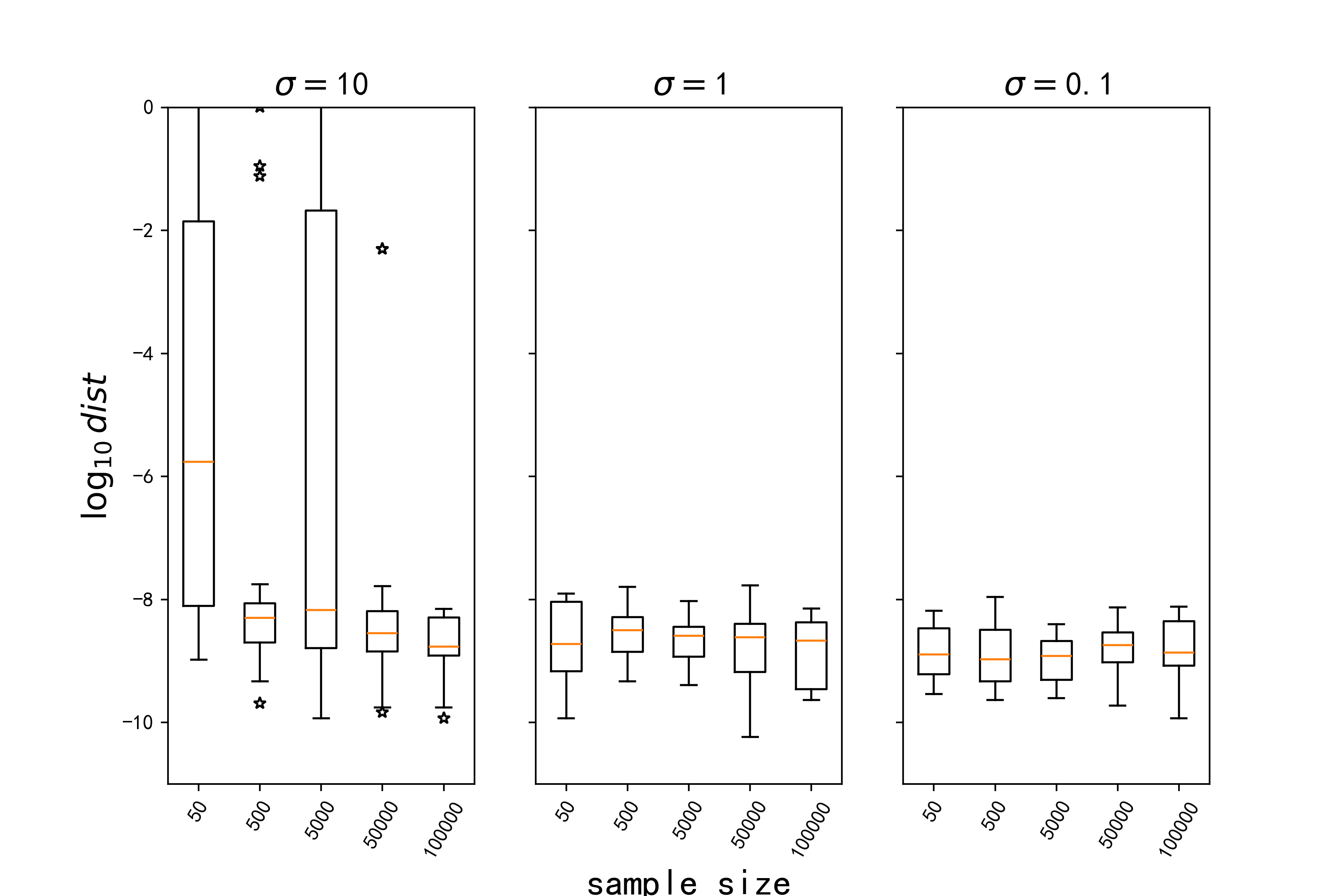}
    \caption{\scriptsize{Performance on HS12 w.r.t after 50 iterations.}}
  \end{minipage}%
  \begin{minipage}[t]{0.5\linewidth}
    \centering
    \includegraphics[scale=0.35]{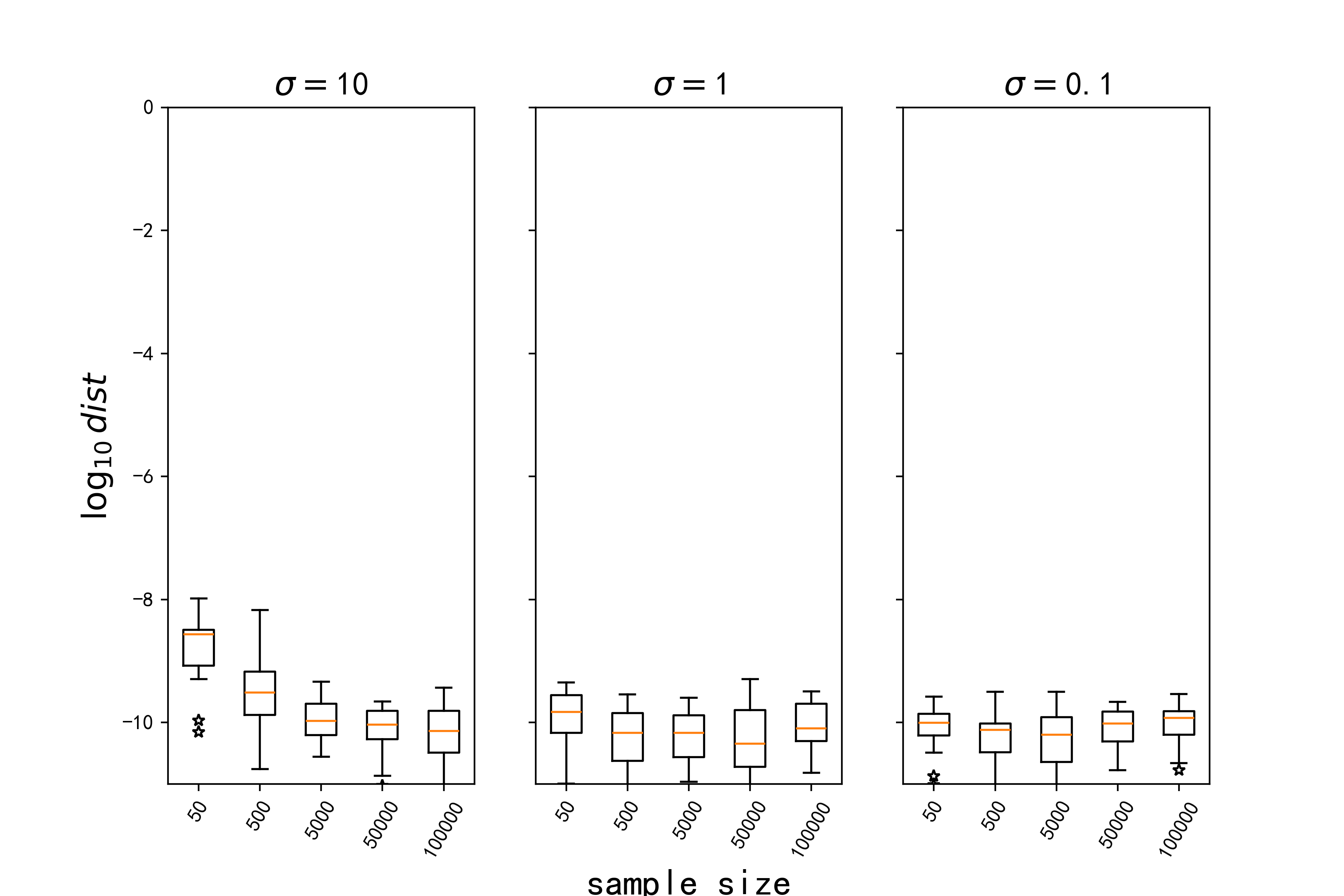}
    \caption{\scriptsize{Performance on HS12 w.r.t after 1 500 iterations.}}
  \end{minipage}
\end{figure*}
\clearpage

\begin{figure*}[!h]
  \begin{minipage}[t]{0.5\linewidth}
    \centering
    \includegraphics[scale=0.35]{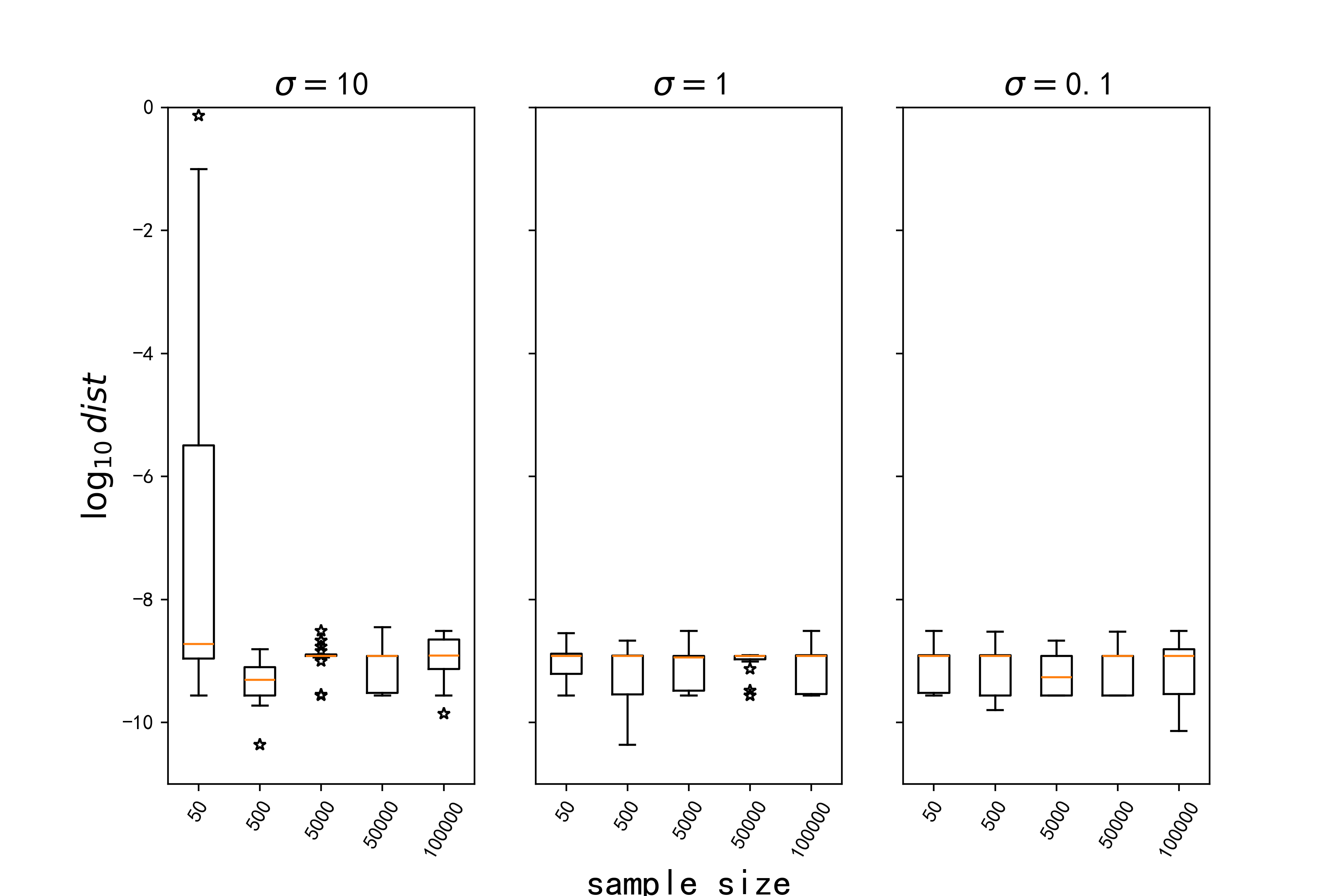}
    \caption{\scriptsize{Performance on HS14 w.r.t after 50 iterations.}}
  \end{minipage}%
  \begin{minipage}[t]{0.5\linewidth}
    \centering
    \includegraphics[scale=0.35]{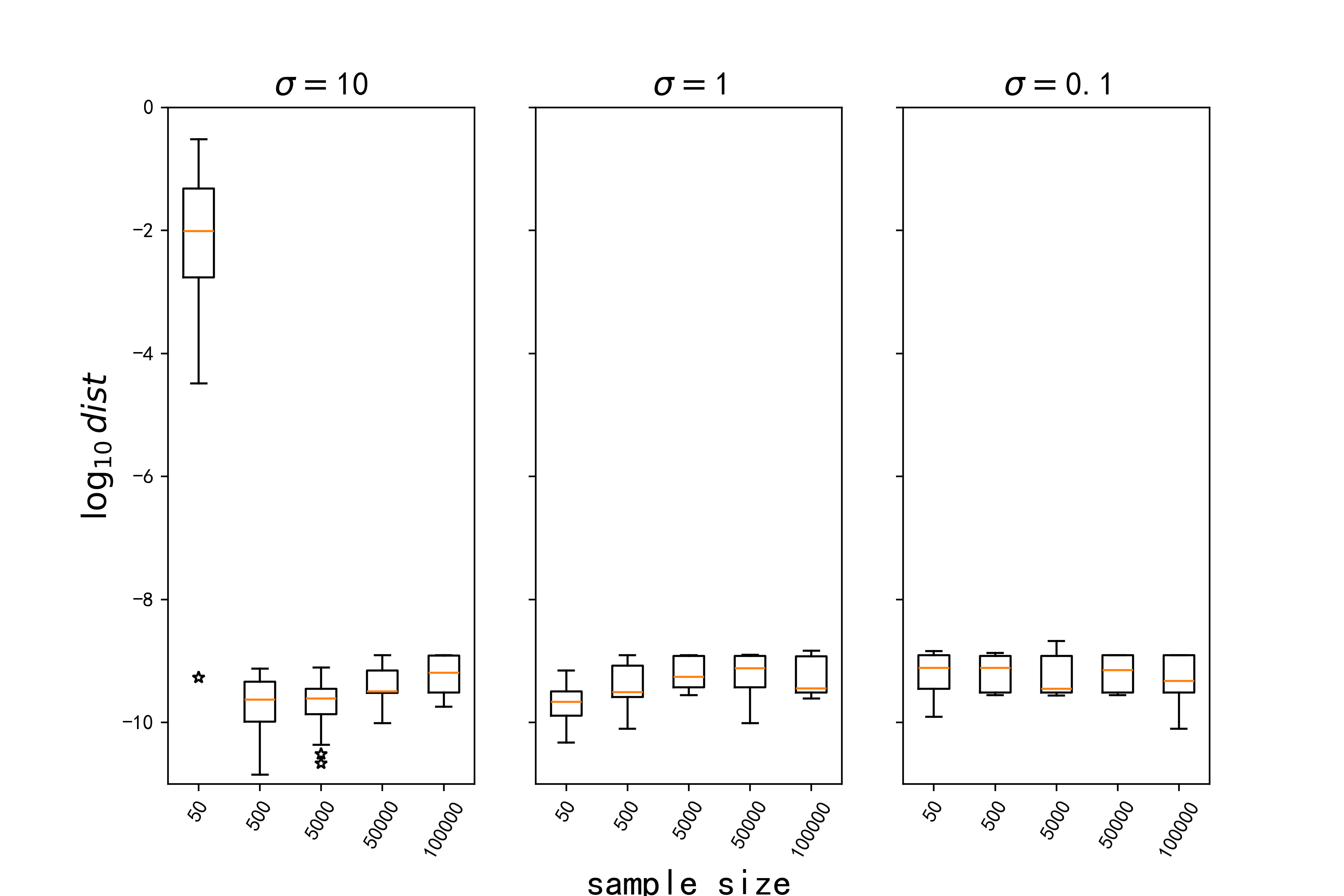}
    \caption{\scriptsize{Performance on HS14 w.r.t after 1 500 iterations.}}
  \end{minipage}
\end{figure*}

\begin{figure*}[!h]
  \begin{minipage}[t]{0.5\linewidth}
    \centering
    \includegraphics[scale=0.35]{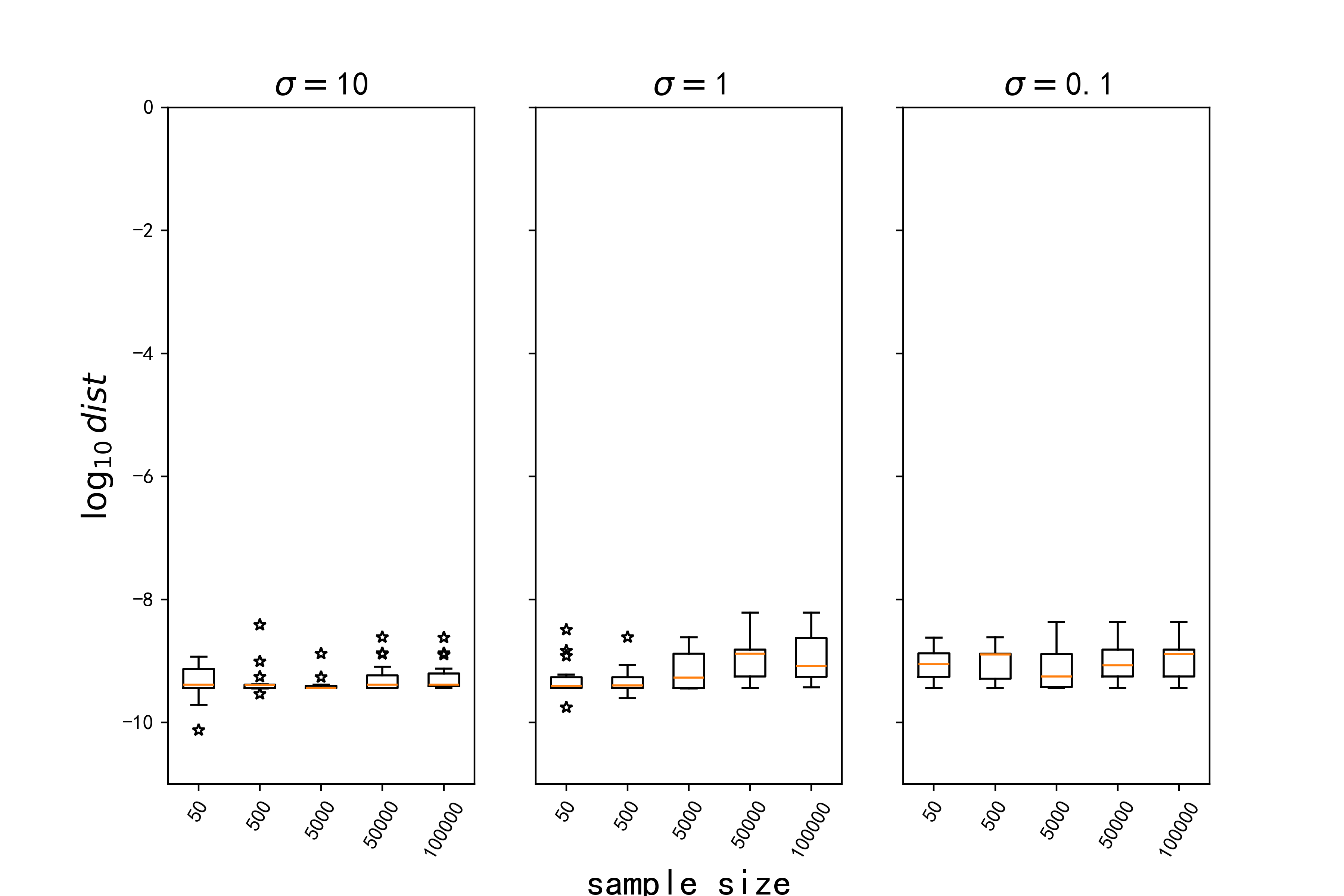}
    \caption{\scriptsize{Performance on HS15 w.r.t after 50 iterations.}}
  \end{minipage}%
  \begin{minipage}[t]{0.5\linewidth}
    \centering
    \includegraphics[scale=0.35]{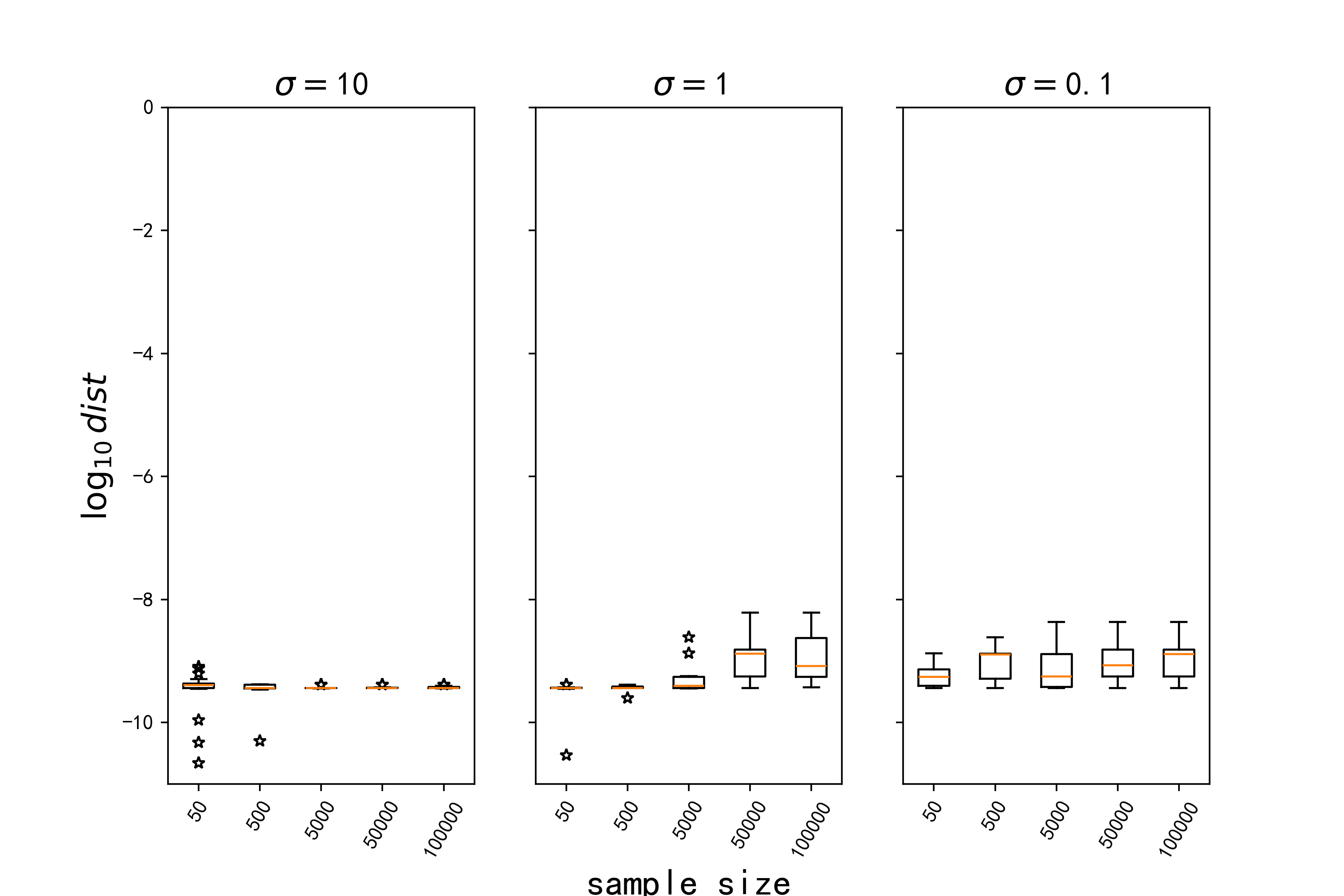}
    \caption{\scriptsize{Performance on HS15 w.r.t after 1 500 iterations.}}
  \end{minipage}
\end{figure*}
\clearpage

\begin{figure*}[!h]
  \begin{minipage}[t]{0.5\linewidth}
    \centering
    \includegraphics[scale=0.35]{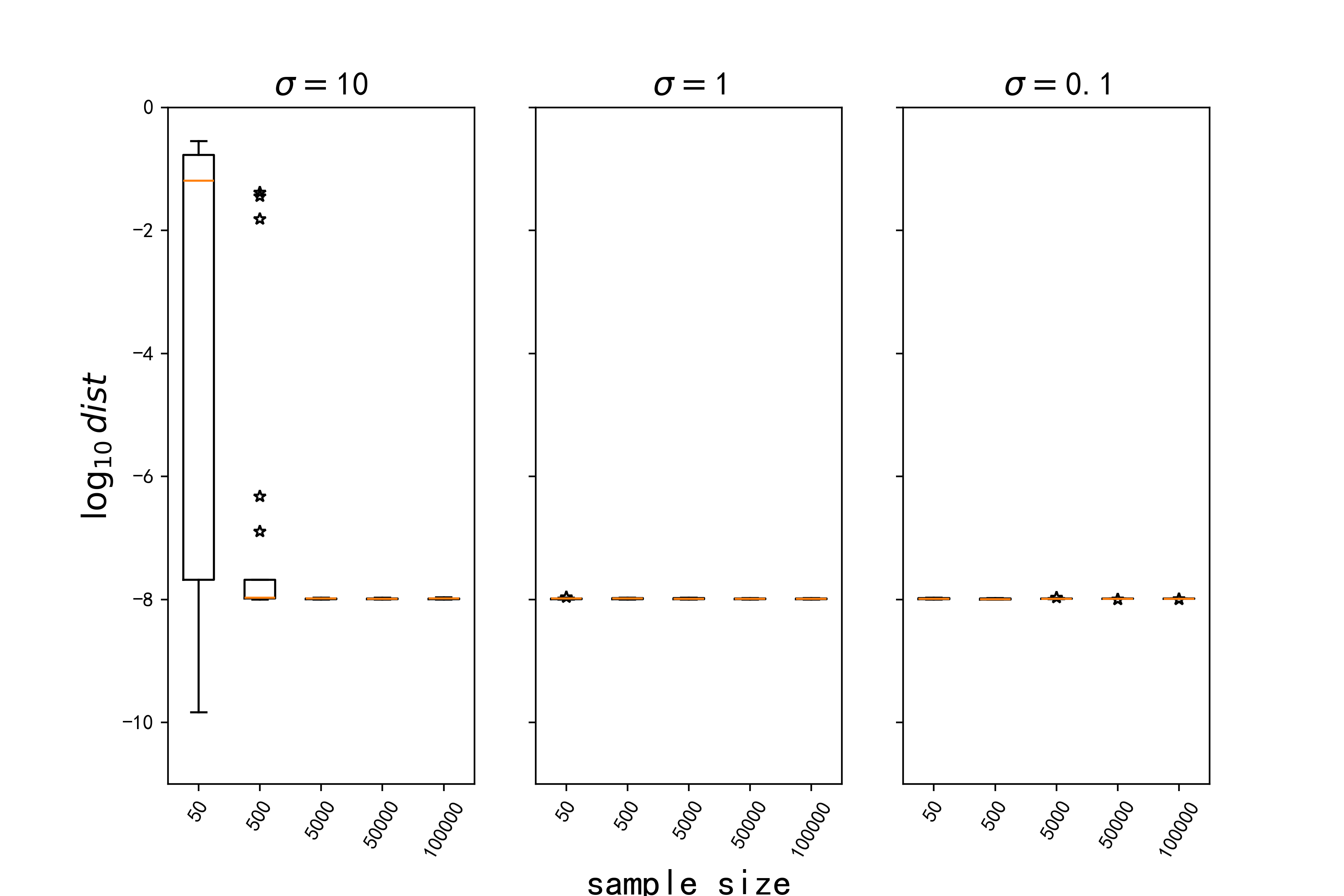}
    \caption{\scriptsize{Performance on HS16 w.r.t after 50 iterations.}}
  \end{minipage}%
  \begin{minipage}[t]{0.5\linewidth}
    \centering
    \includegraphics[scale=0.35]{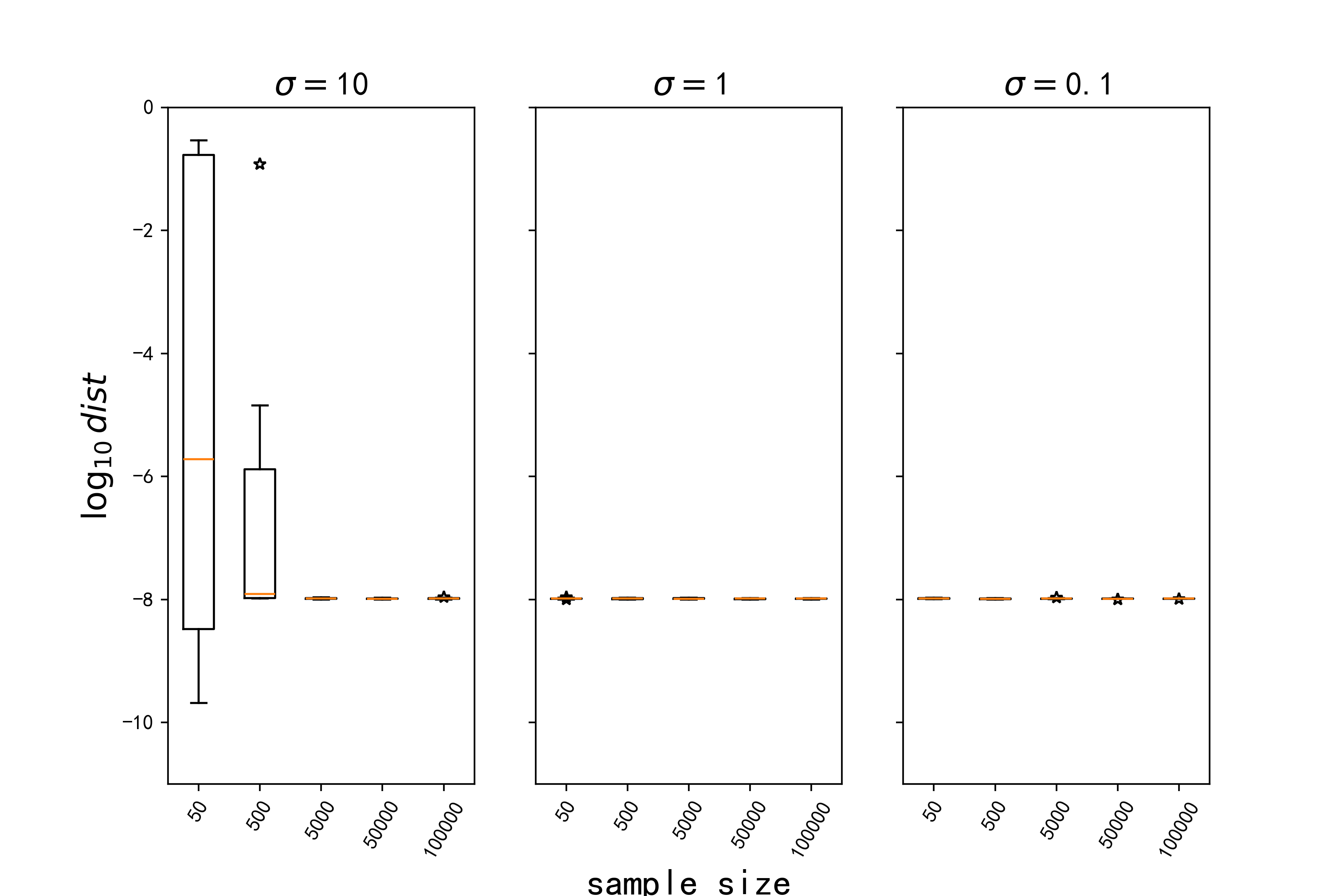}
    \caption{\scriptsize{Performance on HS16 w.r.t after 1 500 iterations.}}
  \end{minipage}
\end{figure*}

\begin{figure*}[!h]
  \begin{minipage}[t]{0.5\linewidth}
    \centering
    \includegraphics[scale=0.35]{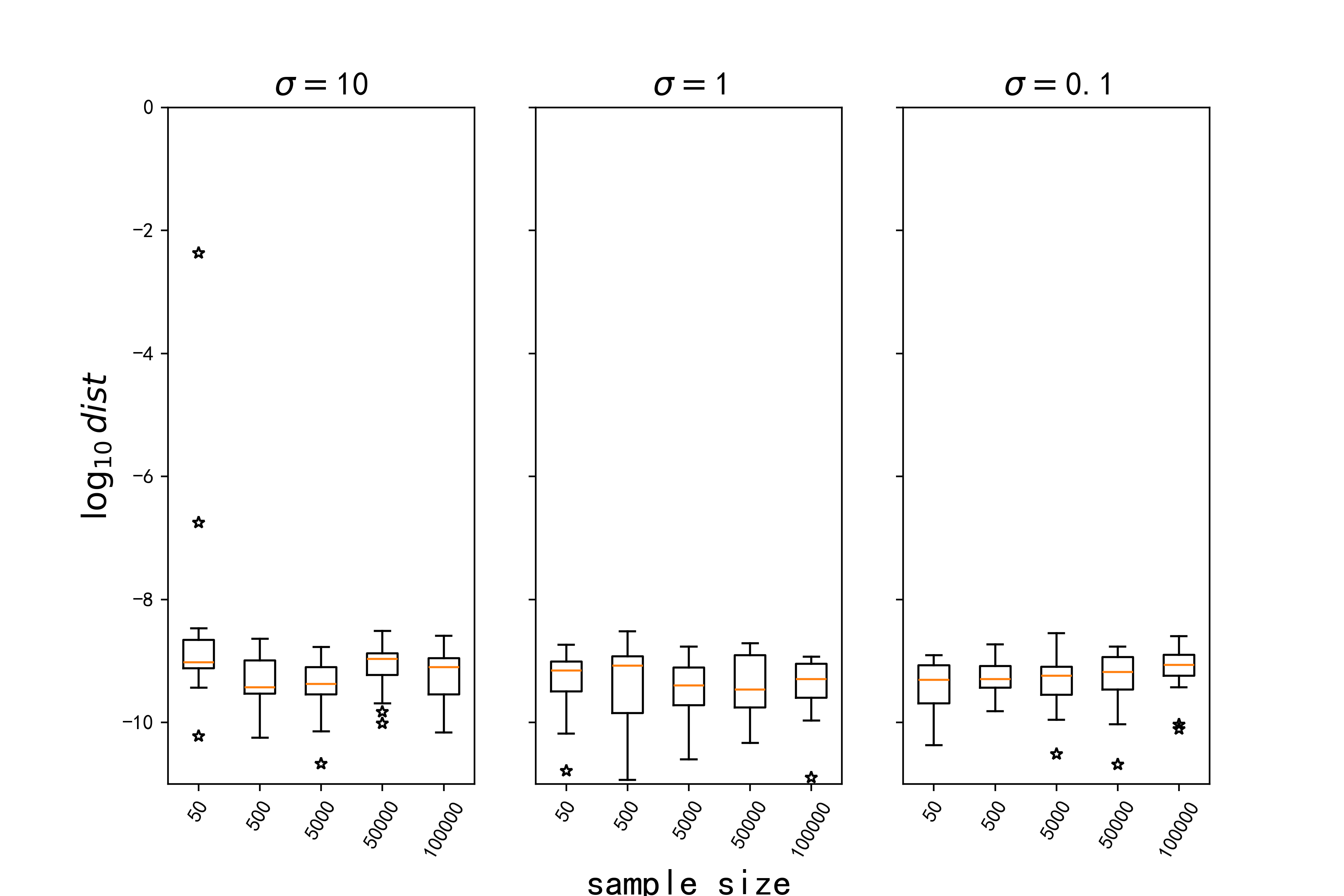}
    \caption{\scriptsize{Performance on HS17 w.r.t after 50 iterations.}}
  \end{minipage}%
  \begin{minipage}[t]{0.5\linewidth}
    \centering
    \includegraphics[scale=0.35]{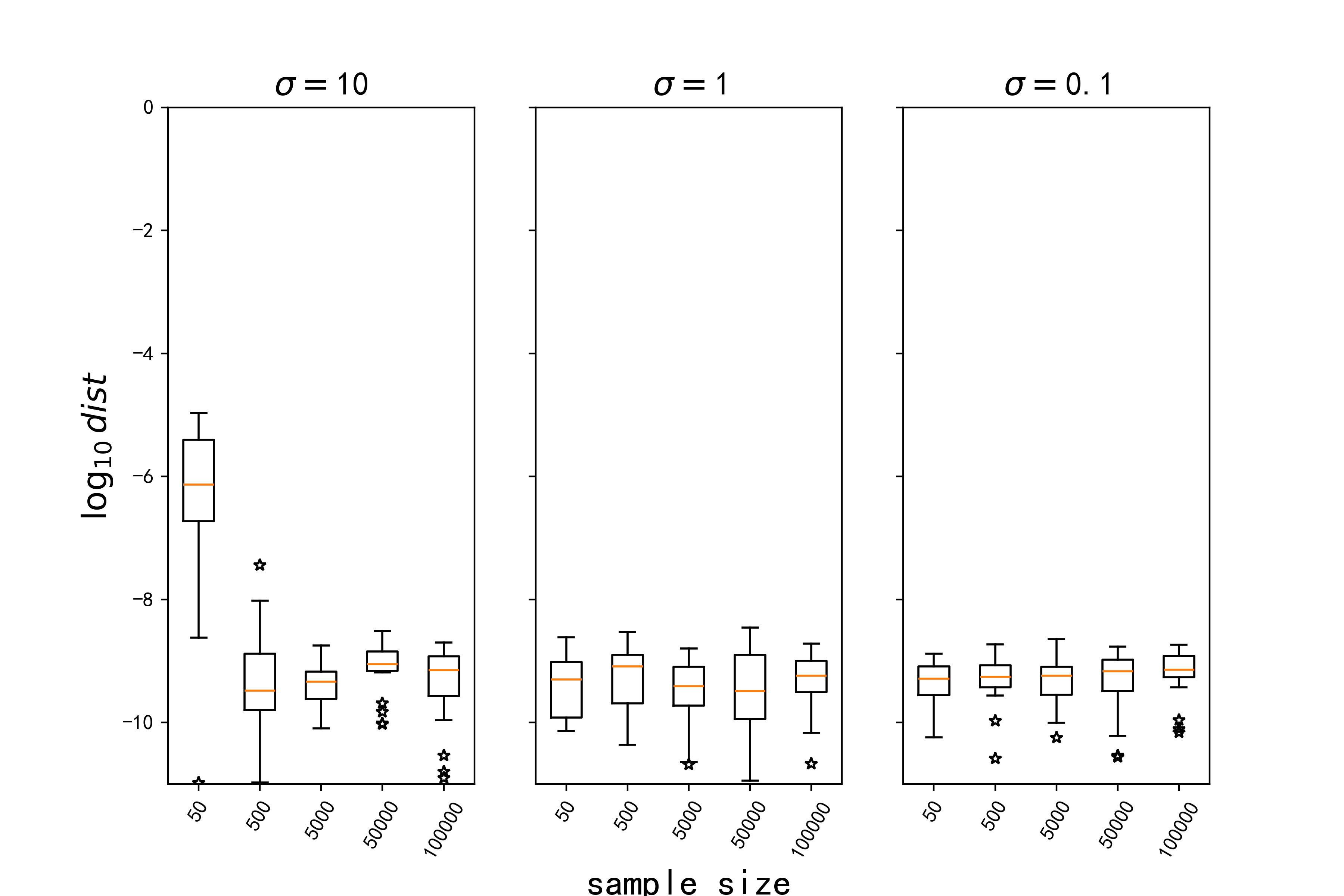}
    \caption{\scriptsize{Performance on HS17 w.r.t after 1 500 iterations.}}
  \end{minipage}
\end{figure*}
\clearpage

\begin{figure*}[!h]
  \begin{minipage}[t]{0.5\linewidth}
    \centering
    \includegraphics[scale=0.35]{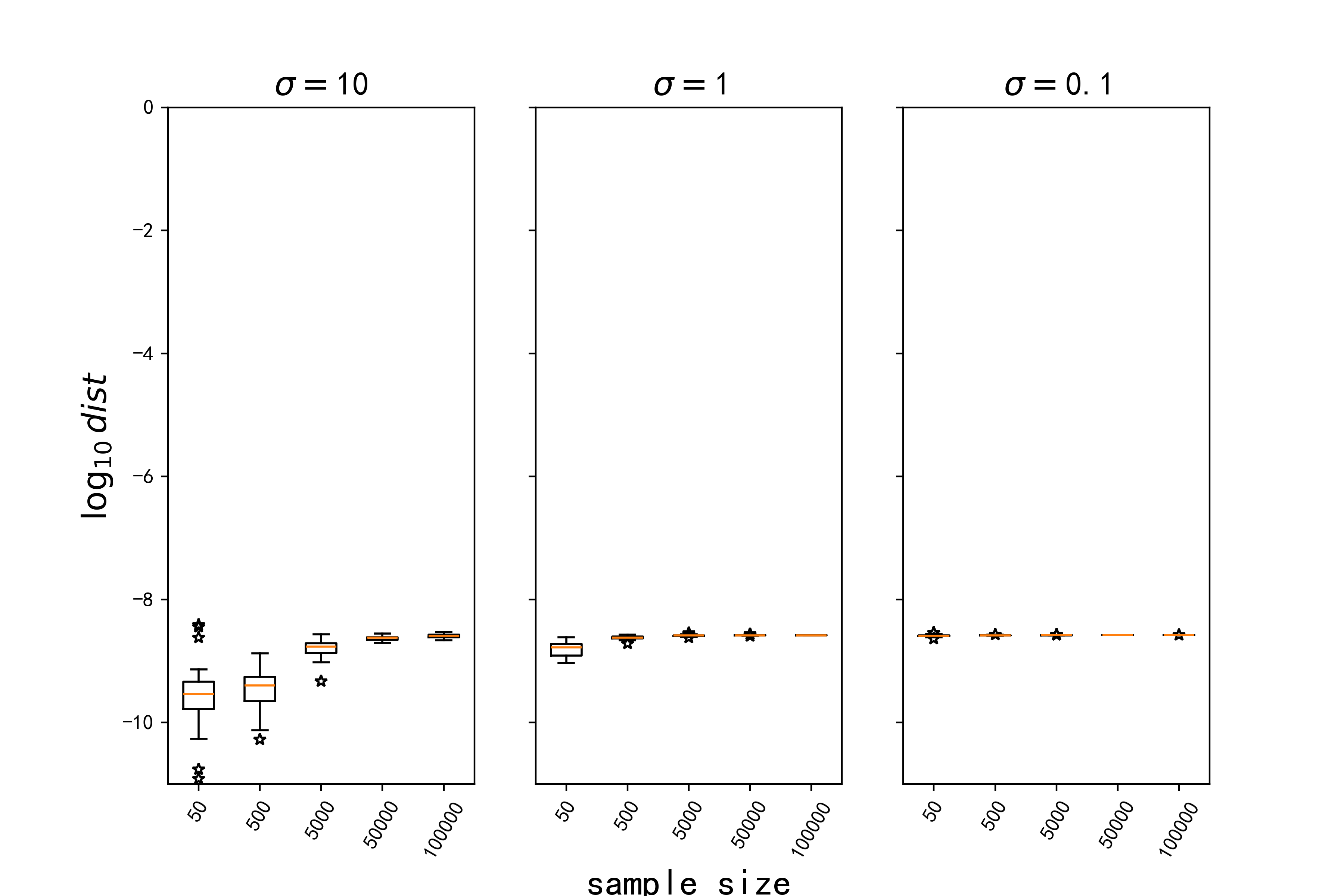}
    \caption{\scriptsize{Performance on HS18 w.r.t after 50 iterations.}}
  \end{minipage}%
  \begin{minipage}[t]{0.5\linewidth}
    \centering
    \includegraphics[scale=0.35]{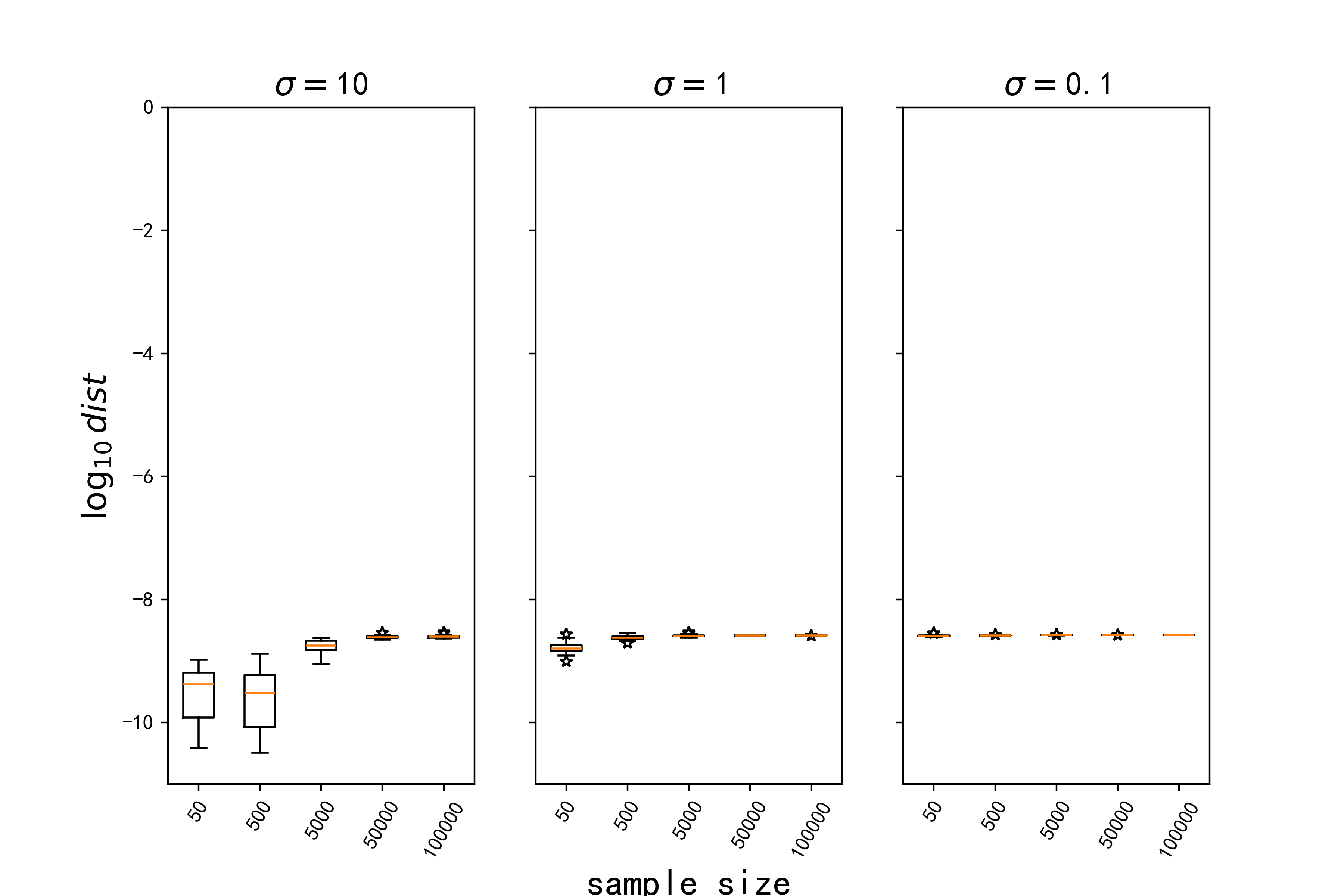}
    \caption{\scriptsize{Performance on HS18 w.r.t after 1 500 iterations.}}
  \end{minipage}
\end{figure*}

\begin{figure*}[!h]
  \begin{minipage}[t]{0.5\linewidth}
    \centering
    \includegraphics[scale=0.35]{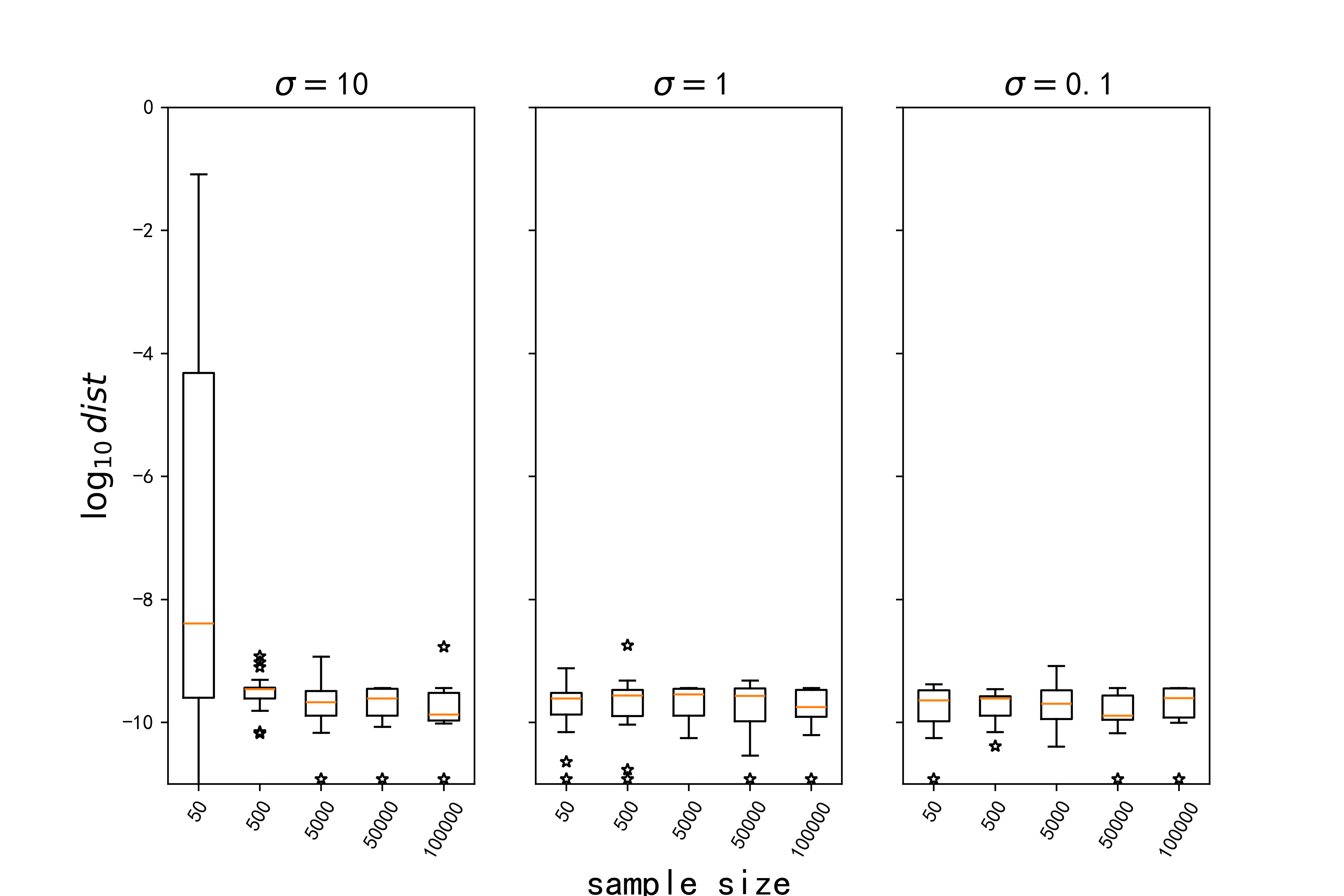}
    \caption{\scriptsize{Performance on HS20 w.r.t after 50 iterations.}}
  \end{minipage}%
  \begin{minipage}[t]{0.5\linewidth}
    \centering
    \includegraphics[scale=0.35]{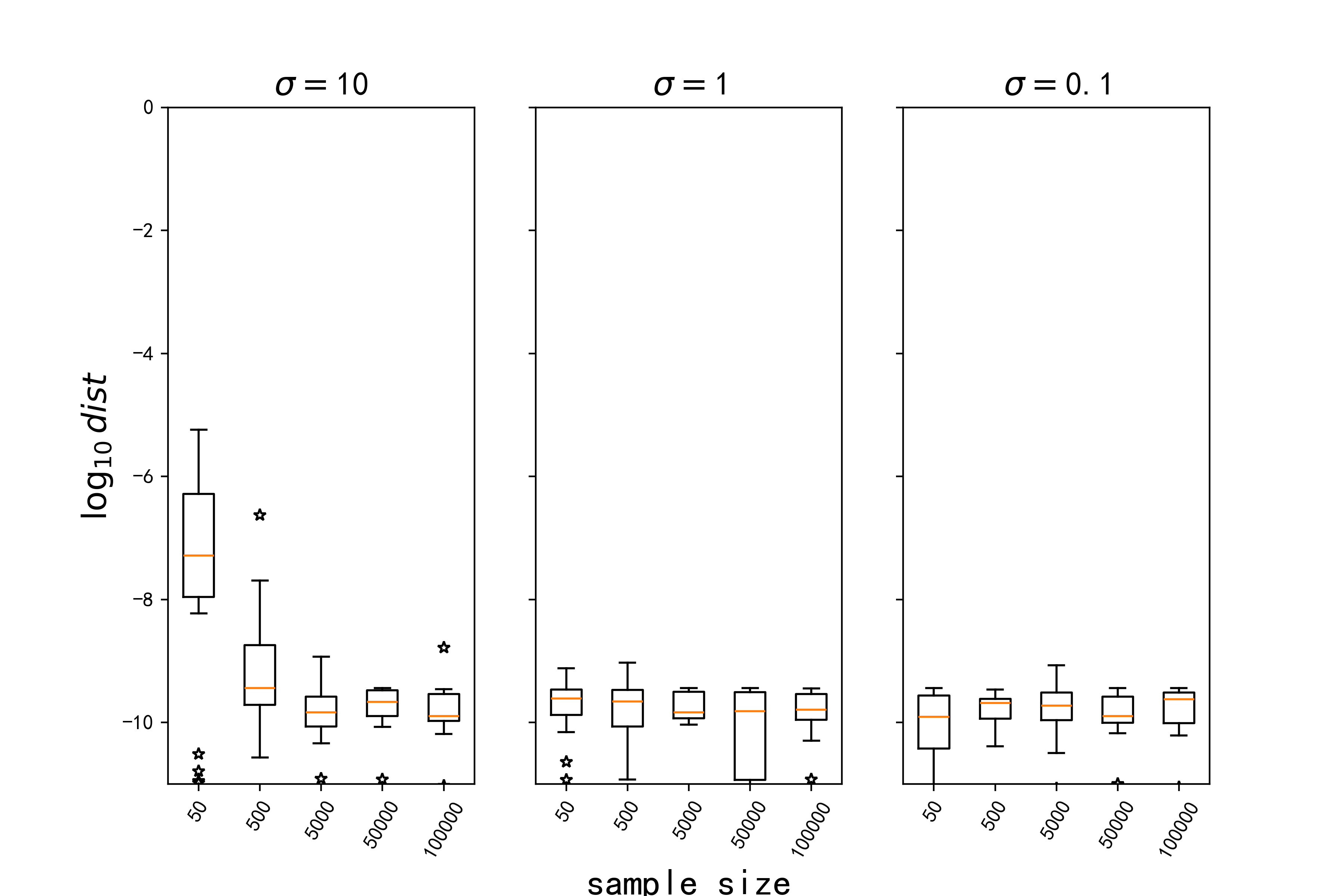}
    \caption{\scriptsize{Performance on HS20 w.r.t after 1 500 iterations.}}
  \end{minipage}
\end{figure*}
\clearpage

\begin{figure*}[!h]
  \begin{minipage}[t]{0.5\linewidth}
    \centering
    \includegraphics[scale=0.35]{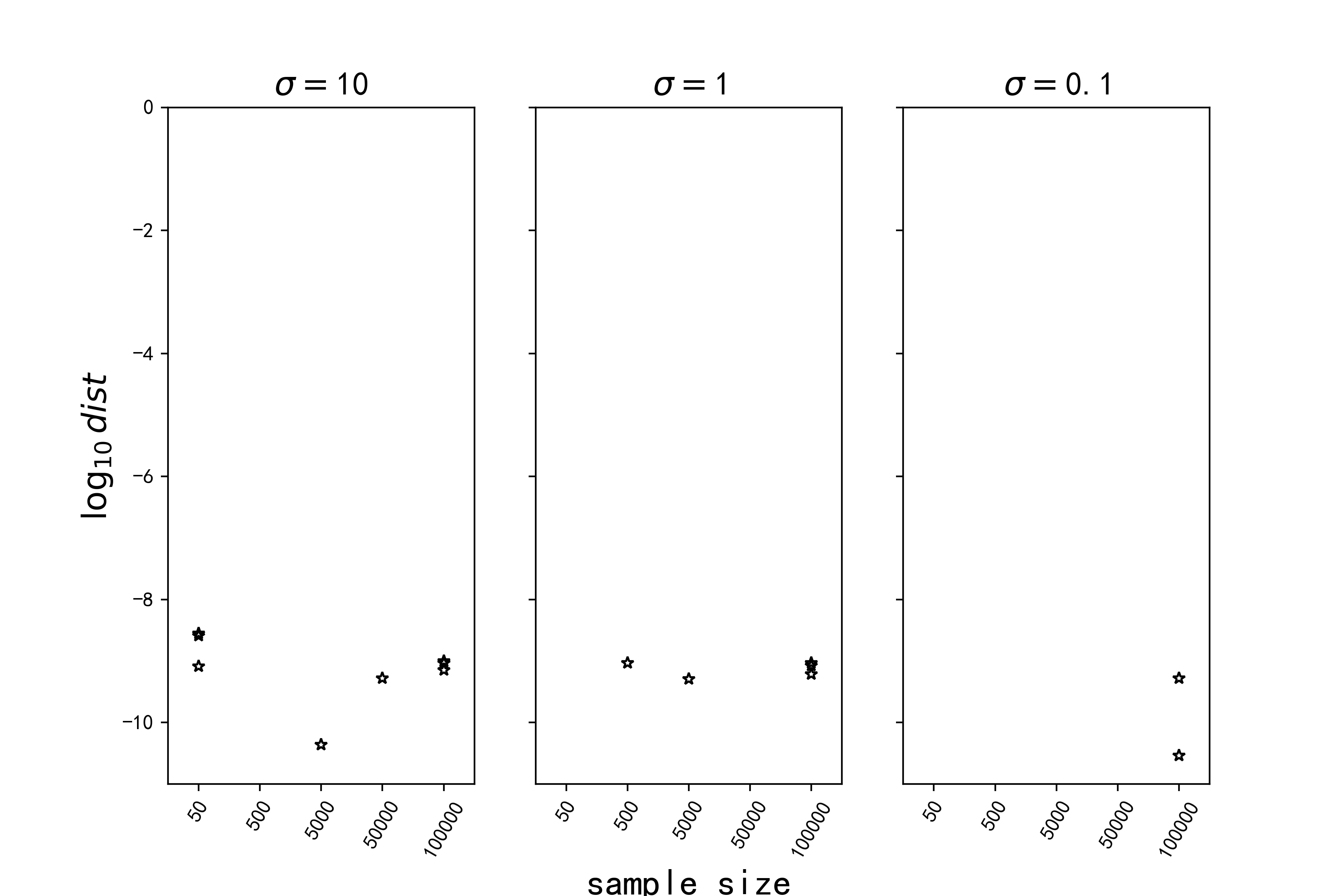}
    \caption{\scriptsize{Performance on HS22 w.r.t after 50 iterations.}}
  \end{minipage}%
  \begin{minipage}[t]{0.5\linewidth}
    \centering
    \includegraphics[scale=0.35]{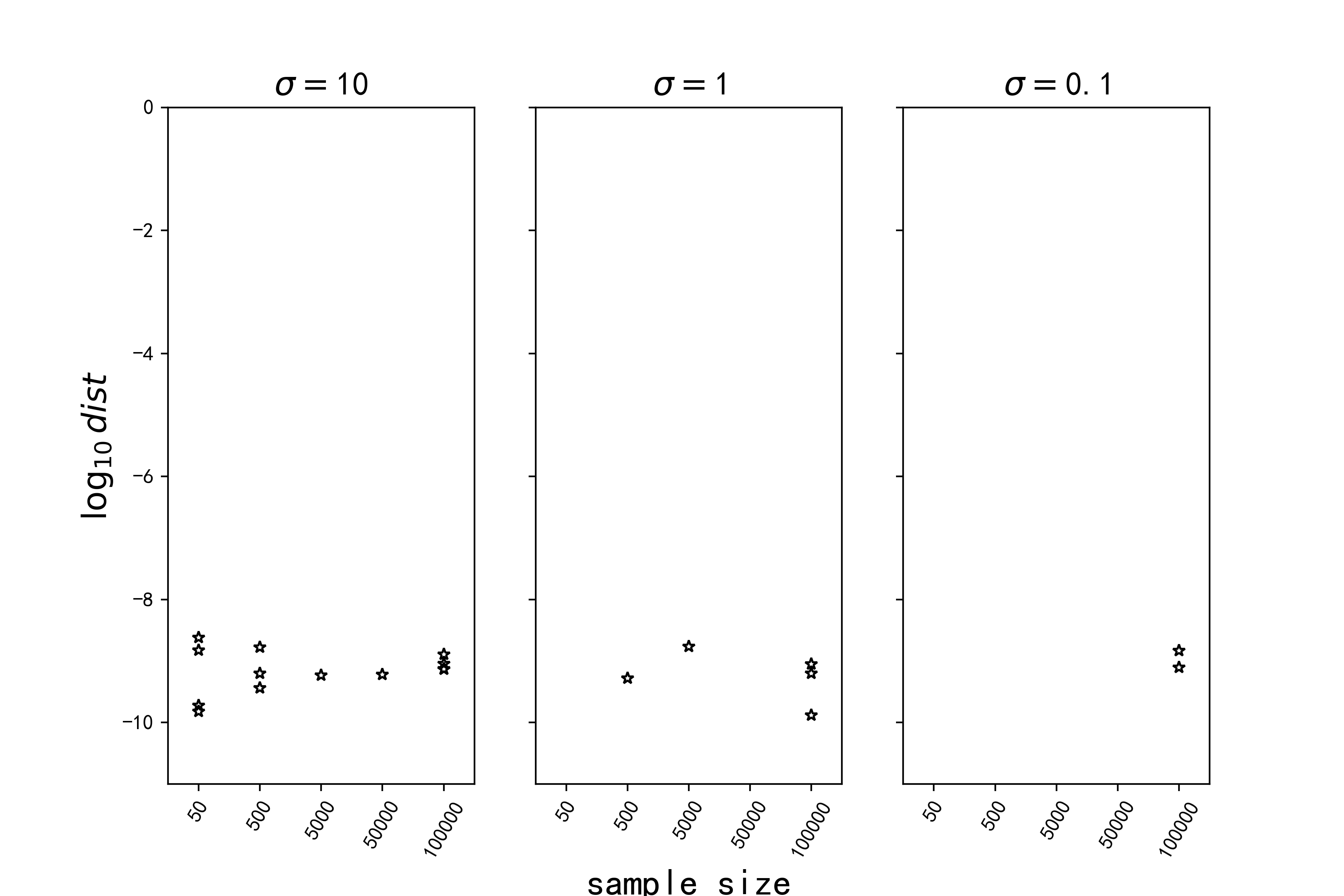}
    \caption{\scriptsize{Performance on HS22 w.r.t after 1 500 iterations.}}
  \end{minipage}
\end{figure*}

\begin{figure*}[!h]
  \begin{minipage}[t]{0.5\linewidth}
    \centering
    \includegraphics[scale=0.35]{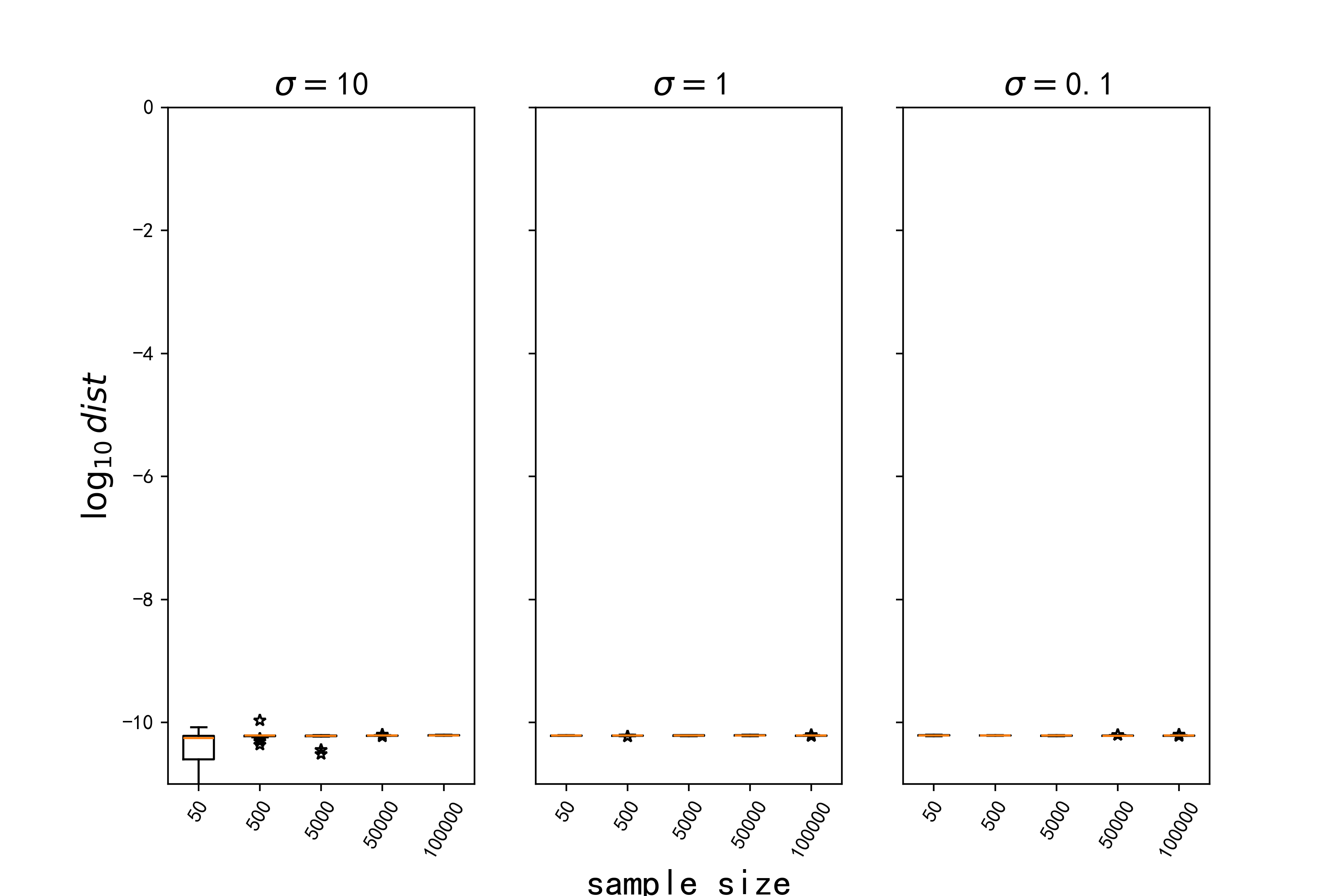}
    \caption{\scriptsize{Performance on HS23 w.r.t after 50 iterations.}}
  \end{minipage}%
  \begin{minipage}[t]{0.5\linewidth}
    \centering
    \includegraphics[scale=0.35]{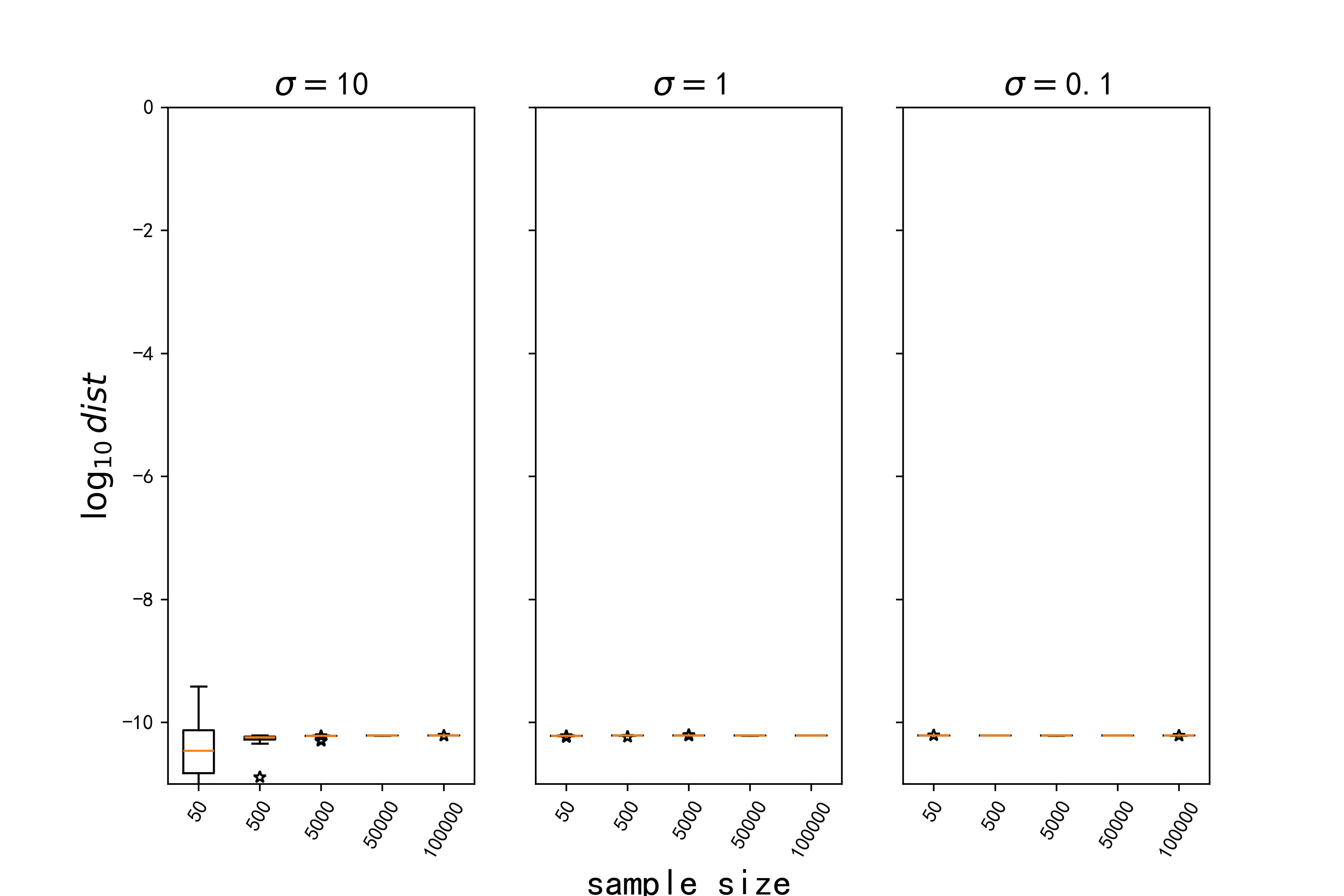}
    \caption{\scriptsize{Performance on HS23 w.r.t after 1 500 iterations.}}
  \end{minipage}
\end{figure*}
\clearpage

\begin{figure*}[!h]
  \begin{minipage}[t]{0.5\linewidth}
    \centering
    \includegraphics[scale=0.35]{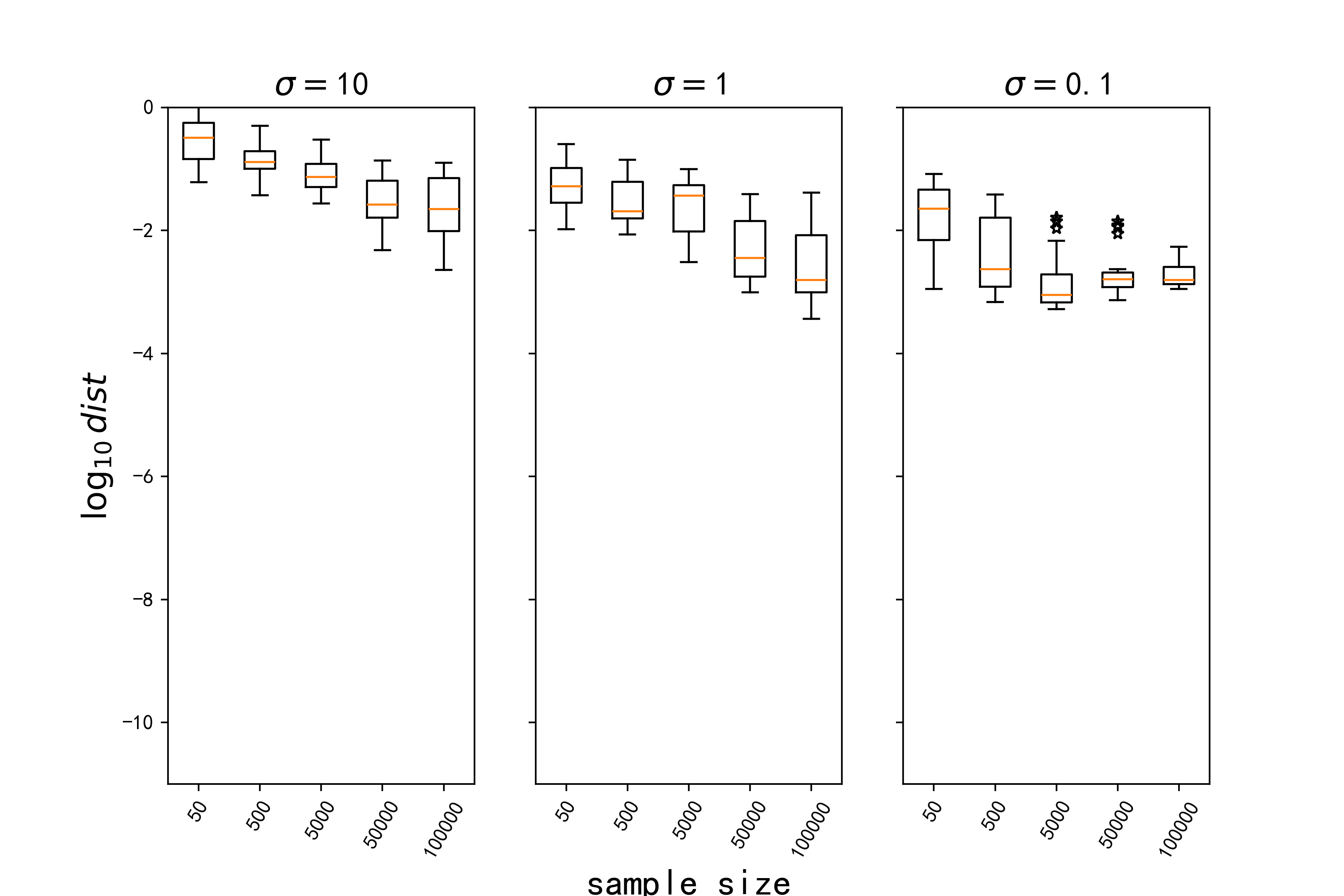}
    \caption{\scriptsize{Performance on HS26 w.r.t after 50 iterations.}}
  \end{minipage}%
  \begin{minipage}[t]{0.5\linewidth}
    \centering
    \includegraphics[scale=0.35]{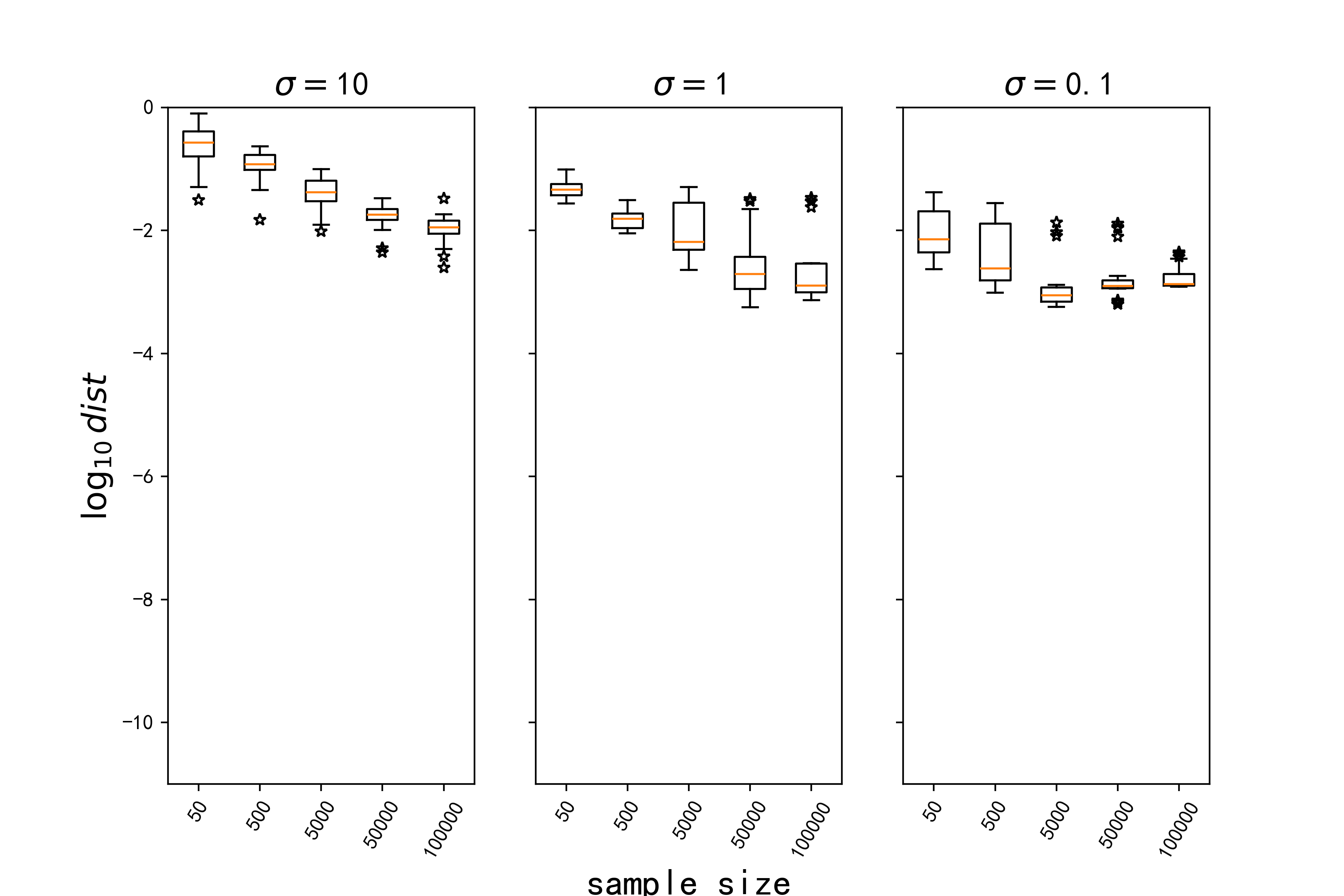}
    \caption{\scriptsize{Performance on HS26 w.r.t after 1 500 iterations.}}
  \end{minipage}
\end{figure*}

\begin{figure*}[!h]
  \begin{minipage}[t]{0.5\linewidth}
    \centering
    \includegraphics[scale=0.35]{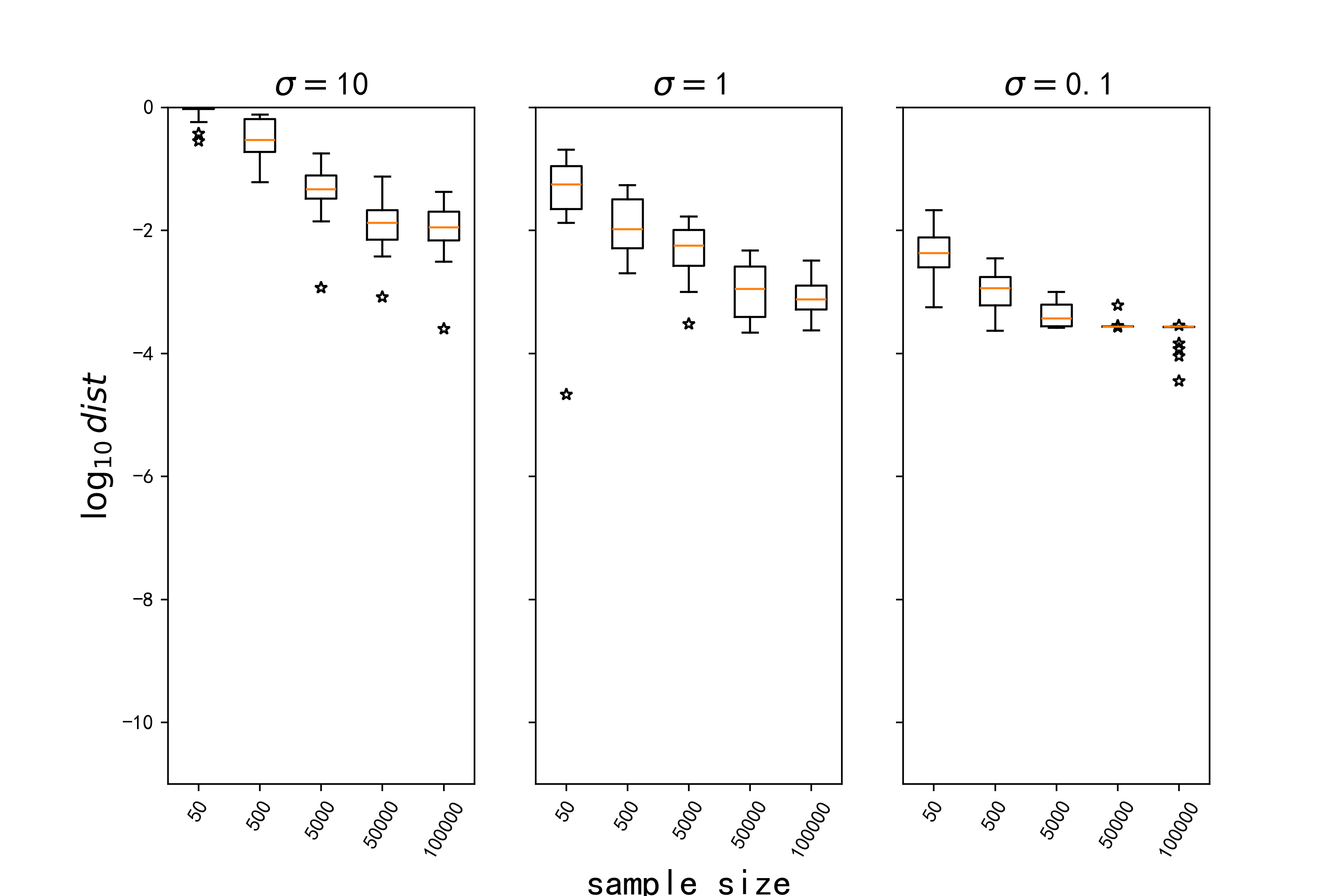}
    \caption{\scriptsize{Performance on HS27 w.r.t after 50 iterations.}}
  \end{minipage}%
  \begin{minipage}[t]{0.5\linewidth}
    \centering
    \includegraphics[scale=0.35]{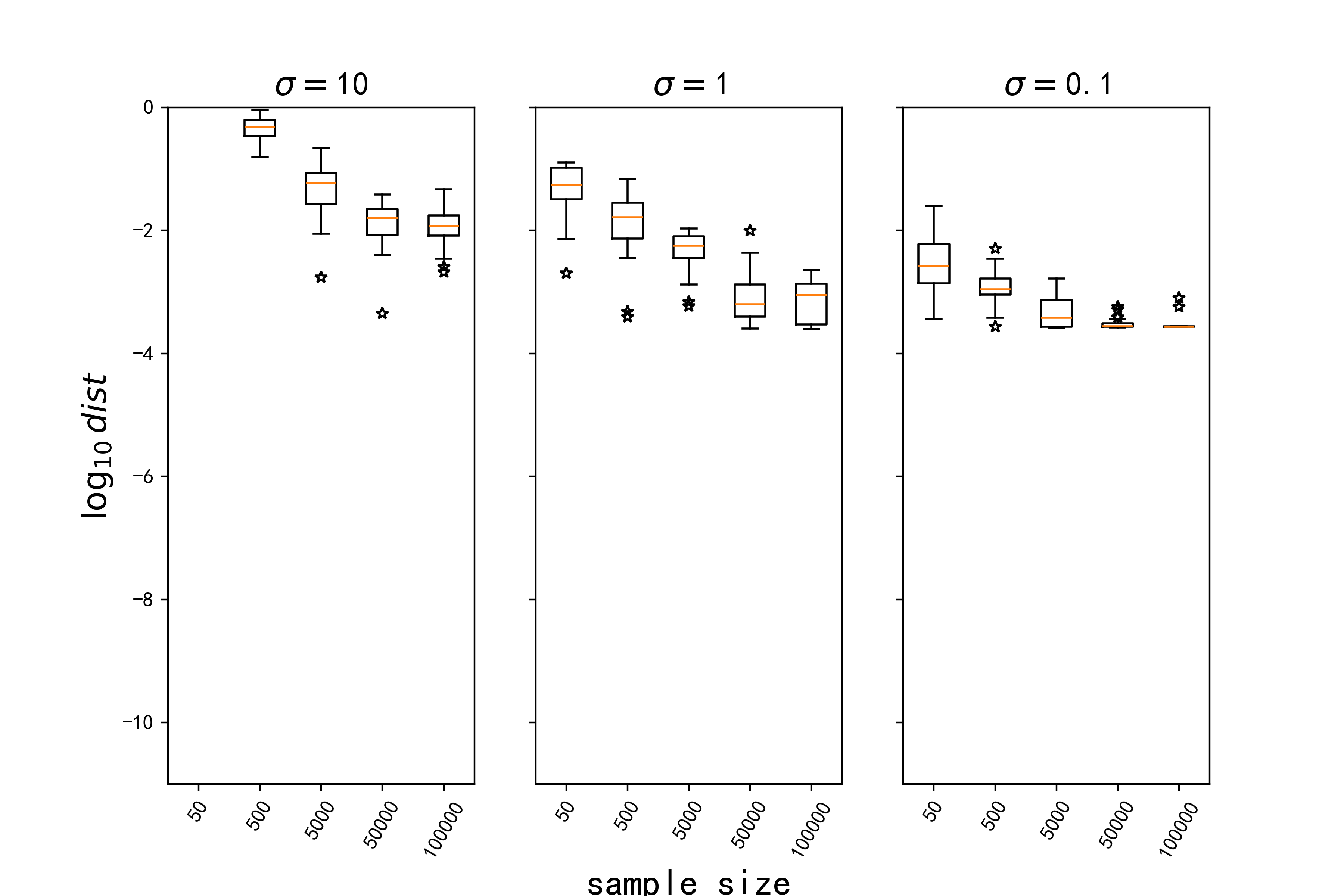}
    \caption{\scriptsize{Performance on HS27 w.r.t after 1 500 iterations.}}
  \end{minipage}
\end{figure*}
\clearpage

\begin{figure*}[!h]
  \begin{minipage}[t]{0.5\linewidth}
    \centering
    \includegraphics[scale=0.35]{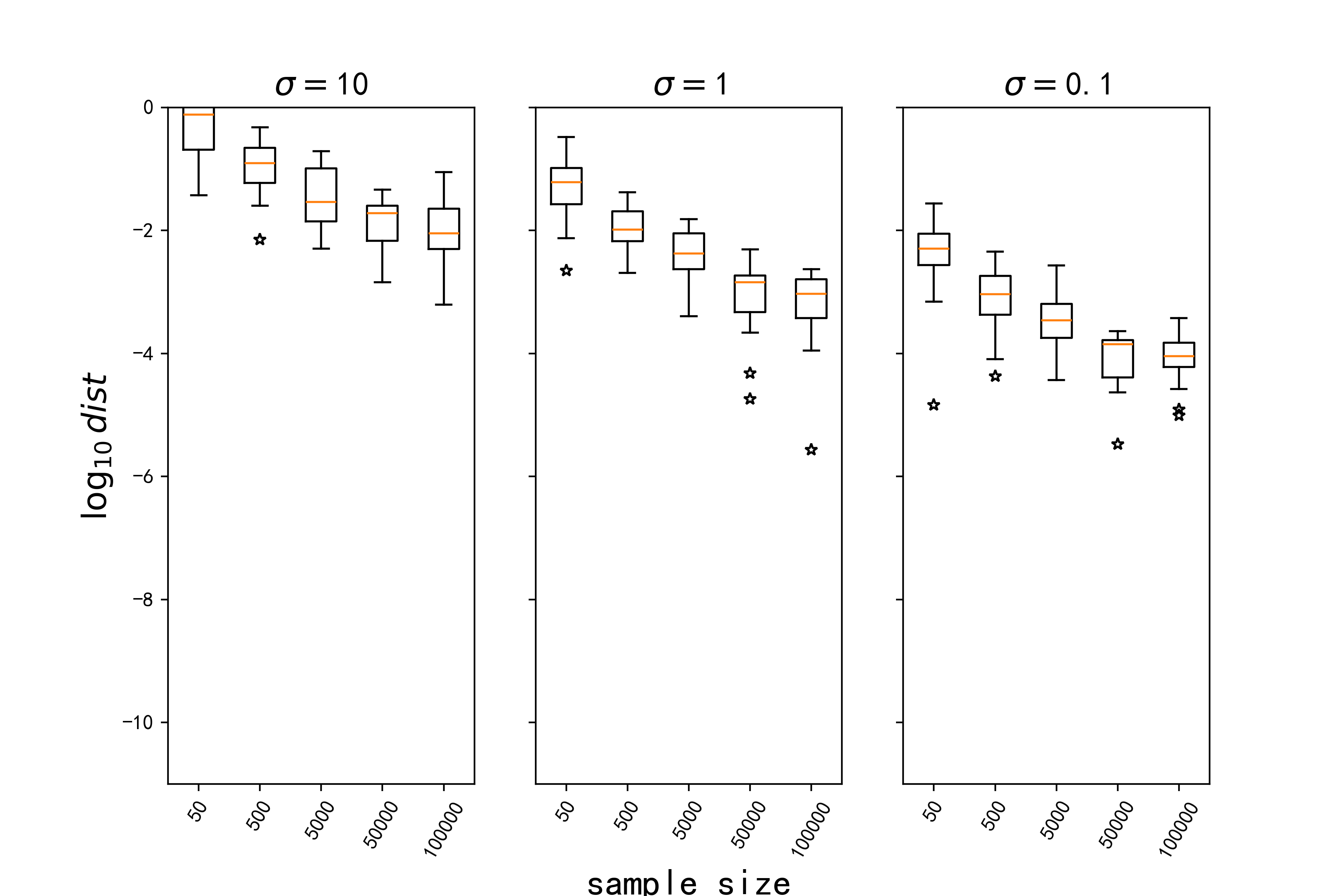}
    \caption{\scriptsize{Performance on HS31 w.r.t after 50 iterations.}}
  \end{minipage}%
  \begin{minipage}[t]{0.5\linewidth}
    \centering
    \includegraphics[scale=0.35]{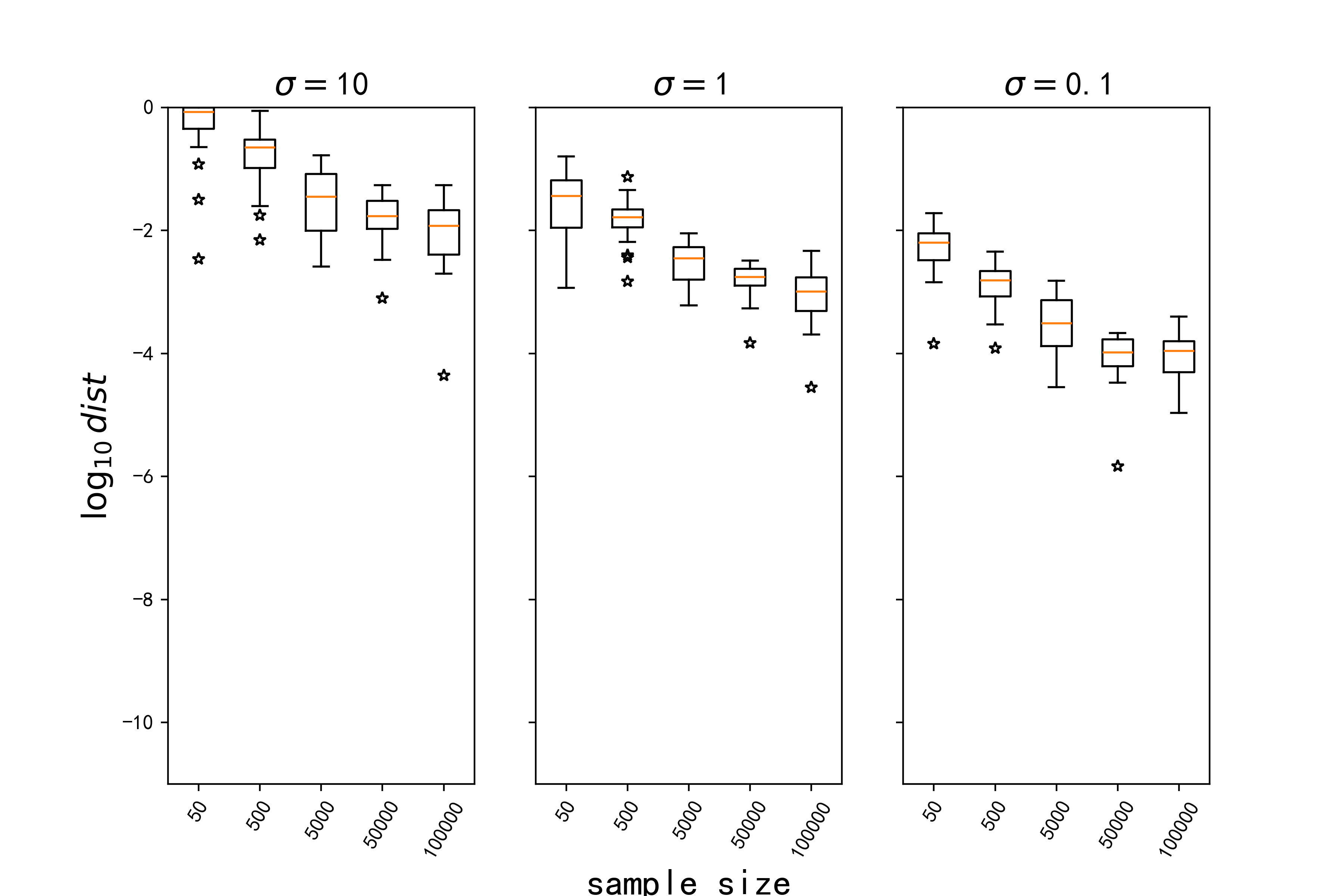}
    \caption{\scriptsize{Performance on HS31 w.r.t after 1 500 iterations.}}
  \end{minipage}
\end{figure*}

\begin{figure*}[!h]
  \begin{minipage}[t]{0.5\linewidth}
    \centering
    \includegraphics[scale=0.35]{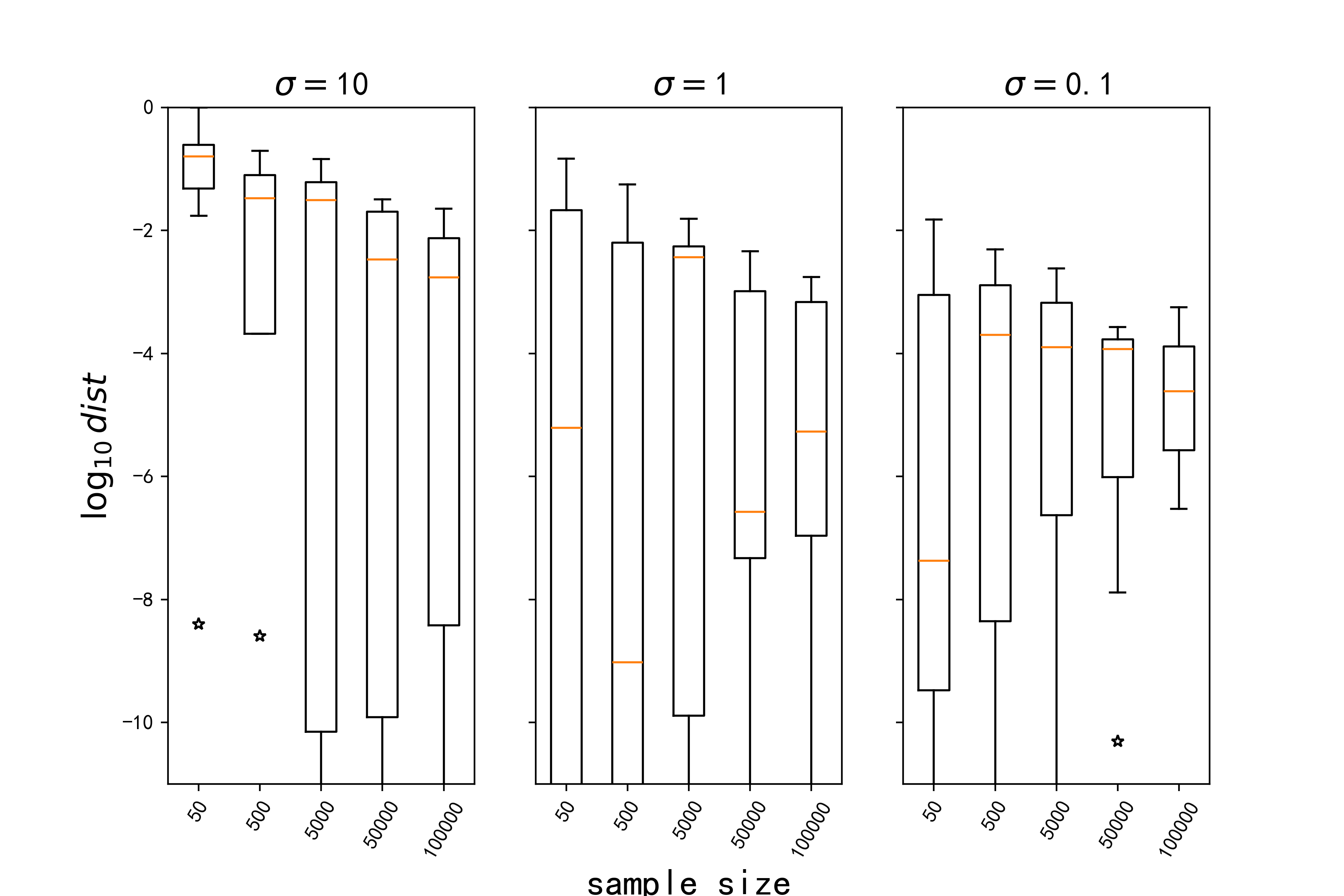}
    \caption{\scriptsize{Performance on HS32 w.r.t after 50 iterations.}}
  \end{minipage}%
  \begin{minipage}[t]{0.5\linewidth}
    \centering
    \includegraphics[scale=0.35]{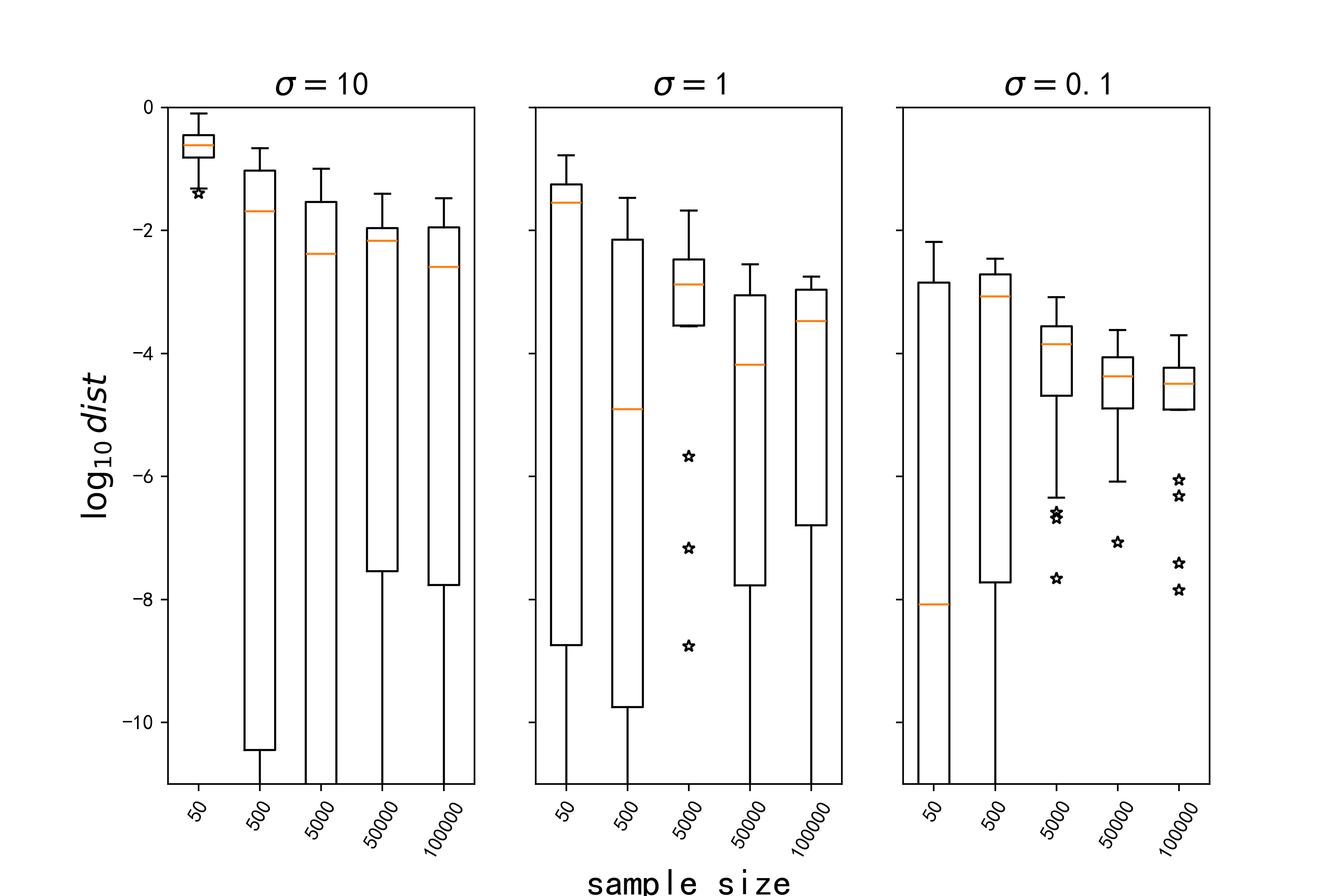}
    \caption{\scriptsize{Performance on HS32 w.r.t after 1 500 iterations.}}
  \end{minipage}
\end{figure*}
\clearpage

\begin{figure*}[!h]
  \begin{minipage}[t]{0.5\linewidth}
    \centering
    \includegraphics[scale=0.35]{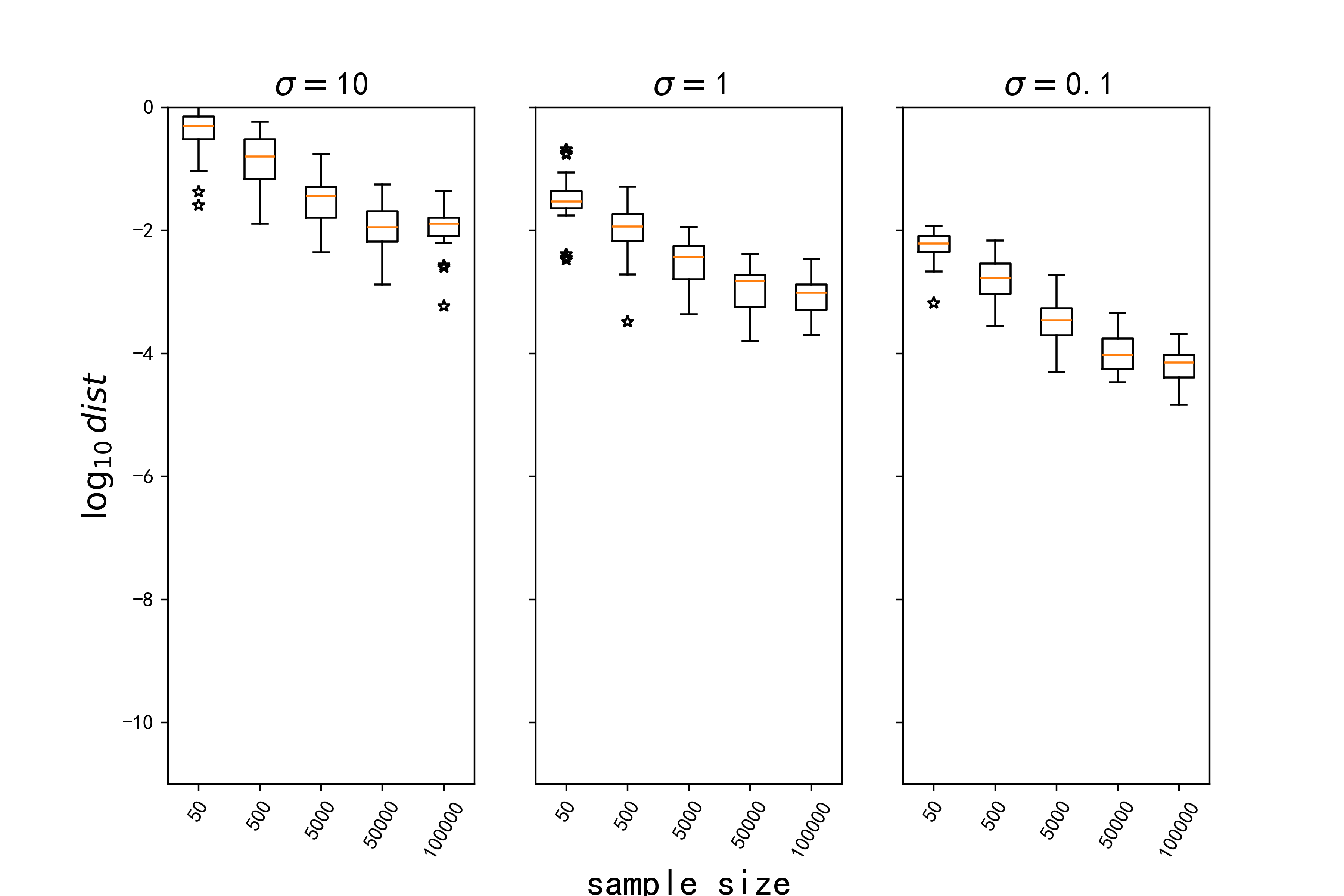}
    \caption{\scriptsize{Performance on HS42 w.r.t after 50 iterations.}}
  \end{minipage}%
  \begin{minipage}[t]{0.5\linewidth}
    \centering
    \includegraphics[scale=0.35]{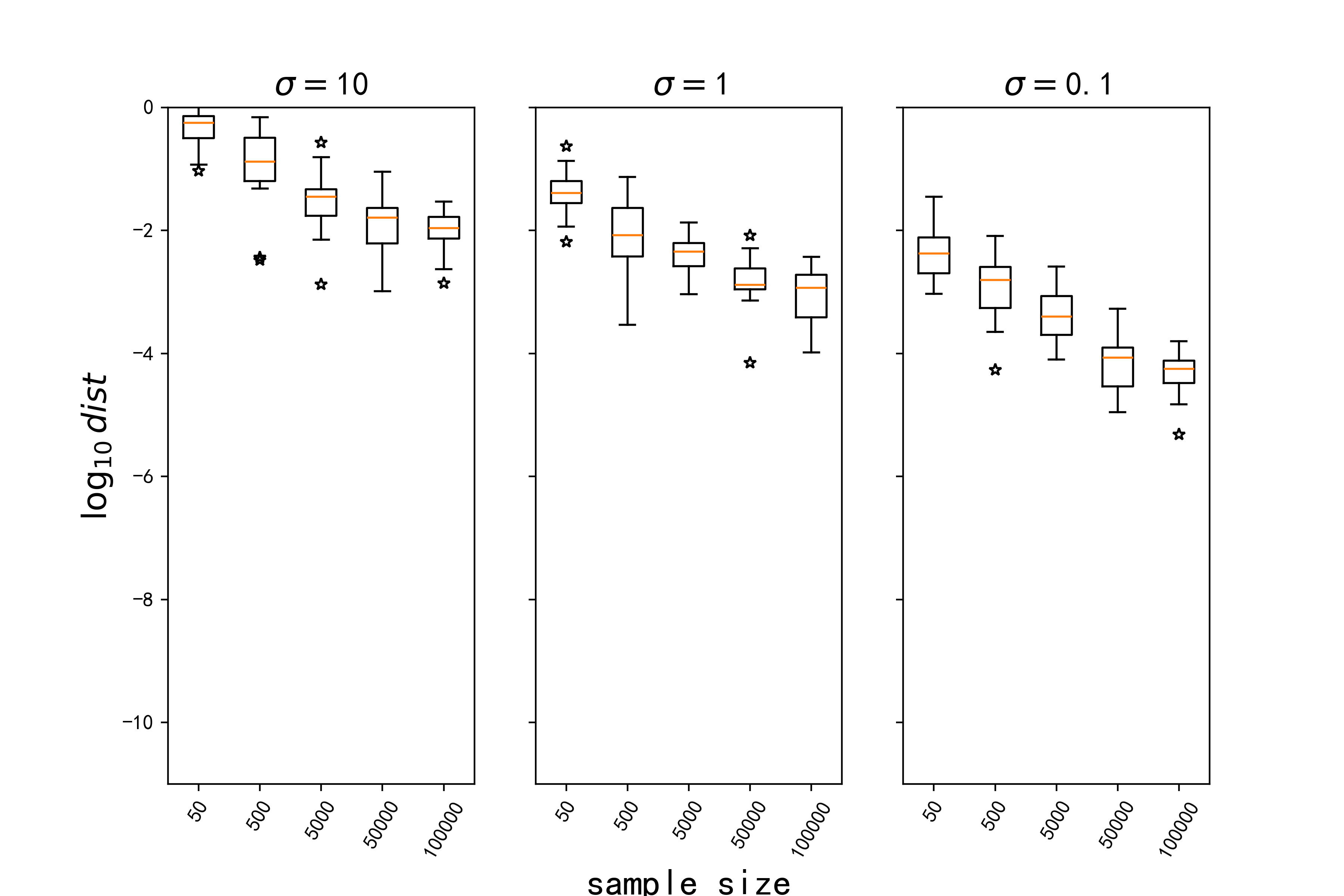}
    \caption{\scriptsize{Performance on HS42 w.r.t after 1 500 iterations.}}
  \end{minipage}
\end{figure*}

\begin{figure*}[!h]
  \begin{minipage}[t]{0.5\linewidth}
    \centering
    \includegraphics[scale=0.35]{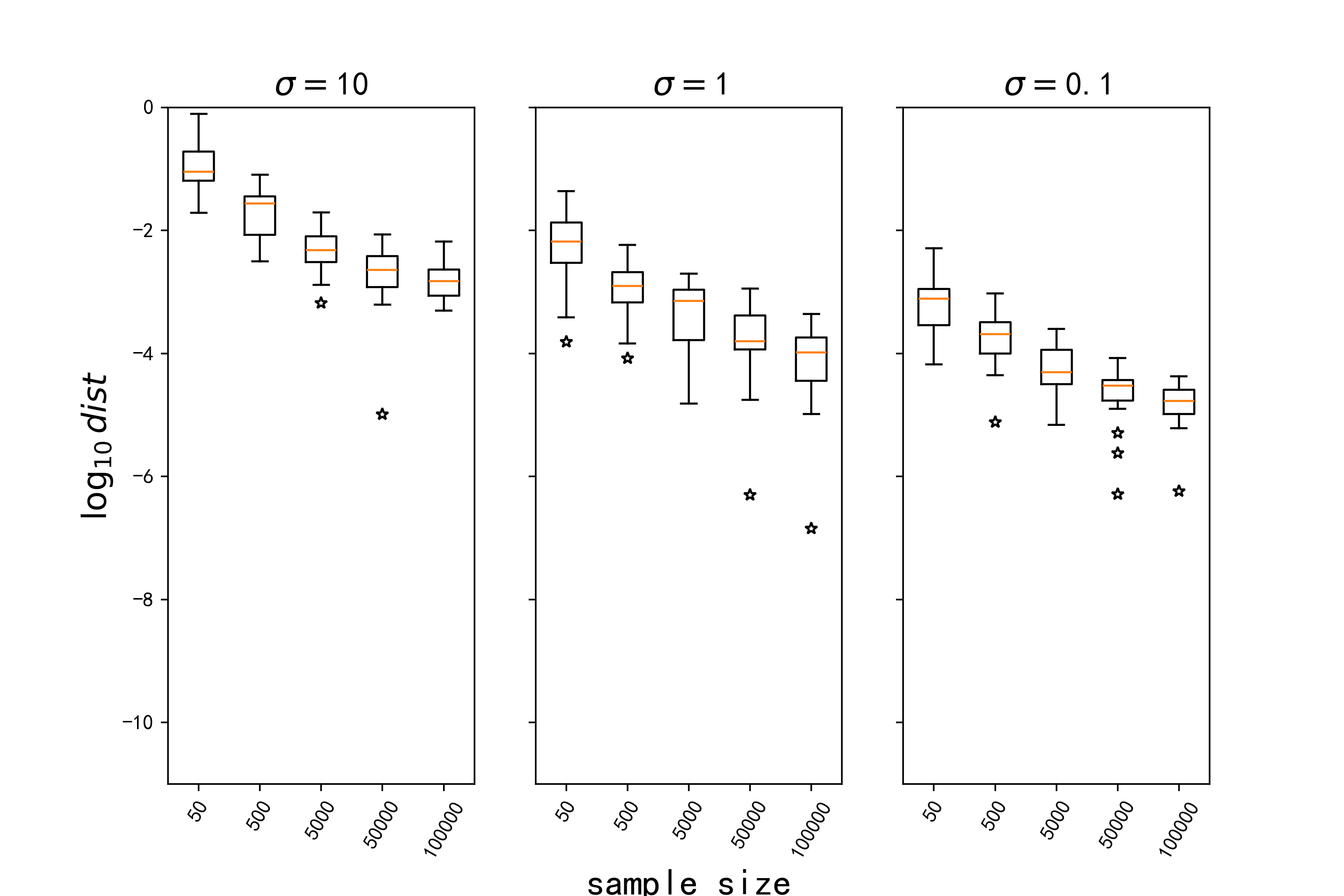}
    \caption{\scriptsize{Performance on HS43 w.r.t after 50 iterations.}}
  \end{minipage}%
  \begin{minipage}[t]{0.5\linewidth}
    \centering
    \includegraphics[scale=0.35]{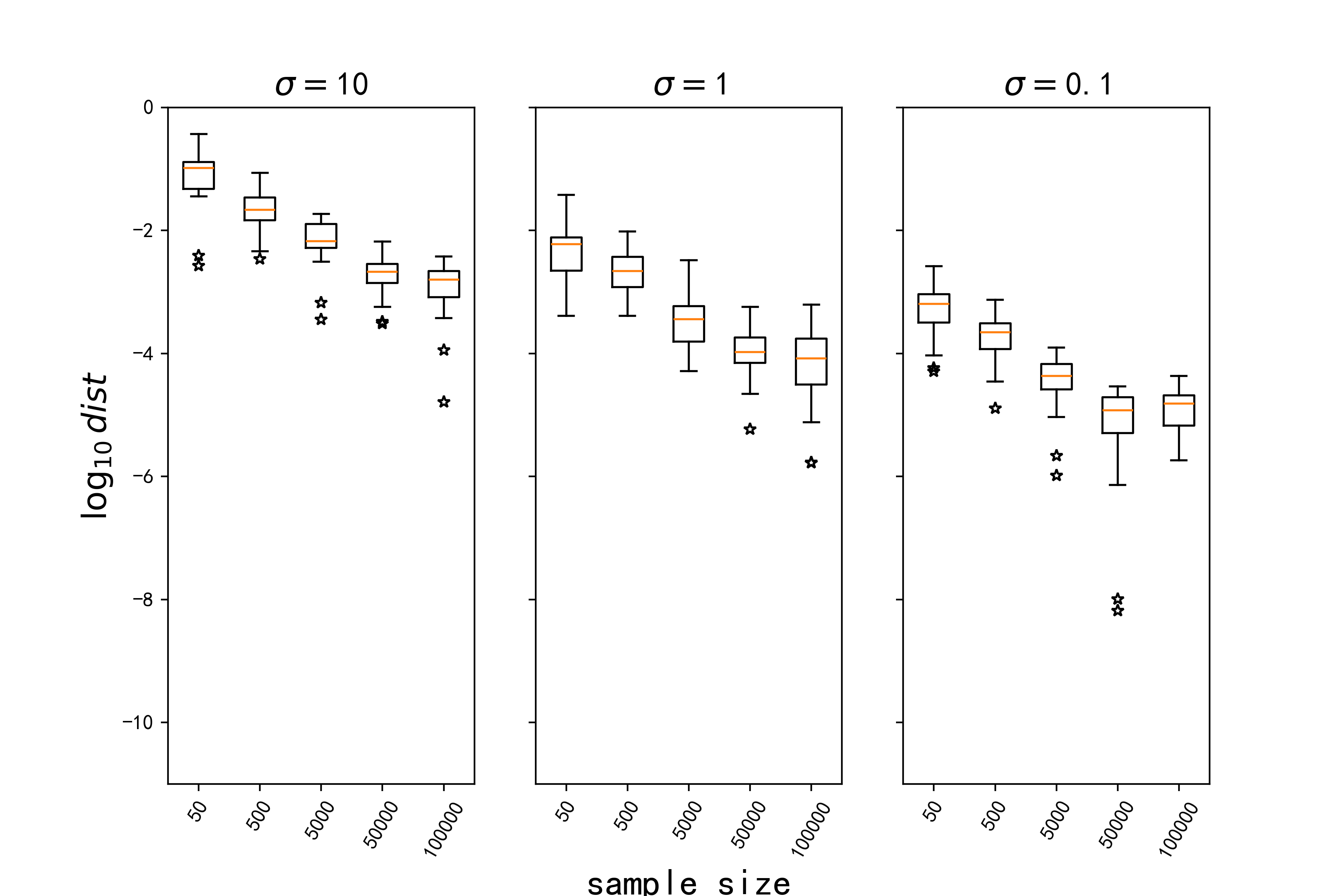}
    \caption{\scriptsize{Performance on HS43 w.r.t after 1 500 iterations.}}
  \end{minipage}
\end{figure*}
\clearpage

\begin{figure*}[!h]
  \begin{minipage}[t]{0.5\linewidth}
    \centering
    \includegraphics[scale=0.35]{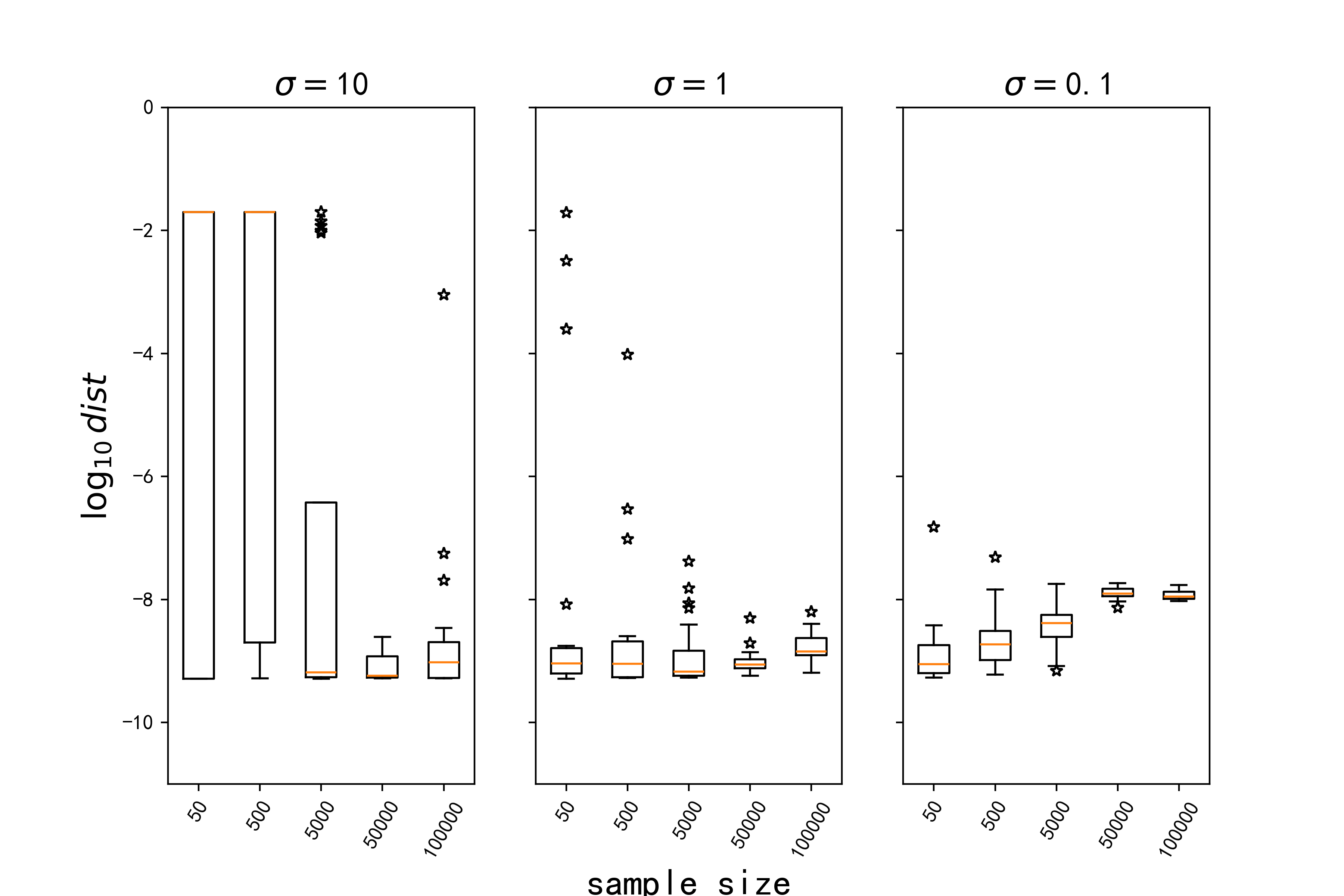}
    \caption{\scriptsize{Performance on HS57 w.r.t after 50 iterations.}}
  \end{minipage}%
  \begin{minipage}[t]{0.5\linewidth}
    \centering
    \includegraphics[scale=0.35]{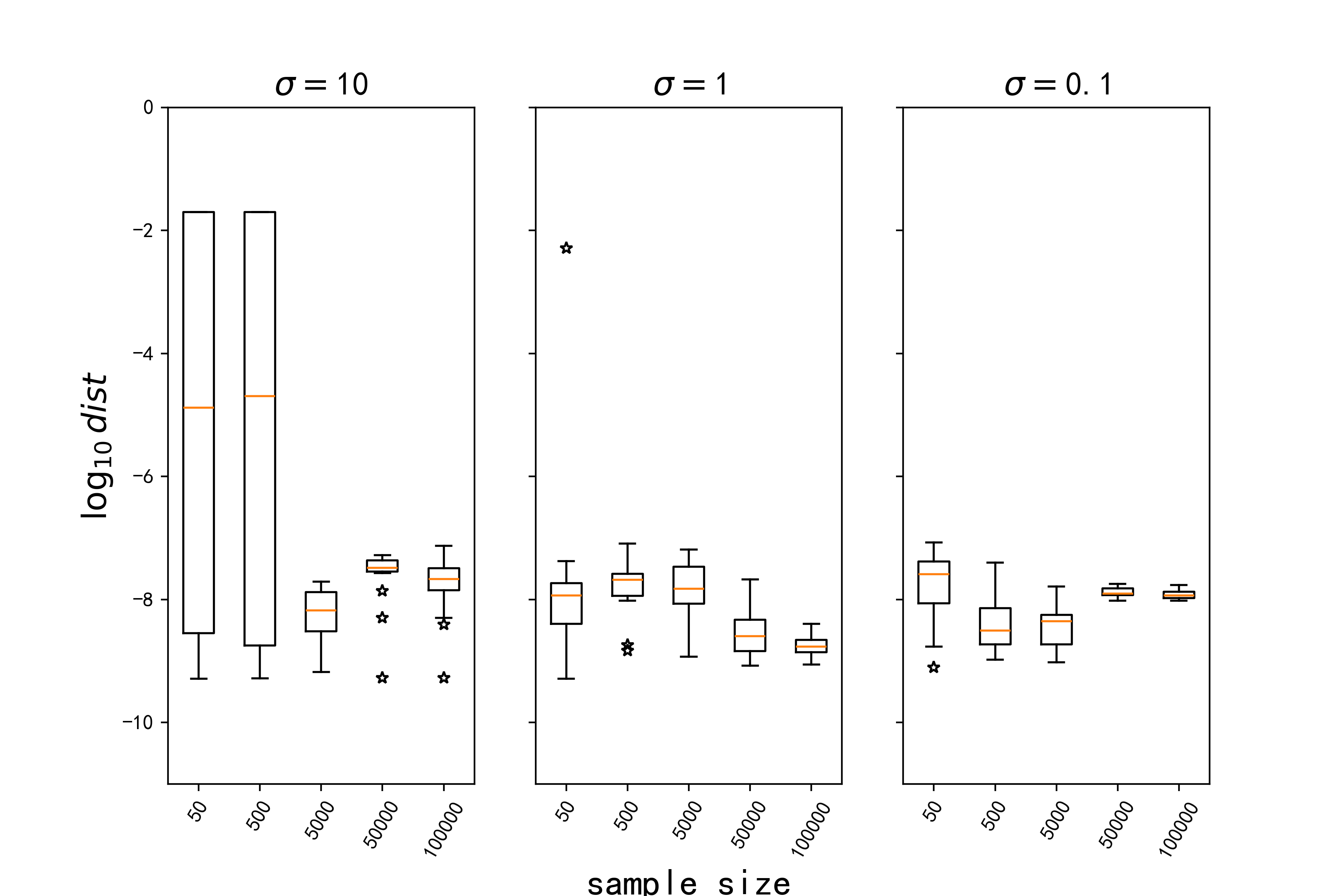}
    \caption{\scriptsize{Performance on HS57 w.r.t after 1 500 iterations.}}
  \end{minipage}
\end{figure*}

\begin{figure*}[!h]
  \begin{minipage}[t]{0.5\linewidth}
    \centering
    \includegraphics[scale=0.35]{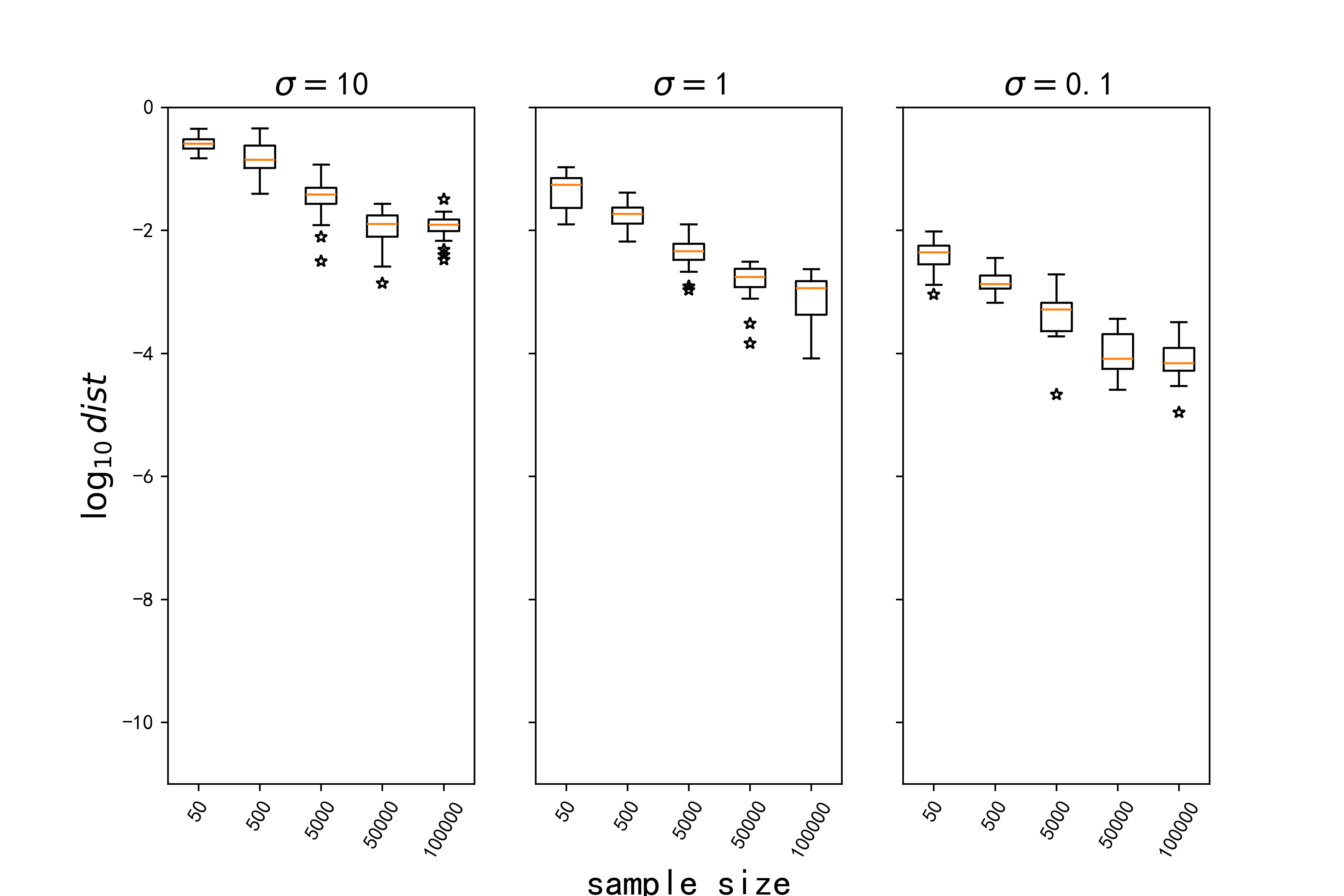}
    \caption{\scriptsize{Performance on HS60 w.r.t after 50 iterations.}}
  \end{minipage}%
  \begin{minipage}[t]{0.5\linewidth}
    \centering
    \includegraphics[scale=0.35]{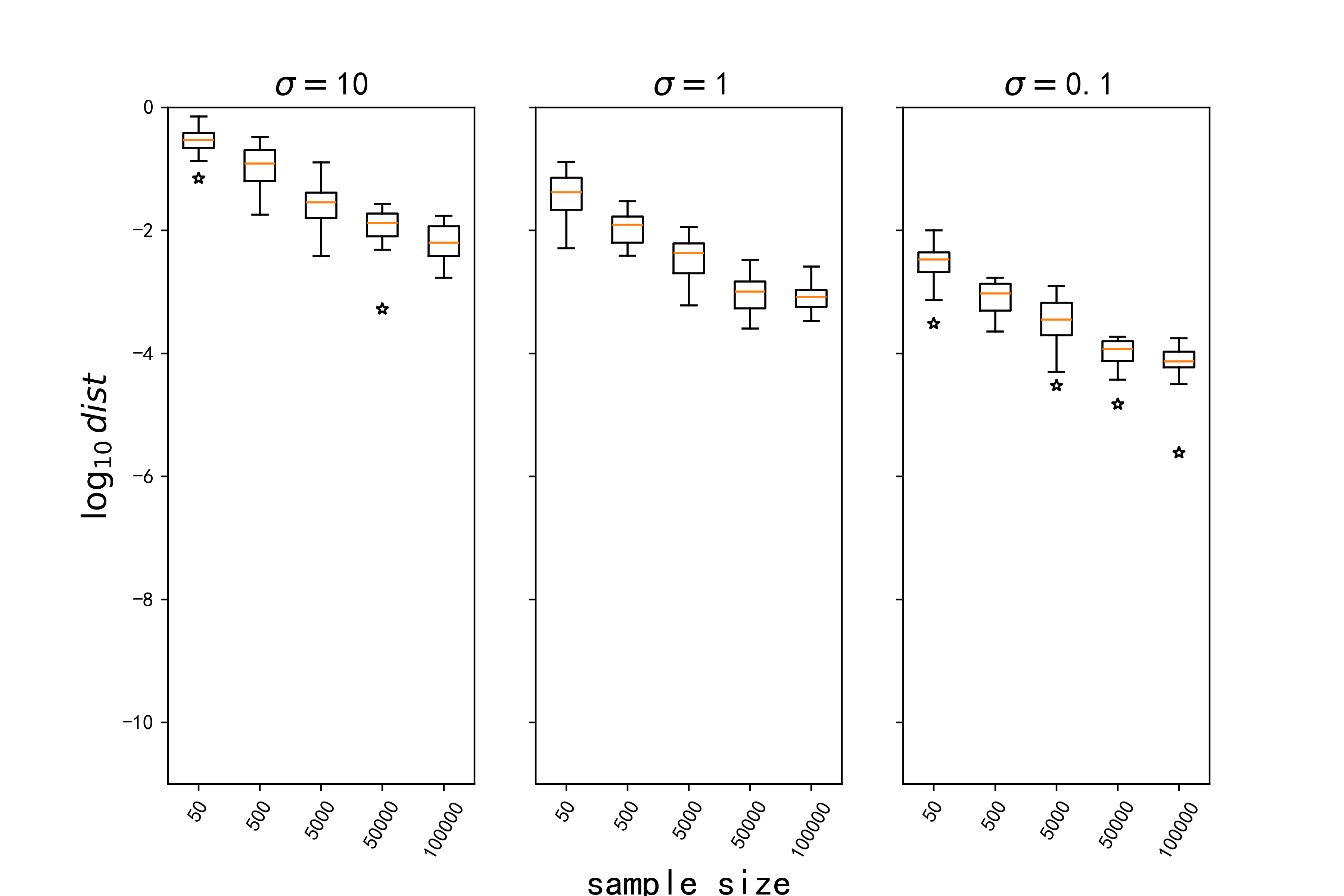}
    \caption{\scriptsize{Performance on HS60 w.r.t after 1 500 iterations.}}
  \end{minipage}
\end{figure*}
\clearpage

\begin{figure*}[!h]
  \begin{minipage}[t]{0.5\linewidth}
    \centering
    \includegraphics[scale=0.35]{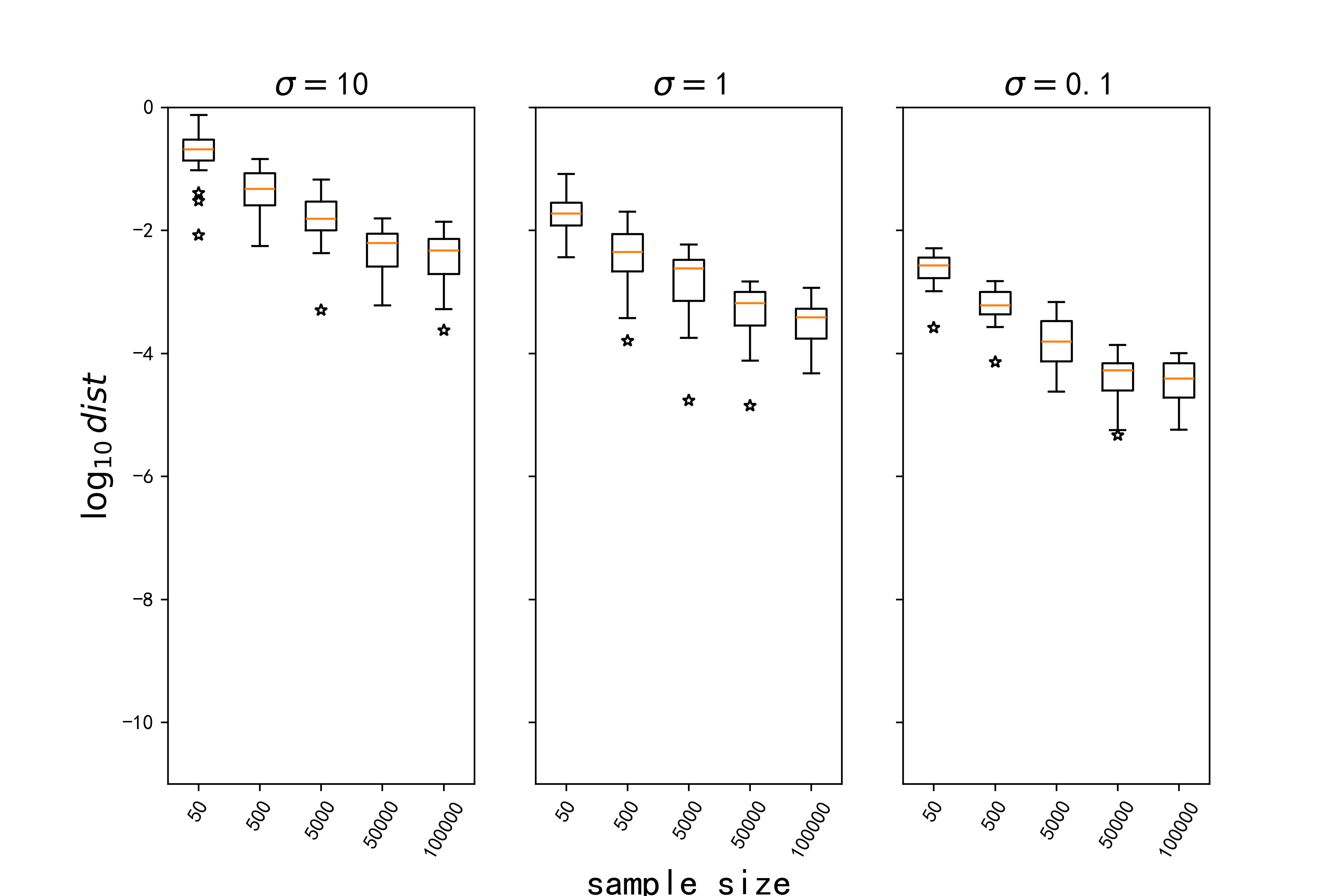}
    \caption{\scriptsize{Performance on HS63 w.r.t after 50 iterations.}}
  \end{minipage}%
  \begin{minipage}[t]{0.5\linewidth}
    \centering
    \includegraphics[scale=0.35]{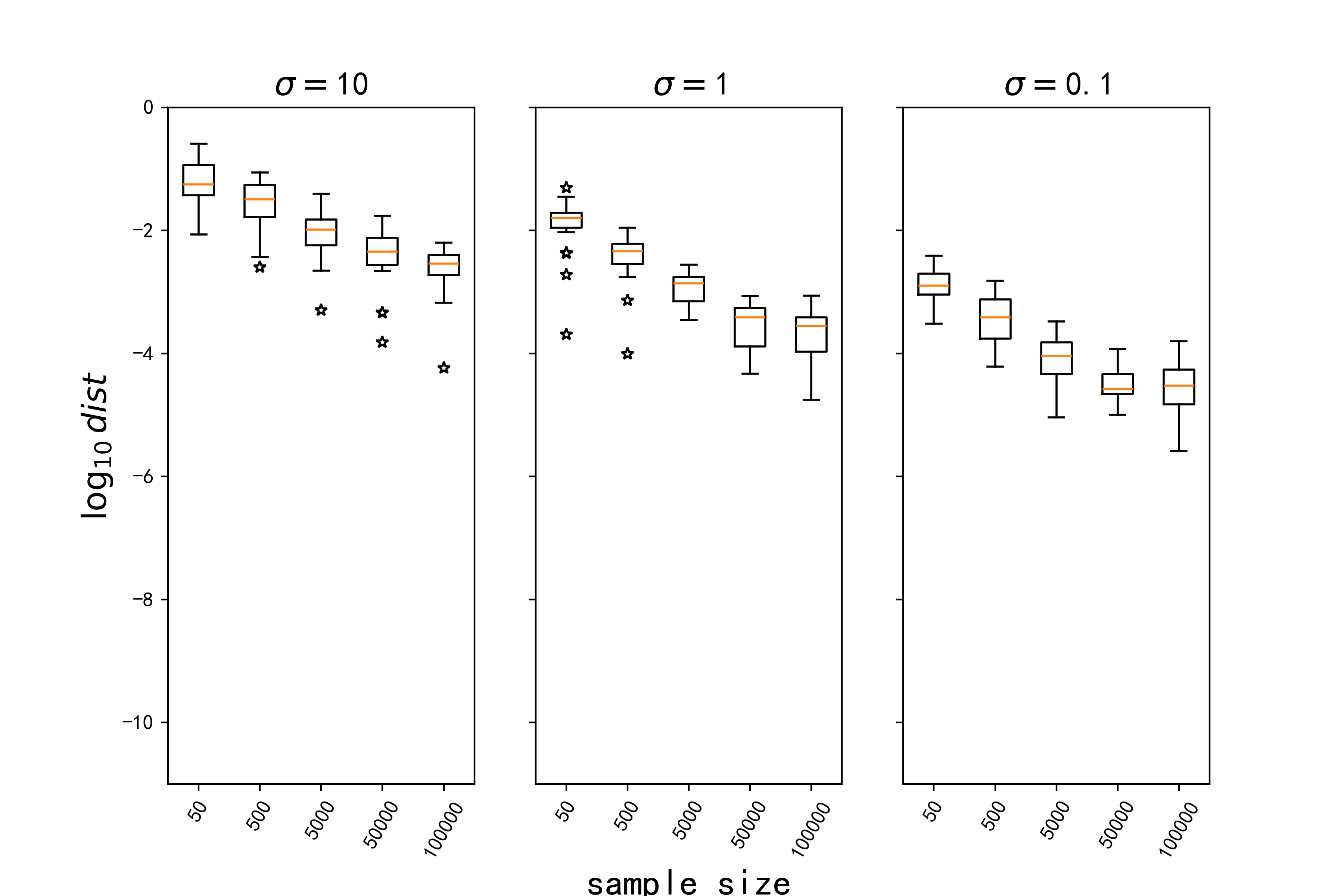}
    \caption{\scriptsize{Performance on HS63 w.r.t after 1 500 iterations.}}
  \end{minipage}
\end{figure*}

\begin{figure*}[!h]
  \begin{minipage}[t]{0.5\linewidth}
    \centering
    \includegraphics[scale=0.35]{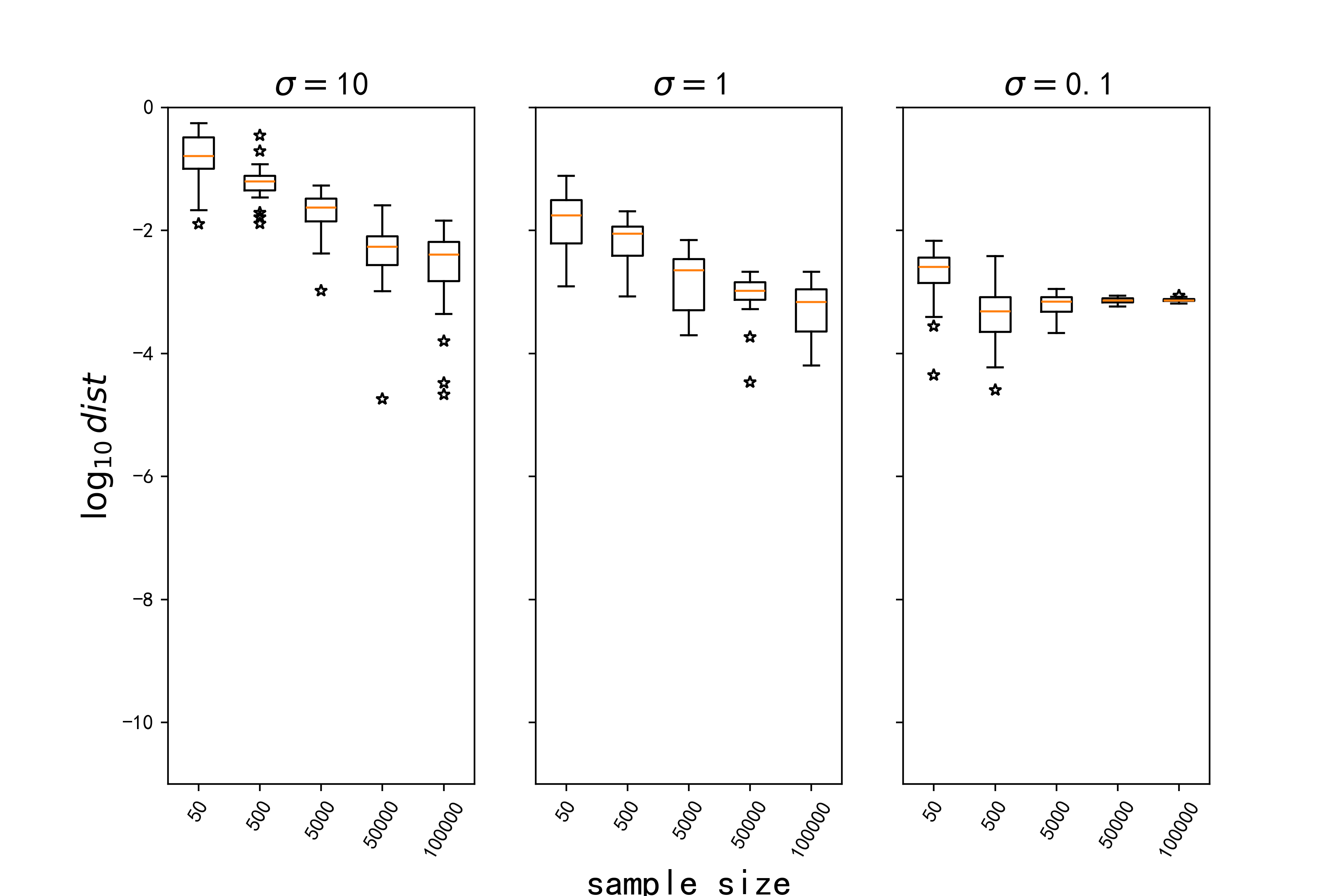}
    \caption{\scriptsize{Performance on HS65 w.r.t after 50 iterations.}}
  \end{minipage}%
  \begin{minipage}[t]{0.5\linewidth}
    \centering
    \includegraphics[scale=0.35]{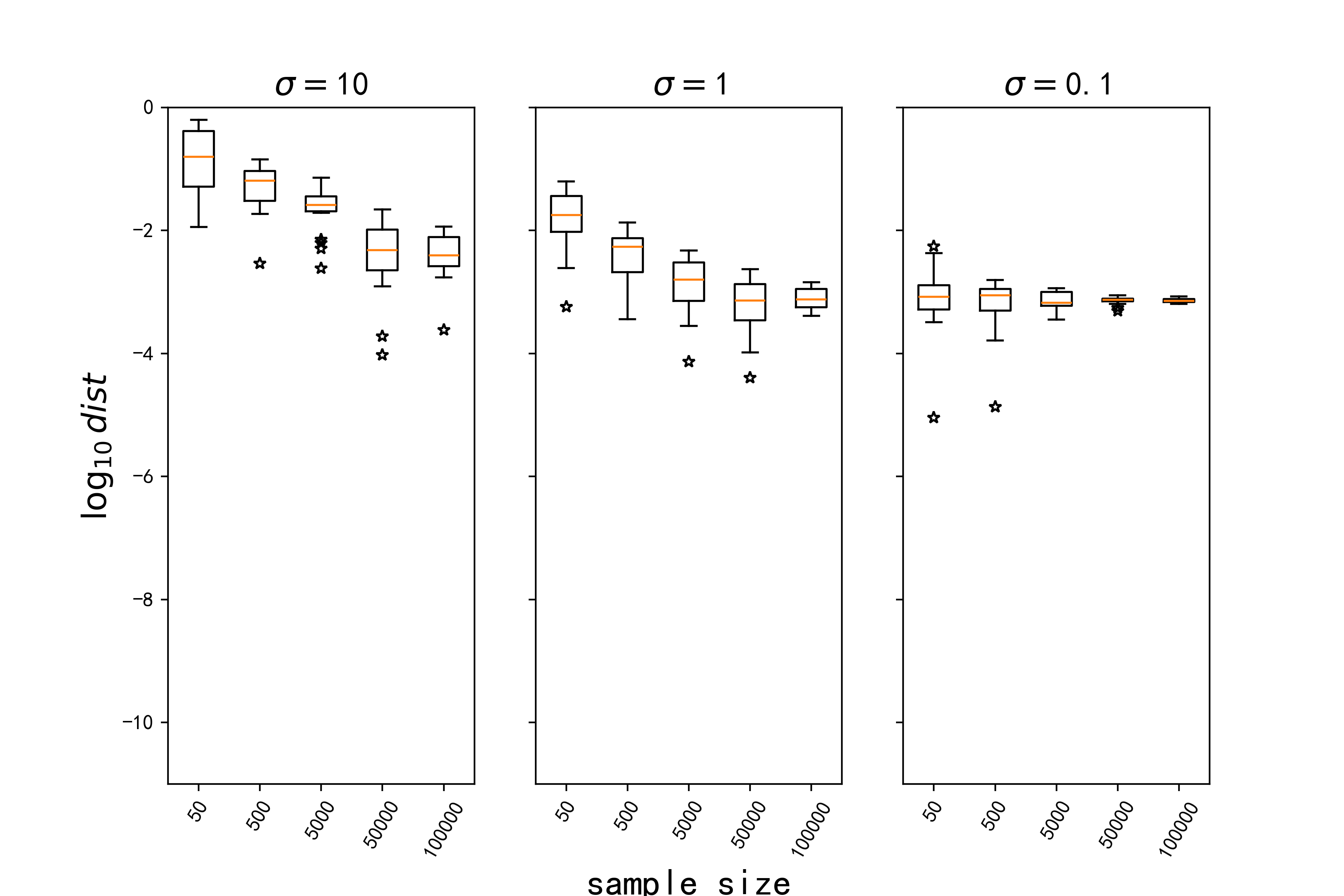}
    \caption{\scriptsize{Performance on HS65 w.r.t after 1 500 iterations.}}
  \end{minipage}
\end{figure*}
\clearpage

\begin{figure*}[!h]
  \begin{minipage}[t]{0.5\linewidth}
    \centering
    \includegraphics[scale=0.35]{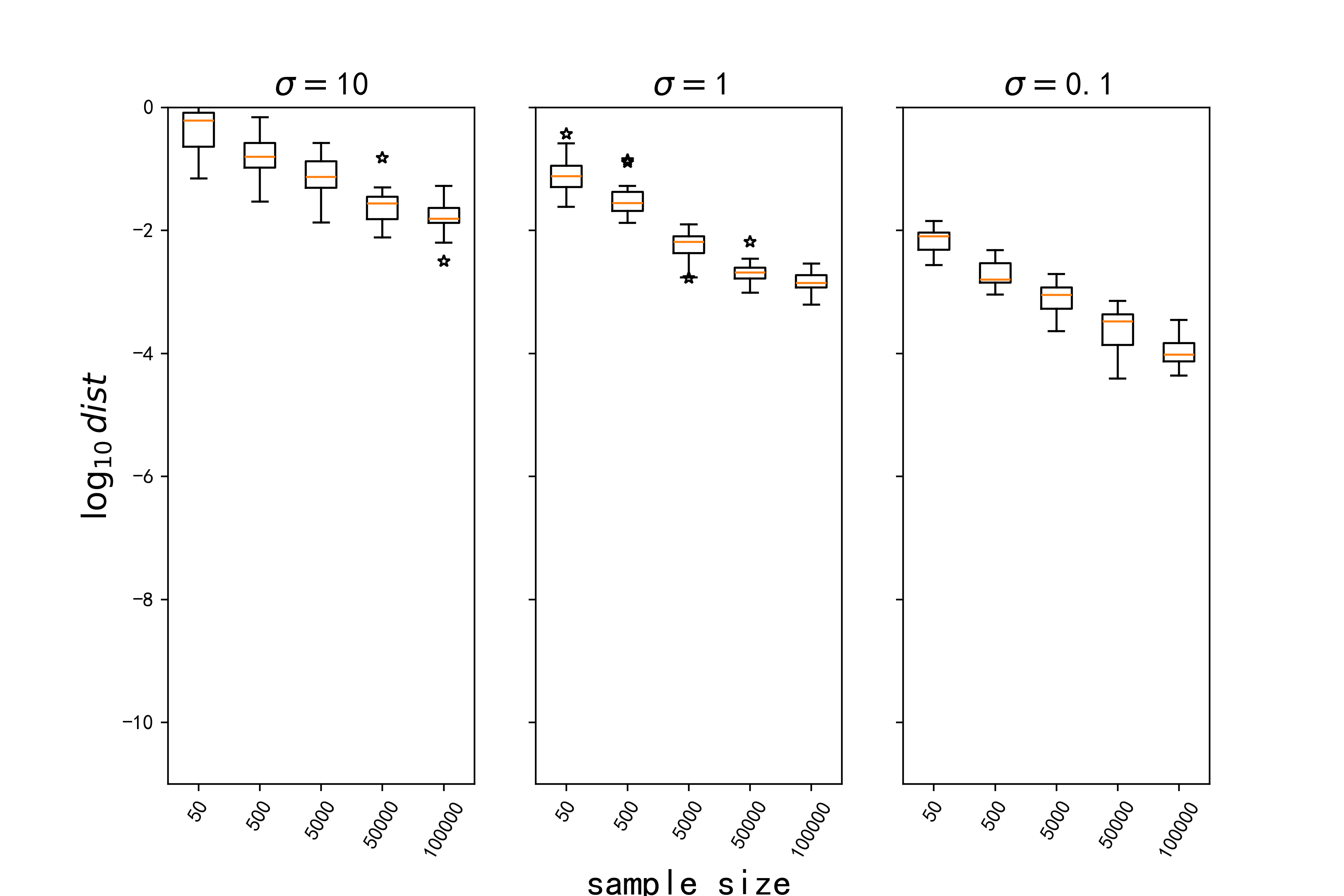}
    \caption{\scriptsize{Performance on HS77 w.r.t after 50 iterations.}}
  \end{minipage}%
  \begin{minipage}[t]{0.5\linewidth}
    \centering
    \includegraphics[scale=0.35]{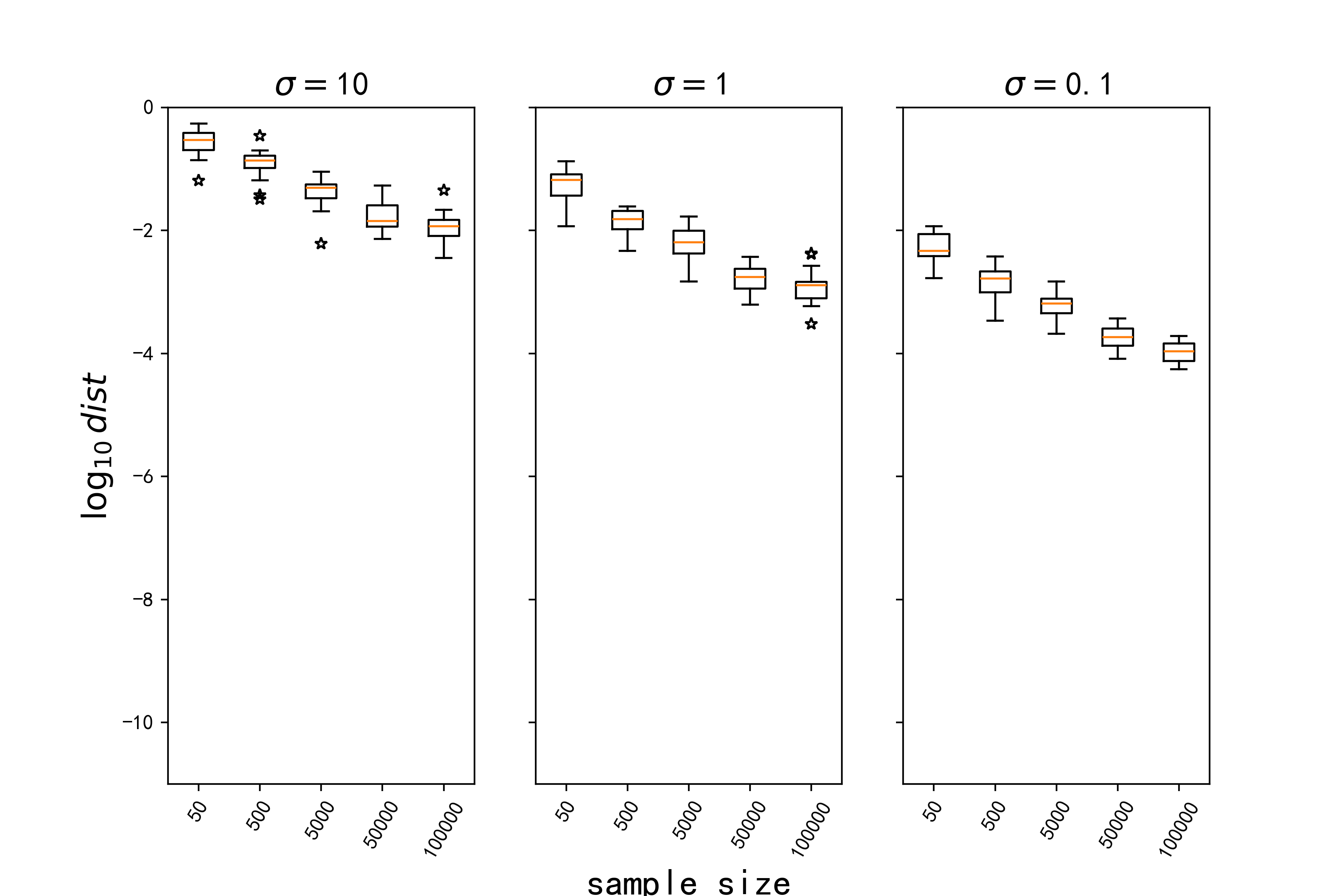}
    \caption{\scriptsize{Performance on HS77 w.r.t after 1 500 iterations.}}
  \end{minipage}
\end{figure*}

\begin{figure*}[!h]
  \begin{minipage}[t]{0.5\linewidth}
    \centering
    \includegraphics[scale=0.35]{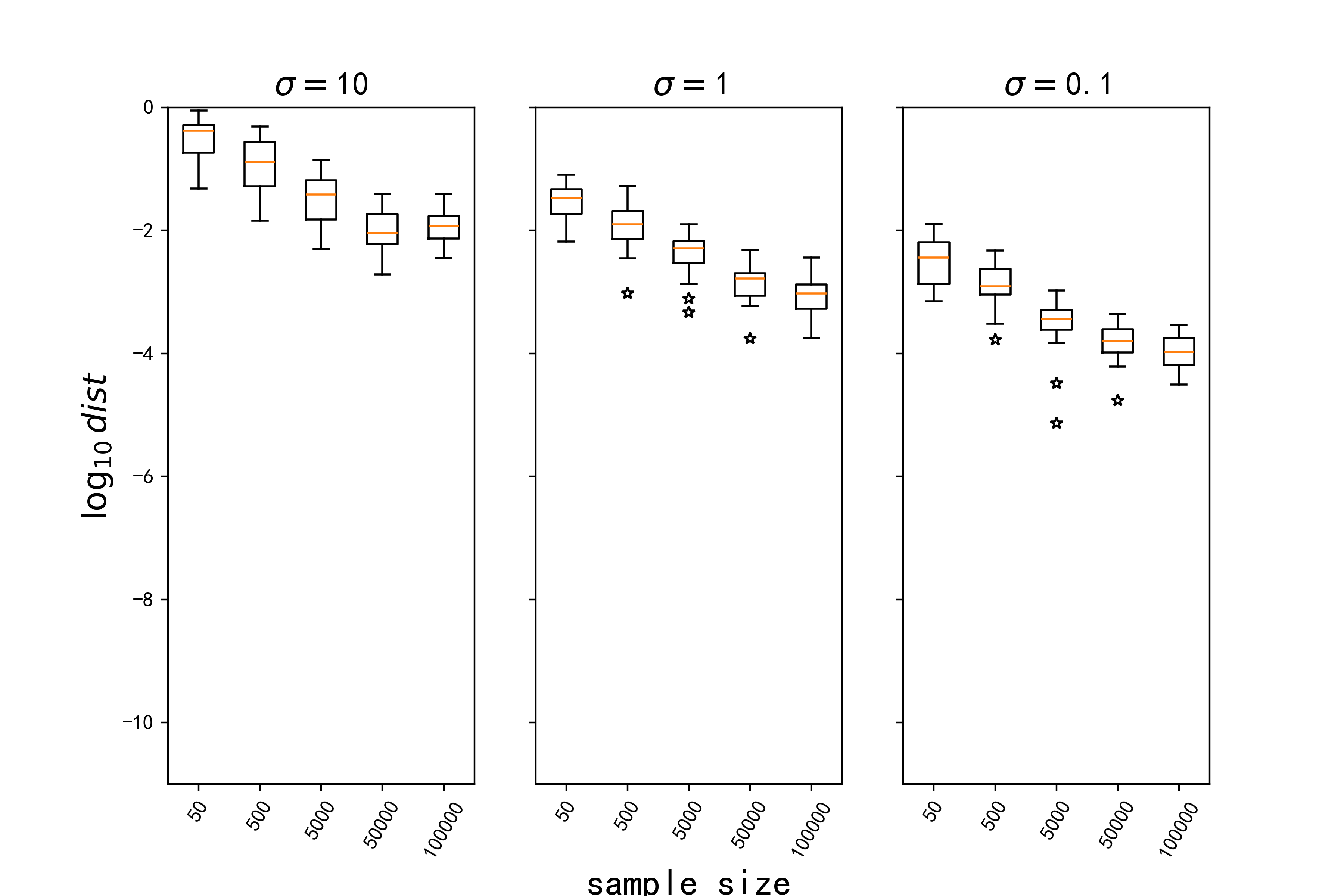}
    \caption{\scriptsize{Performance on HS79 w.r.t after 50 iterations.}}
  \end{minipage}%
  \begin{minipage}[t]{0.5\linewidth}
    \centering
    \includegraphics[scale=0.35]{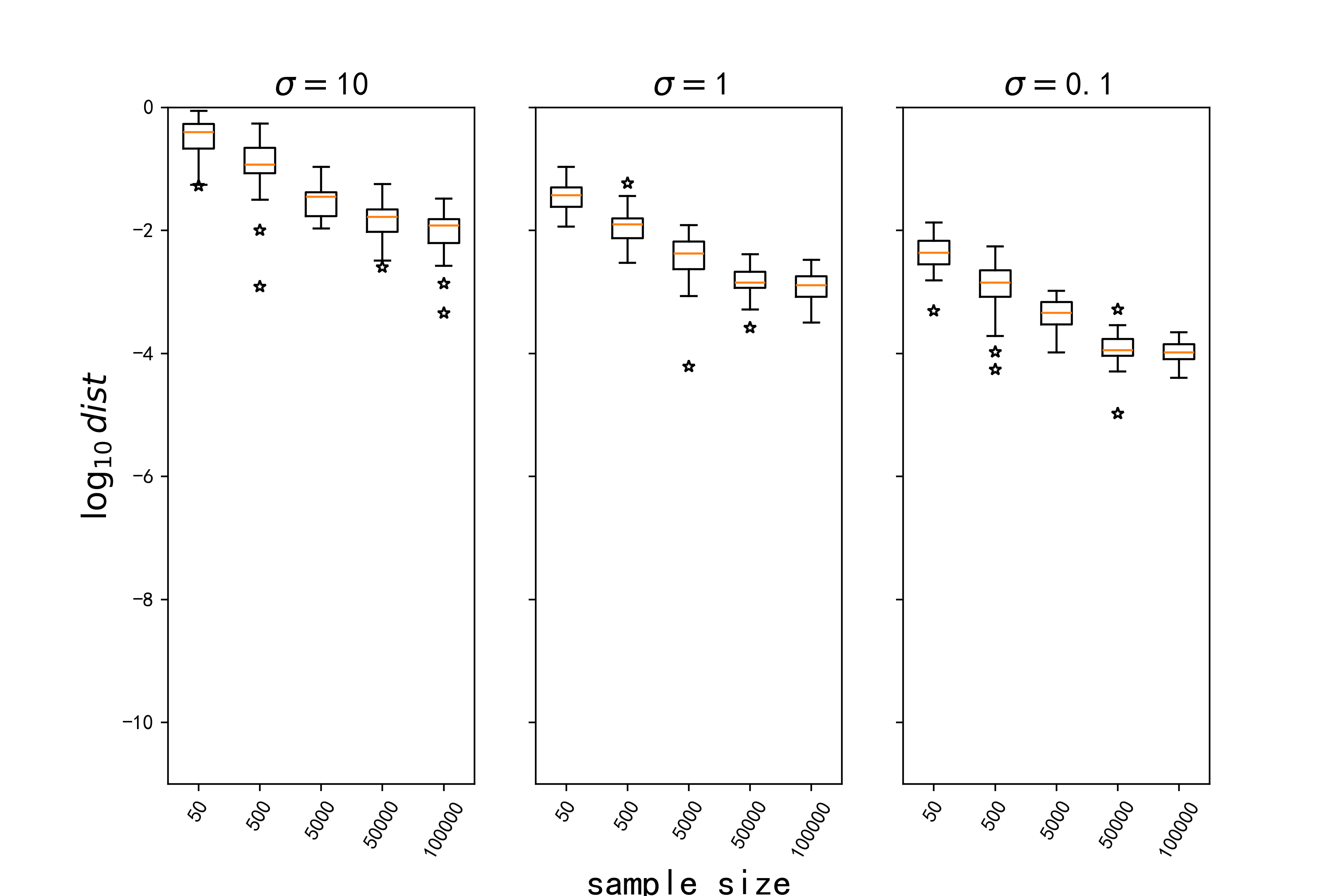}
    \caption{\scriptsize{Performance on HS79 w.r.t after 1 500 iterations.}}
  \end{minipage}
\end{figure*}
\clearpage

\begin{figure*}[!h]
  \begin{minipage}[t]{0.5\linewidth}
    \centering
    \includegraphics[scale=0.35]{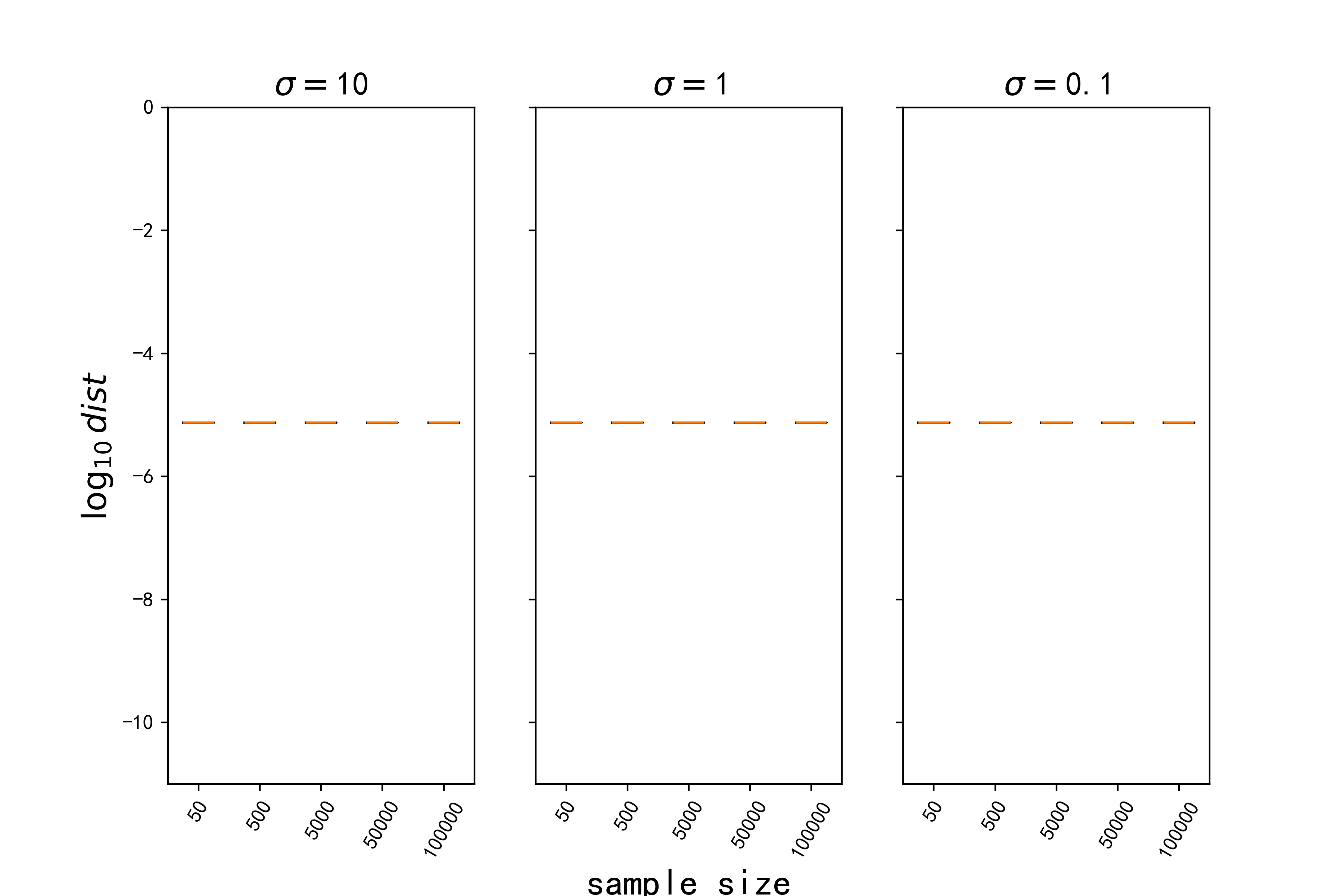}
    \caption{\scriptsize{Performance on HS99 w.r.t after 50 iterations.}}
  \end{minipage}%
  \begin{minipage}[t]{0.5\linewidth}
    \centering
    \includegraphics[scale=0.35]{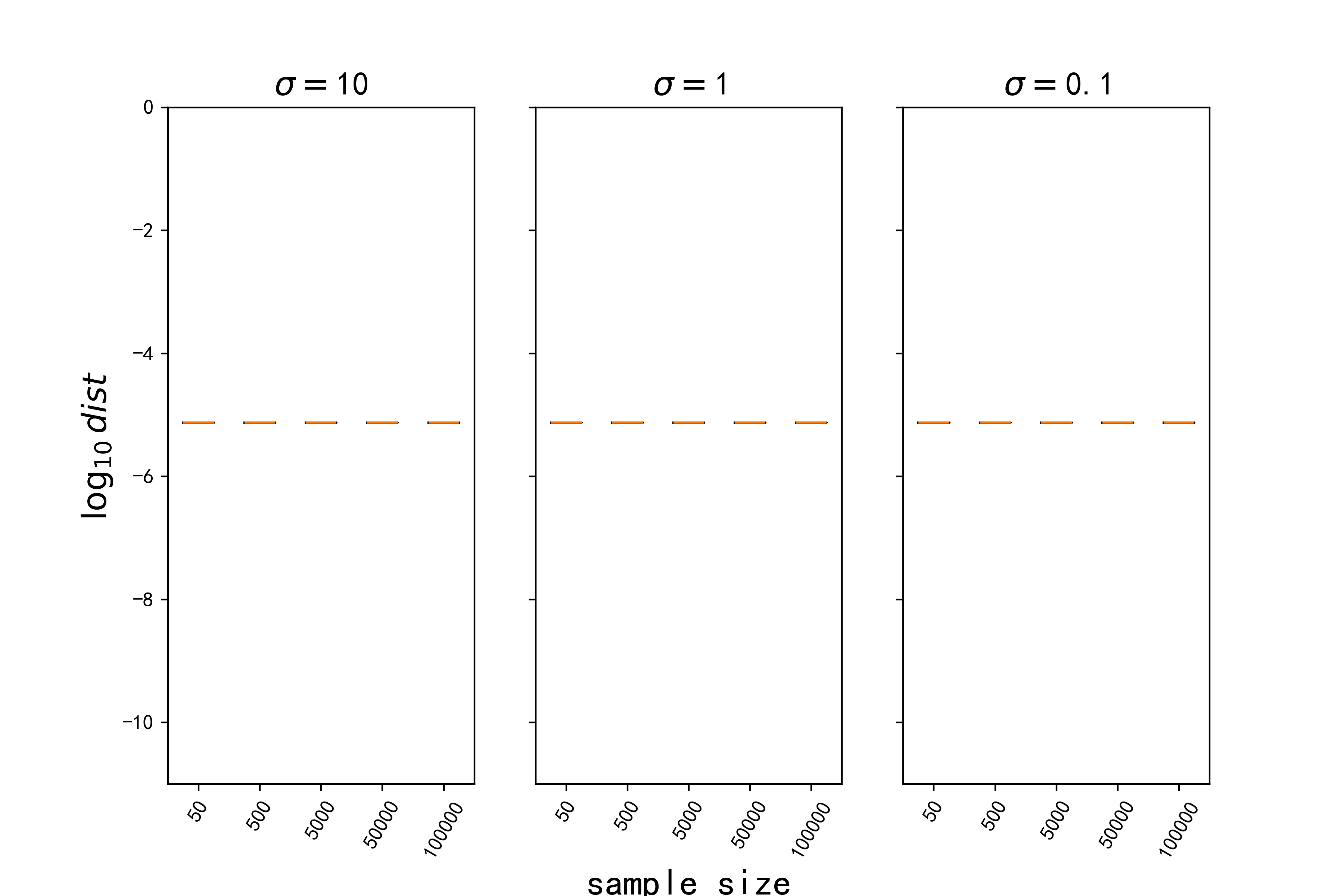}
    \caption{\scriptsize{Performance on HS99 w.r.t after 1 500 iterations.}}
  \end{minipage}
\end{figure*}

\begin{figure*}[!h]
  \begin{minipage}[t]{0.5\linewidth}
    \centering
    \includegraphics[scale=0.35]{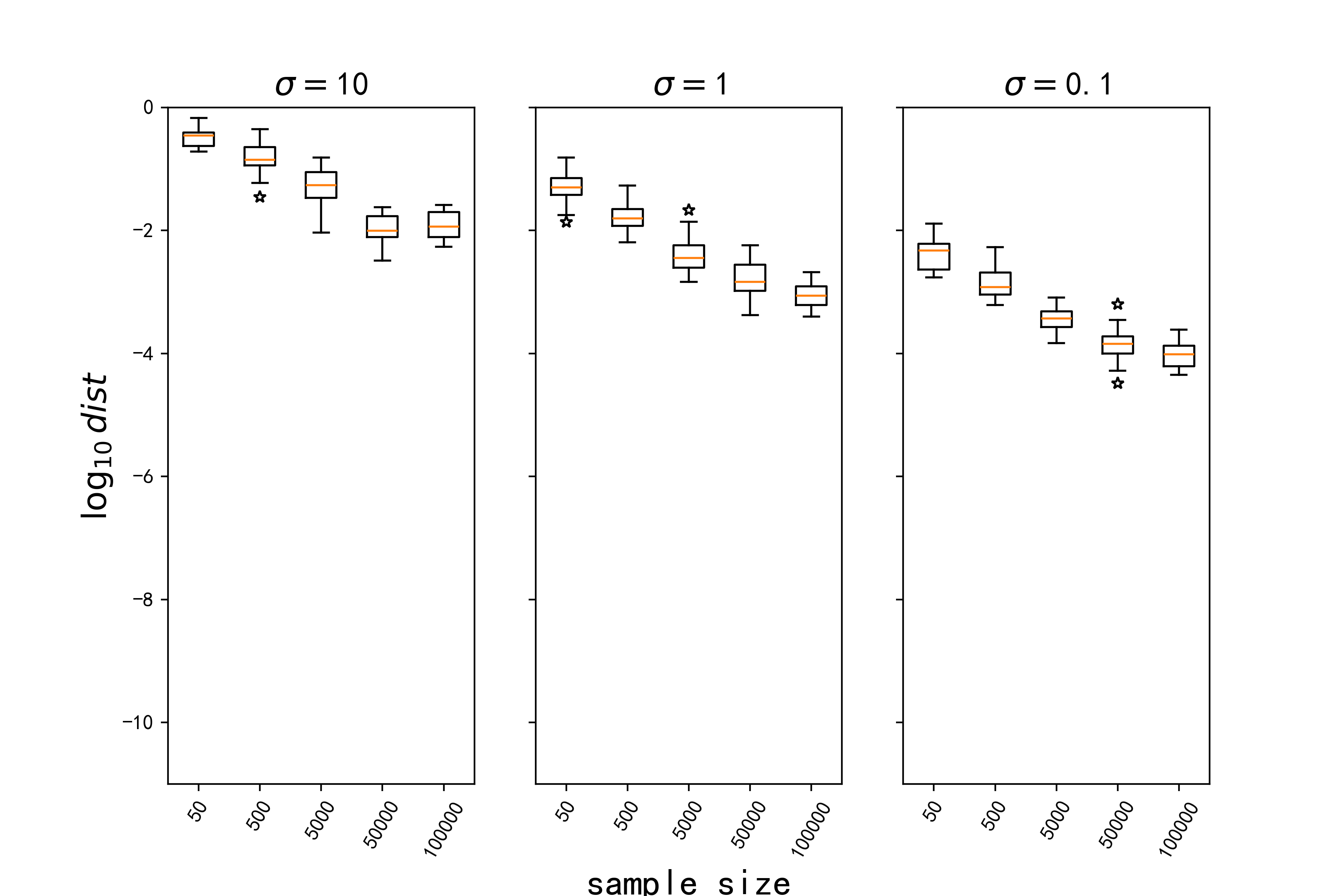}
    \caption{\scriptsize{Performance on HS100 w.r.t after 50 iterations.}}
  \end{minipage}%
  \begin{minipage}[t]{0.5\linewidth}
    \centering
    \includegraphics[scale=0.35]{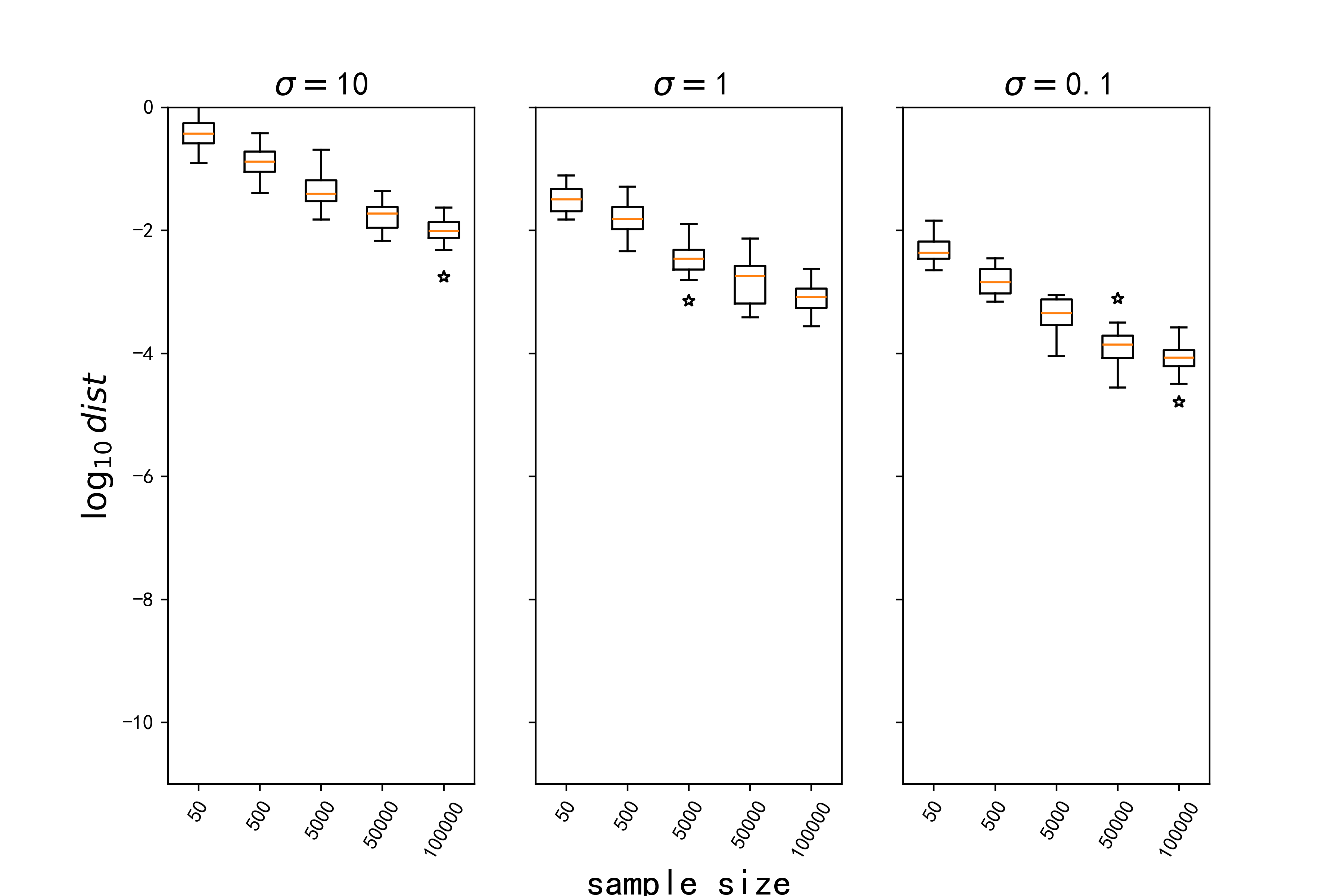}
    \caption{\scriptsize{Performance on HS100 w.r.t after 1 500 iterations.}}
  \end{minipage}
\end{figure*}
\clearpage

\begin{figure*}[!h]
  \begin{minipage}[t]{0.5\linewidth}
    \centering
    \includegraphics[scale=0.35]{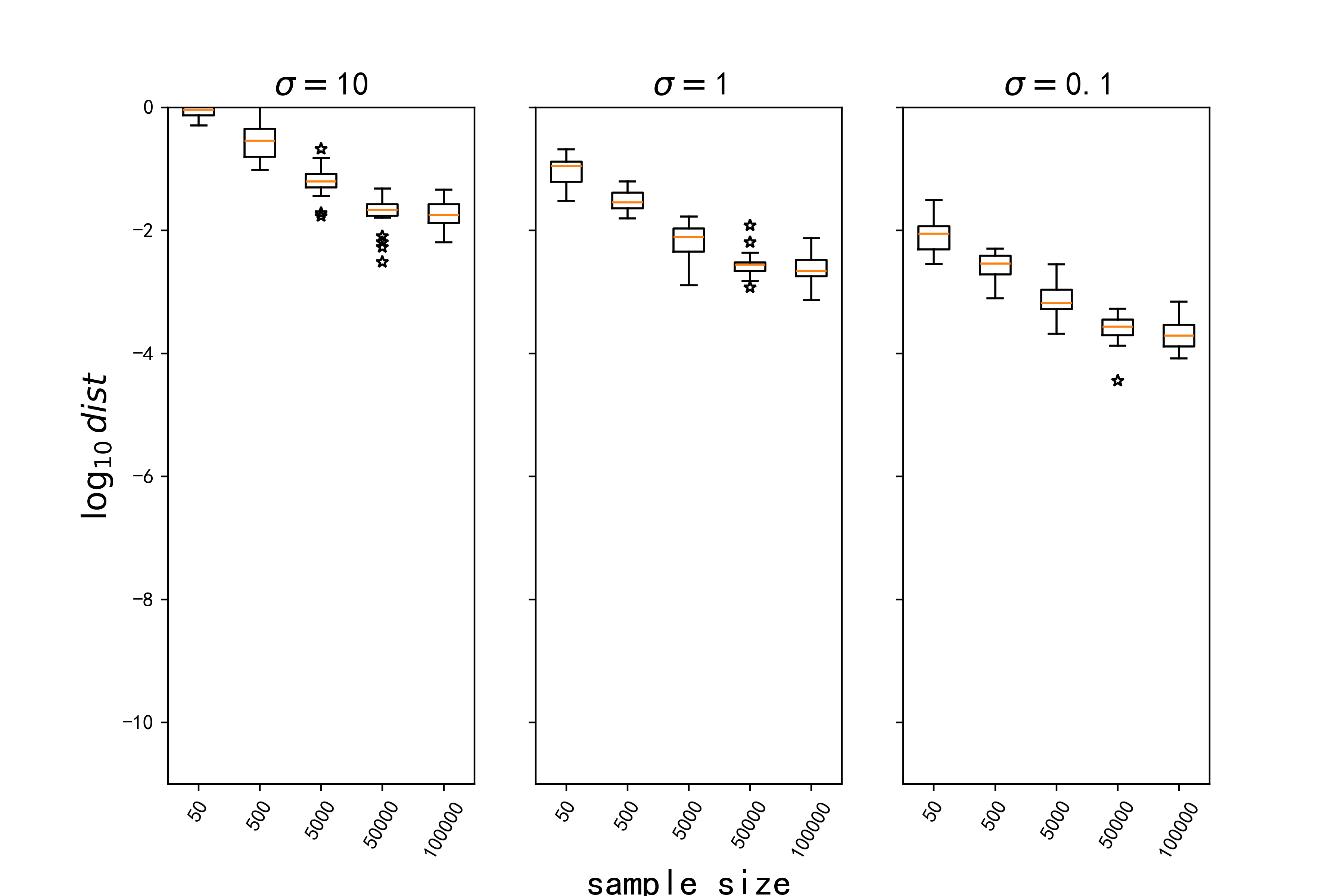}
    \caption{\scriptsize{Performance on HS113 w.r.t after 50 iterations.}}
  \end{minipage}%
  \begin{minipage}[t]{0.5\linewidth}
    \centering
    \includegraphics[scale=0.35]{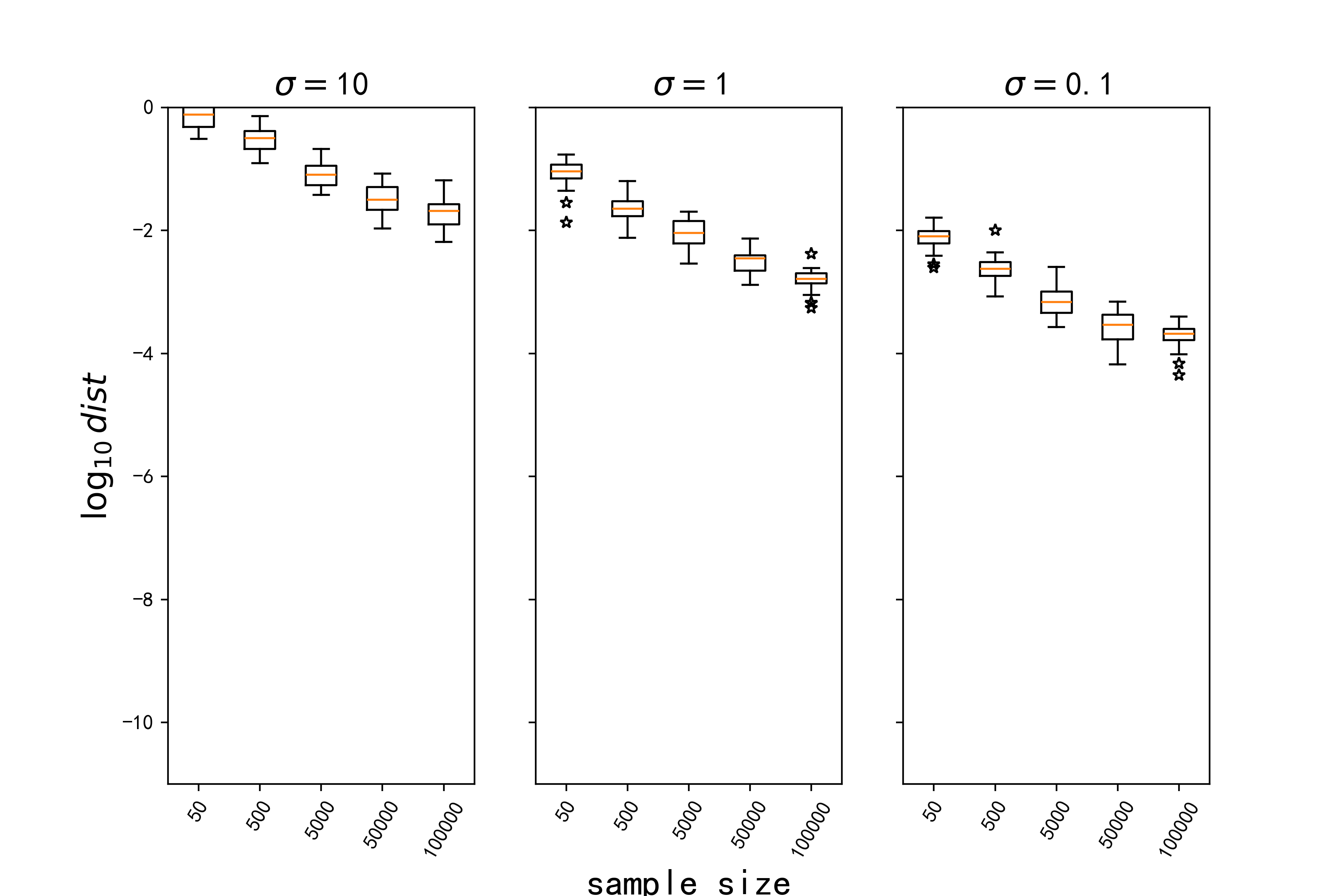}
    \caption{\scriptsize{Performance on HS113 w.r.t after 1 500 iterations.}}
  \end{minipage}
\end{figure*}

\begin{figure*}[!h]
  \begin{minipage}[t]{0.5\linewidth}
    \centering
    \includegraphics[scale=0.35]{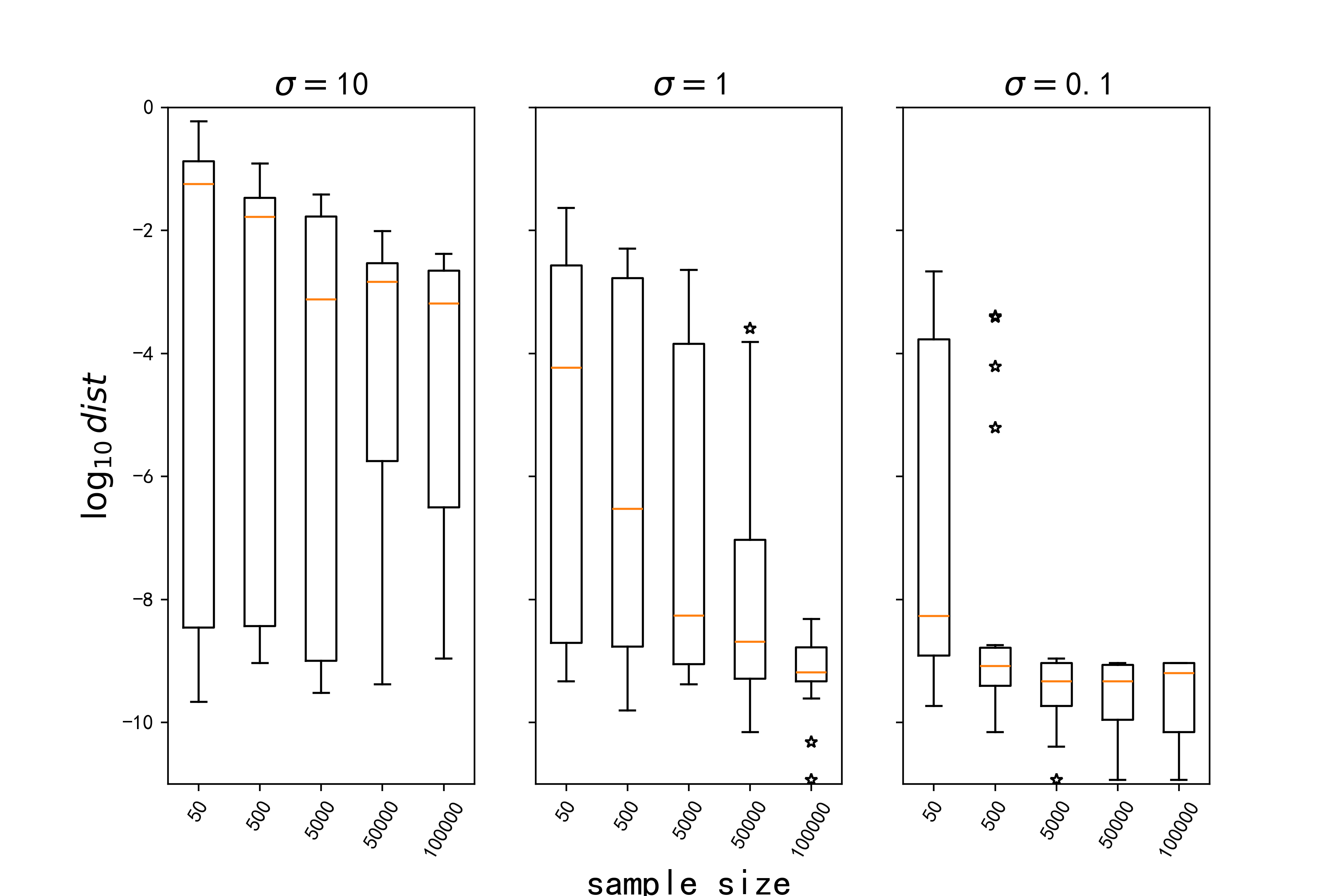}
    \caption{\scriptsize{Performance on S216 w.r.t after 50 iterations.}}
  \end{minipage}%
  \begin{minipage}[t]{0.5\linewidth}
    \centering
    \includegraphics[scale=0.35]{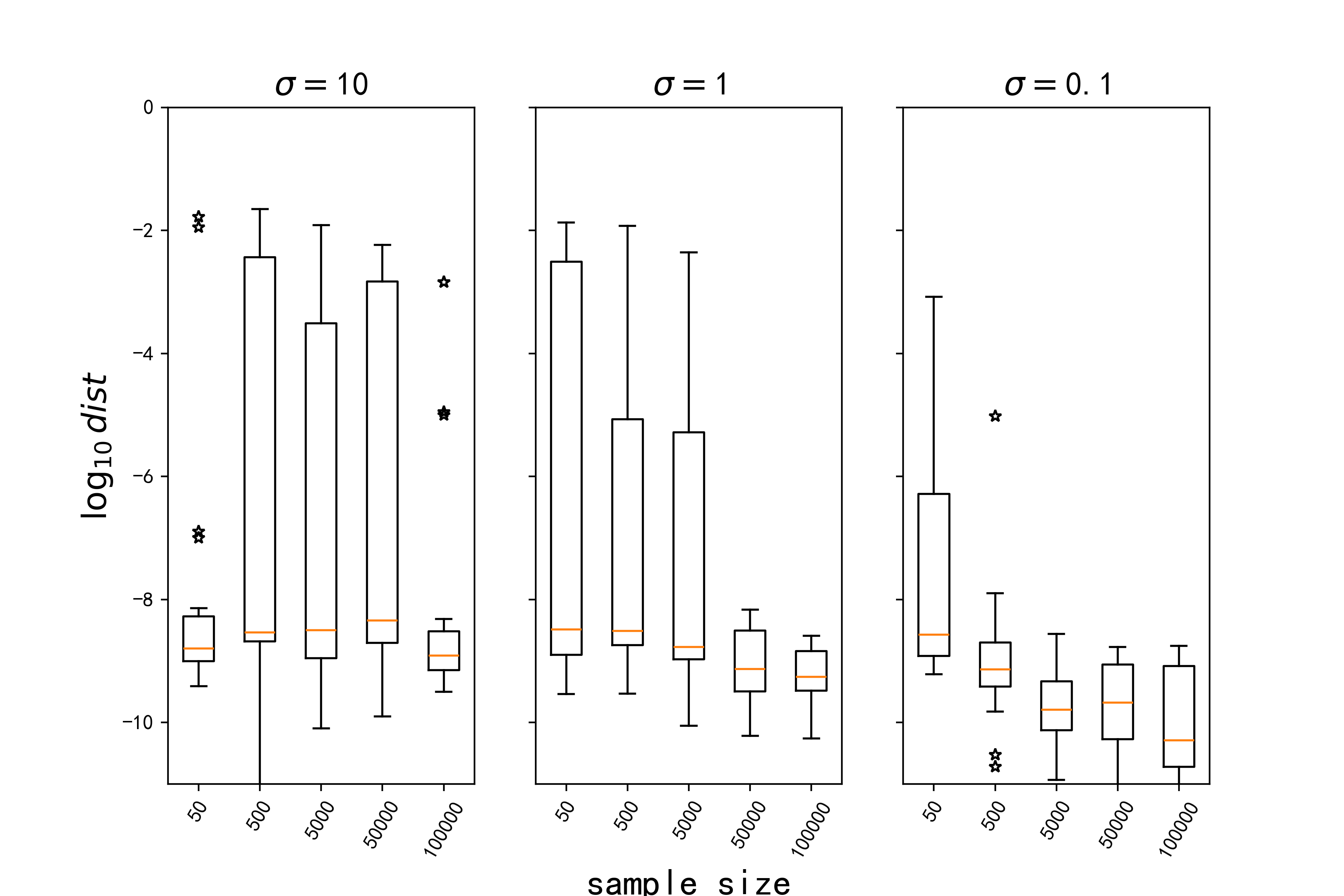}
    \caption{\scriptsize{Performance on S216 w.r.t after 1 500 iterations.}}
  \end{minipage}
\end{figure*}
\clearpage

\begin{figure*}[!h]
  \begin{minipage}[t]{0.5\linewidth}
    \centering
    \includegraphics[scale=0.35]{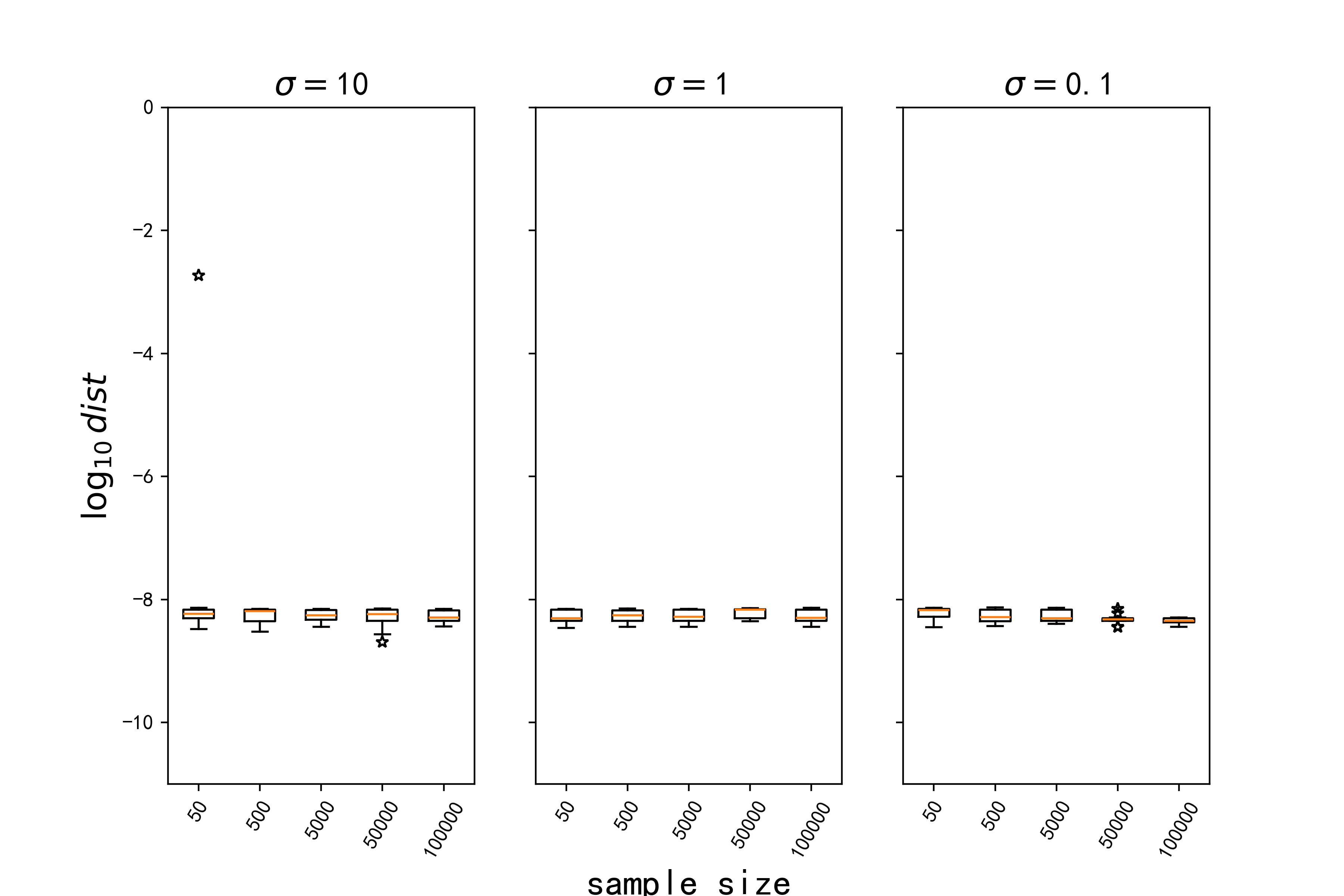}
    \caption{\scriptsize{Performance on S225 w.r.t after 50 iterations.}}
  \end{minipage}%
  \begin{minipage}[t]{0.5\linewidth}
    \centering
    \includegraphics[scale=0.35]{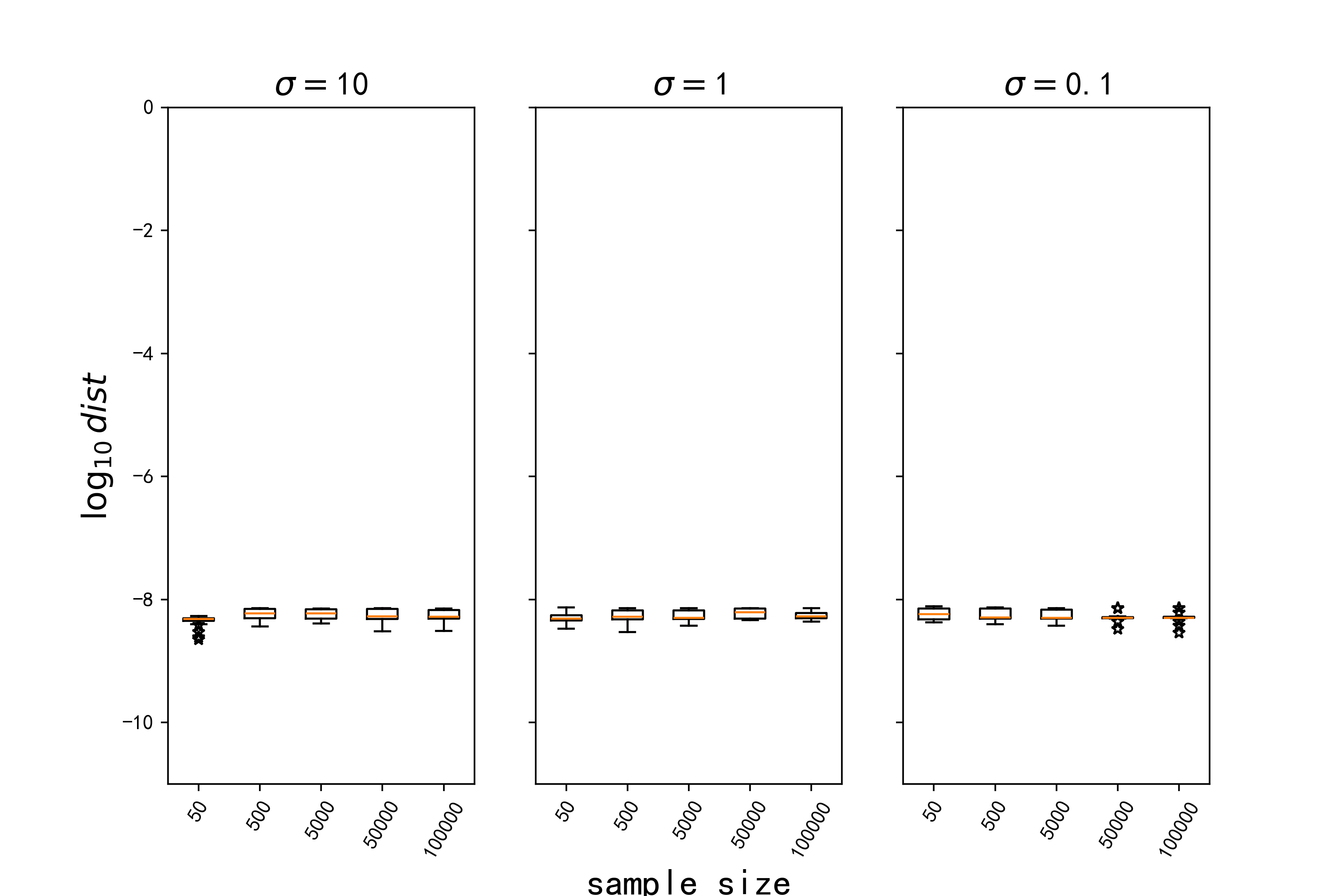}
    \caption{\scriptsize{Performance on S225 w.r.t after 1 500 iterations.}}
  \end{minipage}
\end{figure*}

\begin{figure*}[!h]
  \begin{minipage}[t]{0.5\linewidth}
    \centering
    \includegraphics[scale=0.35]{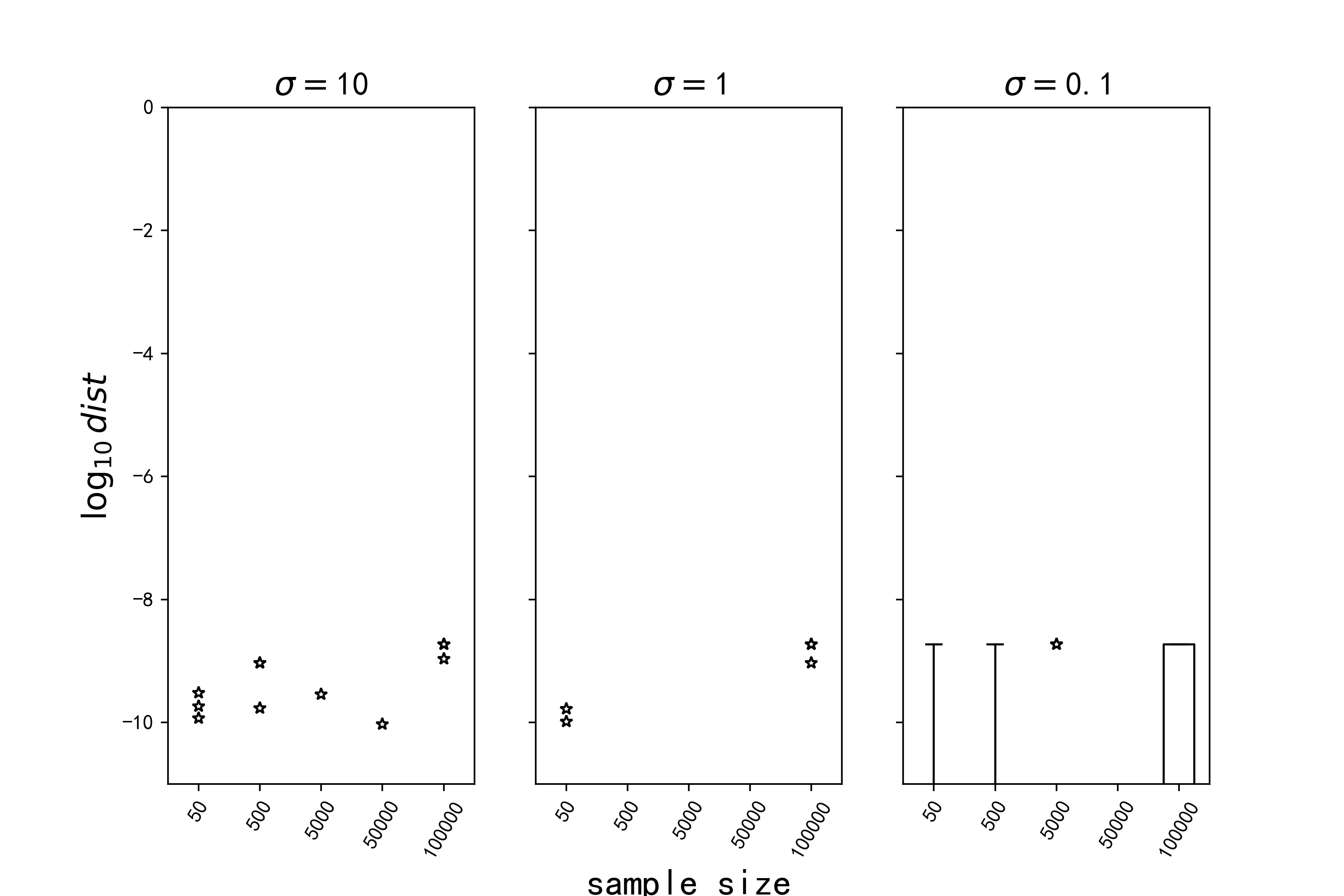}
    \caption{\scriptsize{Performance on S227 w.r.t after 50 iterations.}}
  \end{minipage}%
  \begin{minipage}[t]{0.5\linewidth}
    \centering
    \includegraphics[scale=0.35]{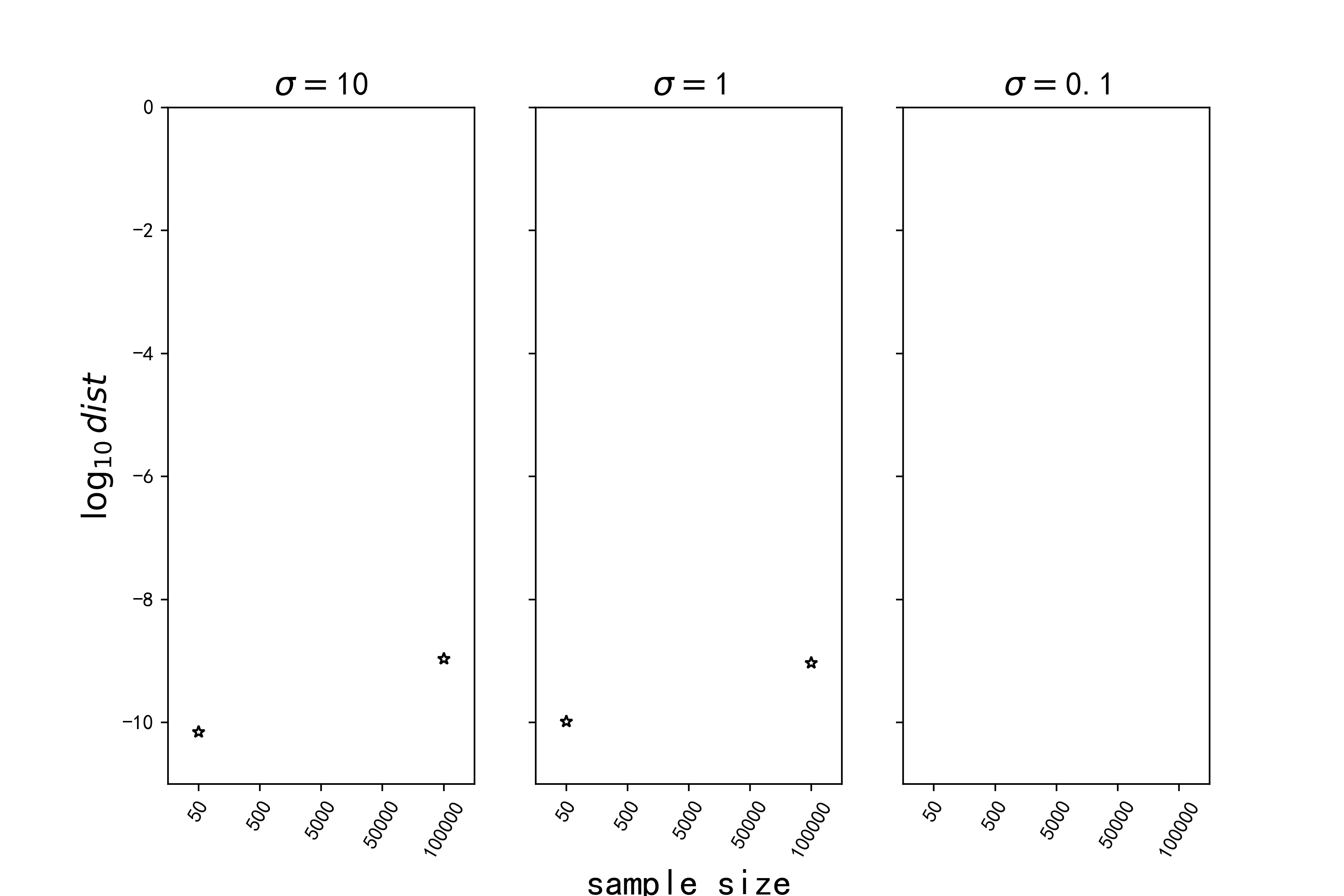}
    \caption{\scriptsize{Performance on S227 w.r.t after 1 500 iterations.}}
  \end{minipage}
\end{figure*}
\clearpage

\begin{figure*}[!h]
  \begin{minipage}[t]{0.5\linewidth}
    \centering
    \includegraphics[scale=0.35]{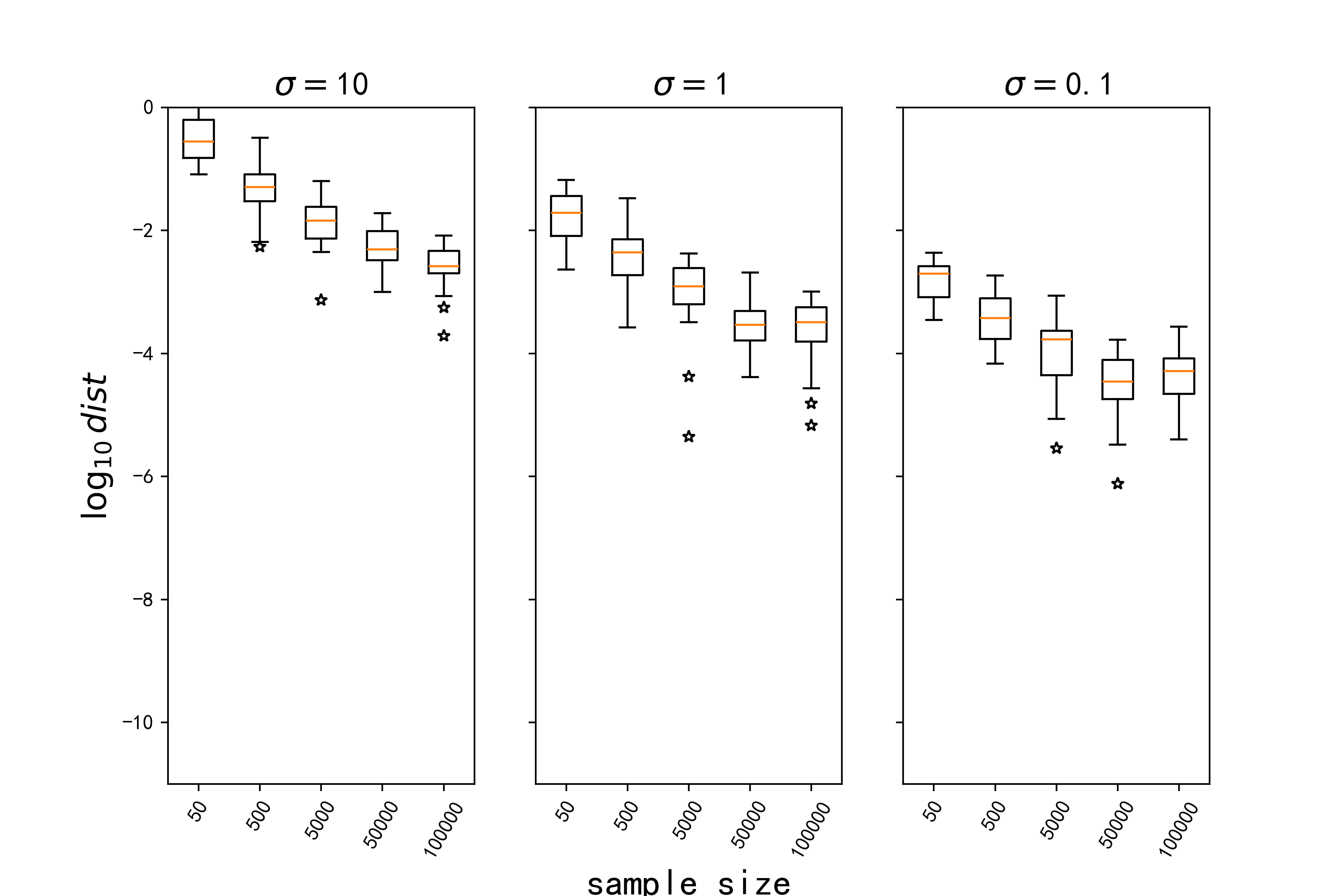}
    \caption{\scriptsize{Performance on S233 w.r.t after 50 iterations.}}
  \end{minipage}%
  \begin{minipage}[t]{0.5\linewidth}
    \centering
    \includegraphics[scale=0.35]{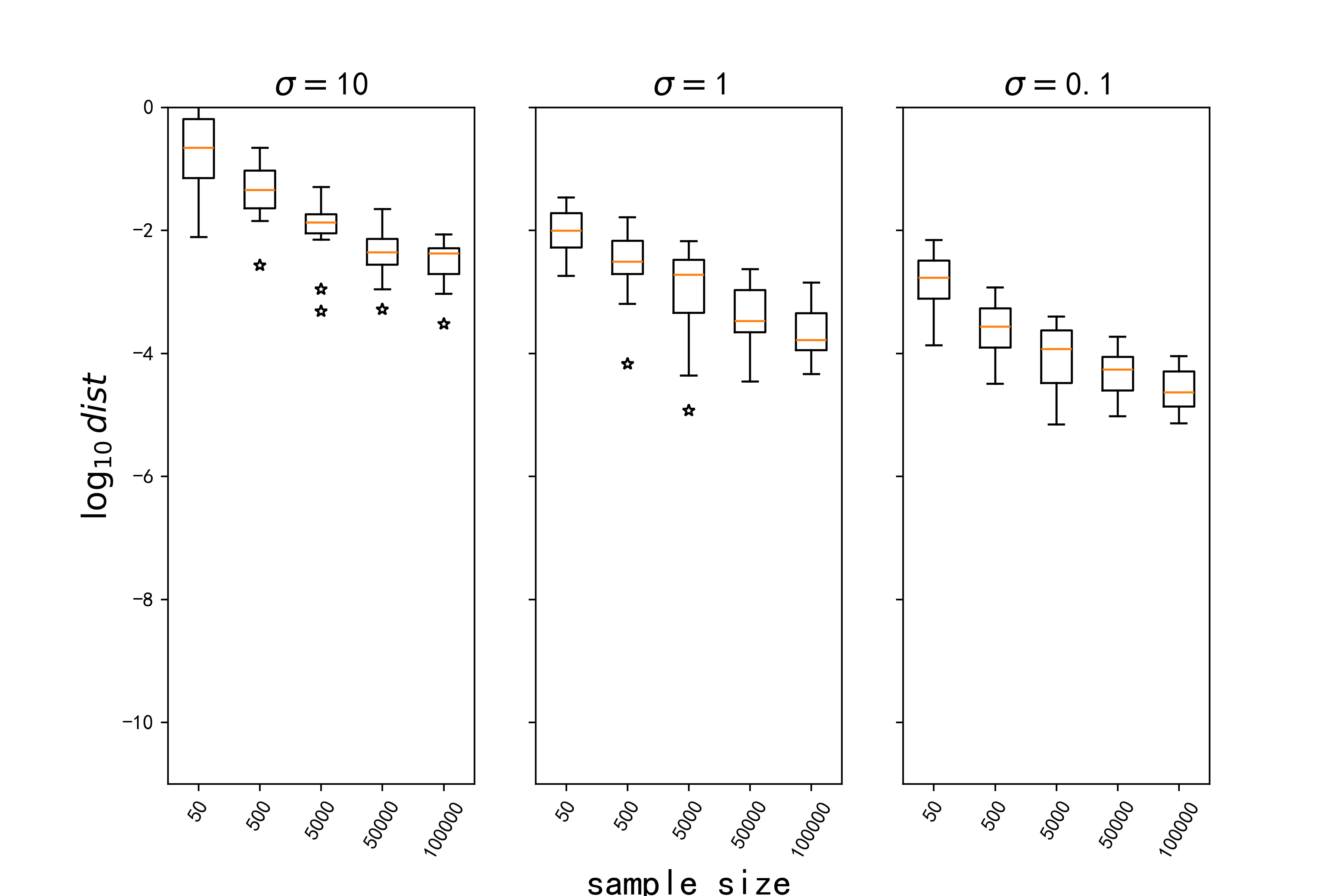}
    \caption{\scriptsize{Performance on S233 w.r.t after 1 500 iterations.}}
  \end{minipage}
\end{figure*}

\begin{figure*}[!h]
  \begin{minipage}[t]{0.5\linewidth}
    \centering
    \includegraphics[scale=0.35]{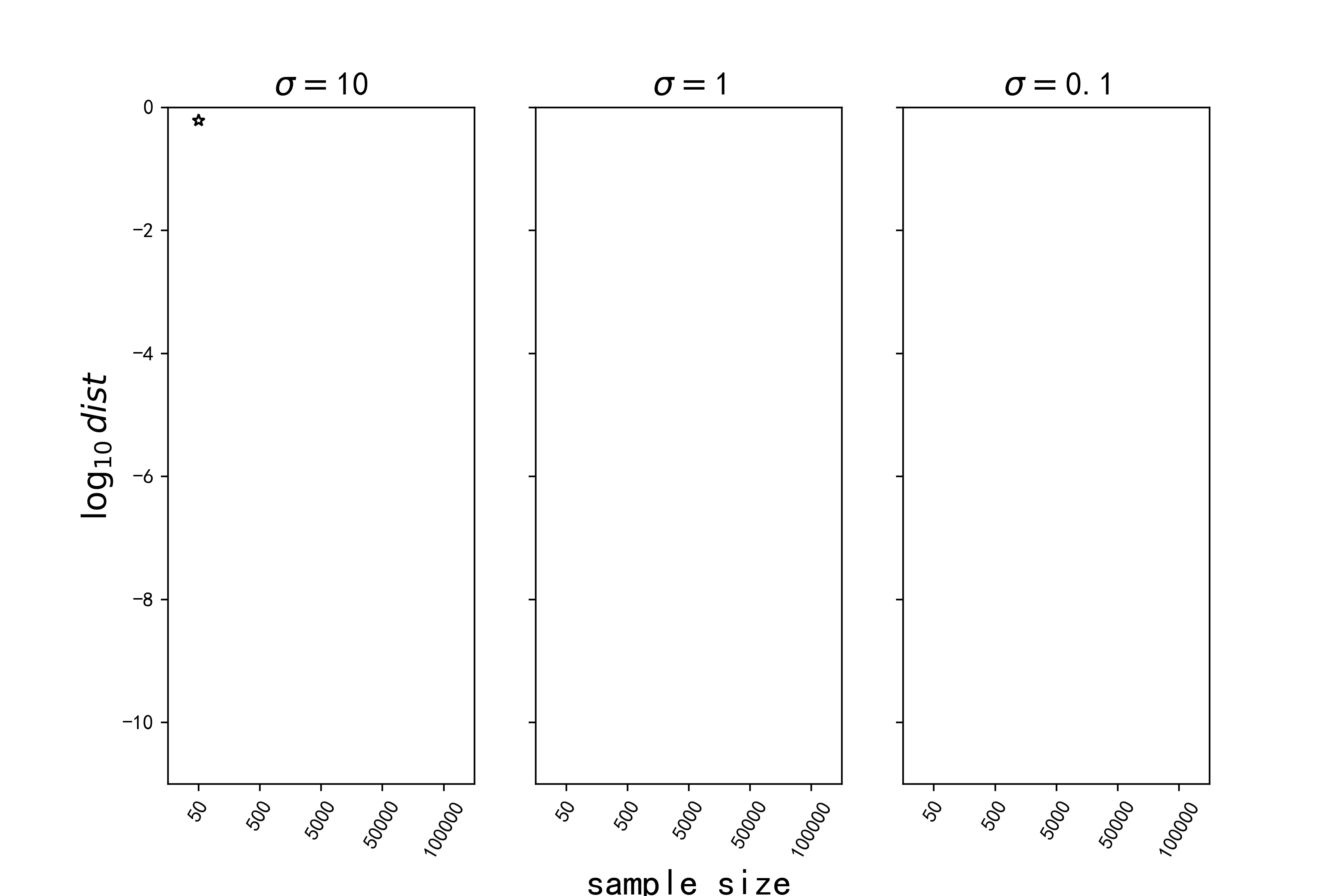}
    \caption{\scriptsize{Performance on S235 w.r.t after 50 iterations.}}
  \end{minipage}%
  \begin{minipage}[t]{0.5\linewidth}
    \centering
    \includegraphics[scale=0.35]{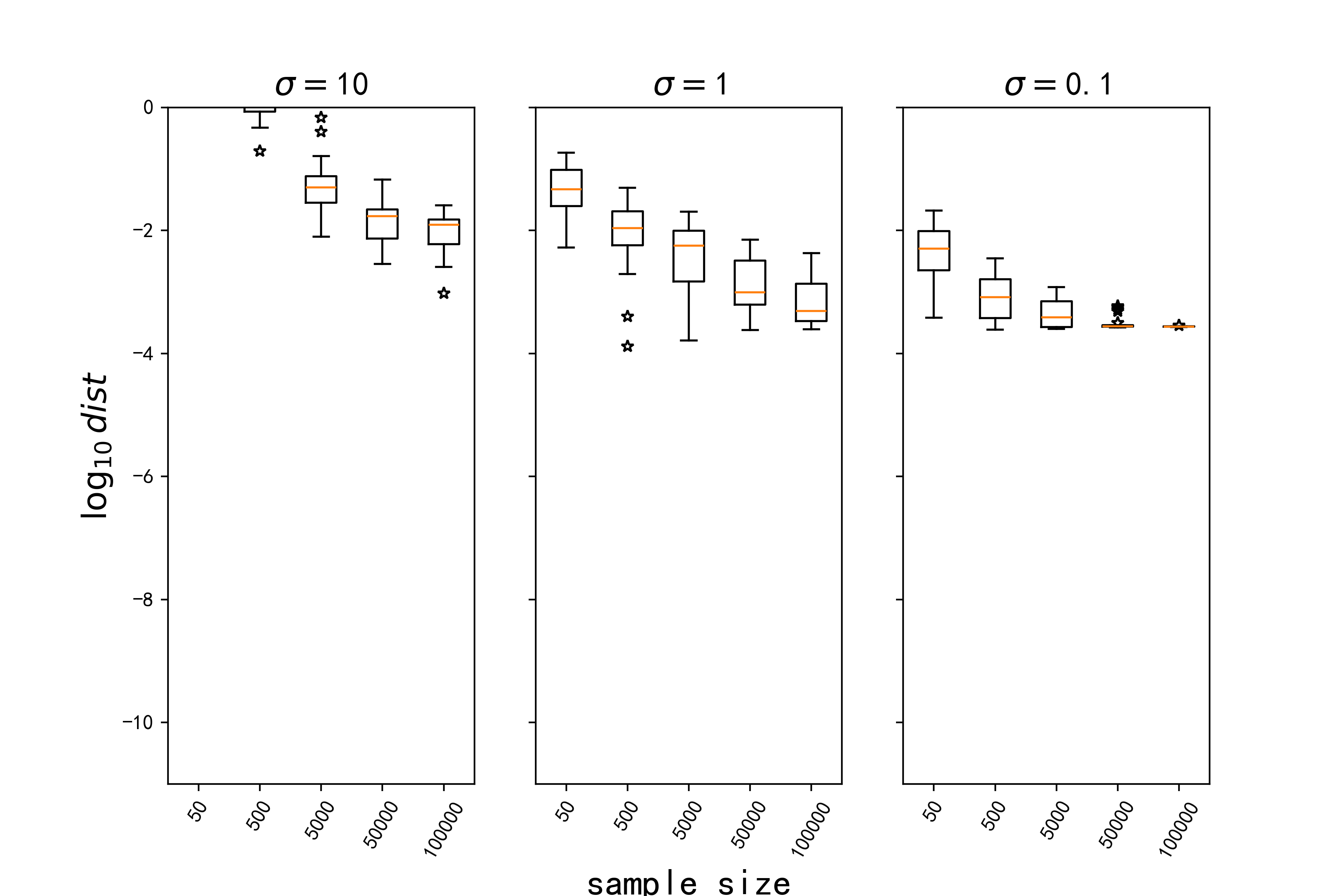}
    \caption{\scriptsize{Performance on S235 w.r.t after 1 500 iterations.}}
  \end{minipage}
\end{figure*}
\clearpage

\begin{figure*}[!h]
  \begin{minipage}[t]{0.5\linewidth}
    \centering
    \includegraphics[scale=0.35]{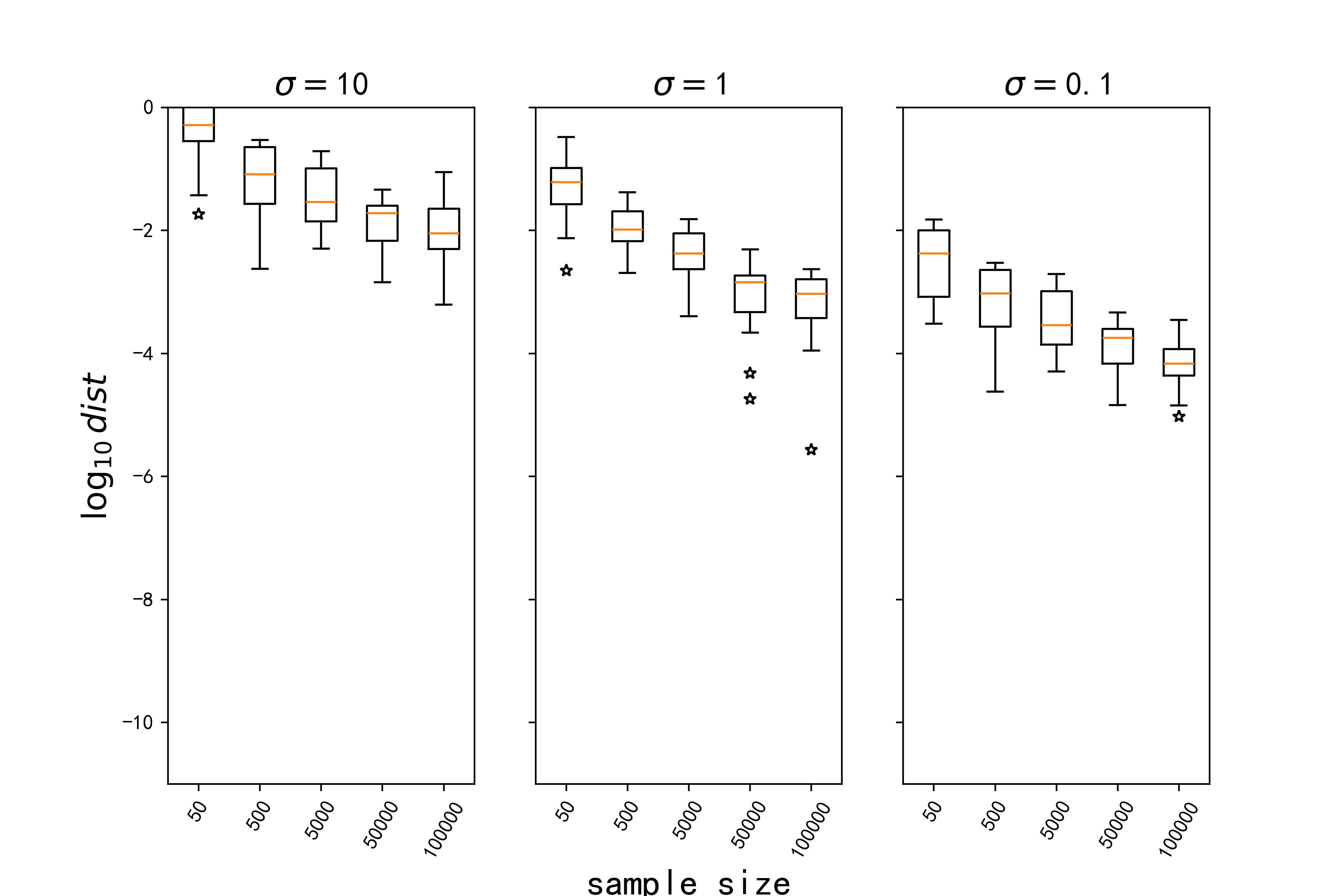}
    \caption{\scriptsize{Performance on S249 w.r.t after 50 iterations.}}
  \end{minipage}%
  \begin{minipage}[t]{0.5\linewidth}
    \centering
    \includegraphics[scale=0.35]{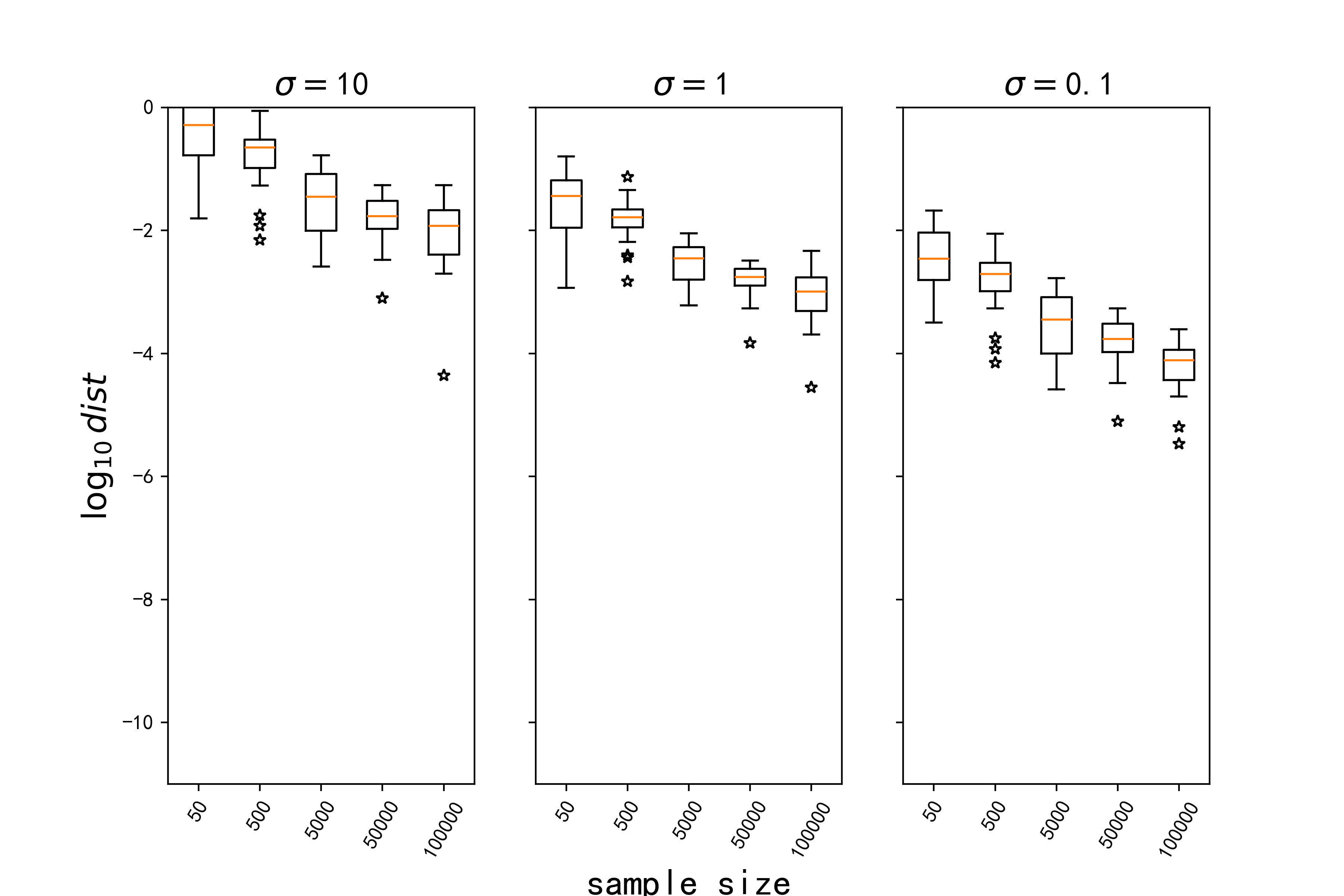}
    \caption{\scriptsize{Performance on S249 w.r.t after 1 500 iterations.}}
  \end{minipage}
\end{figure*}

\begin{figure*}[!h]
  \begin{minipage}[t]{0.5\linewidth}
    \centering
    \includegraphics[scale=0.35]{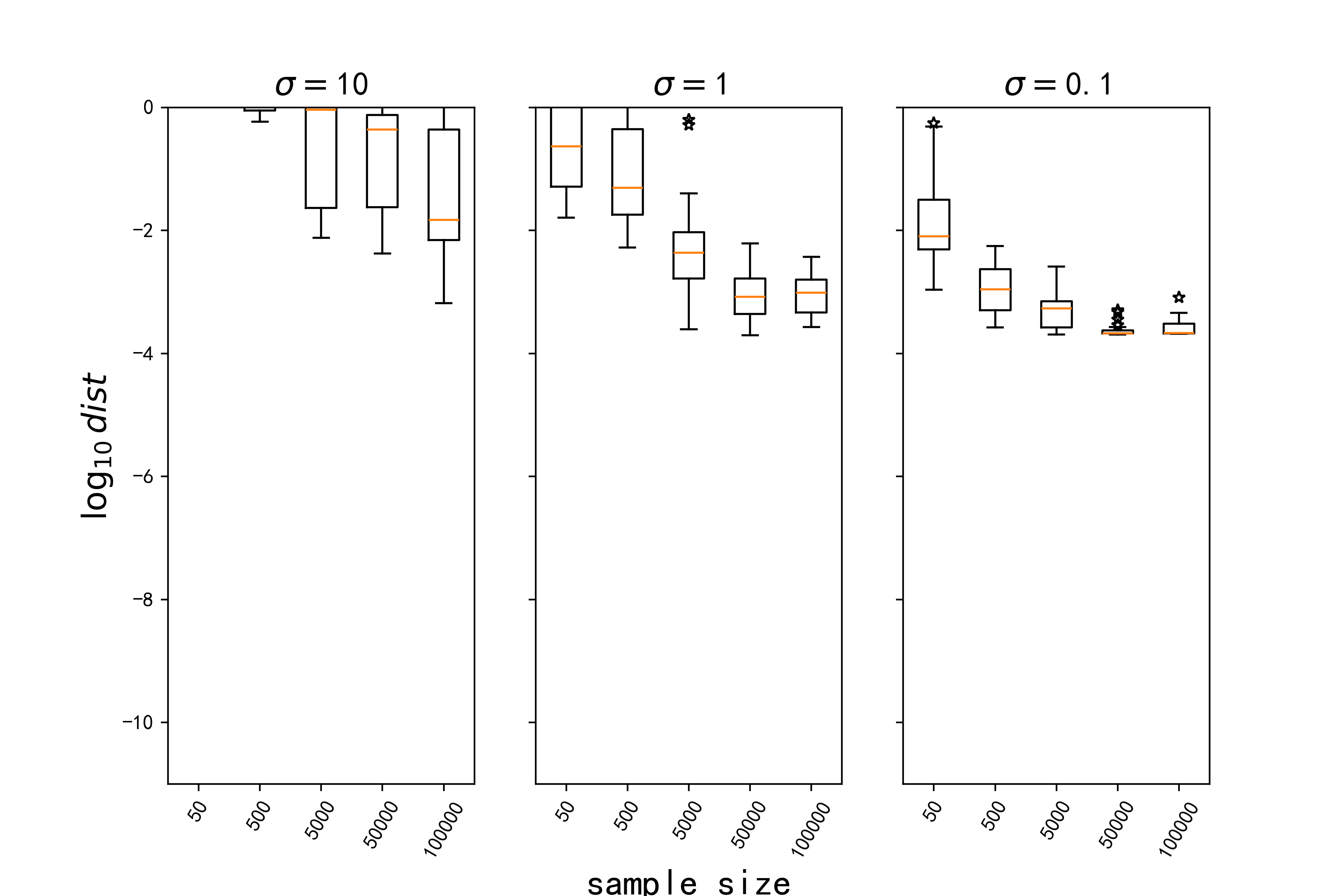}
    \caption{\scriptsize{Performance on S252 w.r.t after 50 iterations.}}
  \end{minipage}%
  \begin{minipage}[t]{0.5\linewidth}
    \centering
    \includegraphics[scale=0.35]{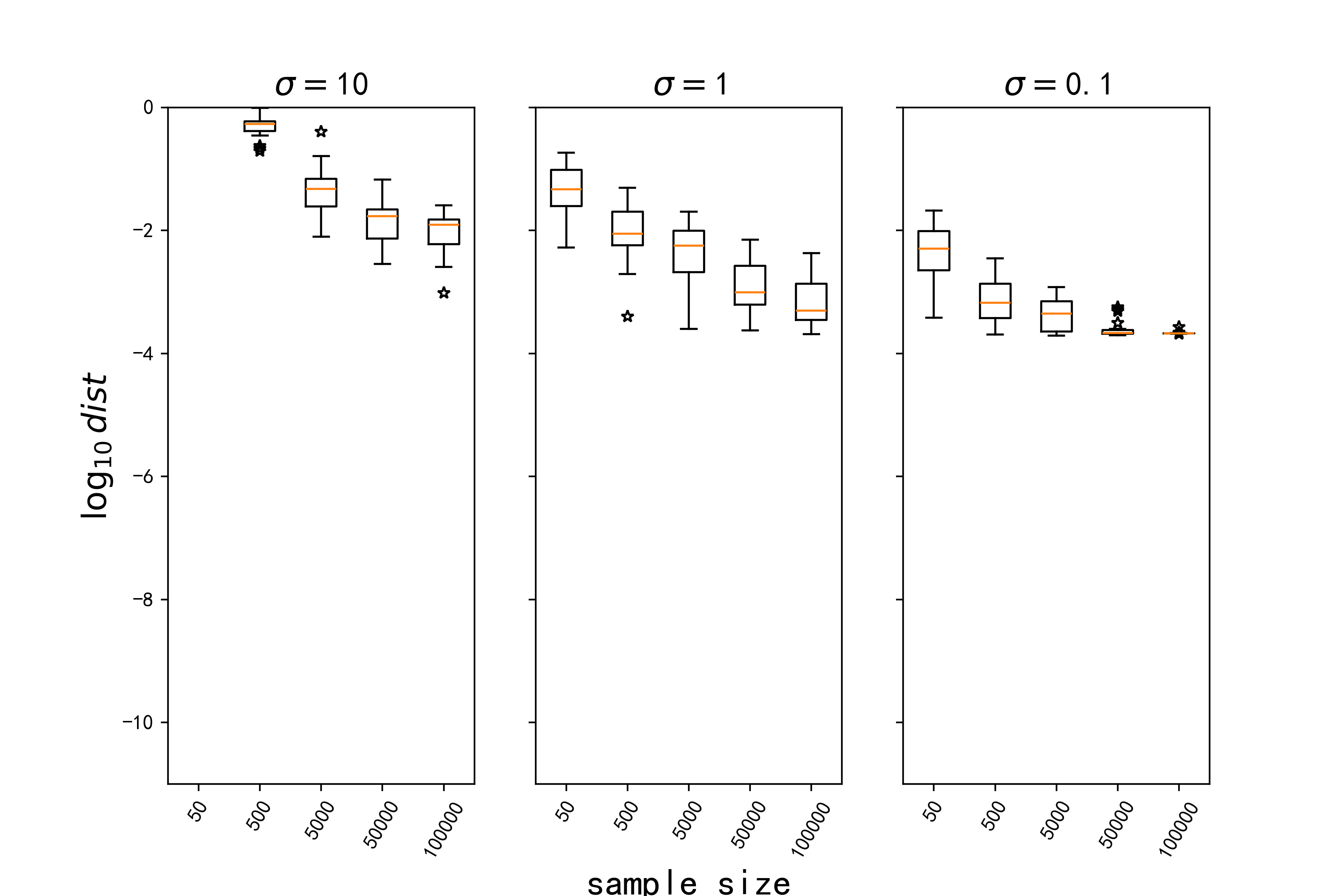}
    \caption{\scriptsize{Performance on S252 w.r.t after 1 500 iterations.}}
  \end{minipage}
\end{figure*}
\clearpage

\begin{figure*}[!h]
  \begin{minipage}[t]{0.5\linewidth}
    \centering
    \includegraphics[scale=0.35]{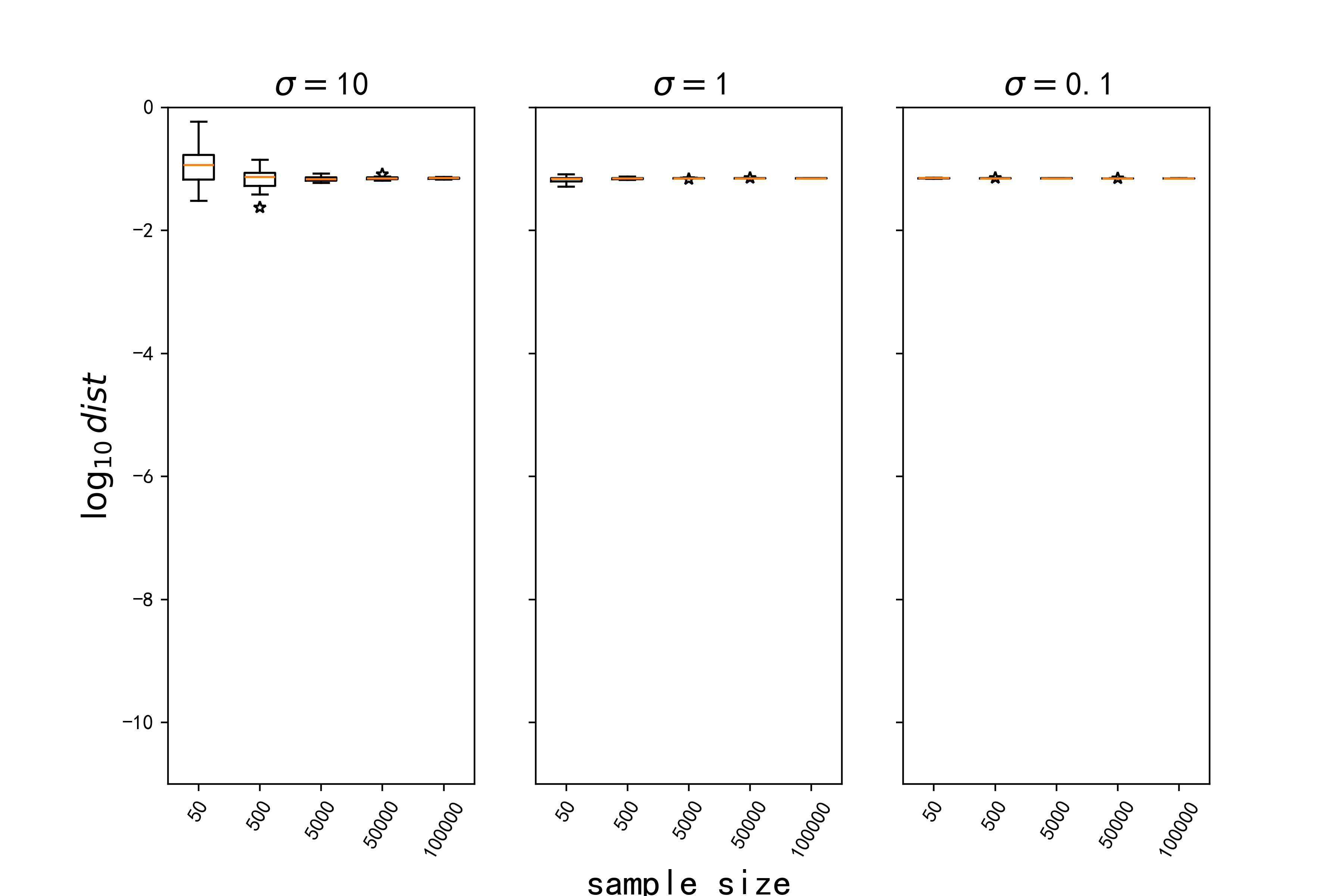}
    \caption{\scriptsize{Performance on S264 w.r.t after 50 iterations.}}
  \end{minipage}%
  \begin{minipage}[t]{0.5\linewidth}
    \centering
    \includegraphics[scale=0.35]{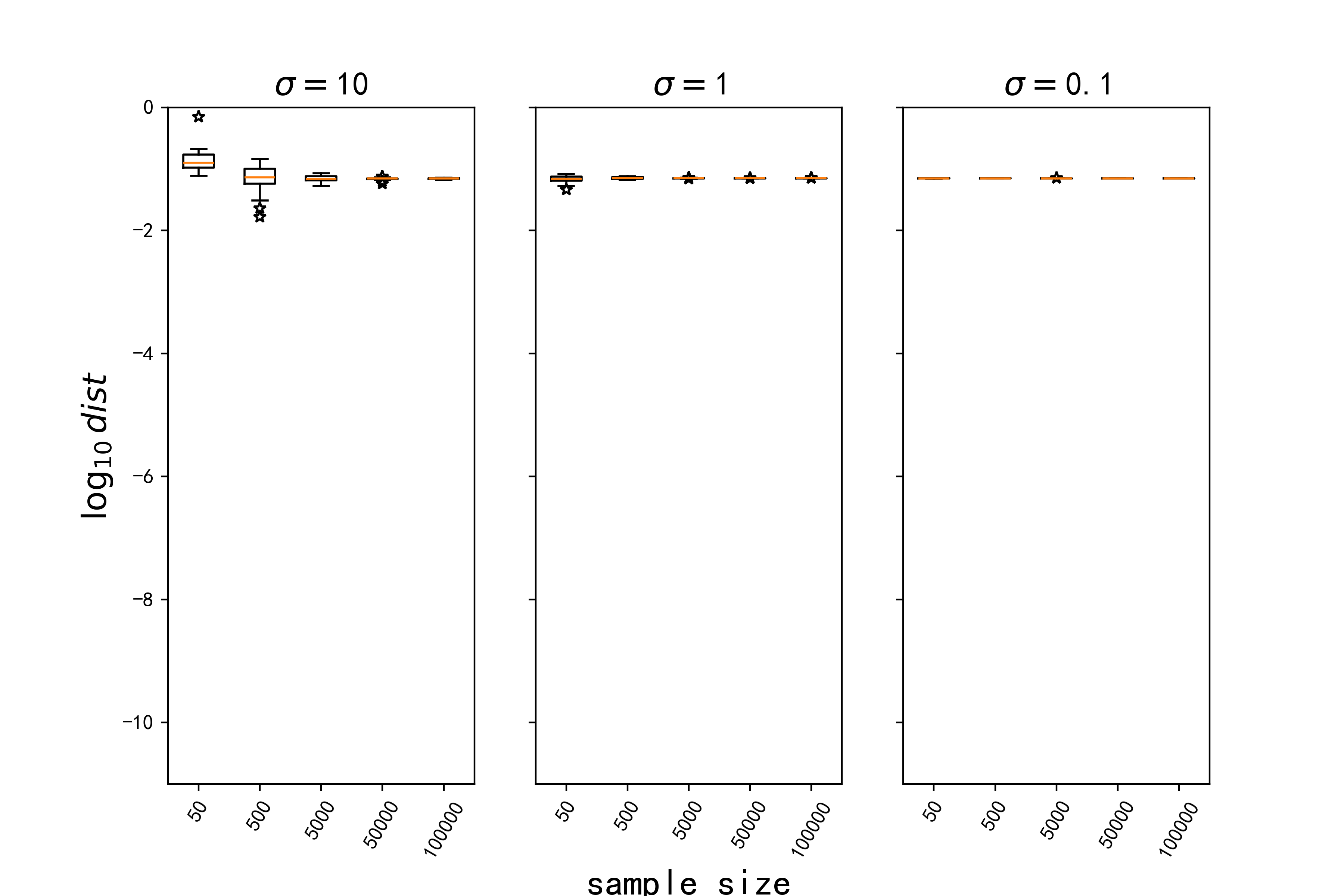}
    \caption{\scriptsize{Performance on S264 w.r.t after 1 500 iterations.}}
  \end{minipage}
\end{figure*}

\begin{figure*}[!h]
  \begin{minipage}[t]{0.5\linewidth}
    \centering
    \includegraphics[scale=0.35]{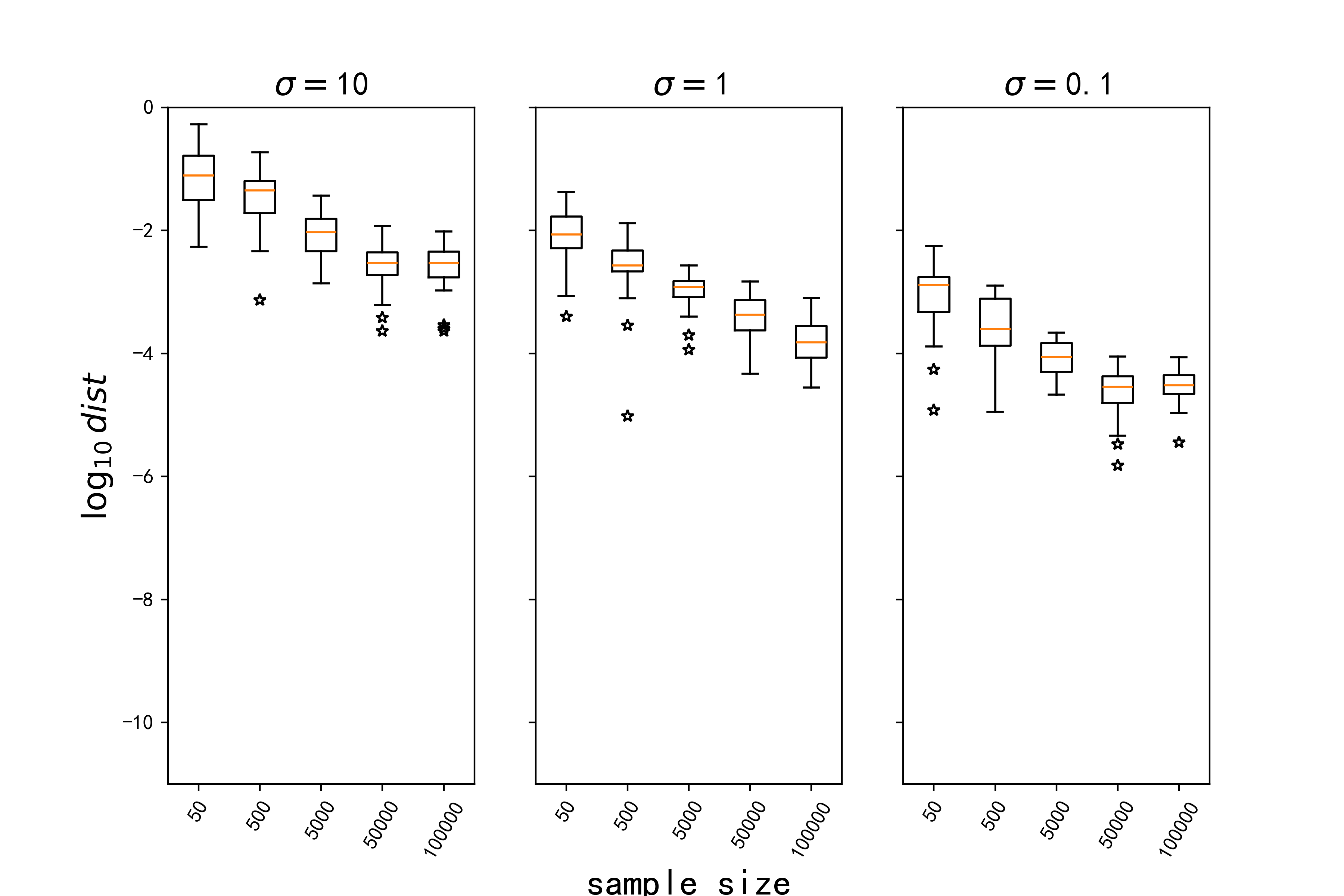}
    \caption{\scriptsize{Performance on S316 w.r.t after 50 iterations.}}
  \end{minipage}%
  \begin{minipage}[t]{0.5\linewidth}
    \centering
    \includegraphics[scale=0.35]{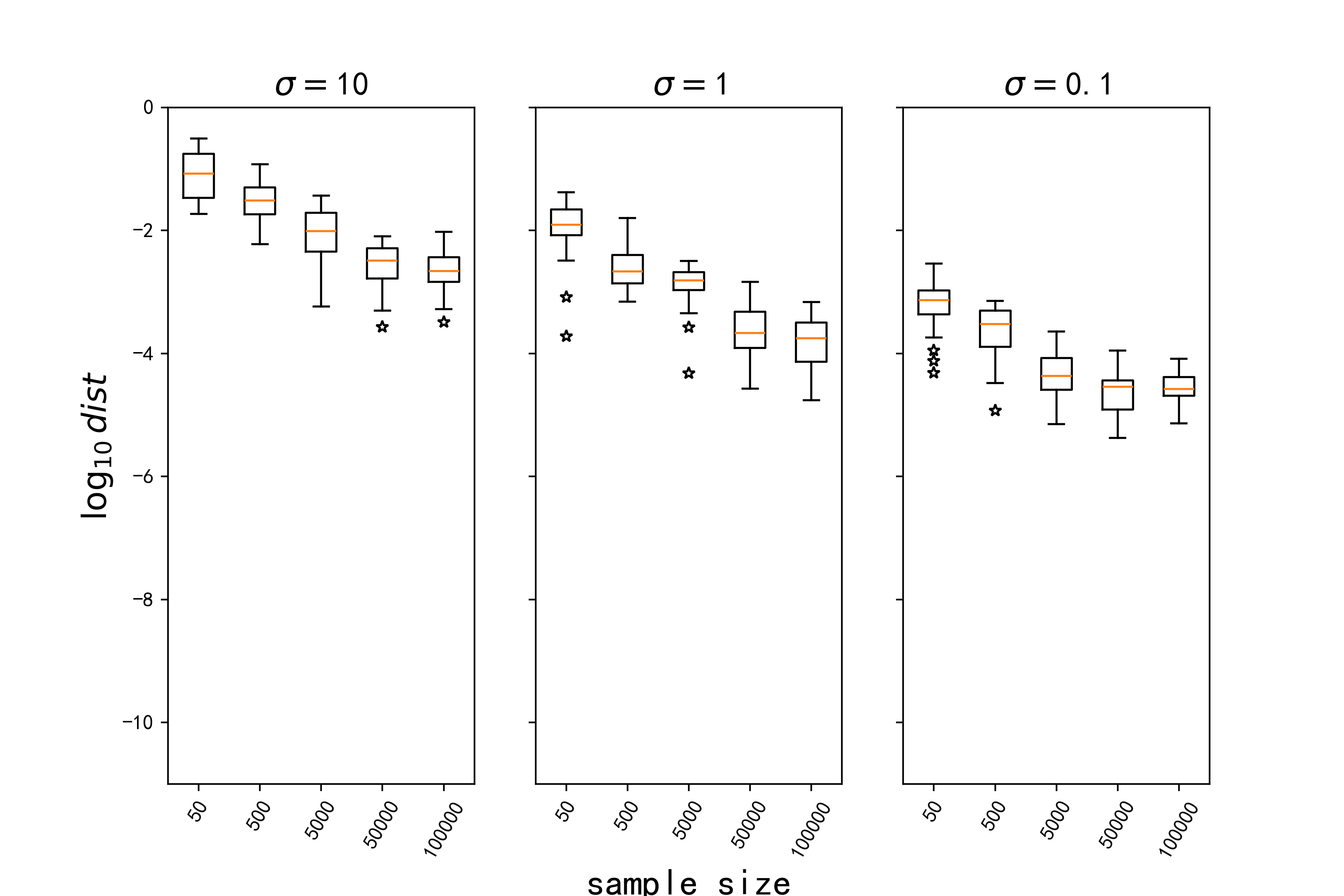}
    \caption{\scriptsize{Performance on S316 w.r.t after 1 500 iterations.}}
  \end{minipage}
\end{figure*}
\clearpage

\begin{figure*}[!h]
  \begin{minipage}[t]{0.5\linewidth}
    \centering
    \includegraphics[scale=0.35]{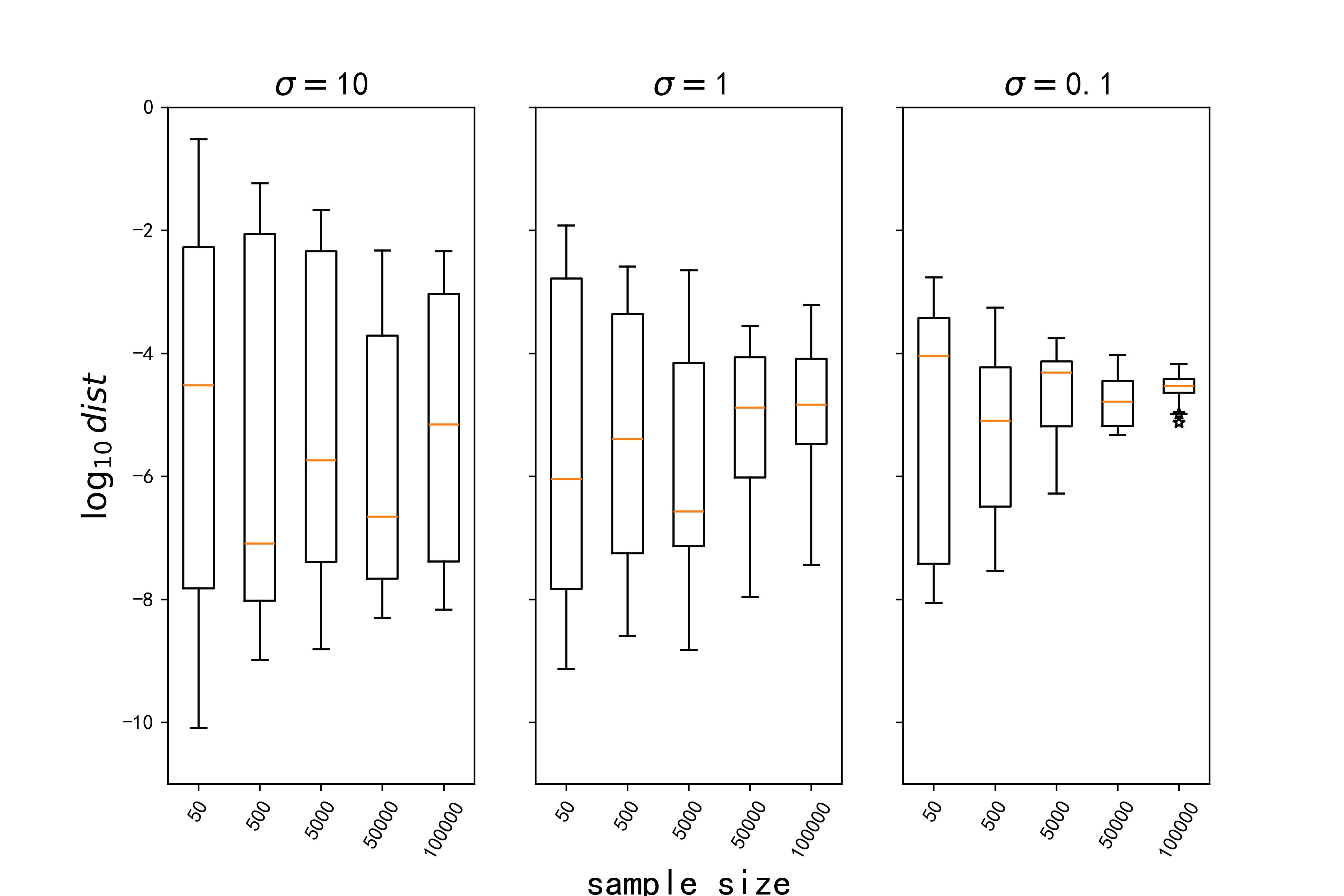}
    \caption{\scriptsize{Performance on S317 w.r.t after 50 iterations.}}
  \end{minipage}%
  \begin{minipage}[t]{0.5\linewidth}
    \centering
    \includegraphics[scale=0.35]{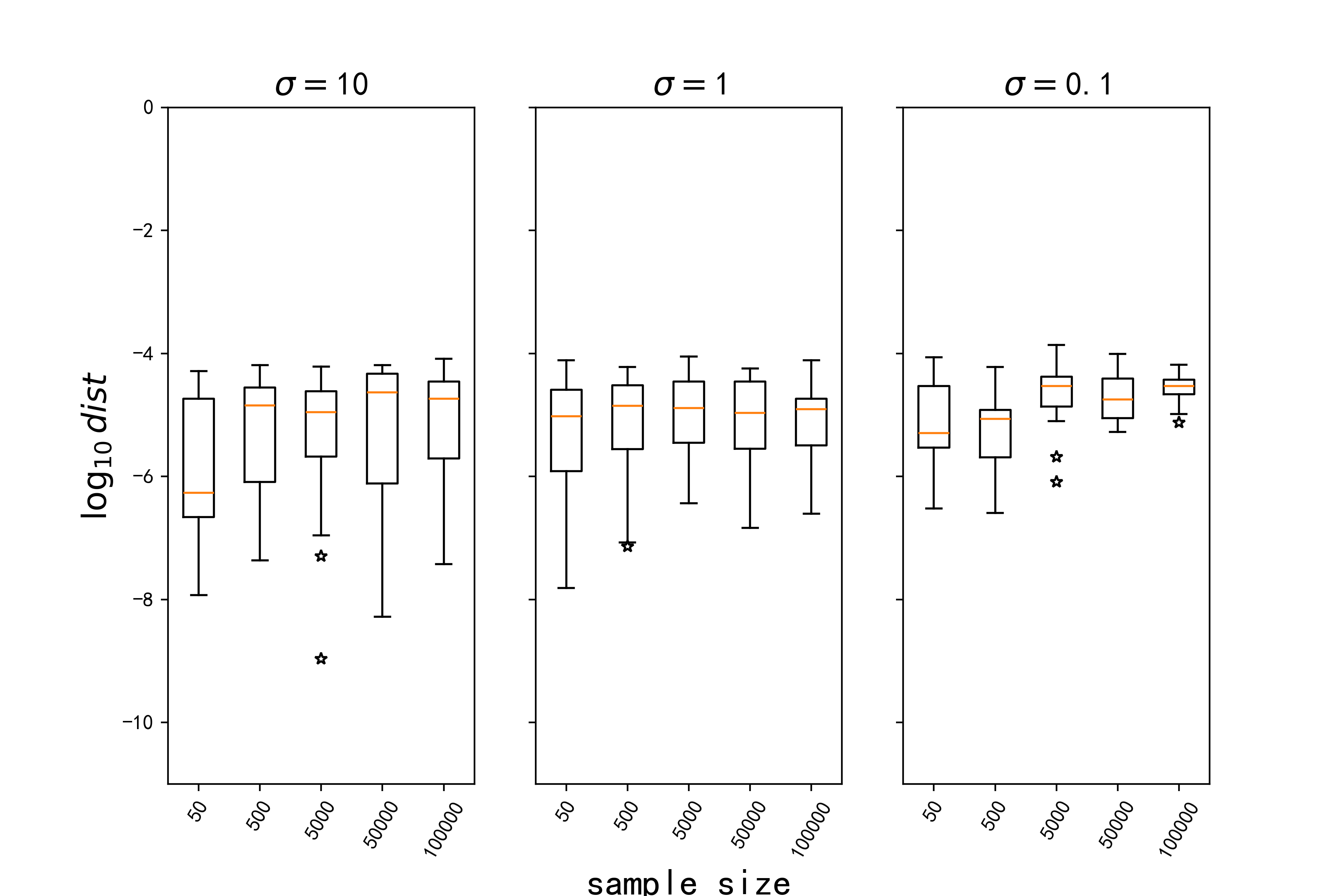}
    \caption{\scriptsize{Performance on S317 w.r.t after 1 500 iterations.}}
  \end{minipage}
\end{figure*}

\begin{figure*}[!h]
  \begin{minipage}[t]{0.5\linewidth}
    \centering
    \includegraphics[scale=0.35]{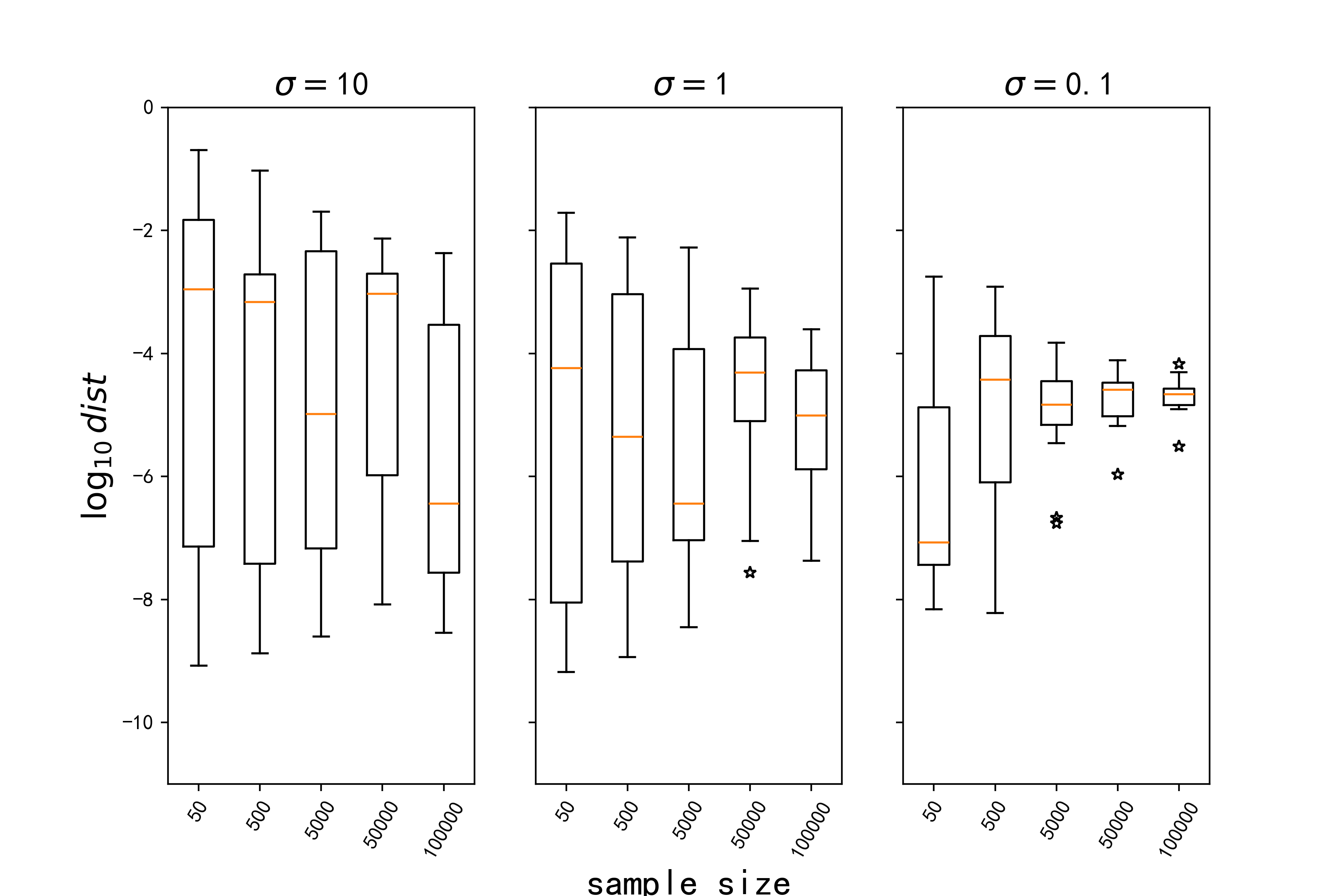}
    \caption{\scriptsize{Performance on S318 w.r.t after 50 iterations.}}
  \end{minipage}%
  \begin{minipage}[t]{0.5\linewidth}
    \centering
    \includegraphics[scale=0.35]{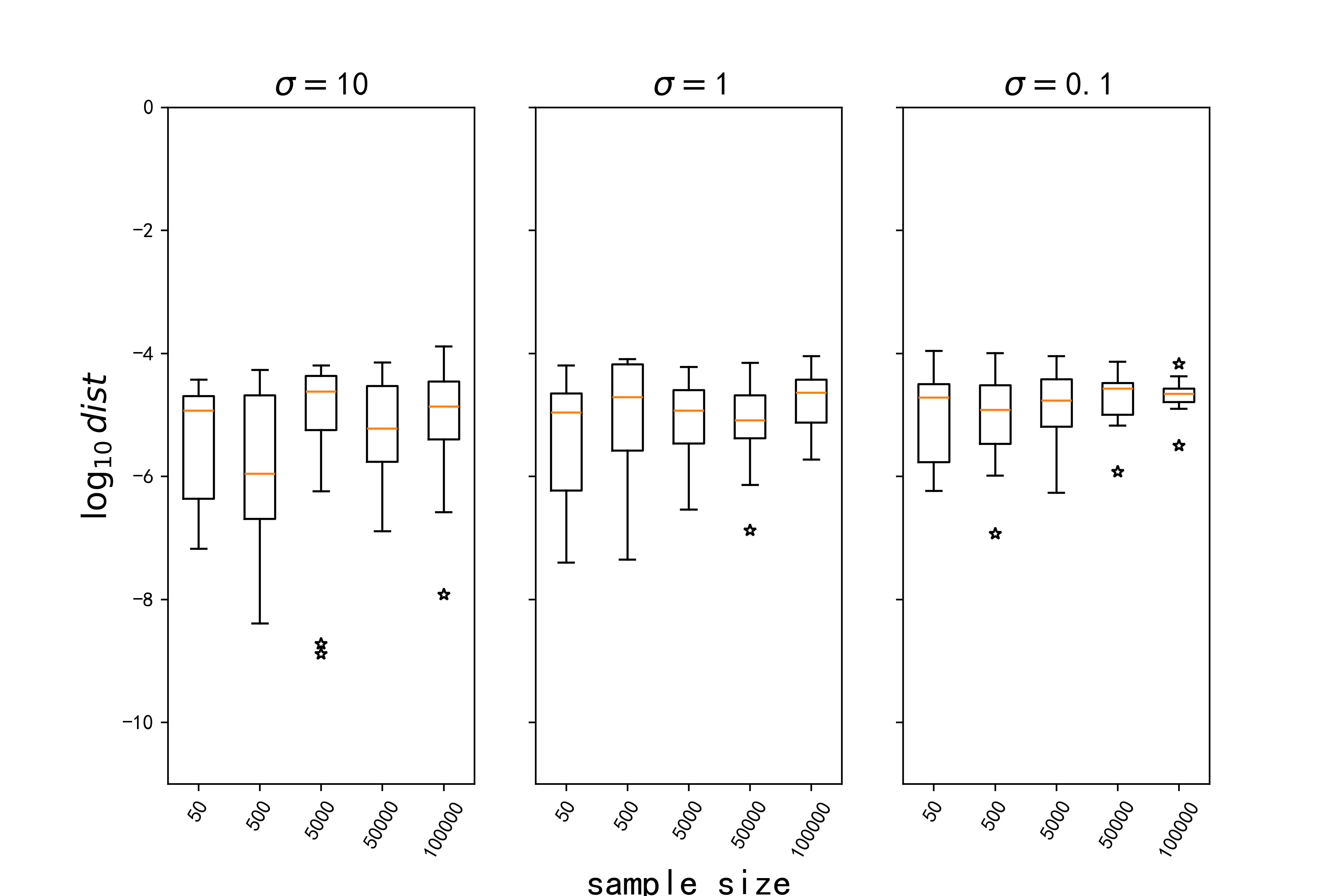}
    \caption{\scriptsize{Performance on S318 w.r.t after 1 500 iterations.}}
  \end{minipage}
\end{figure*}
\clearpage

\begin{figure*}[!h]
  \begin{minipage}[t]{0.5\linewidth}
    \centering
    \includegraphics[scale=0.35]{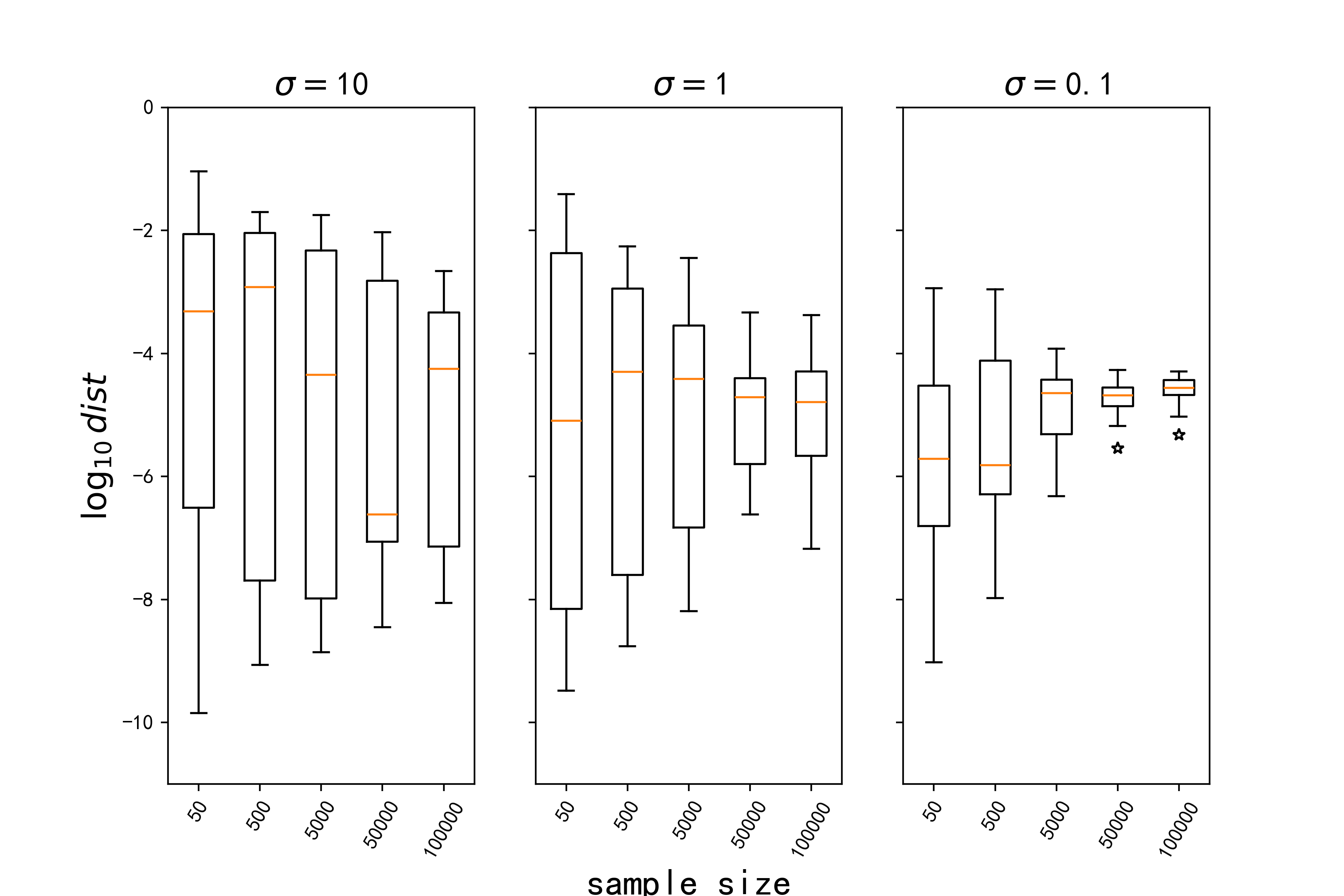}
    \caption{\scriptsize{Performance on S319 w.r.t after 50 iterations.}}
  \end{minipage}%
  \begin{minipage}[t]{0.5\linewidth}
    \centering
    \includegraphics[scale=0.35]{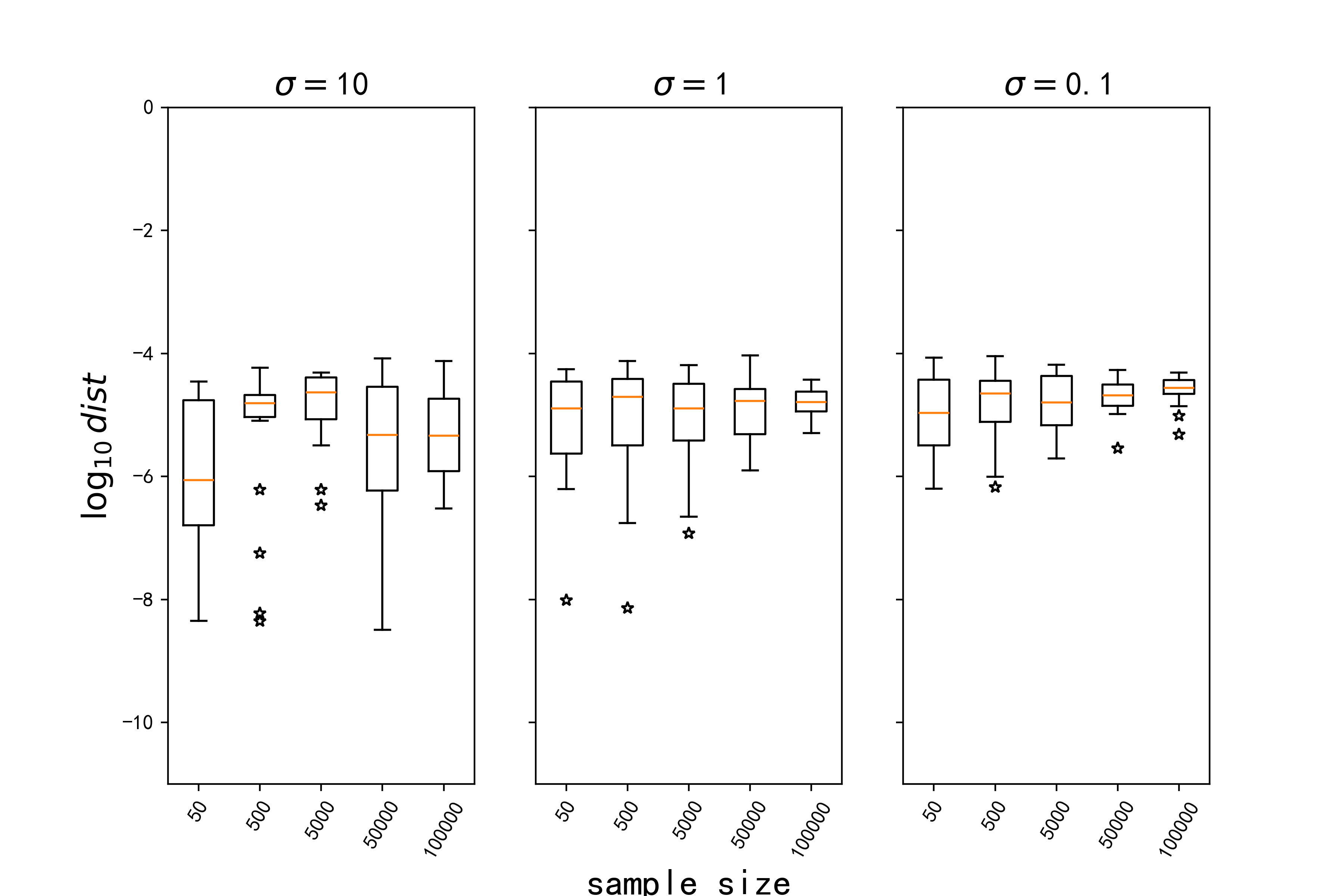}
    \caption{\scriptsize{Performance on S319 w.r.t after 1 500 iterations.}}
  \end{minipage}
\end{figure*}

\begin{figure*}[!h]
  \begin{minipage}[t]{0.5\linewidth}
    \centering
    \includegraphics[scale=0.35]{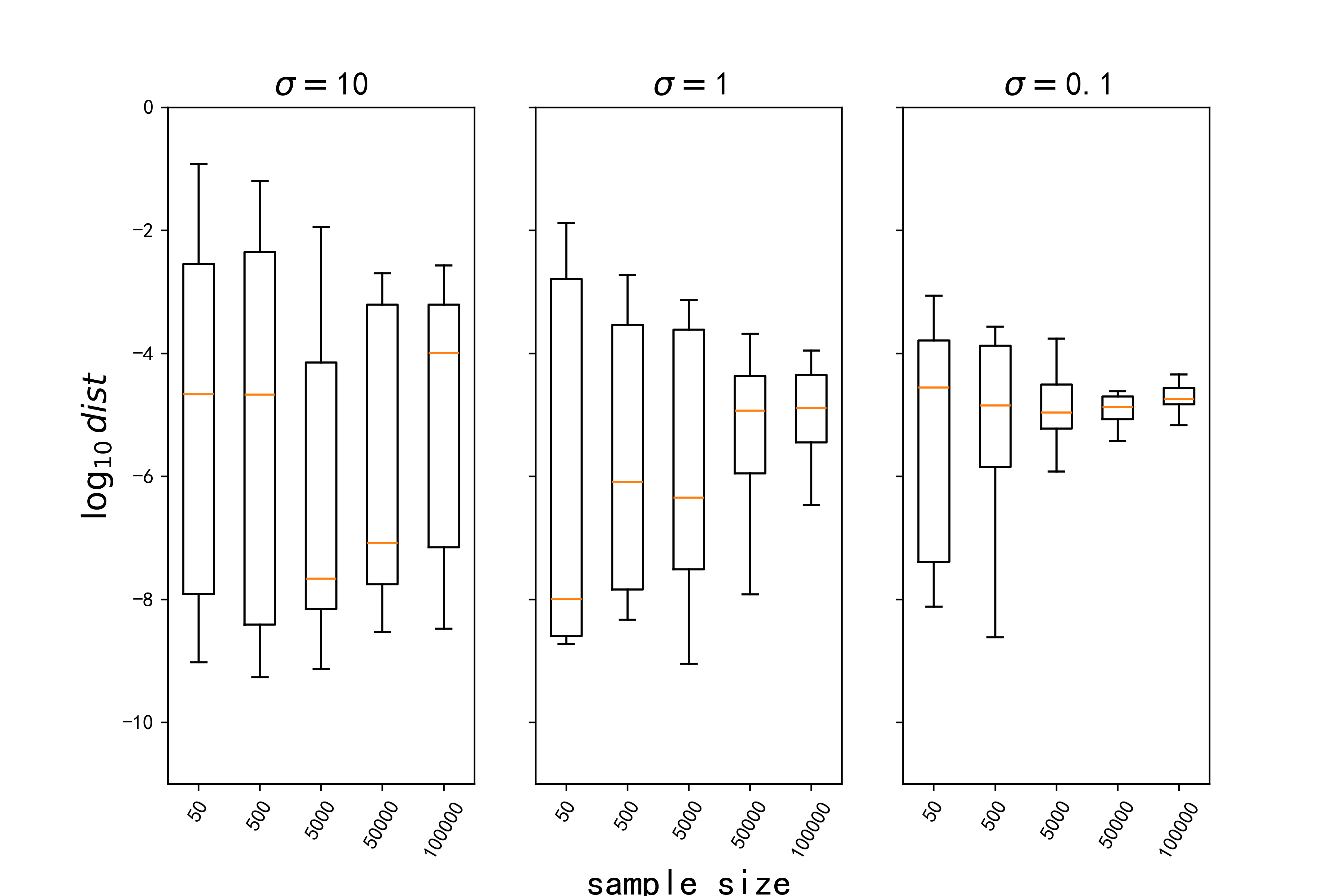}
    \caption{\scriptsize{Performance on S320 w.r.t after 50 iterations.}}
  \end{minipage}%
  \begin{minipage}[t]{0.5\linewidth}
    \centering
    \includegraphics[scale=0.35]{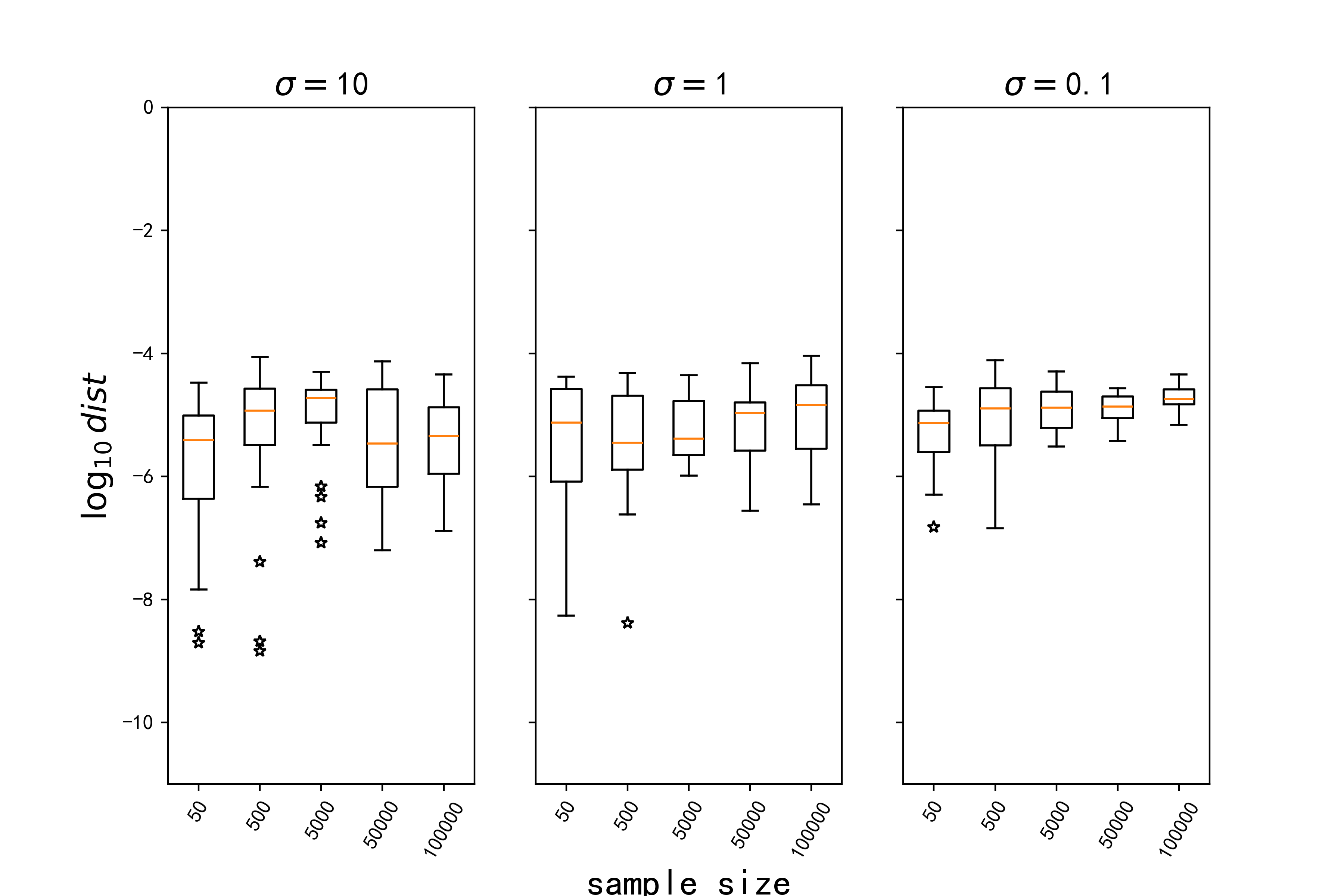}
    \caption{\scriptsize{Performance on S320 w.r.t after 1 500 iterations.}}
  \end{minipage}
\end{figure*}
\clearpage

\begin{figure*}[!h]
  \begin{minipage}[t]{0.5\linewidth}
    \centering
    \includegraphics[scale=0.35]{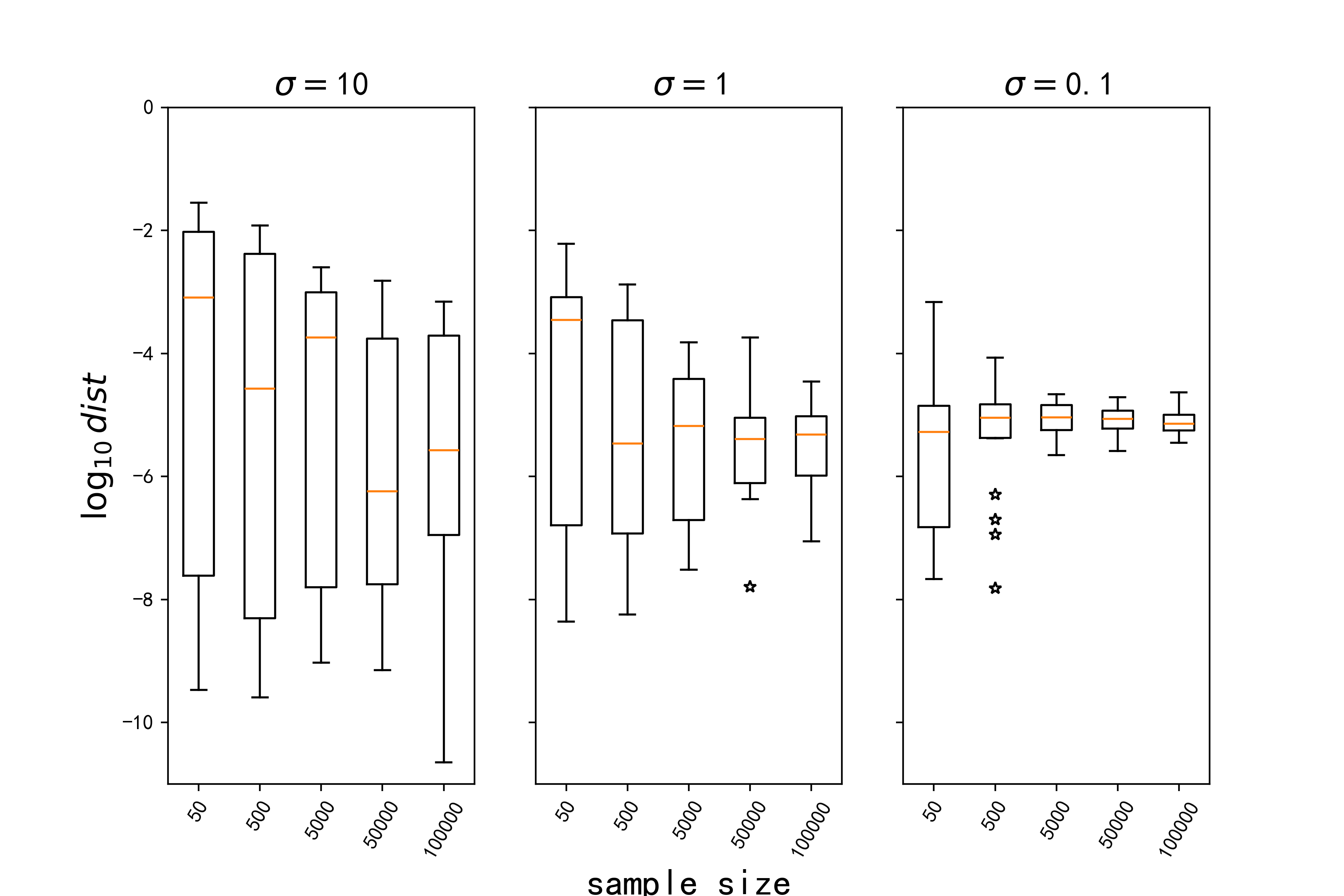}
    \caption{\scriptsize{Performance on S321 w.r.t after 50 iterations.}}
  \end{minipage}%
  \begin{minipage}[t]{0.5\linewidth}
    \centering
    \includegraphics[scale=0.35]{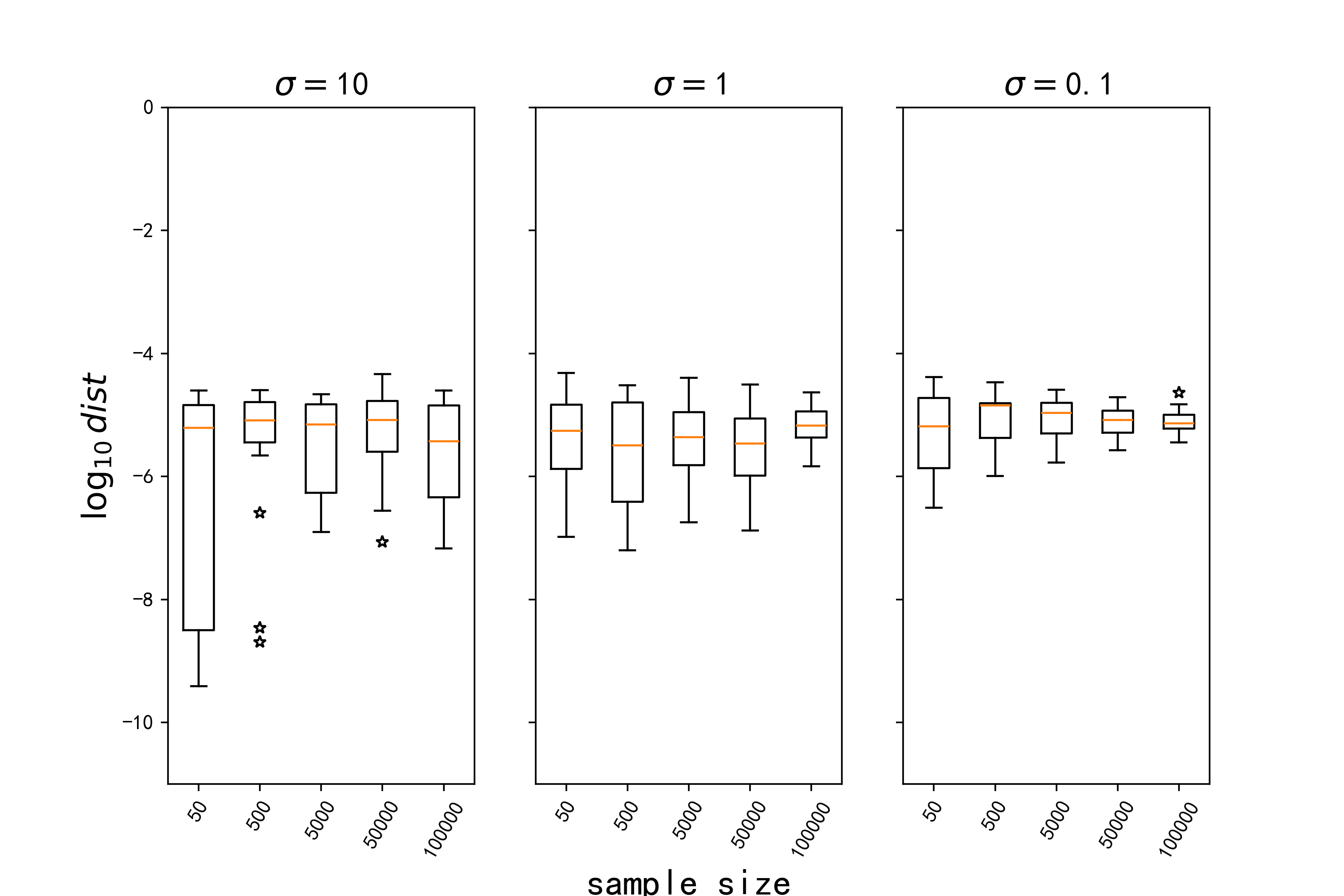}
    \caption{\scriptsize{Performance on S321 w.r.t after 1 500 iterations.}}
  \end{minipage}
\end{figure*}

\begin{figure*}[!h]
  \begin{minipage}[t]{0.5\linewidth}
    \centering
    \includegraphics[scale=0.35]{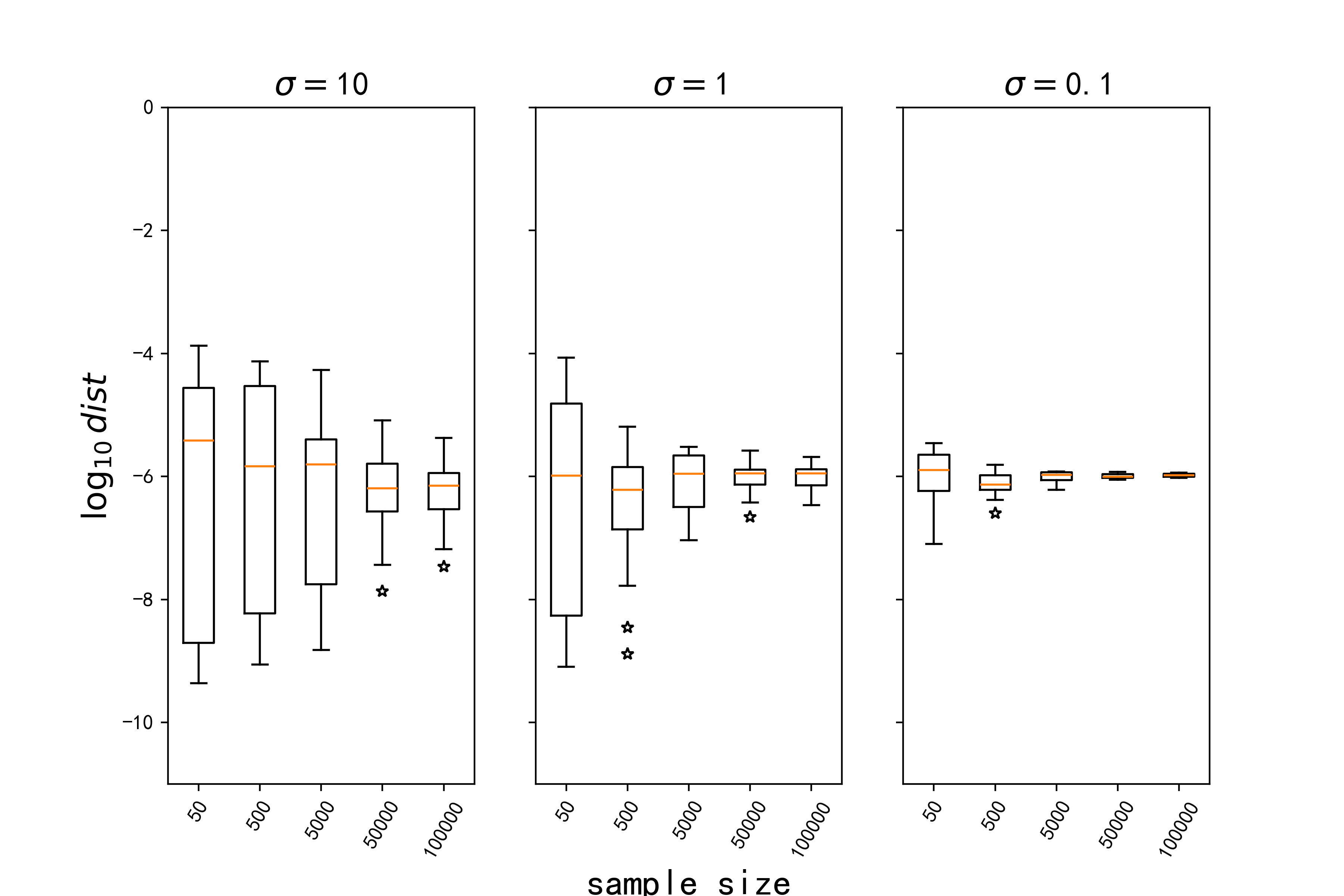}
    \caption{\scriptsize{Performance on S322 w.r.t after 50 iterations.}}
  \end{minipage}%
  \begin{minipage}[t]{0.5\linewidth}
    \centering
    \includegraphics[scale=0.35]{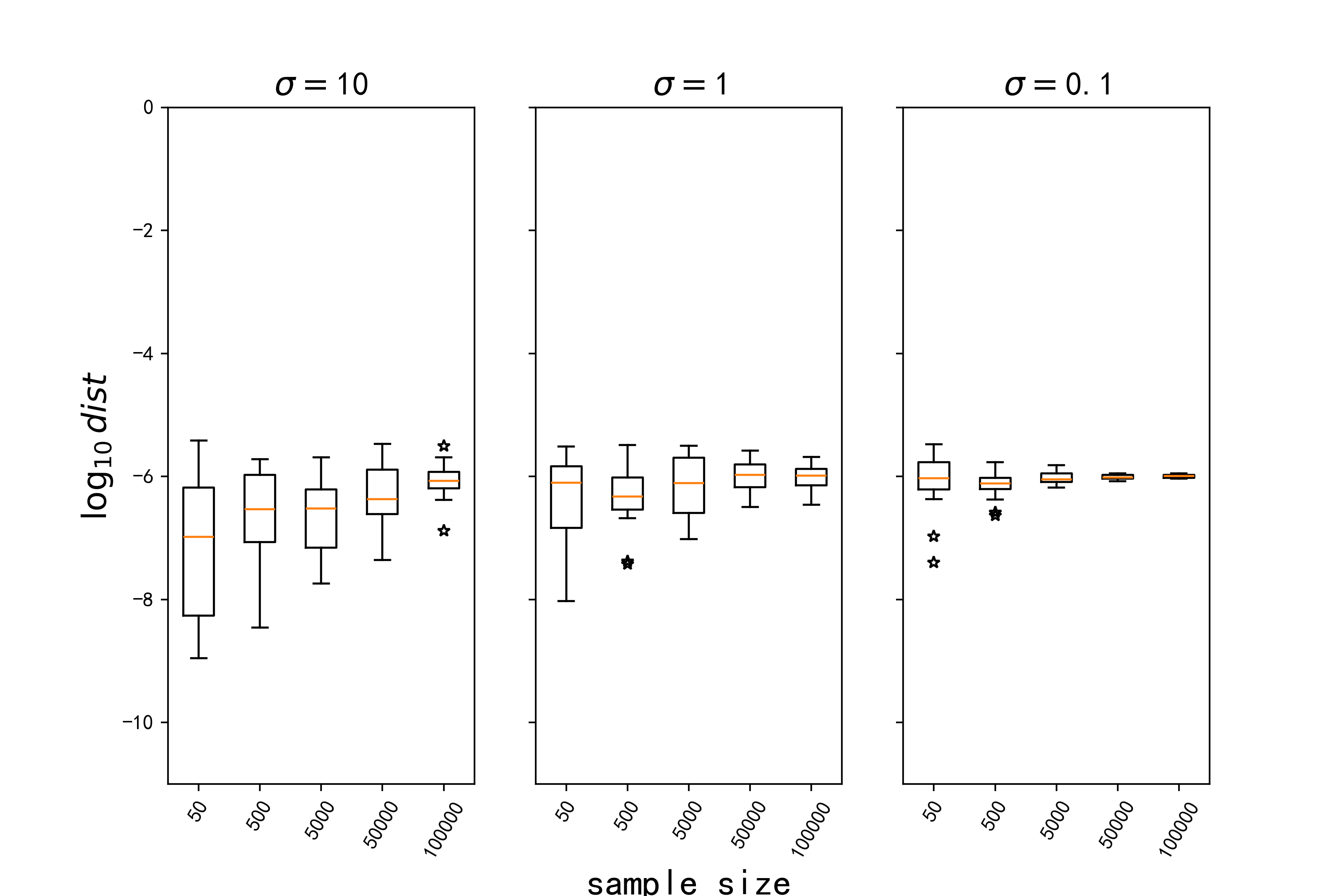}
    \caption{\scriptsize{Performance on S322 w.r.t after 1 500 iterations.}}
  \end{minipage}
\end{figure*}
\clearpage

\begin{figure*}[!h]
  \begin{minipage}[t]{0.5\linewidth}
    \centering
    \includegraphics[scale=0.35]{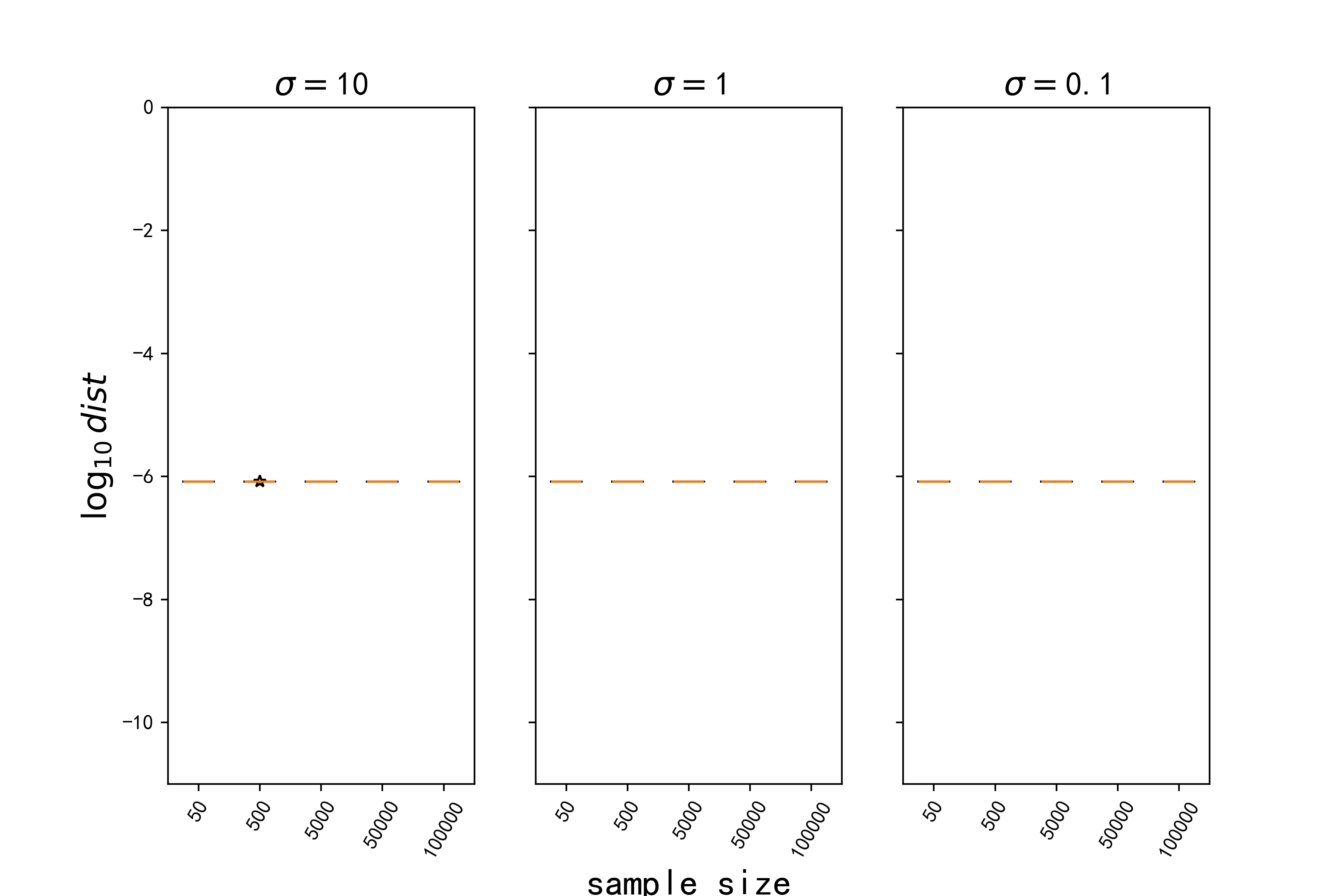}
    \caption{\scriptsize{Performance on S323 w.r.t after 50 iterations.}}
  \end{minipage}%
  \begin{minipage}[t]{0.5\linewidth}
    \centering
    \includegraphics[scale=0.35]{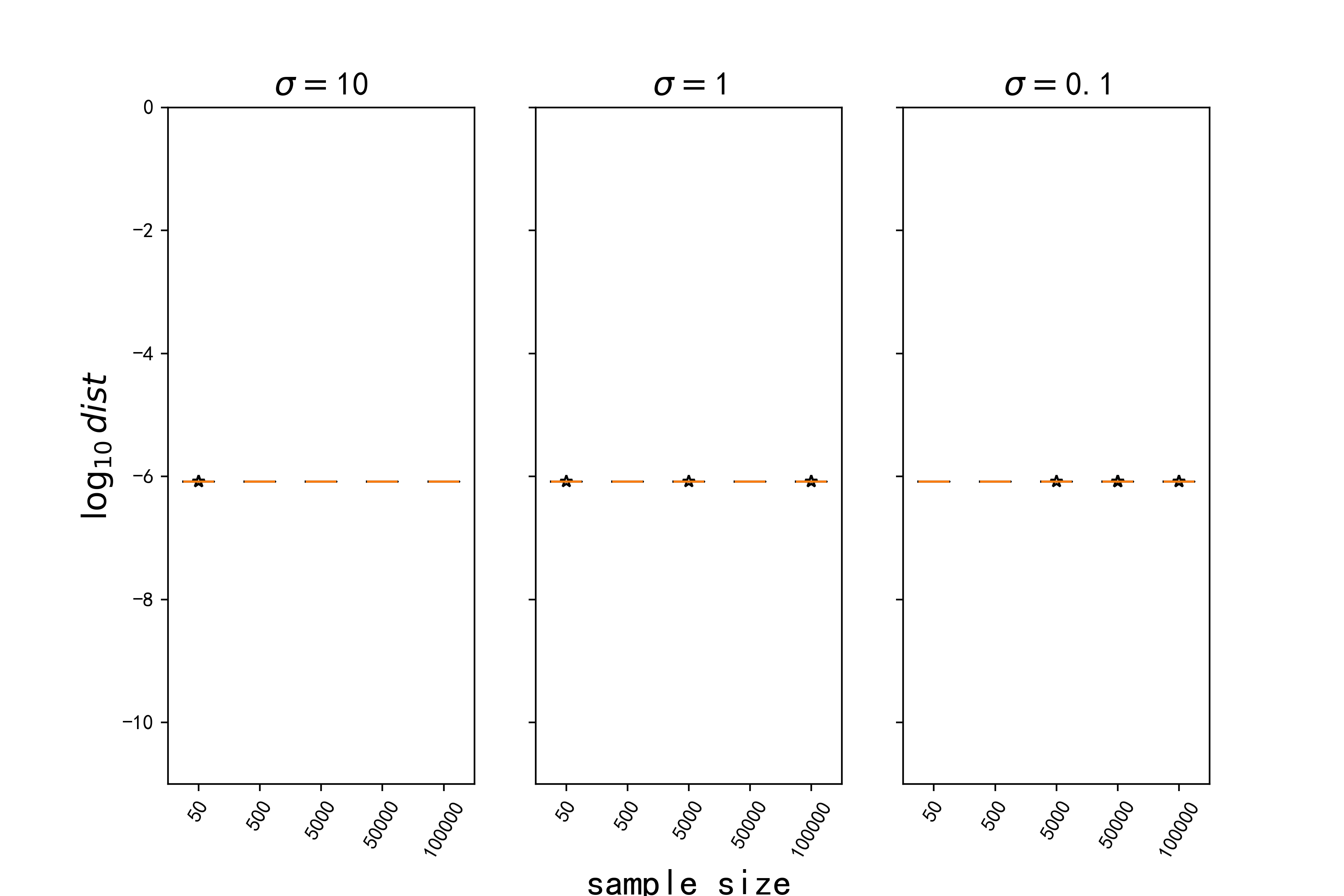}
    \caption{\scriptsize{Performance on S323 w.r.t after 1 500 iterations.}}
  \end{minipage}
  \end{figure*}

\begin{figure*}[!h]
  \begin{minipage}[t]{0.5\linewidth}
    \centering
    \includegraphics[scale=0.35]{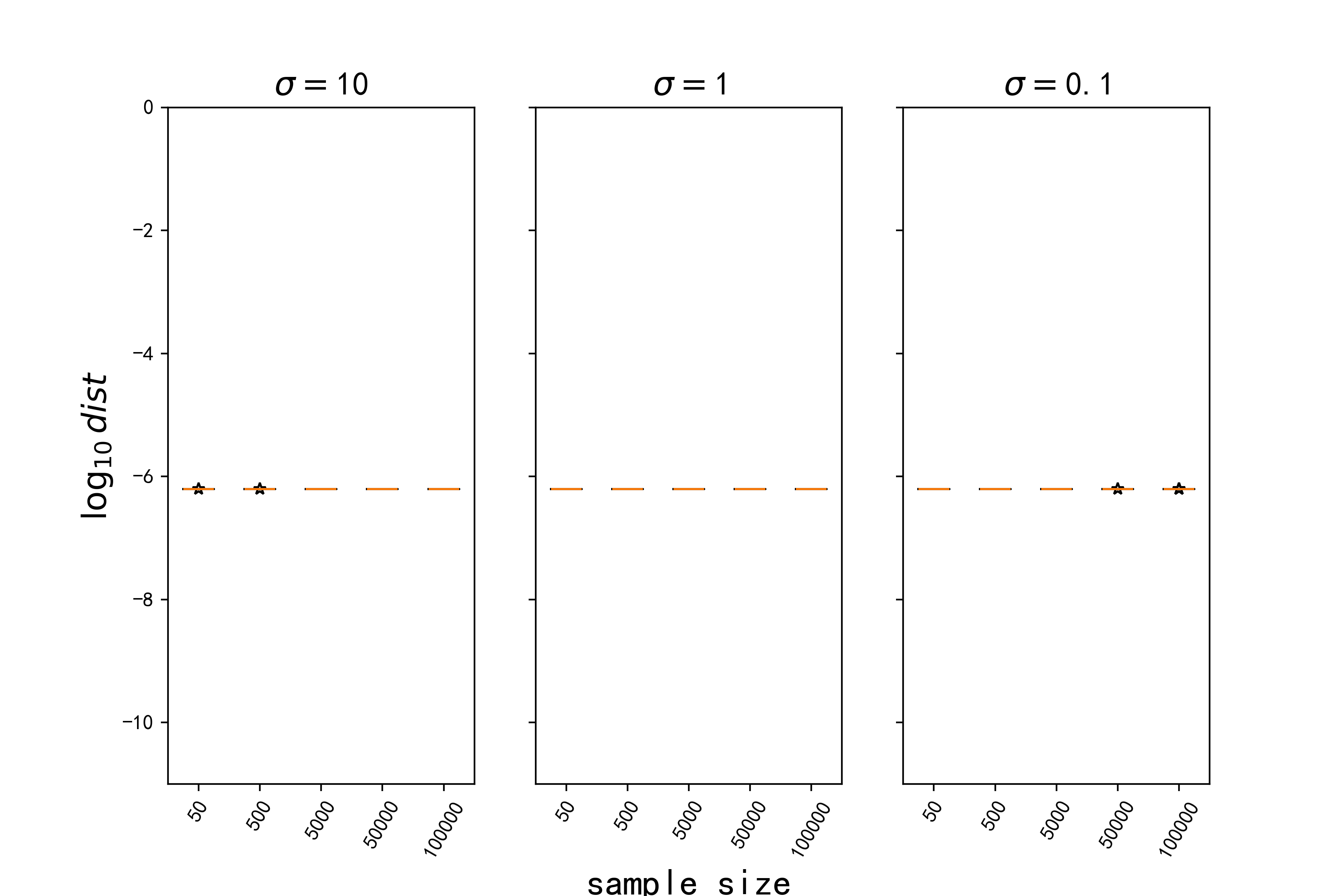}
    \caption{\scriptsize{Performance on S326 w.r.t after 50 iterations.}}
  \end{minipage}%
  \begin{minipage}[t]{0.5\linewidth}
    \centering
    \includegraphics[scale=0.35]{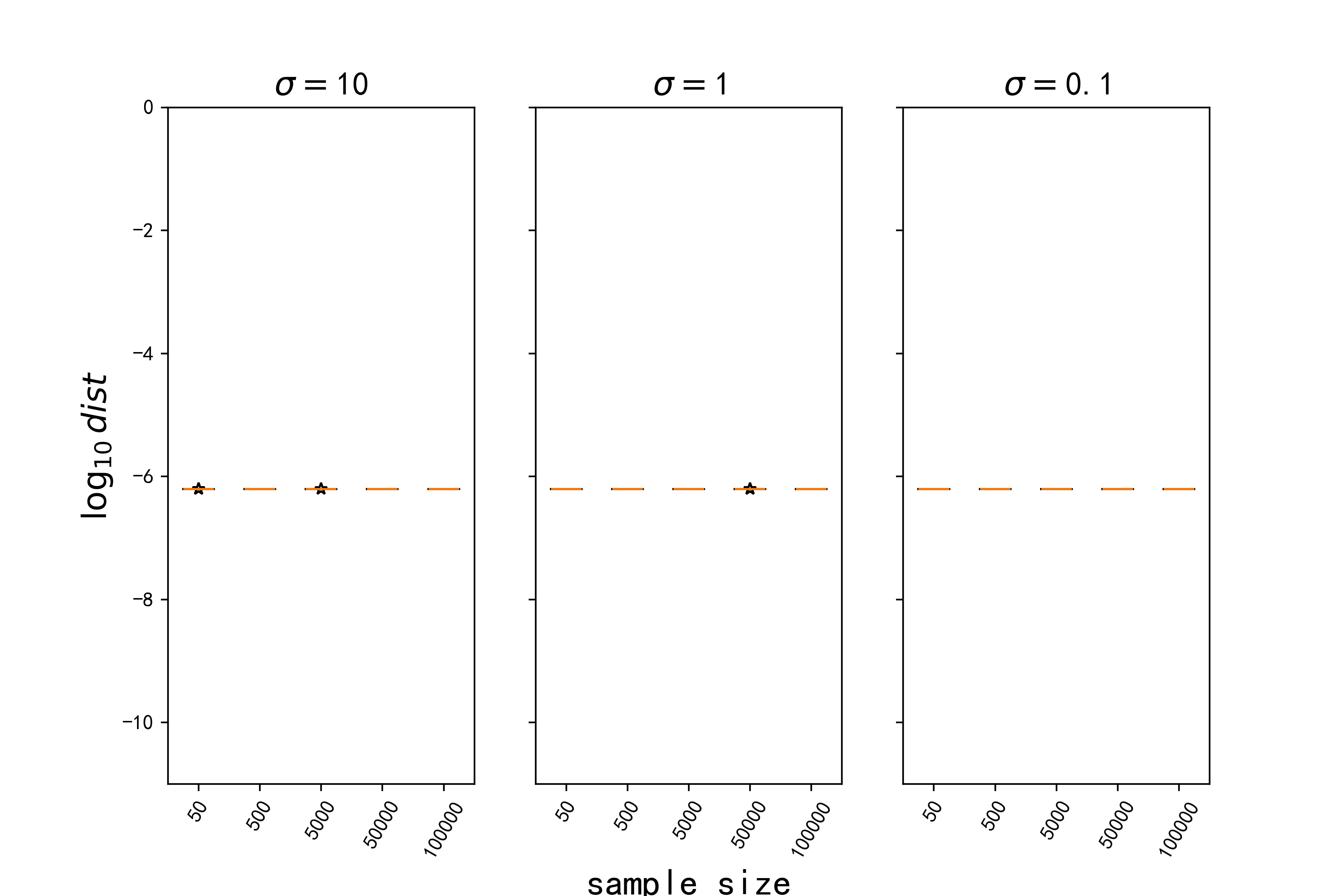}
    \caption{\scriptsize{Performance on S326 w.r.t after 1 500 iterations.}}
  \end{minipage}
\end{figure*}
\clearpage

\begin{figure*}[!h]
  \begin{minipage}[t]{0.5\linewidth}
    \centering
    \includegraphics[scale=0.35]{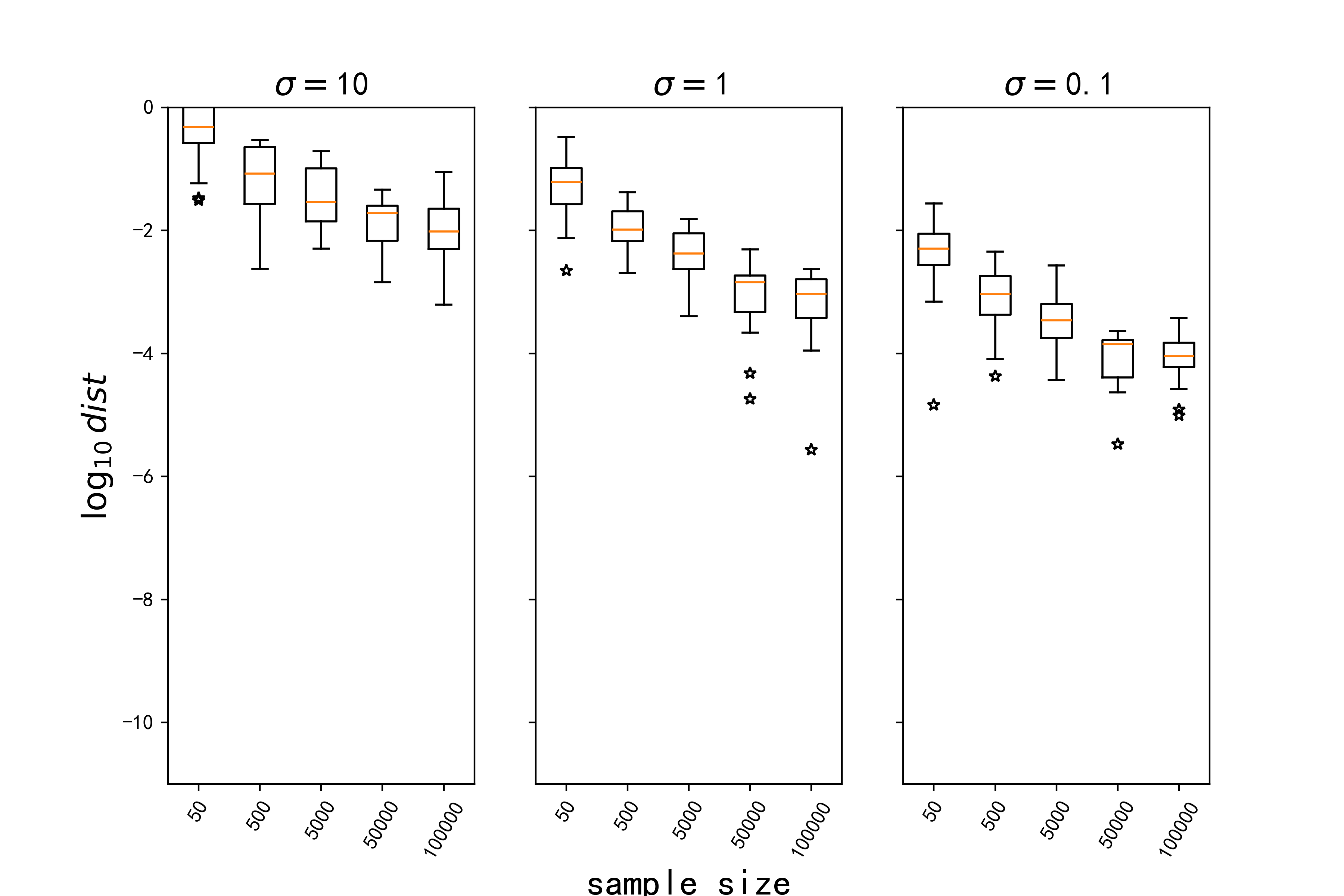}
    \caption{\scriptsize{Performance on S337 w.r.t after 50 iterations.}}
  \end{minipage}%
  \begin{minipage}[t]{0.5\linewidth}
    \centering
    \includegraphics[scale=0.35]{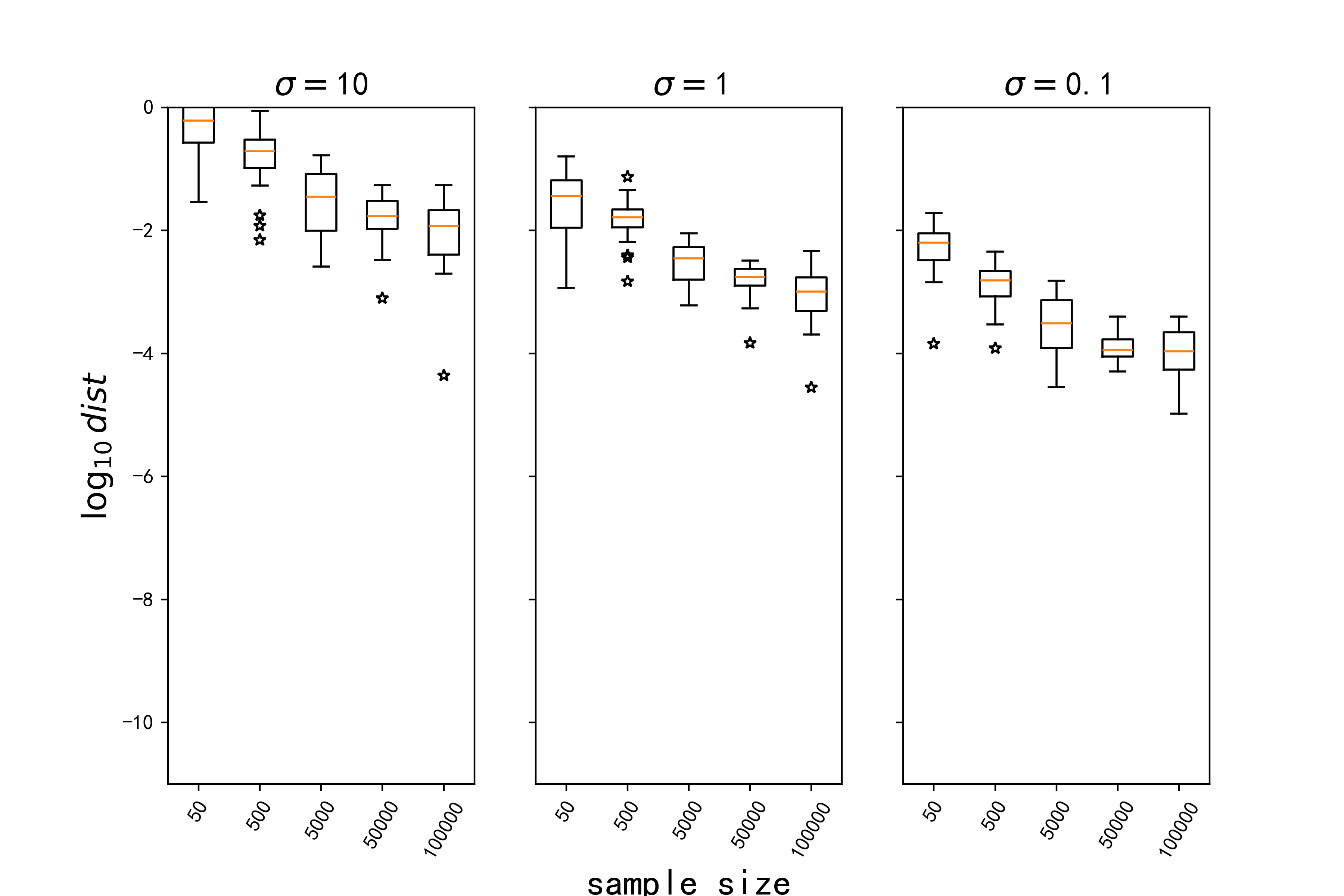}
    \caption{\scriptsize{Performance on S337 w.r.t after 1 500 iterations.}}
  \end{minipage}
\end{figure*}

\begin{figure*}[!h]
  \begin{minipage}[t]{0.5\linewidth}
    \centering
    \includegraphics[scale=0.35]{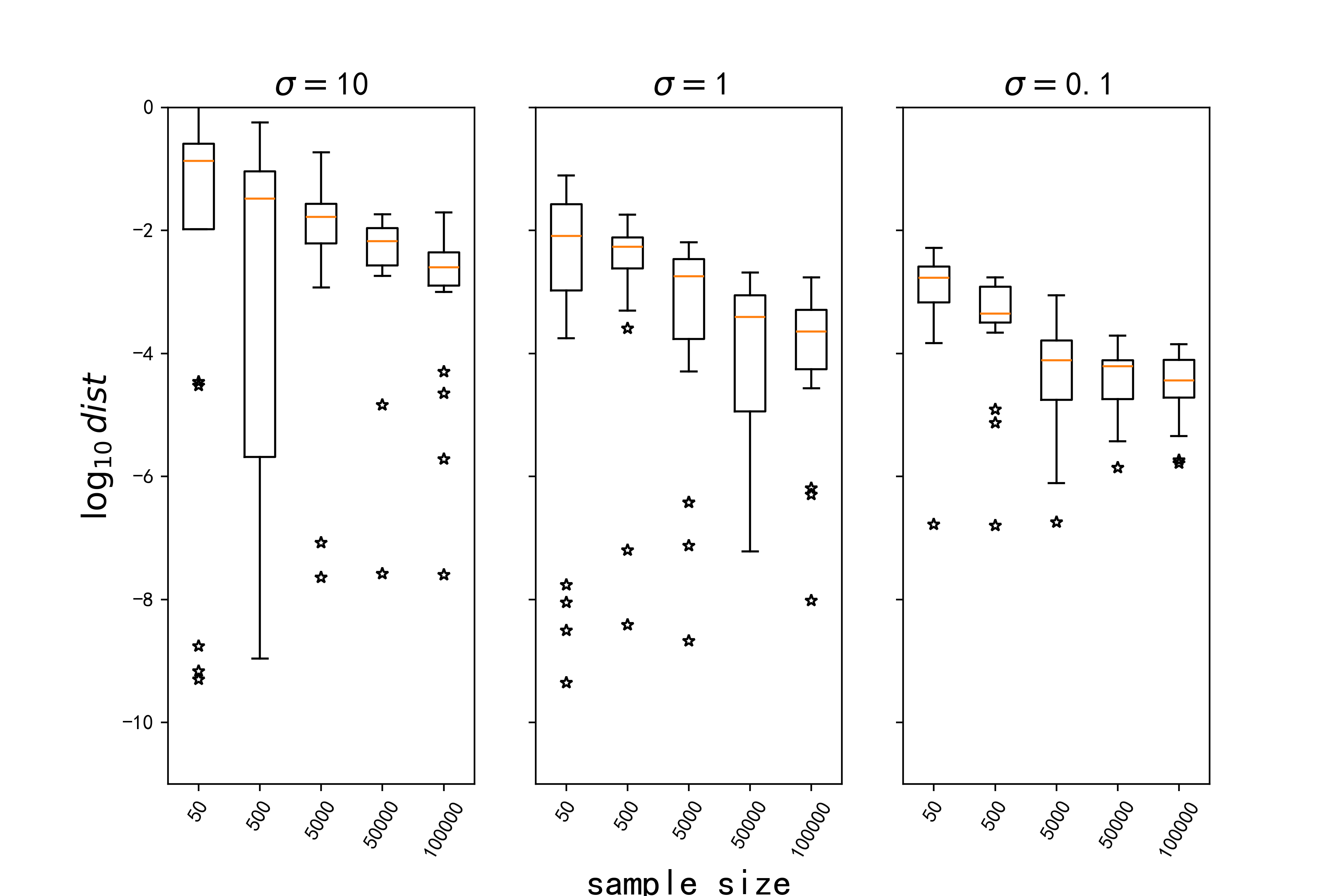}
    \caption{\scriptsize{Performance on S338 w.r.t after 50 iterations.}}
  \end{minipage}%
  \begin{minipage}[t]{0.5\linewidth}
    \centering
    \includegraphics[scale=0.35]{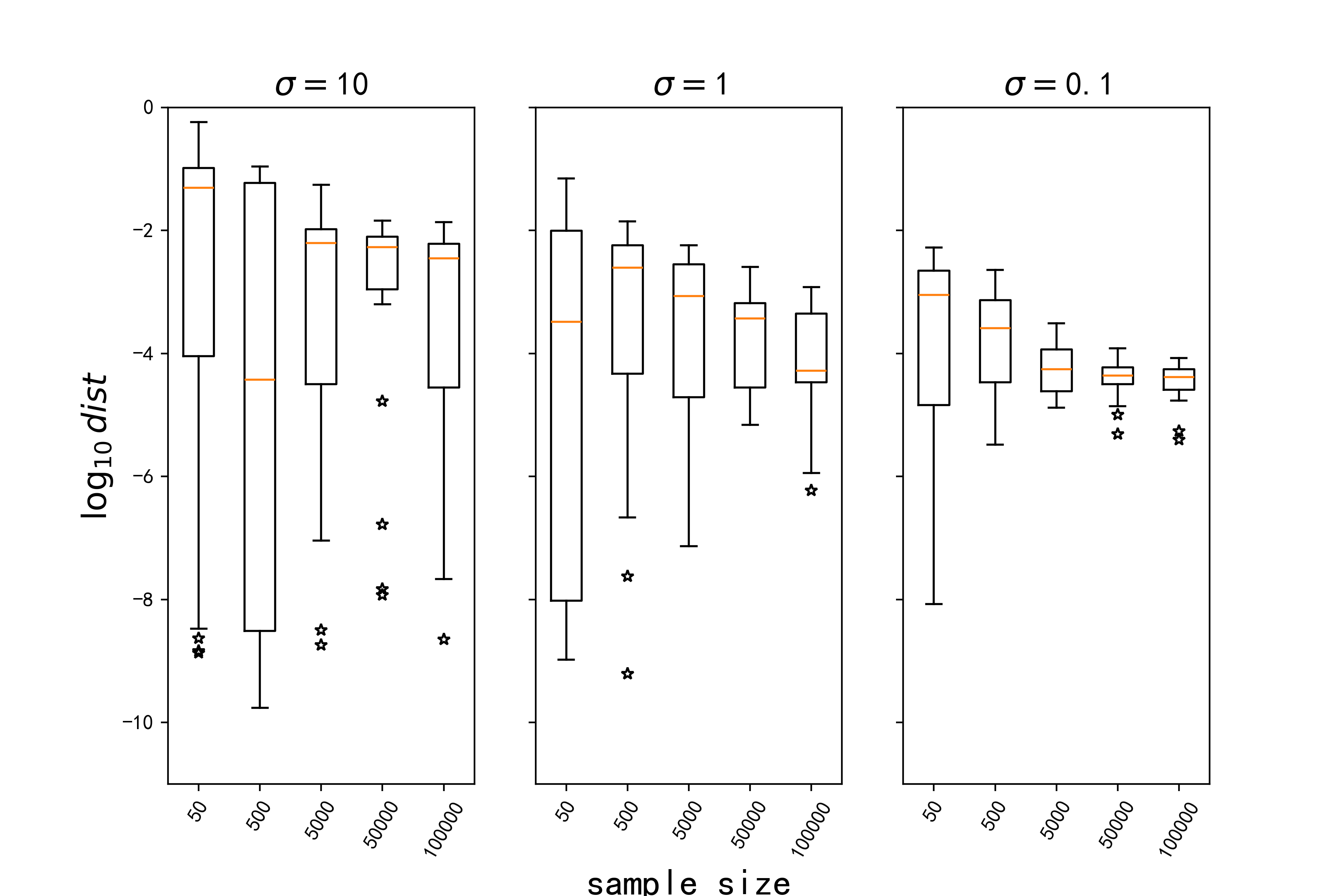}
    \caption{\scriptsize{Performance on S338 w.r.t after 1 500 iterations.}}
  \end{minipage}
\end{figure*}
\clearpage

\begin{figure*}[!h]
  \begin{minipage}[t]{0.5\linewidth}
    \centering
    \includegraphics[scale=0.35]{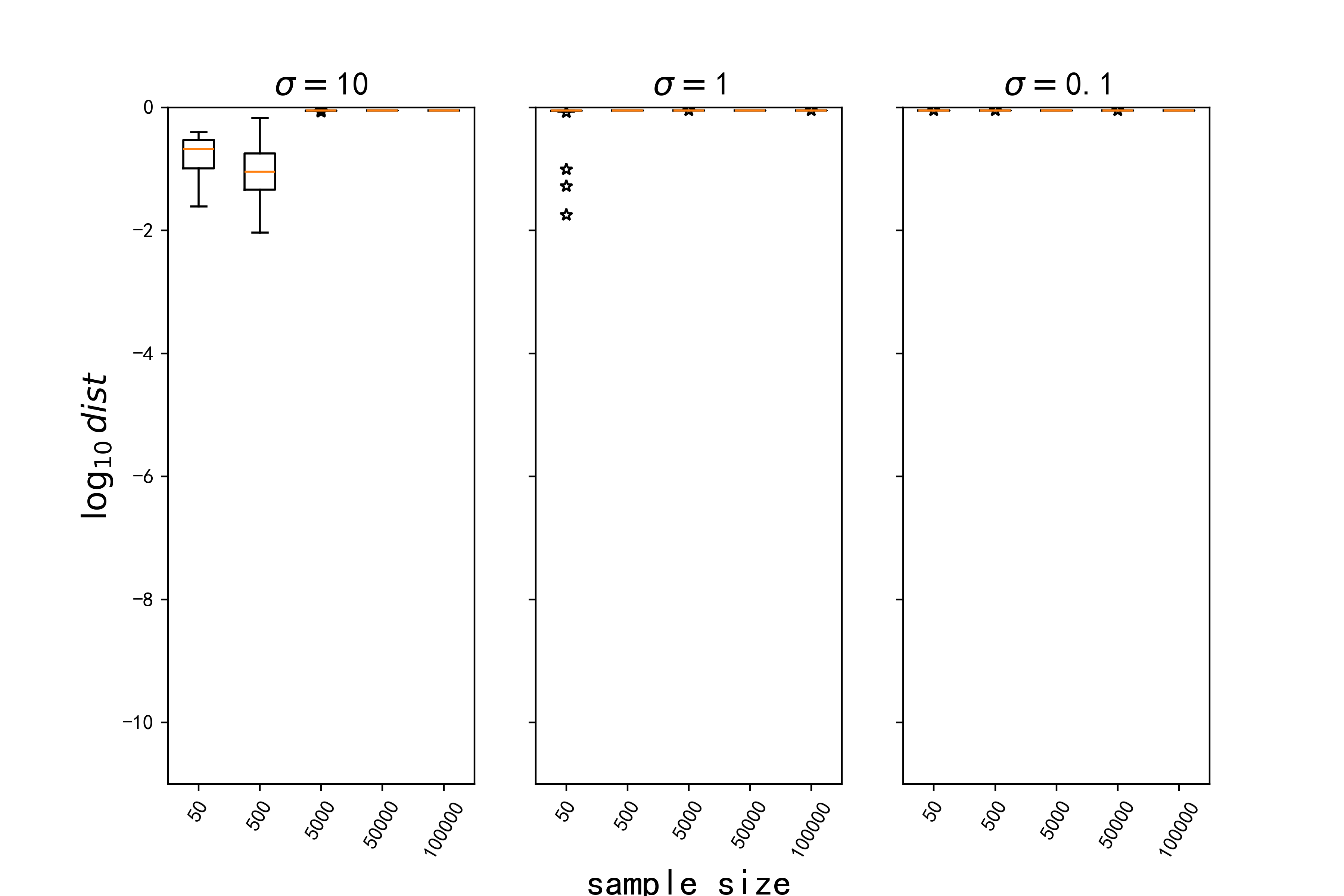}
    \caption{\scriptsize{Performance on S344 w.r.t after 50 iterations.}}
  \end{minipage}%
  \begin{minipage}[t]{0.5\linewidth}
    \centering
    \includegraphics[scale=0.35]{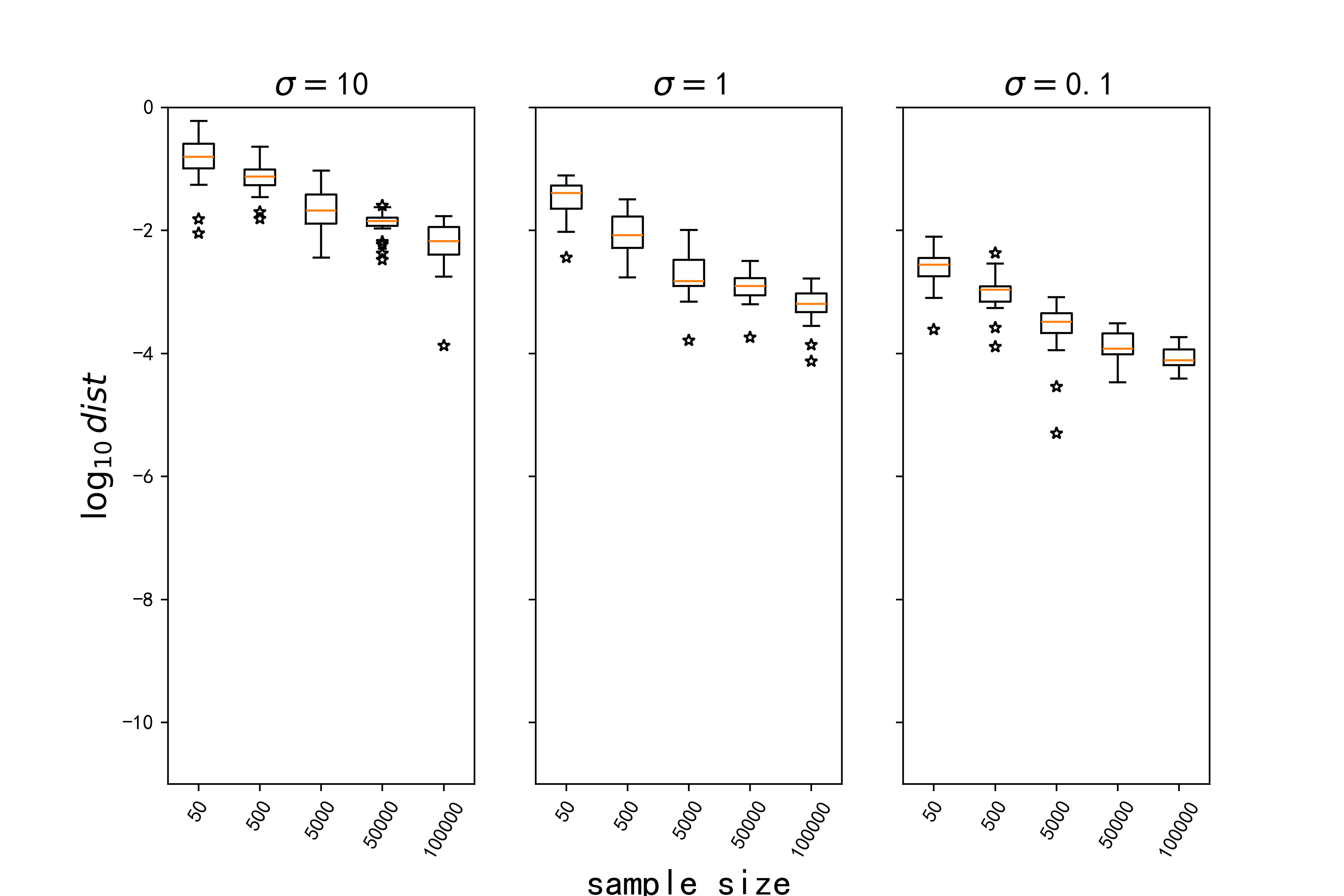}
    \caption{\scriptsize{Performance on S344 w.r.t after 1 500 iterations.}}
  \end{minipage}
\end{figure*}

\begin{figure*}[!h]
  \begin{minipage}[t]{0.5\linewidth}
    \centering
    \includegraphics[scale=0.35]{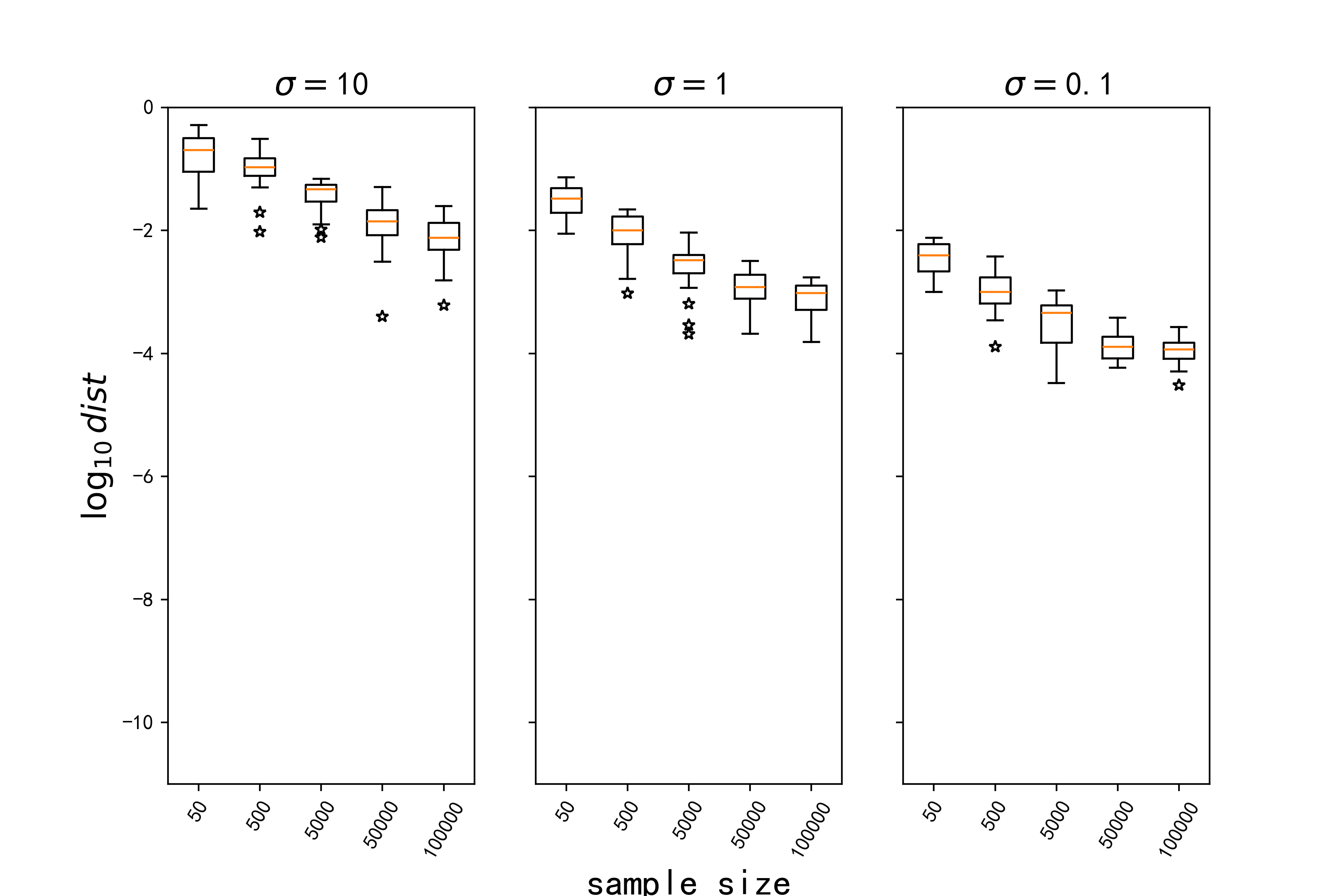}
    \caption{\scriptsize{Performance on S345 w.r.t after 50 iterations.}}
  \end{minipage}%
  \begin{minipage}[t]{0.5\linewidth}
    \centering
    \includegraphics[scale=0.35]{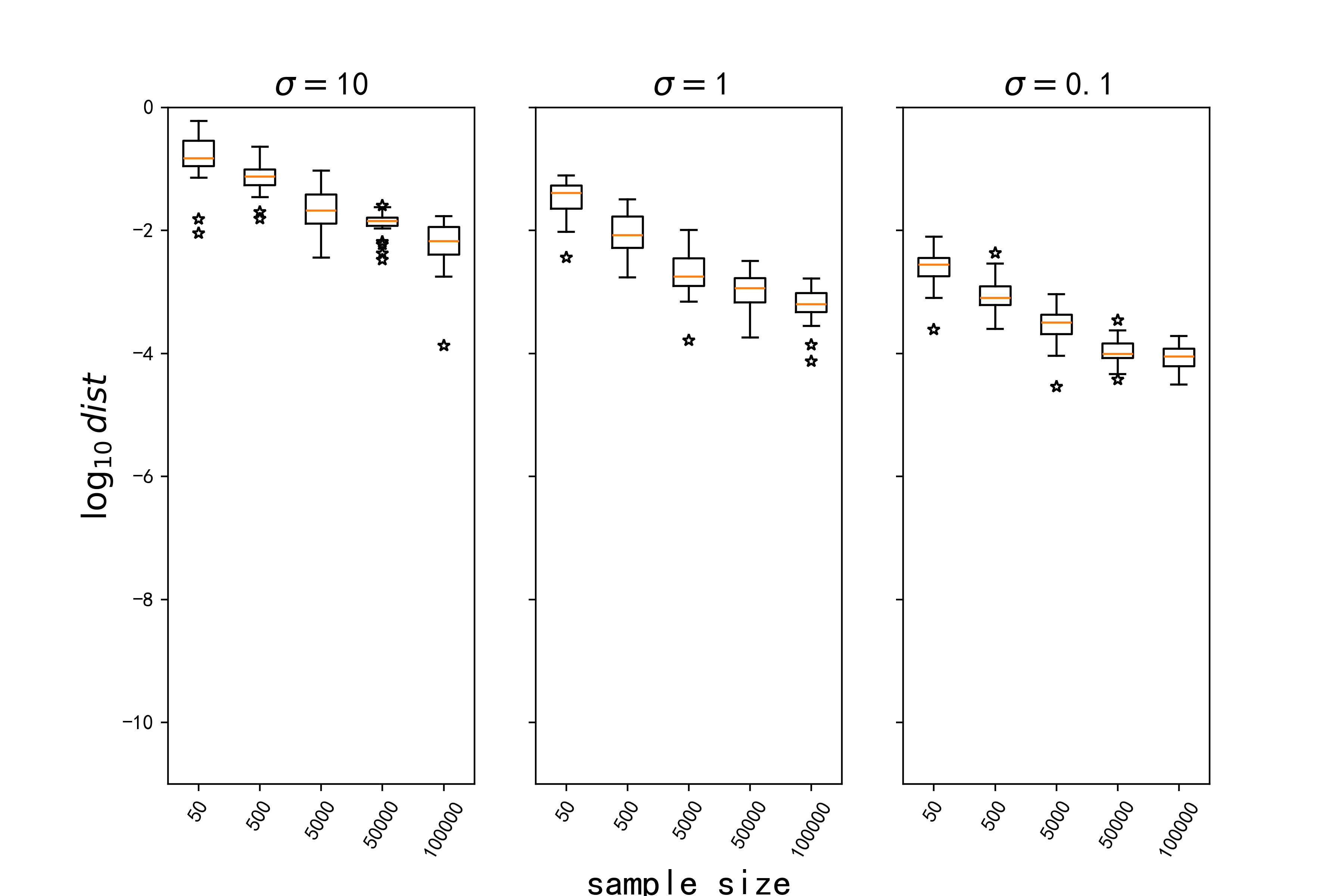}
    \caption{\scriptsize{Performance on S345 w.r.t after 1 500 iterations.}}
  \end{minipage}
\end{figure*}
\clearpage

\begin{figure*}[!h]
  \begin{minipage}[t]{0.5\linewidth}
    \centering
    \includegraphics[scale=0.35]{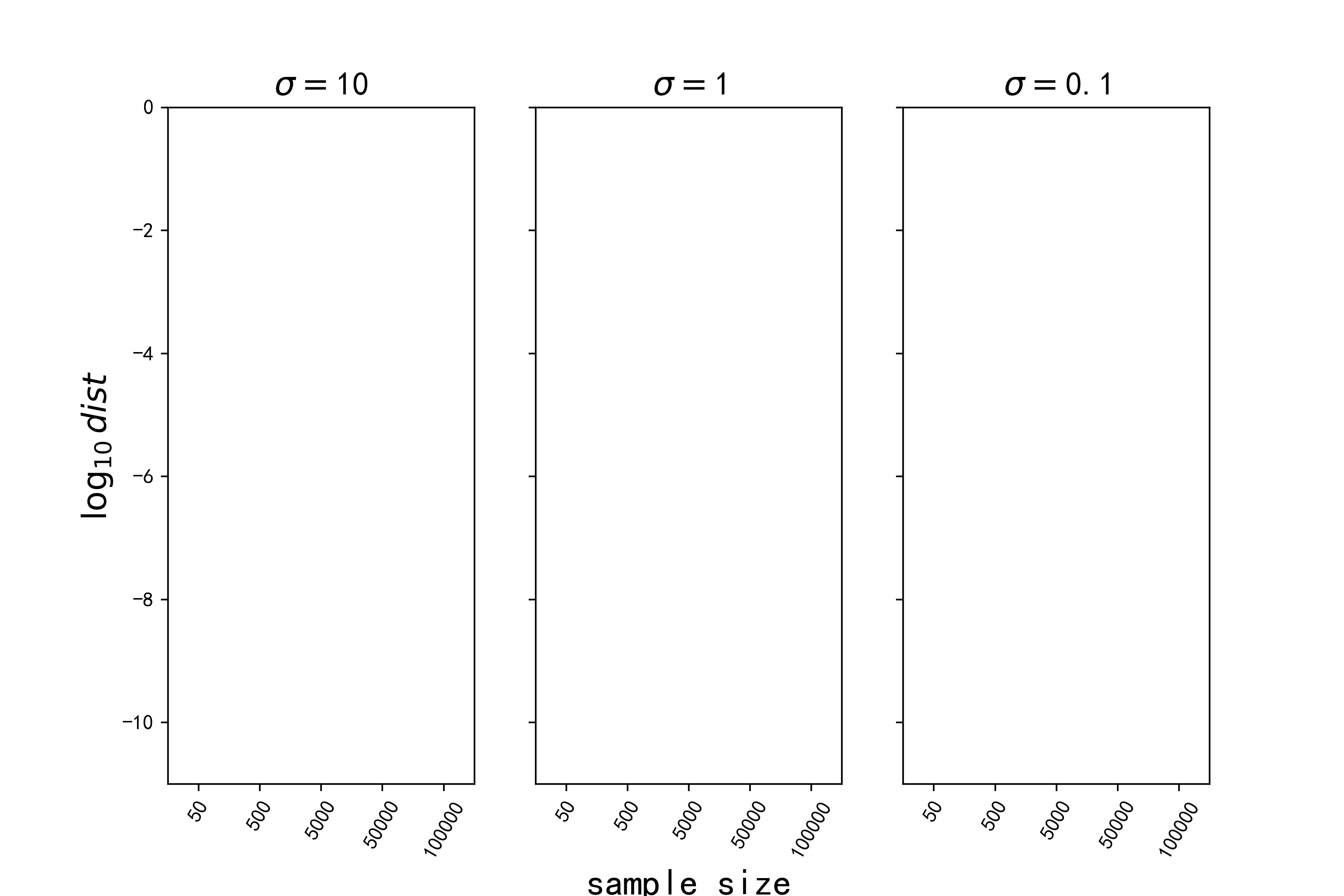}
    \caption{\scriptsize{Performance on S355 w.r.t after 50 iterations.}}
  \end{minipage}%
  \begin{minipage}[t]{0.5\linewidth}
    \centering
    \includegraphics[scale=0.35]{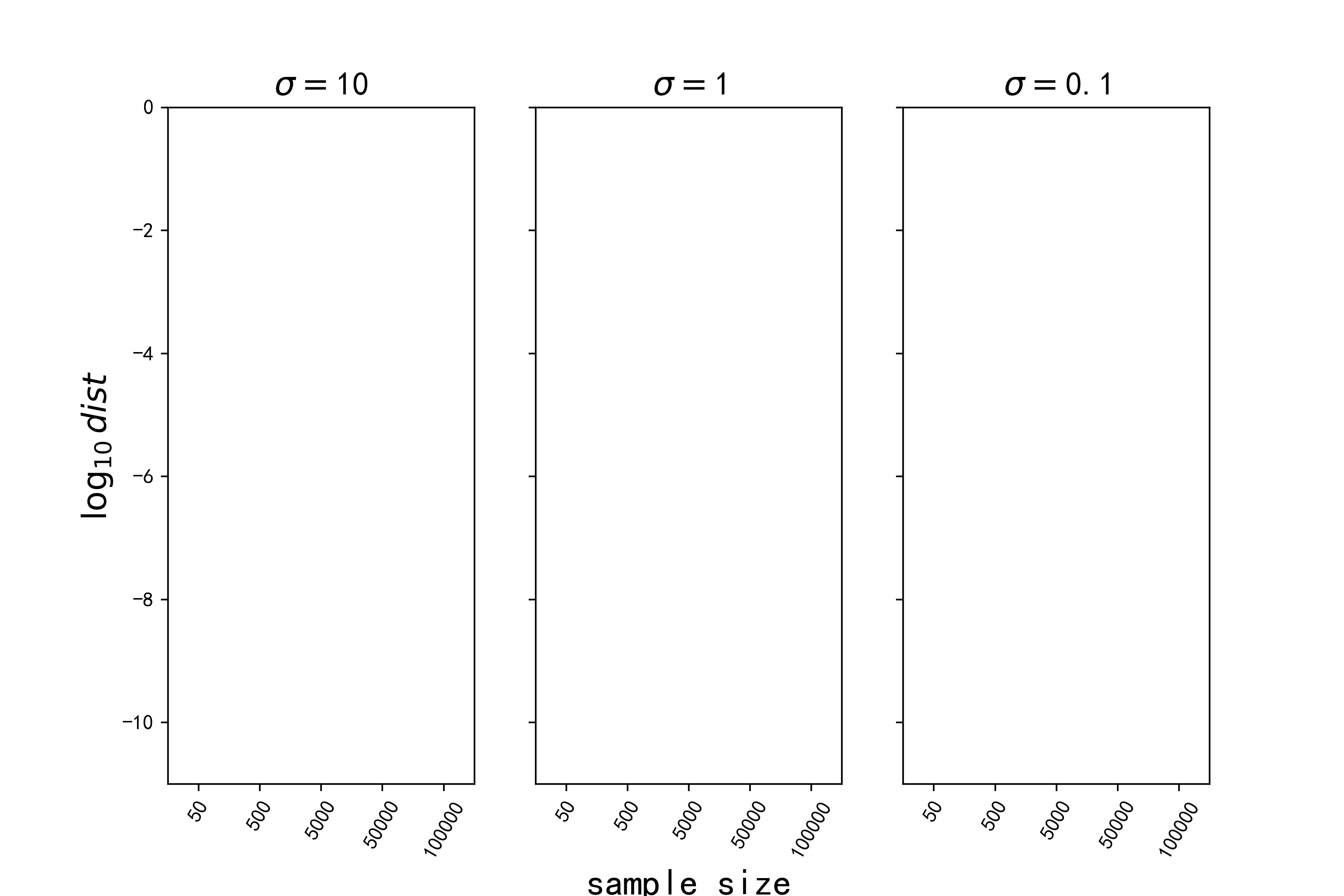}
    \caption{\scriptsize{Performance on S355 w.r.t after 1 500 iterations.}}
  \end{minipage}
\end{figure*}

\begin{figure*}[!h]
  \begin{minipage}[t]{0.5\linewidth}
    \centering
    \includegraphics[scale=0.35]{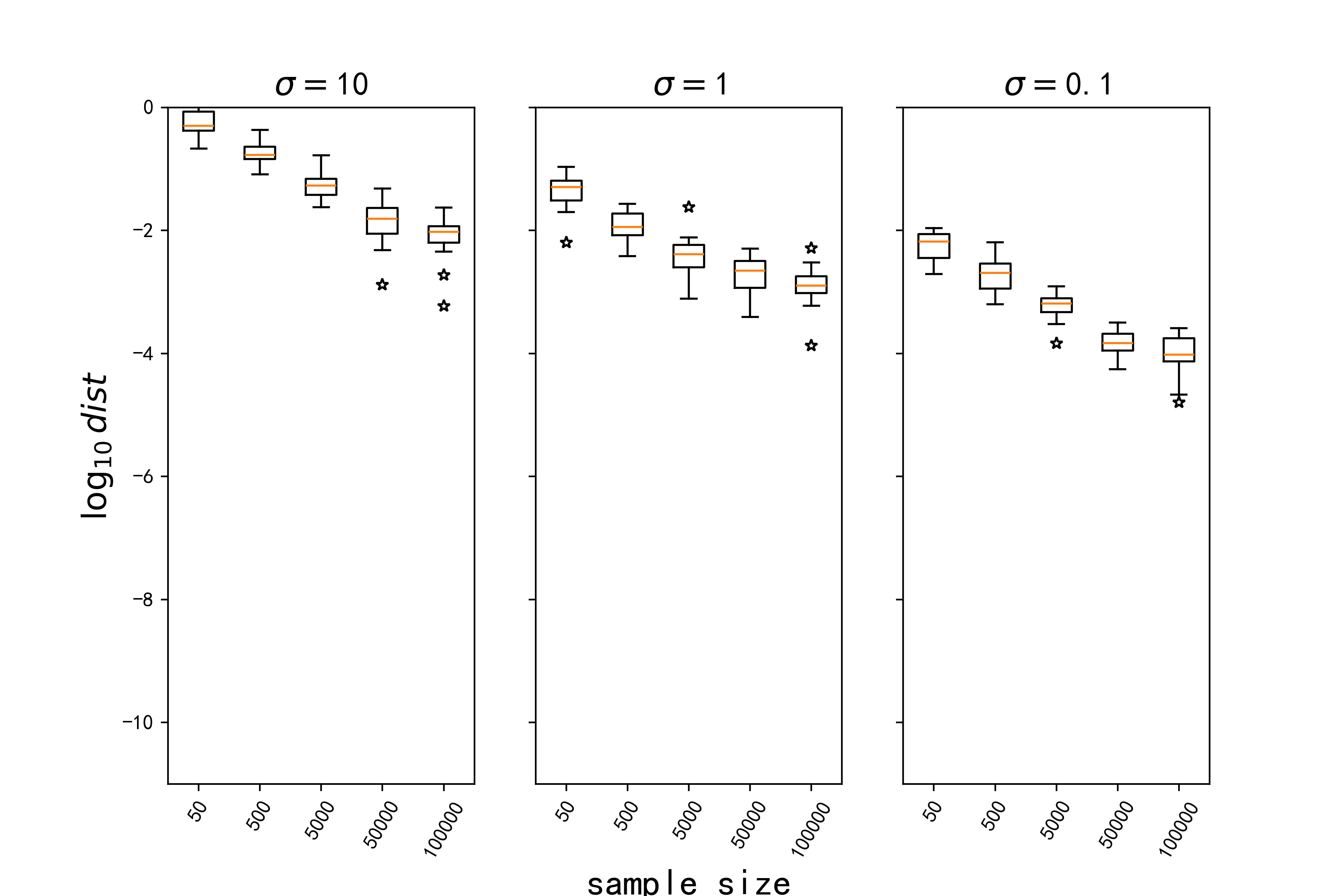}
    \caption{\scriptsize{Performance on S372 w.r.t after 50 iterations.}}
  \end{minipage}%
  \begin{minipage}[t]{0.5\linewidth}
    \centering
    \includegraphics[scale=0.35]{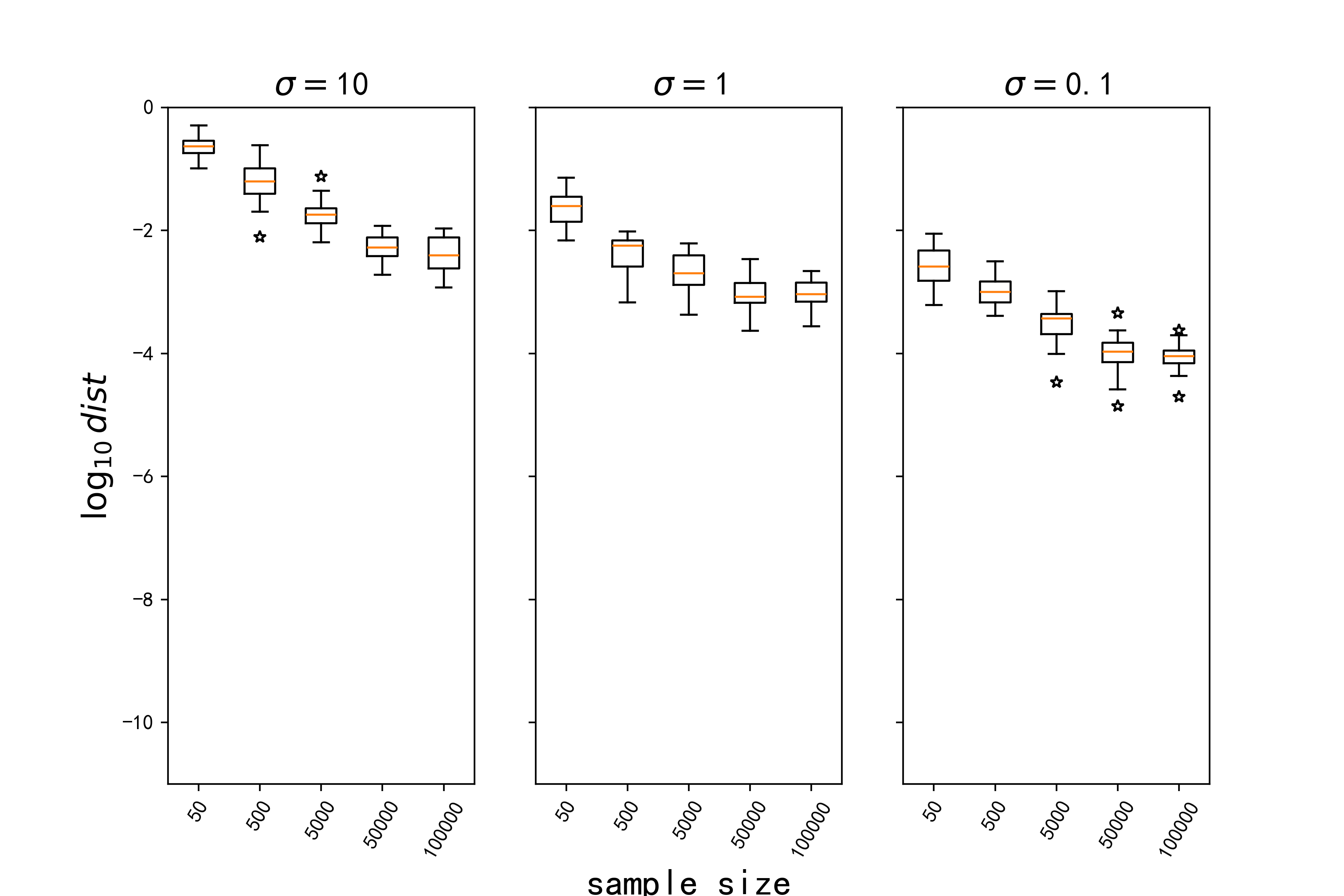}
    \caption{\scriptsize{Performance on S372 w.r.t after 1 500 iterations.}}
  \end{minipage}
\end{figure*}
\clearpage

\begin{figure*}[!h]
  \begin{minipage}[t]{0.5\linewidth}
    \centering
    \includegraphics[scale=0.35]{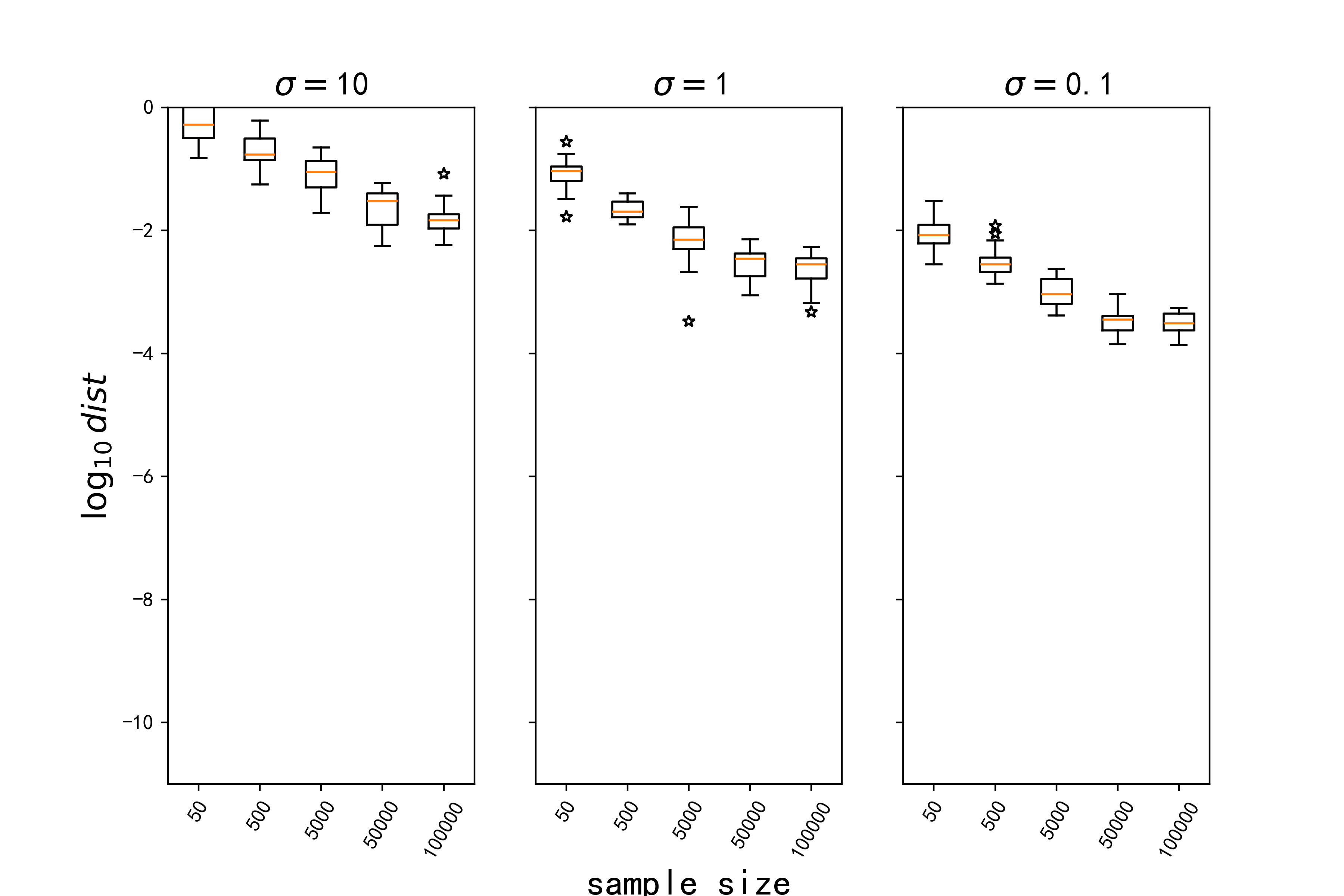}
    \caption{\scriptsize{Performance on S373 w.r.t after 50 iterations.}}
  \end{minipage}%
  \begin{minipage}[t]{0.5\linewidth}
    \centering
    \includegraphics[scale=0.35]{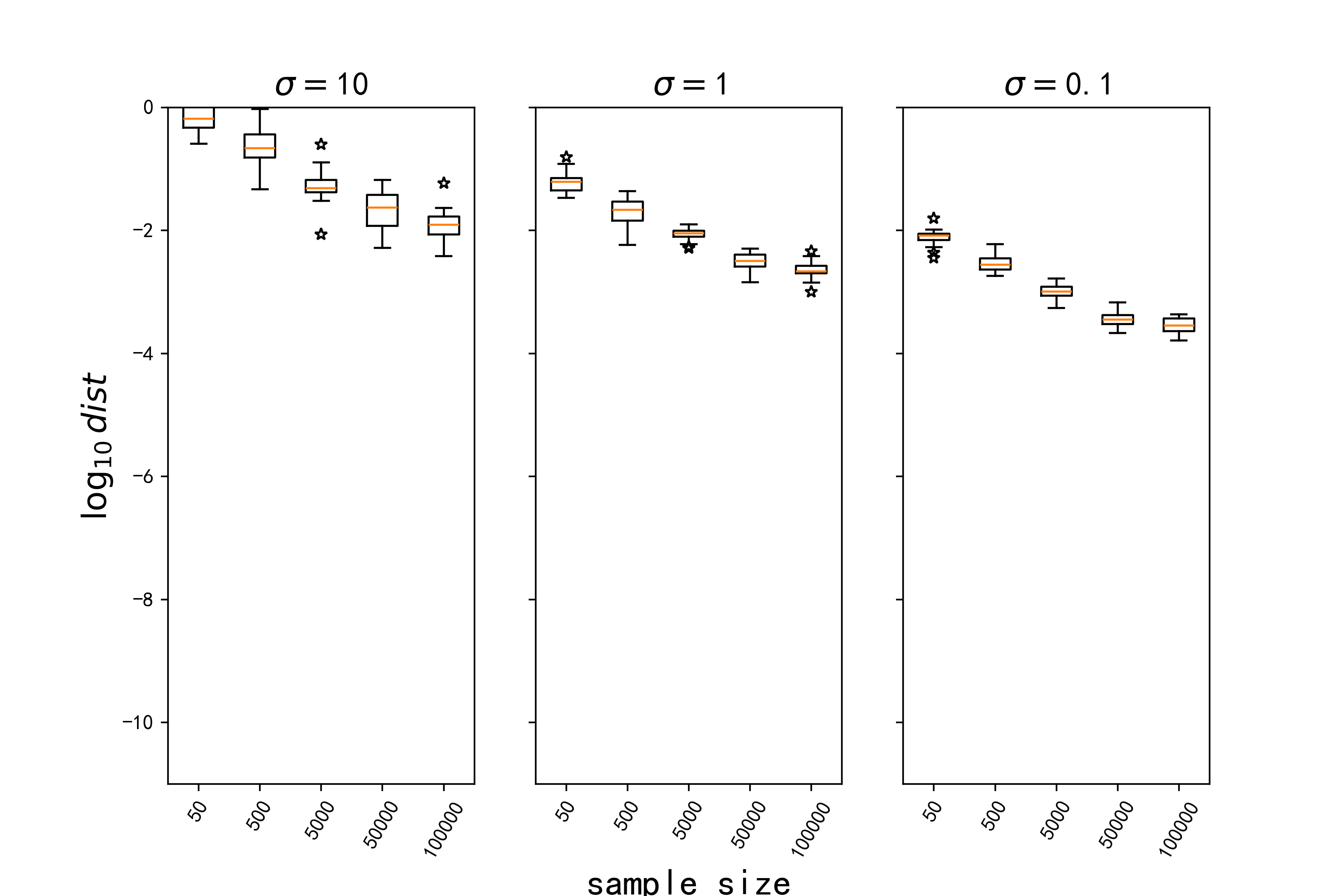}
    \caption{\scriptsize{Performance on S373 w.r.t after 1 500 iterations.}}
  \end{minipage}
\end{figure*}

\begin{figure*}[!h]
  \begin{minipage}[t]{0.5\linewidth}
    \centering
    \includegraphics[scale=0.35]{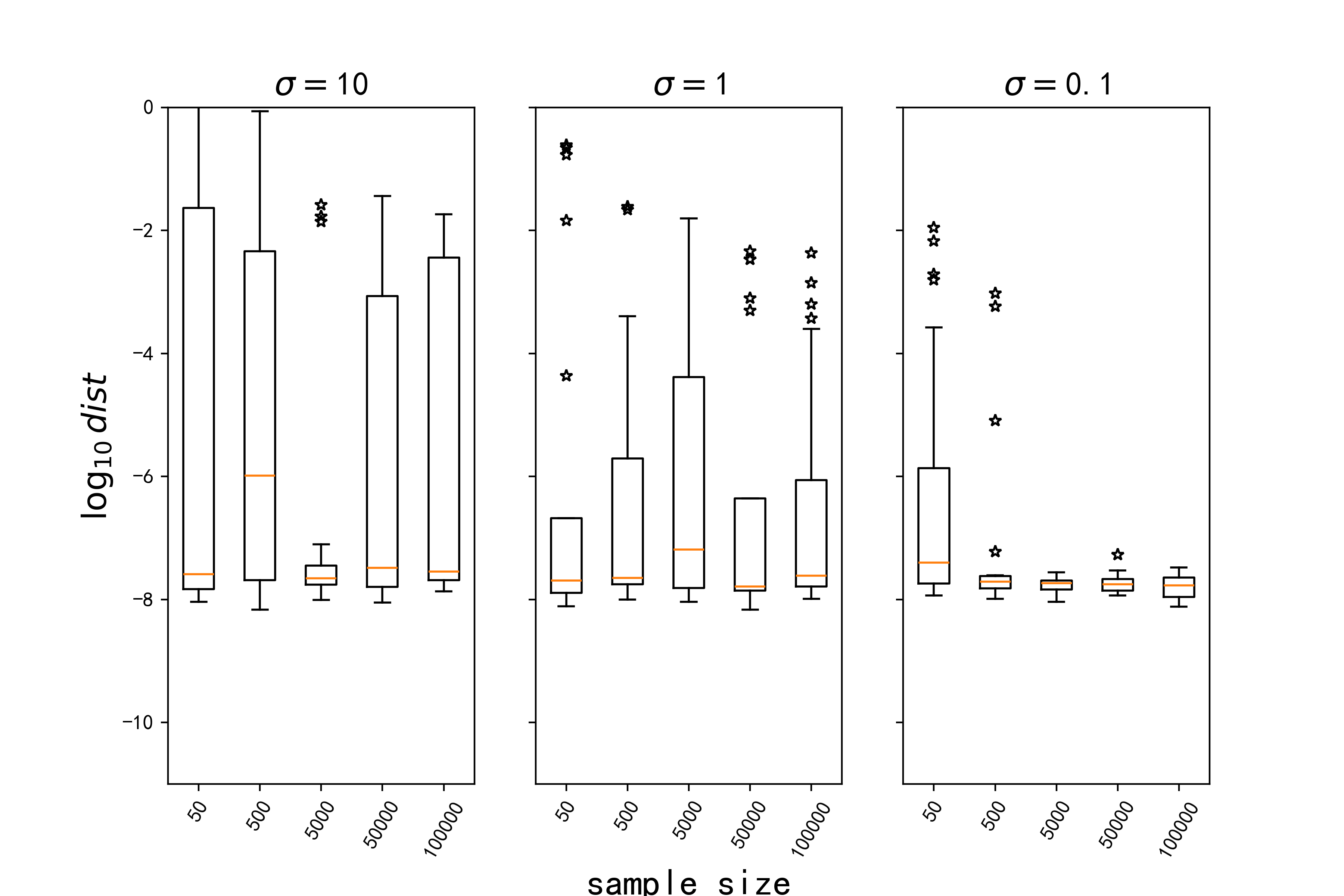}
    \caption{\scriptsize{Performance on S375 w.r.t after 50 iterations.}}
  \end{minipage}%
  \begin{minipage}[t]{0.5\linewidth}
    \centering
    \includegraphics[scale=0.35]{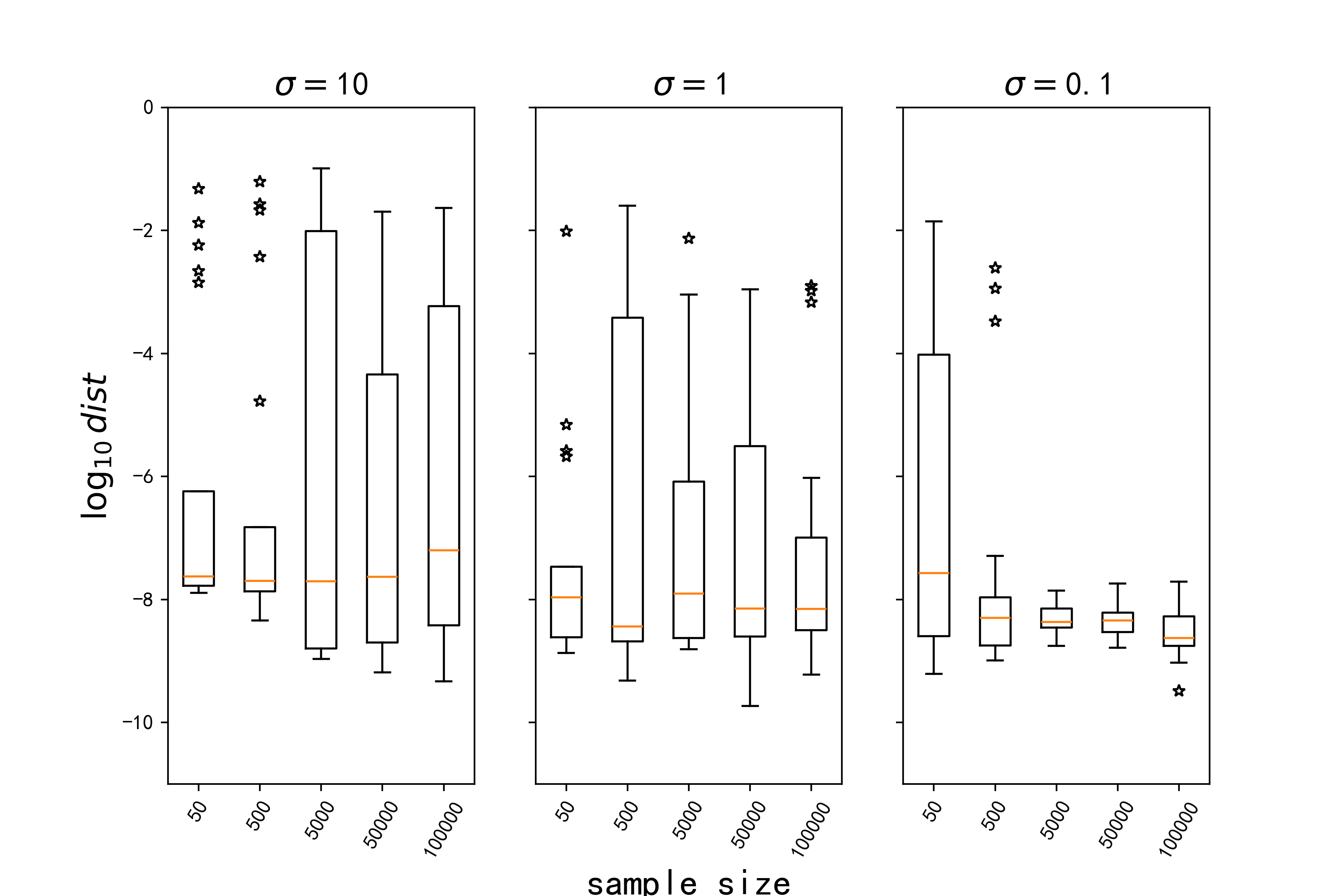}
    \caption{\scriptsize{Performance on S375 w.r.t after 1 500 iterations.}}
  \end{minipage}
\end{figure*}
\clearpage

\begin{figure*}[!h]
  \begin{minipage}[t]{0.5\linewidth}
    \centering
    \includegraphics[scale=0.35]{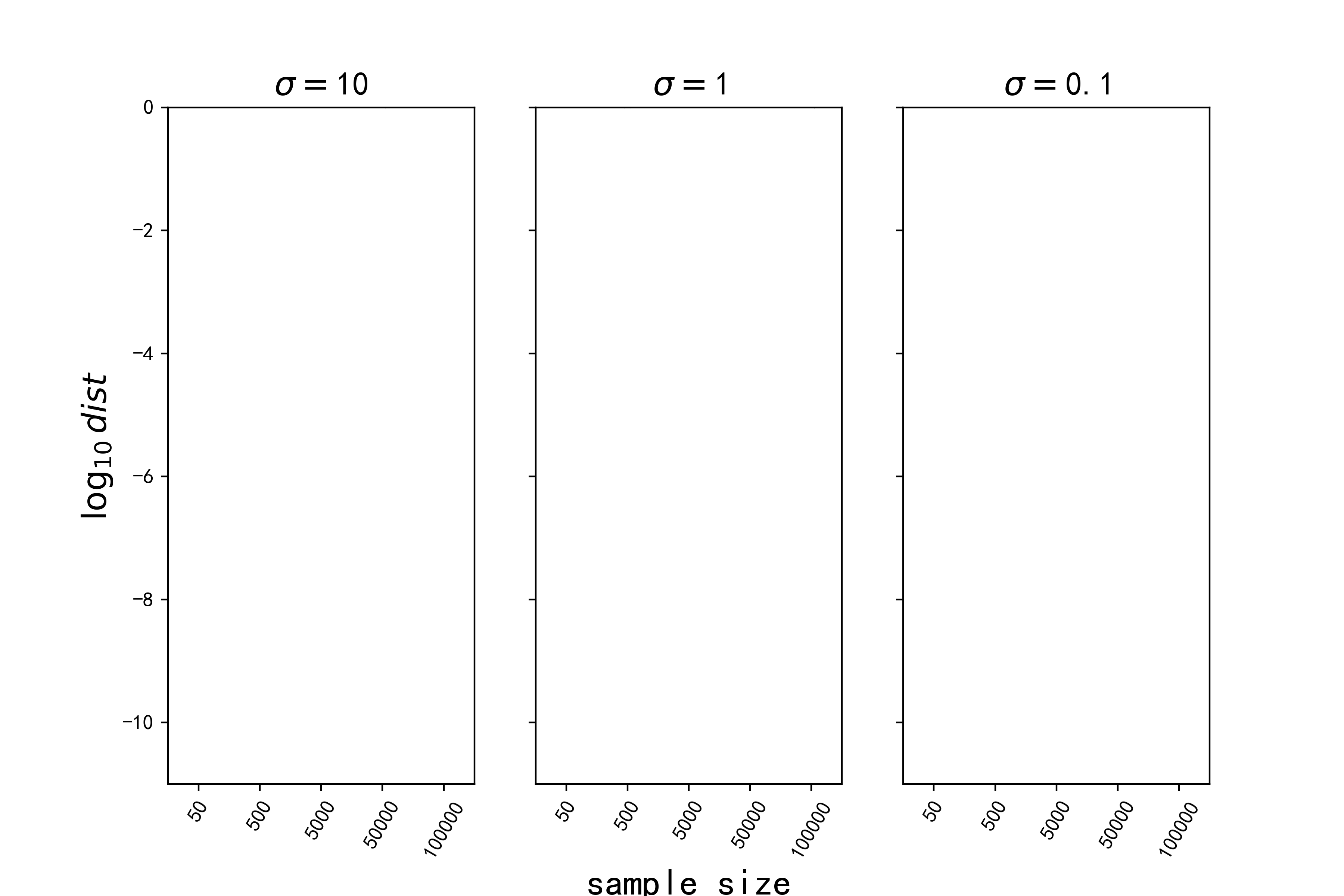}
    \caption{\scriptsize{Performance on S394 w.r.t after 50 iterations.}}
  \end{minipage}%
  \begin{minipage}[t]{0.5\linewidth}
    \centering
    \includegraphics[scale=0.35]{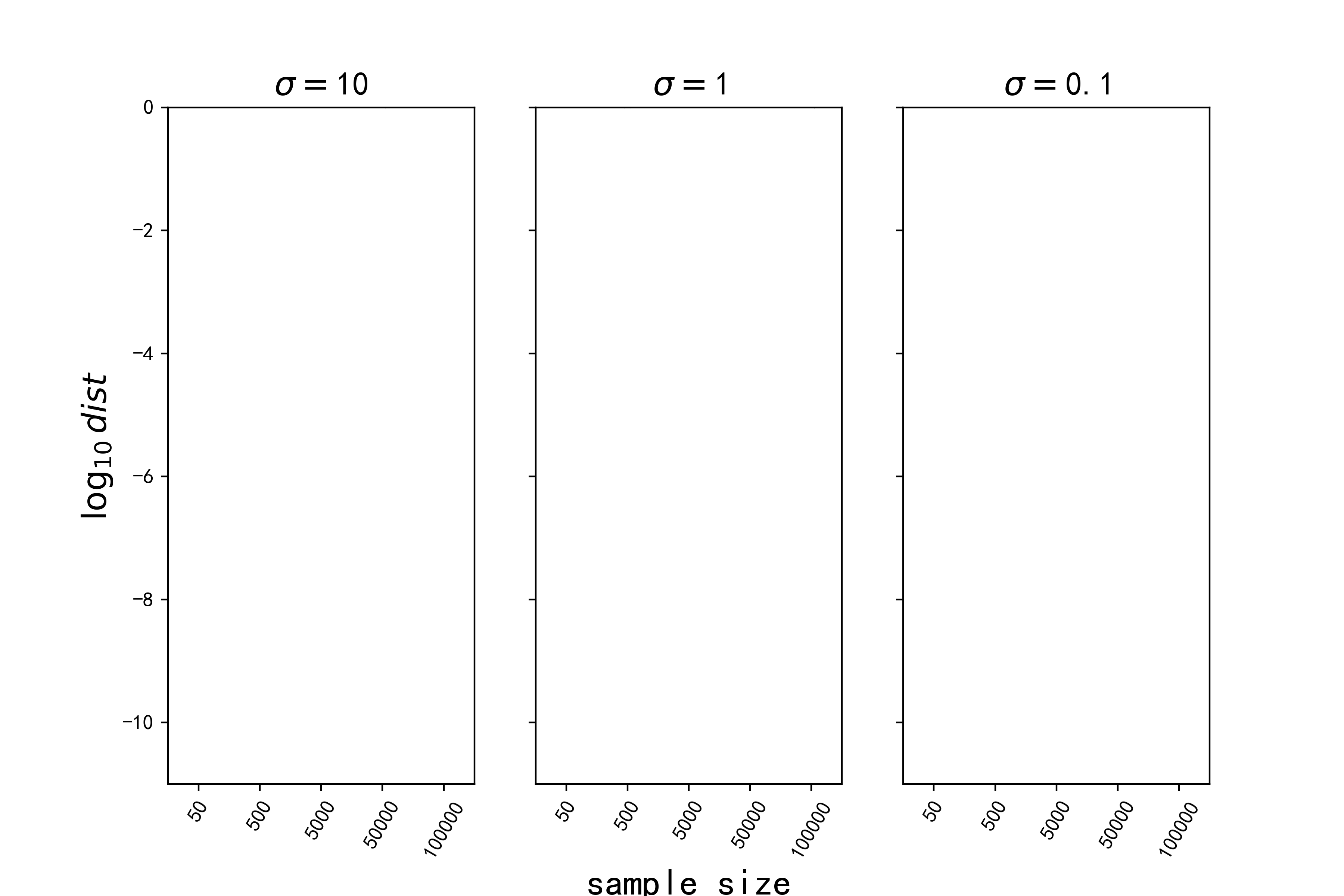}
    \caption{\scriptsize{Performance on S394 w.r.t after 1 500 iterations.}}
  \end{minipage}
\end{figure*}

\begin{figure*}[!h]
  \begin{minipage}[t]{0.5\linewidth}
    \centering
    \includegraphics[scale=0.35]{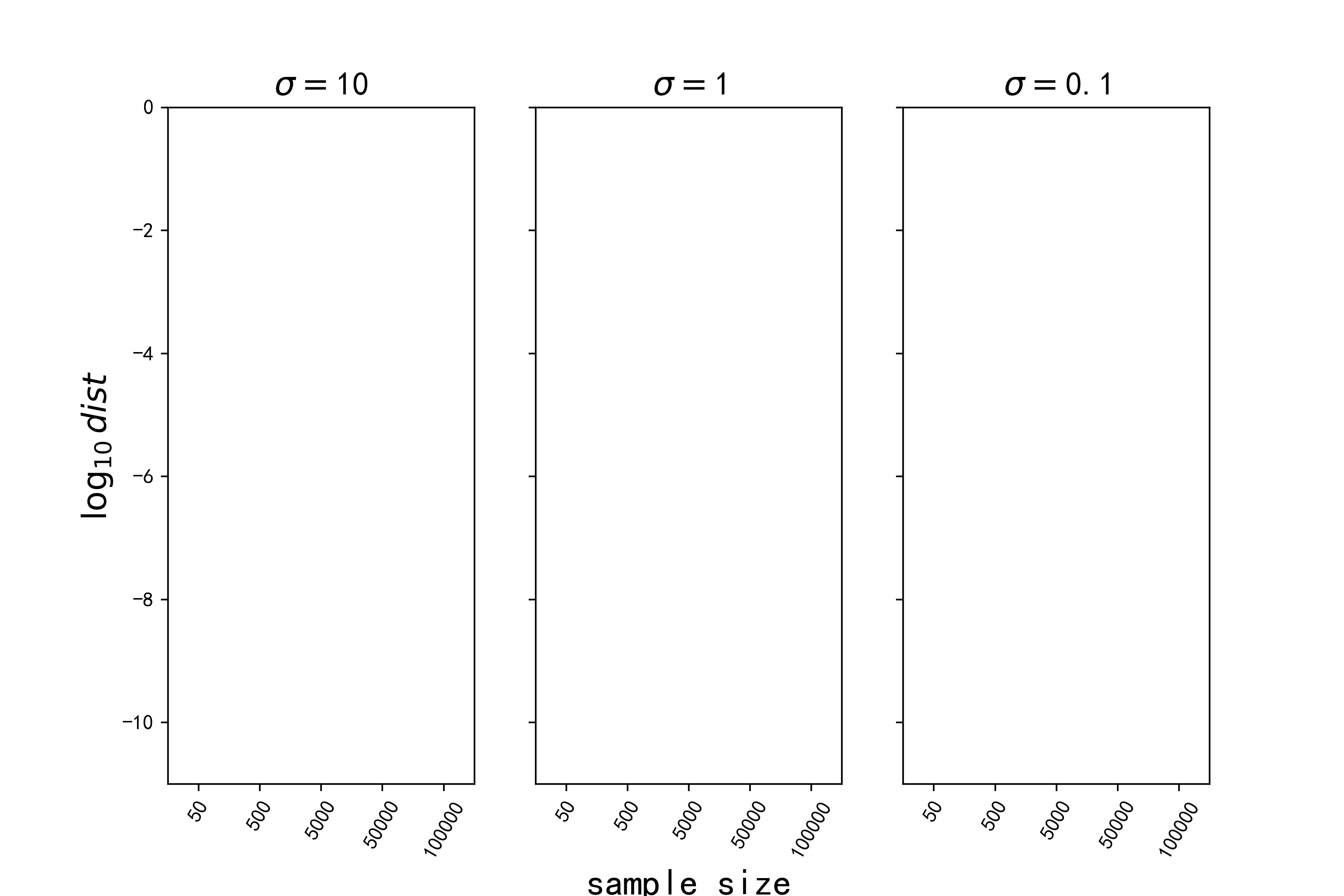}
    \caption{\scriptsize{Performance on S395 w.r.t after 50 iterations.}}
  \end{minipage}%
  \begin{minipage}[t]{0.5\linewidth}
    \centering
    \includegraphics[scale=0.35]{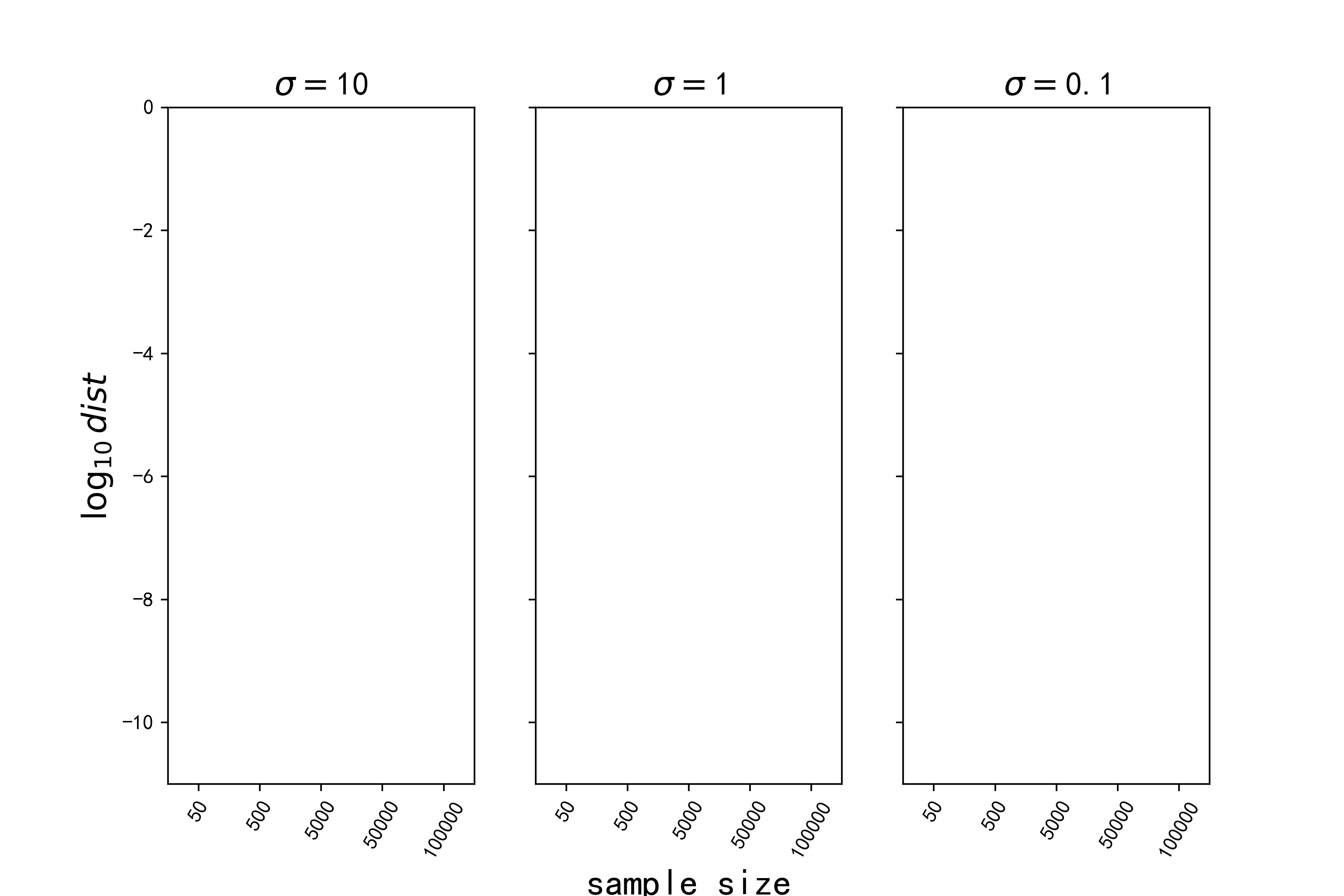}
    \caption{\scriptsize{Performance on S395 w.r.t after 1 500 iterations.}}
  \end{minipage}
\end{figure*}
\clearpage

\end{document}